\documentclass[final,onefignum,onetabnum]{siamart190516}

\usepackage{amsfonts}
\usepackage[lofdepth,lotdepth]{subfig}
\usepackage{thmtools, thm-restate}
\usepackage[margin=1in]{geometry}
\usepackage{tikz}
\usetikzlibrary{arrows,positioning}
\usepackage{wrapfig}
\usepackage{hyperref}
\usepackage{enumitem}

\newtheorem{innercustomlemma}{Lemma}
\newenvironment{customlemma}[1]
  {\renewcommand\theinnercustomlemma{#1}\innercustomlemma}
  {\endinnercustomlemma}
  
\newtheorem{innercustomthm}{Theorem}
\newenvironment{customthm}[1]
  {\renewcommand\theinnercustomthm{#1}\innercustomthm}
  {\endinnercustomthm}

\definecolor{aliceblue}{rgb}{0.94, 0.97, 1.0}
\definecolor{applegreen}{rgb}{0.55, 0.71, 0.0}

\usepackage{lipsum}
\usepackage{amsfonts}
\usepackage{graphicx}
\usepackage{epstopdf}
\usepackage{algorithmic}
\ifpdf
  \DeclareGraphicsExtensions{.eps,.pdf,.png,.jpg}
\else
  \DeclareGraphicsExtensions{.eps}
\fi

\newsiamremark{remark}{Remark}
\newsiamremark{hypothesis}{Hypothesis}
\crefname{hypothesis}{Hypothesis}{Hypotheses}
\newsiamthm{claim}{Claim}

\headers{Neural Parametric Fokker-Planck Equation}{Shu Liu, Wuchen Li, Hongyuan Zha, Haomin Zhou}

\title{Neural Parametric Fokker-Planck Equation}

\author{Shu liu\thanks{School of Mathematics, Georgia Institute of Technology, Atlanta, GA (\email{sliu459@gatech.edu}),}
\and Wuchen Li\thanks{Department of Mathematics, University of South Carolina, Columbia, SC 29208 USA (\email{wuchen@
mailbox.sc.edu}),}
\and Hongyuan Zha \thanks{School of Data Science, Shenzhen Research Institute of Big Data, The Chinese University of
Hong Kong, Shenzhen, China, 518172 (\email{zhahy@cuhk.edu.cn}),}
\and Haomin Zhou \thanks{School of Mathematics, Georgia Institute of Technology, Atlanta, GA (\email{hmzhou@math.gatech.edu}).}}

\usepackage{amsopn}

\ifpdf
\hypersetup{
  pdftitle={Neural parametric Fokker--Planck equation},
  pdfauthor={Liu, Li, Zha, Zhou}
}
\fi

\begin{document}

\maketitle

\begin{abstract}
In this paper, we develop and analyze numerical methods for high dimensional Fokker--Planck equations by leveraging generative models from deep learning. Our starting point is a formulation of the Fokker--Planck equation as a system of ordinary differential equations (ODEs) on finite-dimensional parameter space with the parameters inherited from generative models such as normalizing flows. We call such ODEs {\it neural parametric Fokker--Planck equations}. The fact that the Fokker--Planck equation can be viewed as the $L^2$-Wasserstein gradient flow of Kullback-Leibler (KL) divergence allows us to derive the ODEs as the constrained $L^2$-Wasserstein gradient flow of KL divergence on the set of probability densities generated by neural networks. For numerical computation, we design a variational semi-implicit scheme for the time discretization of the proposed ODE. Such an algorithm is sampling-based, which can readily handle the Fokker--Planck equations in higher dimensional spaces. Moreover, we also establish bounds for the asymptotic convergence analysis of the neural parametric Fokker--Planck equation as well as the error analysis for both the continuous and discrete versions. Several numerical examples are provided to illustrate the performance of the proposed algorithms and analysis.
\end{abstract}


\section{Introduction}
The Fokker--Planck equation is a parabolic partial differential equation (PDE) that plays a crucial role in stochastic calculus, statistical physics, biology and many other disciplines \cite{Nelson2,QiMajda2017lowdimensional,Risken1989fokkerplanck}. Recently, it has seen many applications in machine learning as well \cite{LiuWang2016stein,PavonTabakTrigila2018datadriven,sirignano2018mean}. The Fokker--Planck equation describes the evolution of probability density of a stochastic differential equation (SDE). In this paper, we mainly focus on the following linear Fokker--Planck equation
\begin{equation}
  \begin{split}
  \frac{\partial\rho(t,x)}{\partial t} =& \nabla\cdot(\rho(t,x)\nabla V(x)) + {D} \Delta \rho(t,x), \quad \rho(0,x) = p(x),
  \end{split}
  \label{introduction fpe}
\end{equation}
where $x\in \mathbb{R}^d$, $V\colon \mathbb{R}^d\rightarrow \mathbb{R}$ is a given potential function, ${D}>0$ is a diffusion coefficient, and $p(x)$ is the initial (or reference) density function.
In numerical algorithms, there exist several classical methods \cite{pichler2013numerical} such as finite difference \cite{chang1970practical} or finite element \cite{kumar2006solution} for solving the Fokker Planck equation. Most of the existing methods are grid based, which may be able to approximate the solution accurately if the grid sizes become small. However, they find limited usage in high dimensional problems, especially for $d > 3$, because the number of unknowns grows exponentially fast as the dimension increases. This is known as the curse of dimensionality. The main goal of this 
paper is providing an alternative strategy, with provable error estimates, to solve high dimensional Fokker--Planck equations.  

\subsection{Neural parametric Fokker--Planck equation}
To overcome the challenges imposed by high dimensionality, we leverage the generative models in machine learning \cite{rezende2015variational} and a new interpretation of the Fokker--Planck equation in the theory of optimal transport \cite{villani2008optimal}. We first introduce the KL divergence, also known as relative entropy, defined by
\begin{equation*}
\mathcal{D}_{\textrm{KL}}(\rho||\rho_*) = \int_{\mathbb{R}^d} \rho(x) \log\left( \frac{\rho(x)}{\rho_*(x)} \right) dx \quad \rho_*(x)=\frac{1}{Z_{D}}e^{-\frac{V(x)}{{D}}},~\textrm{with}~ Z_{D}=\int_{\mathbb{R}^d} e^{-\frac{V(x)}{{D}}}~dx.
\end{equation*}
Here $\rho_*(x)$ is the Gibbs distribution. A well-known fact is that the Fokker--Planck equation \eqref{introduction fpe} can be viewed as the gradient flow of the functional ${D}~ \mathcal{D}_{\textrm{KL}}(\rho||\rho_*) $ 
on the probability space $\mathcal{P}$ equipped with Wasserstein metric $g^W$ \cite{JKO,otto2001}.
Recently, this line of research has been extended to parameter space in the field of information geometry \cite{NG,IG,IG2}, leading to an emergent area called transport information geometry \cite{Li2018geometrya, lin2019wasserstein, LiM,LiMontufar2018ricci}.

Inspired by aforementioned work, we study the Fokker--Planck equation defined on parameter manifold (space) $\Theta\subset\mathbb{R}^m$ equipped with metric tensor $G$ which is obtained by pulling back the Wasserstein metric $g^W$ to $\Theta$. Here the metric tensor $G$ can be viewed as an $m\times m$ matrix that contains all the metric information on $\Theta$. In this paper,  we focus on the parameter space from generative models using neural networks. Our line of thoughts can be summarized as following. We start with a given reference distribution $p$, and consider a suitable family of parametric maps $\{T_\theta\}_{\theta \in \Theta}$. Such $T_\theta:\mathbb{R}^d\rightarrow \mathbb{R}^d$ is also called parametric pushforward map since it generates a family of parametric distributions $\{T_{\theta\sharp}p\}$ by pushing forward $p$ using $T_\theta$ (see Definition \ref{def pushforward map}). Then we consider the map $T_{(\cdot)\sharp}:\Theta\rightarrow \mathcal{P},\theta \mapsto T_{\theta\sharp}p$, which can be treated as an immersion from parameter manifold $\Theta$ to probability manifold $\mathcal{P}$. We derive the metric tensor $G(\theta)$ by pulling back the Wasserstein metric via $T_{(\cdot)\sharp}$. Once establishing $(\Theta, G)$, we can compute the $G$-gradient flow of function $H(\theta) = {D}~\mathcal{D}_{\textrm{KL}}(T_{\theta\sharp}p~||~\rho_*)$
defined on the parameter manifold. This leads to an ODE system that can be viewed as a parametric version of Fokker--Planck equation:
\begin{equation}
  \dot\theta_t = - G(\theta_t)^{-1}\nabla_\theta H(\theta_t). \label{introduction paraFPE}
\end{equation}
Here (and for the rest of the paper) dot symbol $\dot\theta$ stands for time derivative $\frac{d\theta_t}{dt}$. Using the pushforward  $\rho_\theta={T_\theta}_{\sharp}p$, in which $\theta$ is the solution of \eqref{introduction paraFPE},  we can approximate the solution $\rho_t$ in \eqref{introduction fpe}.

There are many potential applications for the parameteric Fokker Planck equation. For example, the solution of \eqref{introduction paraFPE} can be immediately used for sampling, which is a crucial task in statistics and machine learning. To be more precise, if the goal is drawing a large number of samples from $\rho_t$ at $N$ different time instances $\{t_1, t_2, ..., t_N\}$ along the solution of \eqref{introduction fpe}, we can acquire $N$ sets of parameters $\theta_{t_1},...,\theta_{t_N}$ from the solution of \eqref{introduction paraFPE}, which provide $N$ pushforward maps $T_{\theta_{t_1}},...,T_{\theta_{t_N}}$. Thus the desired samples at time $t_k$ are $\{T_{\theta_{t_k}}(\boldsymbol{Z}_1), ..., T_{\theta_{t_k}}(\boldsymbol{Z}_M)\}$, in which $\{\boldsymbol{Z}_1,...,\boldsymbol{Z}_M\}$ are samples drawn from the reference distribution $p$. If needed, the pushforward maps can be conveniently reused to generate more samples with negligible additional cost.

\subsection{Computational method}
For the computation of \eqref{introduction paraFPE}, we want to point out that metric tensor $G(\theta)$ doesn't have an explicit form and thus the direct computation of $G(\theta)^{-1}\nabla_\theta H(\theta ) $ is not tractable. To deal with this issue, we design a numerical algorithm based on the semi-implicit Euler scheme of \eqref{introduction paraFPE} with time step size $h$. To be more precise, at each time step, the algorithm seeks to solve the following double-minimization problem:
\begin{align}
\label{introduction scheme}
\begin{split}
  & \min_\theta \left\{\left(\int \left(2 ~\nabla\phi(x)\cdot((T_\theta - T_{\theta_k})\circ T_{\theta_k}^{-1}(x))-|\nabla\phi(x)|^2 \right)\rho_{\theta_k}(x)~dx\right) + 2hH(\theta)\right\} \\ 
  & \textrm{with} ~ \phi ~\textrm{solves:}~~\min_\phi\left\{\int |\nabla\phi(x)-((T_\theta-T_{\theta_k})\circ T_{\theta_k}^{-1}(x))|^2\rho_{\theta_k}(x)~dx\right\}.
\end{split}
\end{align}
Here $\rho_{\theta_k}$ is the density of the pushforwarded distribution $T_{\theta_k\sharp}p$ (cf. Definition \ref{def pushforward map}). And $\phi\colon \mathbb{R}^d\rightarrow \mathbb{R}$ is the Kantorovich dual potential variable for constrained probability models in optimal transport theory. Hence \eqref{introduction scheme} is derived following the semi-implicit Euler scheme in the dual variable. The advantage of using this formulation is that it allows us to design an efficient implementation, purely based on sampling techniques which are computational friendly in high dimensional problems, to compute the solution of the parameteric Fokker--Planck equation \eqref{introduction paraFPE}.
In our implementation,
we endow the pushforward map $T_\theta$ with certain kinds of deep neural network known as Normalizing Flow \cite{rezende2015variational},
because it is friendly to our scheme evaluations. The dual variable $\phi$ in the inner maximization is parametrized by the deep Rectified Linear Unit (ReLU) networks \cite{petersen2018optimal}. Once the network structures for $T_\theta$ and $\phi$ are chosen, the optimizations are carried out by stochastic gradient descent method \cite{ruder2016overview}, in which all terms involved can be computed using samples from the reference distribution $p$. We stress that this is critical in scaling up the computation in high dimensions.
It is worth mentioning that we use neural network as a computational tool without any actual data. Such ``data-poor" computation is in significant contrast to the mainstream of deep learning research. 

\subsection{Major innovations of the proposed method}
There are two main innovative points regarding our proposed method:
\begin{itemize}
    \item (Dimension reduction) Reducing the high dimensional evolution PDE to a finite dimensional ODE system on parameter space. Equivalently, we use the dynamics in a finite dimensional  to approximate the density evolution of particles that follow the Vlasov-type SDE 
    \begin{equation*}
       \dot{\boldsymbol{X}_t}=-\nabla V(\boldsymbol{X}_t)-{D}\nabla\log\rho_t(\boldsymbol{X}_t), \quad \rho_t ~ \textrm{is the density function of distribution of }~\boldsymbol{X}_t.
    \end{equation*}
    Here $D$ is the diffusion coefficient as mentioned in \eqref{introduction fpe}. The density function $\rho_t$ corresponds to the Fokker--Planck equation \eqref{introduction fpe}.
    \item (Sampling-friendly) We distill the information of $\rho_t$ into parameters $\{\theta_t\}$ by solving the parametric Fokker--Planck equation \eqref{introduction paraFPE}.
    By doing so, we are able to obtain an efficient sampling technique to generate samples from $\rho_t$ for any time step $t$. To be more precise, once we have applied our algorithm to solve \eqref{introduction paraFPE} for the time-dependent parameters $\{\theta_t\}$, we can then generate samples from $\rho_t$ by pushing forward the samples drawn from a reference distribution $p$ using the pushforward map $T_{\theta_t}$ with very little computational cost. Such ``implementing once for free future uses" mechanism is one of the significant advantages of our proposed algorithm. It is worth mentioning that in the view of both theoretical derivation and numerical implementation, our method is very different from Langevin Monte Carlo (LMC, MALA) methods \cite{grenander1994representations, roberts1996exponential}, which aims at targeting the stationary distribution of the SDE associated to \eqref{introduction fpe}; or moment methods \cite{QiMajda2017lowdimensional}
    , which focuses on keeping track of certain statistical information of the density $\rho_t$.
    
\end{itemize}

\subsection{Sketch of numerical analysis}
In addition to the method proposed for solving \eqref{introduction fpe}, we also conducted a mathematical analysis on \eqref{introduction paraFPE} and our algorithm. We established asymptotic convergence and error estimates for the parametric Fokker--Planck equation \eqref{introduction paraFPE}, which are summarized in the following two theorems:
\begin{customthm}{5.1}[Asymptotic convergence]\label{introduction convergence}
Consider the Fokker--Planck equation \eqref{introduction fpe} with potential $V$ and diffusion coefficient $D$. Suppose $V$ can be decomposed as $V=U+\phi$ with $U\in \mathcal{C}^2(\mathbb{R}^d)$, $\nabla^2 U\succeq K I\footnote{The matrix $\nabla^2 U(x)-K I_{d\times d}$ is non-negative definite for any $x\in\mathbb{R}^d$.}$ with $K>0$ and $\phi\in L^\infty(\mathbb{R}^d)$, and $\{\theta_t\}$ solves \eqref{introduction paraFPE}.  Then the following inequality holds,
\begin{equation*}
 \mathcal{D}_{\textrm{KL}}(\rho_{\theta_t}\|\rho_*)\leq \frac{\delta_0}{\tilde{\lambda}_{D} {D}^2}(1-e^{-{D}\tilde{\lambda}_{D} t}) +  \mathcal{D}_{\textrm{KL}}(\rho_{\theta_0}\|\rho_*) e^{-{D}\tilde{\lambda}_{D} t},
\end{equation*}
where $\rho_*$ is the Gibbs distribution, $\tilde{\lambda}_{D}>0$ is a constant related to the potential function $V$ and ${D}$. $\delta_0$ is a constant depending on the approximation power of pushforward map $T_\theta$.
\end{customthm}

\begin{customthm}{5.11}[Approximation error]\label{introduction theorem 1}
Consider the Fokker--Planck equation \eqref{introduction fpe} with potential $V$, diffusion coefficient $D$ and initial density $\rho_0$. Assume that $\lambda$ is a lower bound of Hessian of potential $V$, i.e. $\nabla^2 V\succeq \lambda I$, $\delta_0$ is defined in Theorem \ref{introduction convergence},  $E_0=W_2(\rho_{\theta_0},\rho_0)$, and  $\delta_0,E_0\ll 1$, then the following uniform bounds for the $L^2$-Wasserstein error $W_2(\rho_{\theta_t}, \rho_t)$ hold:
\begin{itemize}
 \item When $\lambda>0$, $W_2(\rho_{\theta_t},\rho_t)\leq \max\{\sqrt{\delta_0}/\lambda, E_0\}\sim O(\sqrt{\delta_0}+E_0)$, 
 \item When $\lambda = 0$, $W_2(\rho_{\theta_t}, \rho_t) \leq \frac{\sqrt{\delta_0}}{\mu_{D} }\log\frac{B}{\sqrt{\delta_0}+ E_0}+E_0\sim O(\sqrt{\delta_0}\log \frac{1}{\sqrt{\delta_0}+E_0}+E_0)$,
 \item When $\lambda<0$,  $W_2(\rho_{\theta_t},\rho_t) \leq  A\sqrt{\delta_0} + C\left( E_0+\sqrt{\delta_0}/|\lambda| \right)^\alpha\sim O((E_0 +\sqrt{\delta_0})^\alpha)$.
\end{itemize}
Here $\delta_0$ is a constant depending on the approximation power of pushforward map $T_\theta$. $\mu_{D},A,B,C>0$ are constants only depending on $V,{D},\rho_0,\theta_0$. And $\alpha=\frac{\mu_{D}}{|\lambda|+\mu_{D}}$ is a certain exponent between $0$ and $1$.
\end{customthm}

\noindent
This result reveals that the difference between the solutions of the parametric Fokker--Planck equation \eqref{introduction paraFPE} and the original equation \eqref{introduction fpe}, measured by their Wasserstein distance $W_2(\rho_{\theta_t},\rho_t)$, has a {\it uniformly} small upper bound if both the initial error $E_0$ and ${\delta_0}$ are small enough. 
Most of the techniques used in our analysis for establishing such a result rely on the theory of optimal transport and Wasserstein manifold, which are still not commonly used for numerical analysis in relevant literature.
Besides error analysis for the continuous version of \eqref{introduction paraFPE}, we are able to provide the order of $W_2$-error for the numerical scheme when \eqref{introduction paraFPE} is computed at discrete time by numerical schemes. To be more precise, if we apply forward-Euler scheme to \eqref{introduction paraFPE} and compute $\{\theta_k\}$ at different time nodes $\{t_k\}$, we can show that error at $t_k$: $W_2(\rho_{\theta_k},\rho_{t_k})$ is of order $O(\sqrt{\delta_0})+O(C h)+O(E_0)$ for finite time $t$. This is summarized in the following theorem:
\begin{customthm}{5.14}[Error for discrete scheme]\label{introduction theorem 2}
Assume that $\{\rho_t\}_{t\geq 0}$ is the solution of \eqref{introduction fpe} with potential satisfying $\lambda I \preceq \nabla^2 V\preceq\Lambda I$, $\{\theta_k\}_{k=0}^N$ is the numerical solution of \eqref{introduction paraFPE} at time nodes $t_k = kh$ for $k=0,1,...,N$ computed by forward Euler scheme with time step $h$. Recall $\delta_0$ as mentioned in Theorem \ref{introduction convergence} and we denote $E_0=W_2(\rho_{\theta_0}, \rho_0)$, then we have:
\begin{equation*}
W_2 (\rho_{\theta_k}, \rho_{ t_k  })\leq (\sqrt{\delta_0} h + C h^2)\frac{(1-e^{-\lambda t_k})}{1-e^{-\lambda h }}+e^{-\lambda t_k }E_0\sim O(\sqrt{\delta_0})+O(Ch)+O(E_0),\quad 0\leq k\leq N,
\end{equation*}
where $C$ is a constant depending on $N$ and $h$. 
\end{customthm}
This indicates that the $W_2$-error is dominated by three different terms: $O(\sqrt{\delta_0})$ is the intrinsic error originated from the approximation mechanism of the parametric Fokker--Planck equation; $O(Ch)$ term is induced by the time discretization; and $O(E_0)$ term is the initial error. We further prove that the difference between the forward Euler scheme and our semi-implicit Euler scheme is of order $O(h^2)$, which implies that the proposed semi-implicit Euler scheme can achieve a similar error bounds as the one presented in Theorem \ref{introduction theorem 2}.

It is worth mentioning that we establish Theorem \ref{introduction theorem 2} based on totally different techniques than those used for Theorem \ref{introduction theorem 1}. Since the ODE \eqref{introduction paraFPE} contains the term $G(\theta)^{-1}$, which is hard to handle by the traditional strategies, we interpret it as a particle system governed by a stochastic differential equations (SDEs) of Vlasov type, and obtain the analysis results shown in  Theorem \ref{introduction theorem 2}. 

\subsection{Literature review}

Numerous works exist for solving the Fokker--Planck equations. A finite difference scheme is proposed in \cite{chang1970practical} so that it preserves  the equilibrium of the original equation. A more general class of equations possessing Wasserstein gradient flow structures is solved in \cite{carrillo2019primal}. in which the method is based on a space discretization of a proximal-typed scheme (also known as JKO method \cite{jordan1998variational}). Besides direct solutions, particle simulation techniques also serve as an efficient way of solving the equation. The so-called ``Blob" method is proposed in \cite{carrillo2019blob} and solves the equations by evolving a certain interacting particle systems. Related swarming system is also studied in \cite{leverentz2009asymptotic, carrillo2011global, klar2014multiscale, jabin2014review, carrillo2014derivation}. In \cite{maoutsa2020interacting}, the authors propose another type of interacting systems in order to approximate $\nabla\log\rho$, which plays the role of the diffusion term in the Fokker--Planck equation, with higher accuracy and less fluctuation. In \cite{pathiraja2019discrete, reich2020fokkerplanck}, the authors mainly focus on exploiting the gradient flow structure, i.e. a particle discretization of the Fokker–Planck equation, to deal with Bayesian inference problems.

In addition to the literature focusing on solving the Fokker--Planck equations, There are existing works on applying neural networks to solve PDE of various types in high dimensional spaces \cite{weinan2017deep, raissi2019physics, khoo2017solving, khoo2019solving, zang2019weak, nusken2020solving}. Among the listed works, algorithms for general types of high dimensional PDEs are provided in \cite{raissi2019physics, khoo2017solving}; a sampling friendly method is proposed in \cite{nusken2020solving} to deal with the general optimal control problem of diffusion processes. This is equivalent to solving an associated Hamilton-Jacobi-Bellman equation and such technique can also be applied to importance sampling and rare event simulation. Moreover, numerical methods for high dimensional parabolic PDEs, to which the Fokker--Planck equation belongs, are studied in \cite{weinan2017deep} and \cite{khoo2019solving}. Our approach differs from these existing works in many aspects, including motivations, strategies, and the associated numerical analysis.

For example, in \cite{weinan2017deep}, the authors propose to use the non-linear Feynmann-Kac formula to re-write certain parabolic PDEs as the Backward Stochastic Differential Equation (BSDE), which is then reformulated as a stochastic control problem (also known as reinforcement learning in machine learning community). By applying deep neural network as the control function and optimizing over network parameters, the solution at any given space-time location can be evaluated. Another example is \cite{khoo2019solving}, which mainly focuses on computing the committor function that solves a steady-state (time-independent) Fokker--Planck equation with specific boundary conditions. This committor function can be treated as the solution to a variational problem associated with an energy functional. A neural network is used to replace the solution in the variational problem. When optimizing over network parameters, the neural network can be used to approximate the committor function.

In this paper, we focus on designing a sampling-friendly method for the time dependent Fokker--Planck equation. There are two main reasons that motivate us for this investigation. One, as mentioned before, is to design sample based algorithm to solve PDEs in high dimensions. The other is to provide an alternative sampling strategy that can be potentially faster than LMC.  
Our approaches are different in terms of how deep networks are leveraged to approximate the solution of the PDE. We use pushforward of a given reference measure by neural networks to create a generative model. This is to approximate the stream of probability distributions, which can be used to generate samples not only at the terminal time, but also any time in between.  
More importantly, we prove results, obtained by using newly developed techniques based on Wasserstein metric on probability manifold, on the asymptotic convergence and error control of our numerical schemes. To the best of our knowledge, similar results are still lacking in existing studies. 

\subsection{Organization of this paper}
We organize the paper as follows. In section \ref{section2}, we briefly introduce some background knowledge of the Fokker--Planck equation, including its relation with SDE and its Wasserstein gradient flow structure. In section \ref{Para_FPE}, we introduce the Wasserstein statistical manifold $(\Theta, G)$ and derive our parametric Fokker--Planck equation as the manifold gradient flow of relative entropy on $\Theta$. We study the geometric property of this equation, including an insightful particle motion based interpretation of the parametric Fokker--Planck equation. In section \ref{section4}, we design a numerical scheme that is tractable for computing our parametric Fokker--Planck equation using deep learning framework. Some important details of implementation will be discussed. We present asymptotic convergence  and error estimates for the parametric Fokker--Planck equation in section \ref{Analysis}, and provide some numerical examples in section \ref{section 6}.

\section{Background on the Fokker--Planck equation}\label{section2}
In this section, we present two different perspectives regarding the Fokker--Planck equations, More discussion can be found in \cite{liliuzhouzha}.

\subsection{As the density evolution of stochastic differential equation}\label{background}
The general form of the Fokker--Planck equation is \cite{pavliotis2014stochastic,lelievre_stoltz_2016}:
\begin{align}
    \frac{\partial \rho(x,t)}{\partial t}& =-\nabla\cdot(\rho(x,t)\boldsymbol{\mu}(x,t))+\frac{1}{2}\nabla^2:(\boldsymbol{D}(x,t) \rho(x,t)) \label{generalFPE}   \\
    &=-\sum_{i=1}^d\frac{\partial}{\partial x_i}(\rho(x,t)\mu_i(x,t))+\frac{1}{2}\sum_{i,j=1}^d \frac{\partial^2}{\partial x_i\partial x_j}(D_{ij}(x,t)\rho(x,t)), \quad \rho(x,0)=\rho_0(x). \nonumber
\end{align}
Here $\boldsymbol{\mu}=(\mu_1,...,\mu_d)^{\textrm{T}}$ is the drift function and $\boldsymbol{D}=\{D_{ij}\}$ is the $d\times d$ diffusion tensor. Furthermore, $\boldsymbol{D}$ can be written as $\boldsymbol{D}=\boldsymbol{\sigma}\boldsymbol{\sigma}^{\textrm{T}}$, where $\boldsymbol{\sigma}(x,t)$ is a $d\times \tilde{d}$ matrix.
One derivation of the Fokker--Planck equation originates from the following stochastic differential equation (SDE) \cite{pavliotis2014stochastic, lelievre_stoltz_2016},
\begin{equation*}
    d\boldsymbol{X}_t = \boldsymbol{\mu}(\boldsymbol{X}_t,t)~dt+\boldsymbol{\sigma}(\boldsymbol{X}_t, t)~d\boldsymbol{B}_t,\quad  \boldsymbol{X}_0\sim\rho_0, \label{generalSDE}
\end{equation*}
where $\{\boldsymbol{B}_t\}_{t\geq 0}$ is the standard Brownian motion in $\mathbb{R}^{\tilde{d}}$, and $\rho_0$ is the distribution of the initial state. 
It is well known that the evolution of the density $\rho(x,t)$ of the stochastic process $\{\boldsymbol{X}_t\}_{t\geq 0}$ is described by the above the Fokker--Planck equation.

In this paper, we consider a more specific type of \eqref{generalFPE} by setting $\boldsymbol{\mu}(x,t)= - \nabla V(x)$, $\boldsymbol{\sigma}(x,t)=\sqrt{2{D}}~I_{d\times d}$ (${D} > 0 $), where $I_{d\times d}$ is the $d$ by $d$ identity matrix, and so $\boldsymbol{D}=2{D}~I_{d\times d}$. Then \eqref{generalSDE} is,
\begin{equation}
    d\boldsymbol{X}_t = -\nabla V(\boldsymbol{X}_t)~dt+\sqrt{2{D}}~d\boldsymbol{B}_t\quad \boldsymbol{X}_0 \sim \rho_0. \label{Over-damped_Langevin} 
\end{equation}
This equation is also called over-damped Langevin dynamics which has broad applications in computational physics, computational biology, Bayesian statistics \cite{grenander1994representations, schlick2010molecular, welling2011bayesian}. The corresponding Fokker--Planck equation is simplified to
\begin{equation}
    \frac{\partial\rho(x,t)}{\partial t}=\nabla\cdot(\rho(x,t)\nabla V(x))+{D}\Delta\rho(x,t), \quad \rho(x,0) = \rho_0(x). \label{FPE}
\end{equation}
In addition, we would like to mention that there is a Vlasov-type SDE corresponding to the Fokker--Planck equation \eqref{FPE}:
\begin{equation}
  \frac{d \boldsymbol{X}_t}{dt} = -\nabla V(\boldsymbol{X}_t)-{D}~\nabla\log\rho(\boldsymbol{X}_t,t),\quad \boldsymbol{X}_0\sim\rho_0,   \label{Vlasov-type SDE }
\end{equation}
in which $\rho(\cdot,t)$ is the density of $\boldsymbol{X}_t$. 
This Vlasov-type SDE \eqref{Vlasov-type SDE } will be very useful in our proofs for the error 
estimates of our proposed numerical algorithms.

\subsection{As the Wasserstein gradient flow of relative entropy}\label{2.2 }
Another useful viewpoint states that \eqref{FPE} is the Wasserstein gradient flow of relative entropy. We briefly present some of the notations and basic results in this regard. We only provide in sections \ref{wass_mfld} and \ref{background wass_grad} an informal discussion on Wasserstein manifold and Wasserstein gradient flow. More rigorous treatments on the topics can be found in \cite{ambrosio2008gradient}.

\subsubsection{Wasserstein manifold}\label{wass_mfld}
Denote the probability space supported on $\mathbb{R}^d$ with densities having finite second order moments as
\begin{equation*}
    \mathcal{P}=\left\{\rho\colon\int \rho(x)dx=1,~\rho(x)\geq 0,~\int |x|^2\rho(x)~ dx<\infty\right\}.
\end{equation*}
Here the integral is computed over the sample space $\mathbb{R}^d$. In the following discussion, if not specified, we always write $\int_{\mathbb{R}^d}$ as $\int$ for simplicity.

The so-called Wasserstein distance (also known as $L^2$-Wasserstein distance) on $\mathcal{P}$ is defined as \cite{villani2008optimal}
\begin{equation}
  W_2(\rho_1,\rho_2)= \left(\inf_{\pi\in\Pi(\rho_1,\rho_2)}\iint |x-y|^2 ~d\pi(x,y)\right)^{1/2}, \label{def_wass_dist}
\end{equation}
where $\Pi(\rho_1,\rho_2)$ is the set of joint distributions defined on $\mathbb{R}^d\times \mathbb{R}^d$ with fixed marginal distributions whose densities are $\rho_1, \rho_2$.
If we treat $\mathcal{P}$ as an infinite dimensional manifold,  the Wasserstein distance $W_2$ can induce a metric $g^W$ on the tangent bundle $\mathcal{T}\mathcal{P}$, with which $\mathcal{P}$ becomes a Riemannian manifold. For simplicity, here we directly give the definition of $g^W$. 
One can identify the tangent space at $\rho$ as:
\begin{equation*}
\mathcal{T}_\rho\mathcal{P}=\left\{f \colon \int f(x)dx=0\right\}.
\end{equation*}
For a specific $\rho\in\mathcal{P}$ and $f_i\in \mathcal{T}_\rho\mathcal{P}$, $i=1,2$, we define the Wasserstein metric tensor $g^W$ as \cite{Lafferty,otto2001}
\begin{equation}
g^W(\rho)(f_1,f_2)=\int  \nabla\psi_1(x)\cdot\nabla\psi_2(x)\rho(x) ~dx,\label{def_metric}
\end{equation}
where $\psi_1,\psi_2$ satisfies 
\begin{equation}
f_i=-\nabla\cdot(\rho\nabla\psi_i) \quad i=1,2,\label{hodge}
\end{equation}
with boundary conditions
$$\lim_{x\rightarrow\infty}\rho(x)\nabla\psi_i(x)=0\quad i=1,2.$$
Use the above definition, we can also write
\begin{equation*}
    g^W(\rho)(f_1, f_2) = \int\psi_1(-\nabla\cdot(\rho\nabla\psi_2))~dx = \int (-\nabla\cdot(\rho\nabla))^{-1}(f_1)\cdot f_2~dx.
\end{equation*}
Thus, we can identify $g^W(\rho)$ as $(-\nabla\cdot(\rho\nabla))^{-1}$. When $\textrm{supp}(\rho)=\mathbb{R}^d$, $g^W(\rho)$ is a positive definite bilinear form defined on tangent bundle $\mathcal{T}\mathcal{P}=\{(\rho, f)\colon \rho\in \mathcal{P},~f\in \mathcal{T}_\rho\mathcal{P}\}$. Hence we can treat $\mathcal{P}$ as a Riemannian manifold, which we call {\it Wasserstein manifold}, denoted as $(\mathcal{P},g^W)$ \cite{otto2001}. In order to keep our notations concise, in the sequel, we denote $g^W(\rho)$ as $g^W$ if no confusion is caused.

\subsubsection{Wasserstein gradient}\label{background wass_grad}
We denote the Wasserstein gradient $\textrm{grad}_W$ as the manifold gradient on $(\mathcal{P},g^W)$.
In Riemannian geometry, the manifold gradient must be compatible with the metric, implying that for any smooth functional $\mathcal{F}$ defined on $\mathcal{P}$ and any $\rho\in\mathcal{P}$, considering an arbitrary differentiable curve $\{\rho_t\}_{t\in(-\delta,\delta)}$ with $\rho_0=\rho$, we have
\begin{equation*}
 \frac{d}{dt}\mathcal{F}(\rho_t)\Big\vert_{t=0} = g^W(\rho)(\textrm{grad}_W\mathcal{F}(\rho) , ~ \dot\rho_0).
\end{equation*}
Since we can write 
\[ \frac{d}{dt}\mathcal{F}(\rho_t)\Big\vert_{t=0}=\int \frac{\delta \mathcal{F}(\rho)}{\delta\rho(x)}(x) \cdot \dot\rho_0(x)~dx =\left\langle \frac{\delta\mathcal{F}(\rho)}{\delta\rho} , \dot\rho_0\right\rangle_{L^2},\]
here $ \frac{\delta\mathcal{F}(\rho)}{\delta\rho(x)}(x)$ is the $L^2$ variation of $\mathcal{F}$ at point $x\in\mathbb{R}^d$, we then have
\begin{equation*}
 \left\langle \frac{\delta\mathcal{F}(\rho)}{\delta\rho},\dot\rho_0\right\rangle_{L^2} = g^W(\rho)(\textrm{grad}_W\mathcal{F}(\rho) , ~ \dot\rho_0)\quad \forall~\dot\rho_0\in \mathcal{T}_\rho\mathcal{P}.
\end{equation*}
This leads to the following useful formula for computing Wasserstein gradient of functional $\mathcal{F}$
\begin{equation}
\begin{split}
\textrm{grad}_W\mathcal{F}(\rho)={g^{W}(\rho)}^{-1}\left(\frac{\delta\mathcal{F}}{\delta\rho}\right)(x)
  = -\nabla\cdot\left(\rho(x)\nabla~ \frac{\delta\mathcal{F}(\rho)}{\delta\rho(x)}(x)\right).
   \end{split}
   \label{gradflow}
\end{equation}
In particular, if $\mathcal{F}$ is taken as the relative entropy functional given by
\begin{equation}
  \mathcal{H}(\rho) = {D}~ \mathcal{D}_{\textrm{KL}}\left(\rho~ \Big|\Big|~ \rho_*\right) = \left( \int V(x)\rho(x)+{D}\rho(x)\log\rho(x)~dx \right) + {D}~\log Z_{D},
  \label{relative entropy}
\end{equation}
we have $\nabla\frac{\delta\mathcal{H}(\rho)}{\delta\rho}=\nabla V+{D}\nabla\log\rho$. 
Using \eqref{gradflow}, and noticing $\nabla\log\rho=\frac{\nabla\rho}{\rho}$, then $\nabla\cdot(\rho\nabla\log\rho)=\nabla\cdot(\nabla\rho)=\Delta\rho$, the Wasserstein gradient flow of $\mathcal{H}$ can be written as
\begin{equation*}
\frac{\partial \rho}{\partial t}=-\textrm{grad}_W\mathcal{H}(\rho)=\nabla\cdot(\rho\nabla V)+{D}\nabla\cdot(\rho\nabla\log\rho)),
\end{equation*}
which is exactly the Fokker--Planck equation \eqref{FPE}. 

\section{Parametric Fokker--Planck equation}\label{Para_FPE}
In this section, we provide detailed derivation for our parametric Fokker--Planck equation. 
\subsection{Wasserstein statistical manifold}\label{3.1}
Consider a parameter space  $\Theta$ as an open, convex set in $\mathbb{R}^m$, and assume the sample space is $\mathbb{R}^d$. Let $T_\theta$ be a map from $\mathbb{R}^d$ to $\mathbb{R}^d$  parametrized by $\theta$. In our discussion, we always assume the invertibility of $T_\theta(x)$, and it is second order differentiable with respect to $x$ and $\theta$, i.e. $T_\theta(x)\in C^2(\Theta\times \mathbb{R}^d)$.

\begin{remark}
   There are many different choices for $T_\theta$:
   \begin{itemize}
   \item We can set $T_\theta(x)=Ux+b$, with $\theta=(U,b)$, $U$ is a $d\times d$ invertible matrix, $b\in\mathbb{R}^d$; 
   \item We may also choose $T_\theta$ as the linear combination of basis functions $T_\theta(x) = \sum_{k=1}^{m} \theta_k\vec{\Phi}_k(x)$, where $\{\vec{\Phi}_k \}_{k=1}^m$ are the basis functions and the parameter $\theta$ will be the coefficients: $\theta = (\theta_1,...,\theta_m)$; 
   \item We can also treat $T_\theta$ as neural network. Its general structure can be written as the composition of $l$ affine and non-linear activation functions: $T_\theta(x)=\sigma_l(W_l(\sigma_{l-1}(...\sigma_1(W_1x+b_1)...))+b_l)$. In this case, the parameter $\theta$ will be the weight matrices and bias vectors of the neural network, i.e. $\theta = (W_1,b_1,...,W_l,b_l)$.
   
   \end{itemize}
\end{remark}
\begin{definition}\label{def pushforward map}
Suppose $X,Y$ are two measurable spaces, $\lambda$ is a probability measure defined on $X$; let $f:X\rightarrow Y$ be a measurable map. We define $f_{\sharp}\lambda$ as: $f_{\sharp}\lambda(E)=\lambda(f^{-1}(E))$ for all measurable $E\subset Y$. We call $ f_{\sharp}\lambda $ the pushforward of measure $\lambda$ by map $f$.
\end{definition}
Let $p\in \mathcal{P}$ be a reference probability measure with positive density defined on $\mathbb{R}^d$, such as the standard Gaussian. We denote $\rho_\theta$ as the density of ${T_\theta}_{\sharp}p$. Such kind of mechanism of producing parametric probability distributions is also known as \textbf{generative model}, which has broad applications in deep learning research \cite{goodfellow2014generative, arjovsky2017wasserstein, brock2018large}. We further assume our $T_\theta$ satisfy the following two conditions:
\begin{equation}
\textrm{Condition 1:} \quad \int |z|^2\rho_\theta(z)~dz = \int  |T_\theta(x)|^2~dp(x)<\infty \quad \forall ~ \theta\in\Theta.   \label{Theta_condition_A}
\end{equation}
This ensures that $\rho_\theta\in\mathcal{P}$ for each $\theta\in\Theta$. In order to introduce Wasserstein metric to the parameter space $\Theta$, we also assume that the Frobenius norm of the operator $\partial_\theta T_\theta(x):\mathbb{R}^d\rightarrow \mathbb{R}^{d\times m}$ is locally bounded in the following sense: for any fixed $\theta_*\in\Theta$, there exists $r(\theta_*)>0$ and two functions $L_1(\cdot|~\theta_*), L_2(\cdot|~\theta_*)$ satisfying 
\begin{align}
 \textrm{Condition 2:} \quad  \|\partial_\theta T_\theta(x)\|_F\leq L_1(x|~\theta_*), \|\partial_\theta T_\theta(x)\|_F^2\leq L_2(x|~\theta_*), & ~\forall ~\theta, |\theta-\theta_*|<r(\theta_*)~ \textrm{and}~ x\in \mathbb{R}^d, ~\textrm{and}  \nonumber \\
 \int L_1(x|~\theta_*)~dp(x)<\infty &\quad \int L_2(x|~\theta_*)~dp(x)<\infty. \label{Theta_condition_B}
\end{align}

We define the parametric submanifold $\mathcal{P}_\Theta\subset\mathcal{P}$ as:
$$\mathcal{P}_{\Theta}=\{\rho_\theta ~\textrm{is density function of }~ {T_\theta}_{\sharp}p~ |~\theta\in\Theta\}.$$

Clearly, the connection between $\mathcal{P}$ and $\Theta$ is through the pushforward operation $T_{\theta \sharp}:\Theta\rightarrow\mathcal{P}_\Theta,\theta\mapsto\rho_\theta$. Hence it is natural to define the Wasserstein metric $G(\theta)$ on parameter space $\Theta$ as the pullback of $g^W$ by $T_{\theta\sharp}$. To be specific, we define  $G(\theta) = (T_{\theta\sharp})^* g^W$. Using this definition, $T_{\theta\sharp}$ becomes an isometric immersion from $\Theta$ to $\mathcal{P}$. For each $\theta$, $G(\theta)$ is a bilinear form defined on $\mathcal{T}_\theta\Theta\simeq\mathbb{R}^m$, which can be identified as an $m\times m$ matrix.

Before computing $G(\theta)$, we introduce a lemma which can help us to better understand $G(\theta)$. 
\begin{restatable}{lemma}{lemmaA}
\label{lemma:local err analys}
Suppose $\vec{u},\vec{v}$ are two vector fields defined on $\mathbb{R}^d$, suppose $\varphi,\psi$ solves $-\nabla\cdot(\rho\nabla\varphi)=-\nabla\cdot(\rho\vec{u})$ and $-\nabla\cdot(\rho\nabla\psi) = -\nabla\cdot(\rho\vec{v})$, or equivalently, $\textrm{Proj}_{\rho}[\vec{u}]=\nabla\varphi$ and $\textrm{Proj}_{\rho}[\vec{v}]=\nabla\psi$ (cf. Definition \ref{Hodge Decomp}). Then:
\begin{align}
   & \int \vec{u}(x)\cdot \nabla\psi(x)\rho(x)~dx = \int  \nabla\varphi(x)\cdot \nabla\psi(x) \rho(x)~dx; \label{loc err analys lemma 1}\\
  &  \int |\nabla\psi(x)|^2\rho(x)~dx\leq \int |\vec{v}(x)|^2\rho(x)~dx. \label{loc err analys lemma 2}
\end{align}
\end{restatable}
We prove Lemma \ref{lemma:local err analys} in Appendix \ref{pf hodge lemma}. The metric tensor $G(\theta)$ is computed in the following theorem.

\begin{theorem}\label{thm_about_computing_metric_of_Theta}
Assume $\Theta$ satisfies \eqref{Theta_condition_A},\eqref{Theta_condition_B}. $T_\theta$ is invertible and $T_\theta(x)\in C^2(\Theta\times\mathbb{R}^d)$. Then $\Theta$ can be equipped with the metric tensor $G = (T_{\theta\sharp})^* g^W$, which is an $m\times m$ non-negative definite symmetric matrix of the form:
\begin{equation}
  G(\theta) = \int  \nabla \boldsymbol{\Psi}(T_\theta(x))\nabla \boldsymbol{\Psi}(T_\theta(x))^{\textrm{T}}~dp(x)\label{Metric tensor D dimension}
\end{equation}
at every $\theta\in\Theta$. More precisely, in entry-wise form,
\begin{equation}
  G_{ij}(\theta) =\int  \nabla\psi_i(T_\theta(x))\cdot\nabla\psi_j(T_\theta(x))~dp(x),~~1\leq i,j\leq m, \nonumber
\end{equation}
in which $\boldsymbol{\Psi}=(\psi_1,\cdots,\psi_m)^{\textrm{T}}$ and $\nabla\boldsymbol{\Psi}$ is an $m\times d$ Jacobian matrix of $\boldsymbol{\Psi}$. For each $j=1,2,\cdots,m$, $\psi_j$ solves the following equation:
\begin{equation}
  \nabla\cdot(\rho_\theta\nabla\psi_j(x)) = \nabla\cdot(\rho_\theta~\frac{\partial T_\theta}{\partial\theta_j}(T^{-1}_\theta(x))).\label{Hodge Dcom}
\end{equation}
with boundary conditions
$$\lim_{x\rightarrow\infty}\rho_\theta(x)\nabla\psi_j(x)=0.$$
\end{theorem}
\begin{proof}
\noindent
Suppose $\xi\in \mathcal{T}\Theta$ is a vector field on $\Theta$, for a fixed $\theta\in\Theta$,
we first compute the pushforward $(T_{\theta\sharp})_*\xi(\theta)$ of $\xi$ at point $\theta$: We choose any smooth curve $\{\theta_t\}_{t\geq 0}$ on $\Theta$ with $\theta_0=\theta$ and $\dot{\theta}_0 = \xi(\theta)$. If we denote $\rho_{\theta_t}={T_{\theta_t}}_{\sharp}p$, we have $(T_{\theta \sharp})_*\xi(\theta) = \frac{\partial\rho_{\theta_t}}{\partial t}\Bigr|_{t=0}$.

To compute $\frac{\partial\rho_{\theta_t}}{\partial t}\Bigr|_{t=0}$, we consider an arbitrary $\phi\in C^{\infty}_0(M)$.

On one hand, $\frac{\rho_{\theta_{\Delta t}}(y)-\rho_{\theta_0}(y)}{\Delta t}=\frac{\partial}{\partial t}\rho(\theta_{\tilde{t}_1},y)$, where $\tilde{t}_1$ is some point between $0,\Delta t$, since $\phi\in C_0^\infty$ and $\rho(\theta_t, x)$ is at least $C^1$ with respect to $t,y$, we can show that the function $\varphi(x)=\sup_{s\in[0,\Delta t]}|\phi(x)\frac{\partial}{\partial t}\rho(\theta_s,y)|$ is continuous on a compact set and thus integrable on $\mathbb{R}^d$. Using dominated convergence theorem, we have:
\begin{equation}
\frac{\partial}{\partial t}\left(\int \phi(y)\rho_{\theta_t}(y)~dy\right)\Big\vert_{t=0} =  \int\phi(y)\frac{\partial\rho_{\theta_t}(y)}{\partial t}\Big\vert_{t=0}~dy.  \label{proof_1}
\end{equation}
On the other hand, we have:
\begin{equation}
\frac{\phi(T_{\theta_{\Delta t}}(y))-\phi(T_{\theta_0}(y))}{\Delta t} = \dot\theta^{\textrm{T}}_{\tilde{t}_2}~\partial_\theta T_{\theta_{\tilde{t}_2}}(x)^{\textrm{T}}~\nabla\phi(T_{\theta_{\tilde{t}_2}}(y)),  \label{time difference}
\end{equation}
in which $\tilde{t}_2$ is also between $0,\Delta t$. For any $\Delta t$ small enough and $\tilde{t} \in [0, \Delta t]$, we can easily find upper bounds for $\|\dot\theta_{\tilde{t}}\|\leq A$ and 
$\|\nabla\phi(\cdot)\|_\infty\leq B$. Recall the condition \eqref{Theta_condition_B}, when $\Delta t$ is small enough, we have $|\theta_{\Delta t}-\theta_0|<r(\theta_0)$, thus we obtain the following upper bound for \eqref{time difference}
\[| \dot\theta^{\textrm{T}}_{\tilde{t}}~\partial_\theta T_{\theta_{\tilde{t}}}(x)^{\textrm{T}}~\nabla\phi(T_{\theta_{\tilde{t}}}(y)) | \leq A B \|\partial_\theta T_{\theta_{\tilde{t}}}(x)\|_F\leq ABL_1(x|\theta_0).\]
By \eqref{Theta_condition_B}, we know $L_1(\cdot|\theta_0)\in L^1(p)$, we can apply dominated convergence theorem to obtain: 
\begin{equation}
  \frac{\partial}{\partial t}\left(\int  \phi(T_{\theta_t}(x))dp\right)\Big\vert_{t=0} = \int  \dot{\theta_t}^{\textrm{T}} \partial_\theta T_{\theta_t}(x)^{\textrm{T}}\nabla\phi(T_{\theta_t}(x))\vert_{t=0}dp.   \label{proof_2}
\end{equation}
Since $\frac{\partial}{\partial t}\int\phi(y)\rho_{\theta_t}(y)~dy=\frac{\partial}{\partial t}\int\phi(T_{\theta_t}(x))~dp(x)$, we use \eqref{proof_1} and \eqref{proof_2} to get:
\begin{align}
    \int \phi(y)\frac{\partial \rho_{\theta_t}}{\partial t}(y)\Big\vert_{t=0}~dy & = \int  \dot{\theta_t}^{\textrm{T}} \partial_\theta T_{\theta_t}(x)^{\textrm{T}}\nabla\phi(T_{\theta_t}(x))\vert_{t=0}~dp(x)\nonumber \\ 
    & = \int \dot{\theta}_t^{\textrm{T}} \left( \frac{\partial T_{\theta_t}}{\partial \theta}(T^{-1}_{\theta_t}(x))\right)^{\textrm{T}}\nabla\phi(x)~\rho_{\theta_t}(x)\vert_{t=0}~dx\nonumber\\
    & = \int  \phi(x)\left(-\nabla\cdot\left(\rho_{\theta_t}(x) \frac{\partial  T_{\theta_t}}{\partial \theta}(T_{\theta_t}^{-1}(x))~\dot{\theta}_t\right)\right)\vert_{t=0}~dx. \nonumber
\end{align}
Because $\phi(x)$ is arbitrary, this weak formulation reveals that
\begin{equation}
    (T_{\theta\sharp})_*\xi(\theta)=\frac{\partial\rho_{\theta_t}}{\partial t}\Bigr|_{t=0}=-\nabla\cdot \left(  \rho_\theta(x)~ \frac{\partial T_\theta}{\partial\theta}(T_\theta^{-1}(x)) \xi(\theta)\right).  \label{compute_drhodt}
\end{equation}
Now let us compute the metric tensor $G$. Since $T_{\theta \sharp}$ is isometric immersion from $\Theta$ to $\mathcal{P}$, the pullback of $g^W$ by $T_{\theta \sharp}$ gives $G$, i.e. $(T_{\theta \sharp})^*g^W = G(\theta)$. By definition of pullback map, for any $\theta\in\Theta$ and $\xi(\theta)\in \mathcal{T}_\theta \Theta$, we have:
\begin{equation}
  G(\theta)(\xi(\theta), \xi(\theta)) = g^W(\rho_\theta)((T_{\theta \sharp})_*\xi(\theta), (T_{\theta \sharp})_*\xi(\theta)). \label{a}
\end{equation}
To compute the right hand side of (\ref{a}), recall \eqref{def_metric}, we need to solve for $\varphi$ from:
\begin{equation}
\frac{\partial\rho_{\theta_t}}{\partial t}\Bigr|_{t=0}=-\nabla\cdot(\rho_\theta(x)\nabla\varphi(x)). \label{b}
\end{equation}
By \eqref{compute_drhodt}, \eqref{b} is:
\begin{equation}
\nabla\cdot(\rho_\theta(x) \nabla\varphi(x)) = \nabla\cdot\left(\rho_\theta(x) \frac{ \partial_\theta T_\theta }{\partial \theta}(T_\theta^{-1}(\cdot))  \xi(\theta) \right). \label{equ}
\end{equation}
We can straightforwardly check that $\varphi(x) = \boldsymbol{\Psi}^{\textrm{T}}(x)\xi(\theta) $ is the solution of (\ref{equ}). Now by definition of $g^W$ as mentioned in \ref{wass_mfld}, we write the right hand side of \eqref{a} as
\begin{align}
  g^W(\rho_\theta)((T_{\theta \sharp})_*\xi(\theta), (T_{\theta\sharp})_*\xi(\theta)) = & \int |\nabla\varphi(y)|^2\rho_\theta(y)~dy = \xi(\theta)^{\textrm{T}}\left(\int \nabla\boldsymbol{\Psi}(y)\nabla\boldsymbol{\Psi}(y)^{\textrm{T}}\rho_\theta(y)~dy\right)\xi(\theta). \label{verify finite metric value g}\\
  = & \sum_{i,j=1}^{m} \left( \int \nabla\psi_i(y)\cdot\nabla\psi_j(y)\rho_\theta(y)~dy\right)\xi_i(\theta)\xi_j(\theta).  \nonumber
\end{align}
Here we assume components of $\xi(\theta)$ as $(\xi_1(\theta), ... ,\xi_m(\theta))^{\textrm{T}}$. Before we compute $G(\theta)$, we first verify that the inner product in \eqref{verify finite metric value g} is finite for any $\xi\in \mathcal{T}\Theta$. To show this, by Cauchy–Schwarz inequality we obtain
\begin{equation*}
  \int\nabla\psi_i(y)\cdot\nabla\psi_j(y)\rho_\theta(y)~dy\leq \left(\int|\nabla\psi_i(y)|^2\rho_\theta(y)~dy\right)^{\frac{1}{2}}\left(\int|\nabla\psi_j(y)|^2\rho_\theta(y)~dy\right)^{\frac{1}{2}}.
\end{equation*}
recall $\psi_j$ defined in \eqref{Hodge Dcom}, then applying \eqref{loc err analys lemma 2} of Lemma \eqref{lemma:local err analys} yields
\begin{equation*}
  \int|\nabla\psi_j(y)|^2\rho_\theta(y) ~dy  \leq \int \left|\frac{\partial T_{\theta}}{\partial \theta_j}(T_\theta^{-1}(y))\right|^2\rho_\theta(y)~dy = \int \left|\frac{\partial T_\theta}{\partial \theta_j}(x)\right|^2dp(x) \leq \int L_2(y|\theta)p(y)~dy<\infty.
\end{equation*}
The last two inequalities are due to condition \eqref{Theta_condition_B}. As a result, we proved the finiteness of \eqref{verify finite metric value g}.

Finally, let us compute:
\begin{align}
  G(\theta)(\xi(\theta),\xi(\theta)) = g^W(\rho_\theta)((T_{\theta\sharp})_*\xi(\theta),(T_{\theta\sharp})_*\xi(\theta))  
   = \xi(\theta)^{\textrm{T}}\left(\int\nabla\boldsymbol{\Psi}(T_\theta(x))\nabla\boldsymbol{\Psi}(T_\theta(x))^{\textrm{T}} dp(x)\right)\xi(\theta).\nonumber
\end{align}
Thus we can verify that 
\begin{equation}
G(\theta) = \int  \nabla\boldsymbol{\Psi}(T_\theta(x))\nabla\boldsymbol{\Psi}(T_\theta(x))^{\textrm{T}}~dp(x),  \nonumber
\end{equation}
which completes the proof.
\end{proof}

Generally speaking, the metric tensor $G$ does not have an explicit form when $d\geq 2$. It is worth mention that $G$ has an explicit form and can be computed directly when $d=1$ \cite{liliuzhouzha}.

\begin{remark}[Well-posedness of \eqref{Hodge Dcom}]\label{rmk wp}
  It is worth commenting on the existence and the regularity question to equations like  \eqref{Hodge Dcom}.  Determining what properties or conditions that  $T_\theta$ should have to guarantee the well-posedness of \eqref{Hodge Dcom} is an interesting and important problem on its own. In references such as \cite{pardoux2001poisson} and \cite{volpert2011elliptic}, there are sufficient conditions that guarantee the well-posedness of elliptic PDEs defined on $\mathbb{R}^d$. Most of the existing results require uniform lower bound on $\rho_\theta$, i.e. $\rho_\theta(x)>\epsilon>0$ for all $x\in\mathbb{R}^d$. Such coercive condition is not applicable in our case since $\int\rho_\theta(x)dx=1$ is finite. On the other hand, section 8.1.2 of \cite{villani2003topics} provides another sufficient condition on the well-posedness of \eqref{Hodge Dcom}: If there exists $\lambda>0$ such that the following Poincar\'{e} inequality \eqref{Poincare inequality} holds for any $\varphi\in C^\infty(\mathbb{R}^d)$ with compact support,
  \begin{equation}
     \int |\nabla \varphi(x)|^2\rho_\theta(x)~dx \geq \lambda \int \left(\varphi(x)-\int\varphi\rho_\theta~dx\right)^2\rho_\theta(x)~dx,  \label{Poincare inequality}
  \end{equation}
  and $-\nabla\cdot(\rho_\theta\frac{\partial T_\theta}{\partial \theta_j}(T^{-1}_{\theta}(\cdot)))\in L^2(\rho_\theta)$, Then \eqref{Hodge Dcom} admits a unique solution $\psi_j$ with $\nabla\psi_j\in L^2(\rho_\theta)$. To the best of our knowledge, it is still unclear that what kind of properties of $T_\theta$ may lead to \eqref{Poincare inequality}.  

 It is worth pointing out that under certain situations discussed in Section \ref{3.4 }, equation \eqref{Hodge Dcom} does have classical solutions. For example, if we select $T_\theta$ as an affine transform and consider the Fokker--Planck equation \eqref{FPE} with quadratic potential $V$ and Gaussian initial $\rho_0$, we can prove that \eqref{Hodge Dcom} is well-posed along the trajectory of the ODE \eqref{wass_grad_flow_on_para_spc}, i.e. the following elliptic equation
\begin{equation*}
  -\nabla\cdot(\rho_{\theta_t}\nabla\psi) = -\nabla\cdot(\rho_{\theta_t}\frac{\partial_\theta T_{\theta_t}}{\partial \theta}(T_{\theta_t}^{-1}(x))\dot\theta_t), \quad \textrm{where} ~ \{\theta_t\} ~\textrm{solves}~\eqref{wass_grad_flow_on_para_spc},
\end{equation*}
always admits a classical solution $\psi(x)=V(x)+D\log\rho_\theta(x)+\textrm{Const}$.

In general, The conditions imposed on $T_\theta$ to guarantee  well-posedness of \eqref{Hodge Dcom} is a fundamental and interesting topic subject to further investigation. A good reference related to the topic can be found in \cite{ambrosio2008gradient}. 
\end{remark}

Following theorem provides several criteria for examining whether $G$ is a Riemannian metric, i.e. whether $G(\theta)$ is positive definite.
\begin{theorem} \label{criteria of pos def G}
For $\theta\in\Theta$, $\{\psi_k\}_{k=1}^m$ satisfies \eqref{Hodge Dcom}, the following four statements are equivalent
\begin{enumerate}
  \item $G(\theta)$ is positive definite;
  \item For any $\xi\in \mathcal{T}_\theta\Theta~(\xi\neq 0)$, there exists $z\in M$ such that $\nabla\cdot(\rho_\theta(z)\frac{\partial  T_\theta}{\partial \theta}(T_\theta^{-1} (z))\xi)\neq 0$;
  \item $\{\nabla\psi_k\}_{k=1}^m$, as $m$ functions in the space $L^2(\mathbb{R}^d;\mathbb{R}^d,\rho_{\theta_k})$, are linearly independent;
  \item $\frac{d}{dt}(T_{\theta+t\xi\sharp}p)\vert_{t=0} \neq 0$ for any $\xi\in \mathbb{R}^m$.
\end{enumerate}
\end{theorem}
\begin{proof}
We first verify that 1 and 2 are equivalent. We need the following identity used in Theorem \ref{thm_about_computing_metric_of_Theta}: For any $\theta,\xi,x$, we have
\begin{equation}
 \nabla\cdot(\rho_\theta(x)\nabla(\xi^{\textrm{T}}\boldsymbol{\Psi}(x))) = \nabla\cdot(\rho_\theta(x) \frac{\partial T_\theta}{\partial\theta}(T_\theta^{-1}(x))\xi). \label{proof_pos_def_hodge_dcom}
\end{equation}
$(\Leftarrow)$: suppose for any $\theta\in\Theta$ and $\xi\in \mathcal{T}_\theta\Theta$, at certain $z\in\mathbb{R}^d$, $\nabla\cdot(\rho_\theta(z)\frac{\partial T_\theta}{\partial \theta}(T_\theta^{-1} (z)\xi)\neq 0$, then $\nabla\cdot(\rho_\theta(z)\nabla(\xi^{\textrm{T}}\boldsymbol{\Psi}(z)))\neq 0$, thus $\rho_\theta\nabla(\xi^{\textrm{T}}\boldsymbol{\Psi})$ is not identically $\boldsymbol{0}$. Using continuity of $\rho_\theta\nabla(\xi^{\textrm{T}}\boldsymbol{\Psi})$, we know that: $|\nabla(\xi^{\textrm{T}}\boldsymbol{\Psi}(x))|^2\rho_\theta(x)>0$ in some small neighbourhood of $z$. Thus we have:
\begin{equation}
\xi^{\textrm{T}} G(\theta)\xi = \int |\nabla\boldsymbol{\Psi}(x)^{\textrm{T}}\xi|^2\rho_\theta(x)~dx > 0,  \label{proof_pos_def}
\end{equation}
holds for any $\theta$ and $\xi$, this leads to the positive definiteness of $G$.

\noindent
$(\Rightarrow)$: Now suppose \eqref{proof_pos_def} holds for all $\theta,\xi$, then we have
\begin{equation*}
    \int-\nabla\cdot(\rho_\theta(x)\nabla(\xi^{\textrm{T}}\boldsymbol{\Psi}(x)))\cdot\xi^{\textrm{T}}\boldsymbol{\Psi}(x)~dx>0.
\end{equation*}
This leads to the existence of a $z\in\mathbb{R}^d$ such that $-\nabla\cdot(\rho_\theta(z)\nabla(\xi^{\textrm{T}}\boldsymbol{\Psi}(z)))\neq 0$. Combining \eqref{proof_pos_def_hodge_dcom}, we have verified the equivalence between 1 and 2.

We recall \eqref{compute_drhodt}, then $\frac{d}{dt}(T_{\theta+t\xi\sharp}p)\vert_{t=0} = (T_{\theta\sharp})_*\xi = -\nabla\cdot(\rho_\theta(x)\frac{\partial  T_\theta}{\partial \theta}(T_\theta^{-1}(x))\xi)$, this verifies the equivalence between 2 and 3.

Finally, as stated before, we can verify $\xi^{\textrm{T}}G(\theta)\xi = \|\sum_{k=1}\xi_k\nabla\psi_k\|^2_{L^2(\rho_\theta)}$, this formula will directly leads to the equivalence between 1 and 4 and we have proved the equivalence among statements 1,2,3 and 4.

\end{proof}

To keep our discussion concise in the following sections, we will always assume $G(\theta)$ is positive definite for every $\theta\in\Theta$.


\subsection{Parametric Fokker--Planck equation} \label{section 3.2}
We consider the pushforward $T_{(\cdot)\sharp}$ induced relative entropy functional $H=\mathcal{H}\circ T_{(\cdot)\sharp}:\Theta\rightarrow \mathbb{R}$, 
\begin{align}
H(\theta)=\mathcal{H}(\rho_\theta)&=\left(\int  V(x)\rho_\theta(x)+{D}\rho_\theta(x)\log\rho_\theta(x)~dx\right) +{D} \log Z_{D} \nonumber \\
& = \left(\int V(T_\theta(x))+{D}\log\rho_\theta(T_\theta(x))~dp(x)\right)+{D} \log Z_{D}.  \label{relative entropy parameter}
\end{align}
Following the theory in \cite{NG}, the gradient flow of $H$ on Wasserstein parameter manifold $(\Theta,G)$ satisfies 
\begin{equation}
\dot{\theta}=-G(\theta)^{-1}\nabla_\theta H(\theta).
\label{wass_grad_flow_on_para_spc}
\end{equation}
We call \eqref{wass_grad_flow_on_para_spc} {\em parametric Fokker--Planck equation}. The ODE \eqref{wass_grad_flow_on_para_spc} as the Wasserstein gradient flow on parameter space $(\Theta,G)$ is closely related to the Fokker--Planck equation on probability submanifold $\mathcal{P}_{\Theta}$. We have the following theorem, which is a natural result derived from submanifold geometry.

\begin{theorem}\label{theorem_submfld}
Suppose $\{\theta_t\}_{t\geq 0 }$ solves (\ref{wass_grad_flow_on_para_spc}). Then $\{\rho_{\theta_t}\}$ is the gradient flow of $\mathcal{H}$ on probability submanifold $\mathcal{P}_{\Theta}$. Furthermore, at any time $t$, $\dot\rho_{\theta_{t}}=\frac{d}{dt} \rho_{\theta_t} \in \mathcal{T}_{\rho_{\theta_{t}}}\mathcal{P}_\Theta$ is the orthogonal projection of $-\textrm{grad}_W\mathcal{H}(\rho_{\theta_t})\in \mathcal{T}_{\rho_{\theta_t}}\mathcal{P}$ onto the subspace $\mathcal{T}_{\rho_{\theta_t}}\mathcal{P}_\Theta$ with respect to the Wasserstein metric $g^W$. 
\end{theorem}

We prove this theorem in the Appendix \ref{pf on subM}. 

The following theorem is an important new statement closely related to Theorem. \ref{theorem_submfld}. 
\begin{theorem}[Wasserstein gradient as solution to a least squares problem]\label{lemma_submfld_grad_err}
For a fixed $\theta\in\Theta$, $\boldsymbol{\Psi}\subset\mathbb{R}^m$ as defined in Theorem \ref{thm_about_computing_metric_of_Theta}, then
\begin{equation}
  G(\theta)^{-1}\nabla_\theta H(\theta) = \underset{\eta\in \mathcal{T}_{\theta}\Theta\cong\mathbb{R}^m}{\arg\min} \left\{ \int |(\nabla\boldsymbol{\Psi}( T_\theta(x)))^{\textrm{T}}\eta-\nabla\left( V+{D}\log\rho_\theta \right)\circ T_\theta(x)|^2 dp(x) \right\}. \label{important_lemma}
\end{equation} 
\end{theorem}
\begin{proof}
Direct computation shows that minimizing the function in \eqref{important_lemma} is equivalent to minimizing:
\begin{equation*}
 \eta^{\textrm{T}}\left(\int \nabla\boldsymbol{\Psi}(T_\theta(x))\nabla\boldsymbol{\Psi}(T_\theta(x))^{\textrm{T}}~dp(x)\right)\eta-2~\eta^{\textrm{T}}\left(\int \nabla\boldsymbol{\Psi}(x)\nabla(V(x)+{D}\log\rho_\theta(x))\rho_\theta(y)~dx\right).
\end{equation*}
For each entry in the second term, we have:
\begin{align*}
  & \int \nabla\psi_k(x)\cdot\nabla(V(x)+{D}\log\rho_\theta(x))\rho_\theta(x)~dx = \int -\nabla\cdot(\rho_\theta(x)\nabla\psi_k(x))\cdot(V(x)+{D}\log\rho_\theta(x))~dx\\
  &=\int-\nabla\cdot(\rho_\theta(x)\partial_{\theta_k}T_\theta(T^{-1}_\theta(x)))\cdot(V(x)+{D}\log\rho_\theta(x))~dx = \int (\nabla V(T_\theta(x))+D\nabla\log\rho_\theta(T_\theta(x)))\cdot \partial_{\theta_k}  T_\theta(x)~dp(x) \\
  &= \int \nabla V(T_\theta(x))\cdot\partial_{\theta_k}T_\theta(x)+\partial_{\theta_k}[D\log\rho_\theta(T_\theta(x))]~dp(x) - \underbrace{\int D~ \partial_{\theta_k}\log\rho_\theta(T_\theta(x))~dp(x)}_{=D\int\nabla_\theta \rho_\theta(x)dx = 0} \\
  &= \partial_{\theta_k}\left(\int (V(T_\theta(x))+{D}\log\rho_\theta(T_\theta(x)))~dp(x)\right)=\partial_{\theta_k}H(\theta).
\end{align*}
Recall the definition \eqref{Metric tensor D dimension} of $G(\theta)$, the target function to be minimized is $\eta^{\textrm{T}} G(\theta)\eta - 2\eta^{\textrm{T}}\nabla_\theta H(\theta)$. And the minimizer is clearly $G(\theta)^{-1}\nabla_\theta H(\theta)$.
\end{proof}

\noindent

In addition to the direct proof, the result in Theorem \ref{lemma_submfld_grad_err} can also be understood in a different way. Let us denote $\xi=G(\theta)^{-1}\nabla_\theta H(\theta)$, $\{\theta_t\}$ solves \eqref{wass_grad_flow_on_para_spc} with initial value  $\theta_0=\theta$. By Theorem \ref{theorem_submfld}, $\frac{d}{dt}\rho_{\theta_t}\Big\vert_{t=0}=(T_{\theta\sharp})_*\xi\in \mathcal{T}_{\rho_\theta}\mathcal{P}_\Theta$ is the orthogonal projection of  $\textrm{grad}_W\mathcal{H}(\rho_\theta)$ onto $\mathcal{T}_{\rho_\theta}\mathcal{P}_\Theta$ with respect to the metric $g^W$. This is equivalent to say that $\eta$ solves the following least square problem:
\begin{equation}
  \min_{\eta} g^W(\textrm{grad}_W\mathcal{H}(\rho_\theta)-(T_{\theta \sharp})_*\eta,~\textrm{grad}_W\mathcal{H}(\rho_\theta)-(T_{\theta\sharp})_*\eta).\label{important_lemma_equiv}
\end{equation}
Recall the definition of $g^W$ in section \ref{wass_mfld} and by \eqref{gradflow}, we have $\textrm{grad}_W\mathcal{H}(\rho_\theta)=-\nabla\cdot(\rho_\theta\nabla(V+{D}\log\rho_\theta))$. Because of  \eqref{compute_drhodt}, $(T_{\theta \sharp})_*\eta=-\nabla\cdot(\rho_\theta\partial_\theta T_\theta(T^{-1}_\theta(\cdot))\eta)$, solving $-\nabla\cdot(\rho_\theta\nabla \varphi)=\textrm{grad}_W\mathcal{H}(\rho_\theta)-(T_{\theta \sharp})_*\eta$ gives 
\begin{equation*}
\varphi = (V+{D}\log\rho_\theta) - \boldsymbol{\Psi}^{\textrm{T}}\eta,
\end{equation*}
and thus least squares problem \eqref{important_lemma_equiv} can be written as
\begin{equation*}
\min_\eta\left\{\int |\nabla\boldsymbol{\Psi}(x)^{\textrm{T}}\eta-\nabla(V(x)+{D}\log\rho_\theta(x))|^2\rho_\theta(x)~dx\right\} , 
\end{equation*}
which is exactly \eqref{important_lemma}.


\subsection{A particle viewpoint of the parametric Fokker Planck Equation  } \label{particle level explanation}
The motion of parameter $\theta_t$ solving \eqref{wass_grad_flow_on_para_spc} naturally induce a stochastic dynamics on $\mathbb{R}^d$ whose density evolution is exactly $\{\rho_{\theta_t}\}$. To see this, notice that $\{\theta_t\}$ directly leads to a time dependent map $\{T_{\theta_t}\}$. Let us denote a random variable $\boldsymbol{Z}\sim p$, i.e. $\boldsymbol{Z}$ is distributed according to the reference distribution $p$. We set $\boldsymbol{Y}_0=T_{\theta_0}(\boldsymbol{Z})\sim\rho_{\theta_0}$. At any time $t$, the map $T_{\theta_t}$ sends $\boldsymbol{Y}_0$ to $\boldsymbol{Y}_t = T_{\theta_t}(T_{\theta_0}^{-1}(\boldsymbol{Y}_0))\sim\rho_{\theta_t}$. Thus, we construct a sequence of random variables $\{\boldsymbol{Y}_t\}$ whose density evolution is exactly $\{\rho_{\theta_t}\}$. We can characterize the dynamical system satisfied by $\{\boldsymbol{Y}_t\}$ by taking time derivative: $\dot{\boldsymbol{Y}}_t = \partial_\theta T_{\theta_t}(\boldsymbol{Z})\dot\theta_t =  \partial_\theta T_{\theta_t}(T_{\theta_t}^{-1}(\boldsymbol{Y}_t))\dot\theta_t$. It is actually more insightful to consider the following dynamic:
\begin{equation}
   \dot{\boldsymbol{X}}_t = \nabla\boldsymbol{\Psi}_t(\boldsymbol{X}_t)^{\textrm{T}}~\dot\theta_t ,  \quad \boldsymbol{X}_0=T_{\theta_0}(\boldsymbol{Z})\sim\rho_{\theta_0}.  \label{X_t dym}
\end{equation}
Here $\boldsymbol{\Psi}_t$ is obtained from \eqref{Hodge Dcom} with parameter $\theta_t$. 
It is not hard to show that for any time $t$, $\boldsymbol{X}_t$ and $\boldsymbol{Y}_t$ has the same distribution. Thus $\boldsymbol{X}_t\sim\rho_{\theta_t}$ for all $t\geq 0$. Recall $\dot\theta_t = -G(\theta_t)^{-1}\nabla_\theta H(\theta_t)$, we are able to rewrite \eqref{X_t dym} as:
\begin{equation}
  \dot{\boldsymbol{X}_t} = \nabla\boldsymbol{\Psi}_t(\boldsymbol{X}_t)^{\textrm{T}} {\underbrace{ \left(
  \int \nabla\boldsymbol{\Psi}_t(x)\nabla\boldsymbol{\Psi}_t(x)^{\textrm{T}}~\rho_{\theta_t}(x)~dx \right)}_{G(\theta_t)}}^{-1} \underbrace{ \left(  \int  \nabla \boldsymbol{\Psi}_t(\eta)(-\nabla V(\eta)-{D}\nabla\log\rho_{\theta_t}(\eta))~\rho_{\theta_t}(\eta)~d \eta  \right)}_{-\nabla_\theta H(\theta_t)}.\label{corrsde}
\end{equation}
If we define the kernel function $K_\theta:\mathbb{R}^d\times\mathbb{R}^d\rightarrow \mathbb{R}^{d\times d}$ as
\begin{equation*}
  K_\theta(x,\eta)=\nabla\boldsymbol{\Psi}^{\textrm{T}}(x)\left(\int\nabla\boldsymbol{\Psi}(x)\nabla\boldsymbol{\Psi}(x)^{\textrm{T}}~\rho_\theta(x)~dx\right)^{-1}\nabla\boldsymbol{\Psi}(\eta).
\end{equation*}
This $K_\theta$ induces a linear operator $\mathcal{K}_\theta:L^2(\mathbb{R}^d;\mathbb{R}^d,\rho_\theta)\rightarrow L^2(\mathbb{R}^d;\mathbb{R}^d,\rho_\theta)$ by:
\begin{equation*}
  \mathcal{K}_\theta[\vec{v}] = (\mathcal{K}_\theta * \vec{v})(\cdot) = \int  K_\theta(\cdot, \eta)~\vec{v}(\eta)~\rho_\theta( \eta )~ d\eta. 
\end{equation*}
It can be verified that $\mathcal{K}_\theta$ is an orthogonal projection defined on the Hilbert space $L^2(\mathbb{R}^d;\mathbb{R}^d,\rho_\theta)$. The range of such projection is the subspace  $\textrm{span}\left\{\nabla\psi_1,...,\nabla\psi_m\right\}\subset L^2(\mathbb{R}^d; \mathbb{R}^d,\rho_\theta ) $. Here $\psi_1,...,\psi_m$ are the $m$ components of $\boldsymbol{\Psi}$ solved from \eqref{Hodge Dcom}.
Using the linear operator, we can rewrite \eqref{corrsde} as:
\begin{equation}
    \dot{\boldsymbol{X}}_t  = - \mathcal{K}_{\theta_t}[\nabla V + {D} \nabla\log\rho_{\theta_t}](\boldsymbol{X}_t), \quad \textrm{where} ~ \rho_{\theta_t} ~ \textrm{is the probability density of } \boldsymbol{X}_t \quad \boldsymbol{X}_0\sim \rho_{\theta_0}. \label{Vlasov SDE with proj}
\end{equation}
We can compare \eqref{Vlasov SDE with proj} with the following dynamic without projection:
\begin{equation}
  \dot{ \tilde{ \boldsymbol{X}_t}} = - (\nabla V + {D}\nabla\log\rho_t)(\tilde{\boldsymbol{X}_t}), \quad \textrm{where} ~ \rho_t ~\textrm{is the probability density of } \tilde{\boldsymbol{X}_t } \quad \tilde{\boldsymbol{X}_0}\sim\rho_0. \label{Vlasov SDE without proj}
\end{equation}
As discussed in section \ref{background}, \eqref{Vlasov SDE without proj} is the Vlasov-type SDE that involves the density of random particle. If assuming \eqref{Vlasov SDE without proj} admits a regular solution, we have $ \rho(x,t)=\rho_t(x)$, which solves the original Fokker Planck equation \eqref{FPE}. From orthogonal projection viewpoint,  we can treat that the approximate solution $\rho_{\theta_t}$ of \eqref{FPE} is actually originated from the projection of vector field that drives the SDE \eqref{Vlasov SDE without proj}. 

We would like to mention that the expectation of $\ell^2$ discrepancy between $\nabla V+{D}\nabla\log\rho$ and its $\mathcal{K}_\theta$ projection is:
\begin{equation}
  \mathbb{E}_{\boldsymbol{X}\sim\rho_\theta} |\mathcal{K}_\theta[\nabla V+{D}\nabla\log\rho_\theta](\boldsymbol{X})-(\nabla V+{D}\nabla\log\rho_\theta)(\boldsymbol{X})|^2=\int|\nabla\boldsymbol{\Psi}(x)^{\textrm{T}}\xi-(-\nabla V -{D}\nabla\log\rho_\theta)(x)|^2\rho_\theta(x)~dx,   \label{expectation of l2 discrepancy particle}
\end{equation}
in which $\xi = -G(\theta)^{-1}\nabla_\theta H(\theta)$. This is an essential term appeared in our error analysis part. 

\begin{remark}
  We should mention the relationship between our kernel $K_{\theta_t}$ and the Neural Tangent Kernel (NTK) introduced in \cite{jacot2018neural}. Using our notation, Neural Tangent Kernel can be written as $K^{NTK}_\theta = \partial_\theta T_\theta(x)\partial_\theta T_\theta(\xi)^{\textrm{T}}$. If we consider the flat gradient flow $\dot{\theta}=-\nabla_\theta H(\theta)$ of relative entropy on $\Theta$, its corresponding particle dynamic is 
  \begin{equation*}
   \dot{\boldsymbol{X}}_t = \int K^{NTK}_{\theta_t}(T^{-1}_{\theta_t}(\boldsymbol{X}_t),T^{-1}_{\theta_t}(\eta))(-\nabla V(\eta)-{D}\nabla\log\rho_{\theta_t}(\eta))\rho_{\theta_t}(\eta)~d\eta
  \end{equation*}
  Different from our $K_\theta$, which introduces an orthogonal projection, Neural Tangent Kernel introduces an non-negative definite transform to the vector field $-\nabla V-{D}\nabla\log\rho_{\theta_t}$.
\end{remark}


\begin{figure}
\centering
\begin{tikzpicture}[node distance=1.1cm, auto]
\centering
\tikzset{
    mynode/.style={rectangle,rounded corners,draw=black,  top color=white, bottom color=yellow!50, very thick, inner sep=0.5em, minimum size=1em, text centered},
    myarrow/.style={-, >=latex', shorten >=1pt, thick},
    mylabel/.style={text width=7em, text centered}
}

\node[] (dummy1) {};

\node[mynode, left=of dummy1] (upper left box) {$ \dot{\boldsymbol{X}_t }  =  - \mathcal{K}_{\theta_t}\left(\nabla \frac{\delta \mathcal{H}(\rho_{\theta_t})}{\delta\rho_{\theta_t}}\right)(\boldsymbol{X}_t) $ on $\mathbb{R}^d$};

\node[mynode, right=of dummy1] (upper right box) {$ \dot{\boldsymbol{X}_t }  =  - \nabla \frac{\delta \mathcal{H}(\rho_t)}{\delta\rho_t}(\boldsymbol{X}_t) $ on $\mathbb{R}^d$};

\node[mynode , below= 2.7cm of upper left box.west,anchor=west] (lower left box) {$ \dot\theta = -G(\theta)^{-1}\nabla_\theta H(\theta)$ on $\Theta $};

\node[mynode , below= 2.7cm of upper right box.east,anchor=east] (lower right box) {$\partial_t\rho = - \textrm{grad}_W \mathcal{H}(\rho)$ on $\mathcal{P}(\mathbb{R}^d)$};

\draw[->, >=latex', shorten >=2pt, shorten <=2pt, thick](upper right box.west) to node[auto, swap, below, text width=6em, text centered] {Projection of vector field} (upper left box.east);

\draw[<->, >=latex', shorten >=2pt, shorten <=2pt, thick](lower left box.north) to node[auto, swap, left, text width=10em, text centered] {How dynamics on $\Theta$ triggers dynamics on $\mathbb{R}^d$} (lower left box.north |- upper left box.south);

\draw[<->, >=latex', shorten >=2pt, shorten <=2pt, thick](upper right box.south) to node[auto, swap, right, text width=10em, text centered] {Density evolution of $\boldsymbol{X}_t$ solves Fokker Planck equation} (upper right box.south |- lower right box.north);

\draw[->, >=latex', shorten >=2pt, shorten <=2pt, thick](lower right box.west) to node[auto, swap, above, text width=11em, text centered] {Projection from $(\mathcal{P},g^W)$ onto $(\Theta,G)$} (lower left box.east);

\node[align=center, above, text width=25 em, text centered] at (0,0.6)  {[Particle point of view]
};

\node[below, text width=25em, text centered] at (0,-3.2) {[Probability manifold point of view]
};

\end{tikzpicture}
\caption{Illustrative diagram}
\label{illustration diagram}
\end{figure}
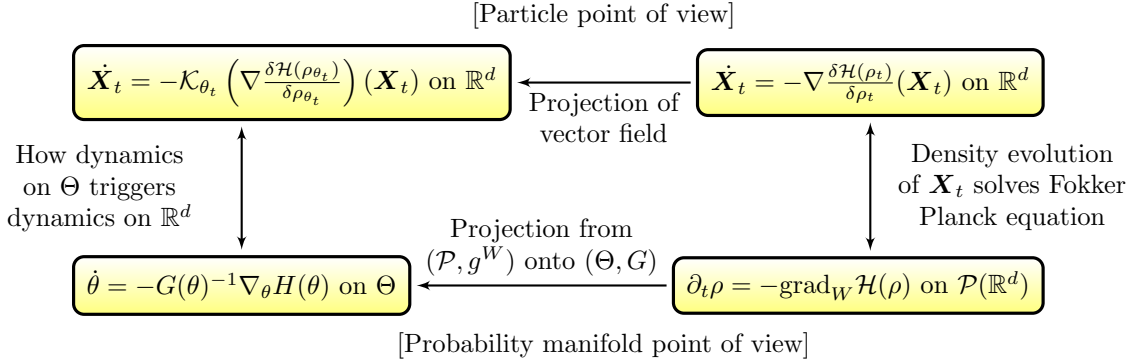

\begin{remark}
Figure \ref{illustration diagram} illustrates the relation between \eqref{FPE}, \eqref{wass_grad_flow_on_para_spc}, \eqref{Vlasov SDE without proj} and \eqref{Vlasov SDE with proj}. It is worth mentioning that the probability manifold point of view discussed in Theorem \ref{theorem_submfld} is useful for our analysis of the continuous dynamics \eqref{wass_grad_flow_on_para_spc}, while particle point of view helps us on establishing the numerical analysis for the time discrete scheme (i.e. forward-Euler) of \eqref{wass_grad_flow_on_para_spc}.
\end{remark}


\subsection{An example of the parametric Fokker--Planck equation with quadratic potential}\label{3.4 }
The solution of the parametric Fokker--Planck equation (\ref{wass_grad_flow_on_para_spc}) can serve as an approximation to the solution of the original equation (\ref{FPE}).  In some special cases, $\rho_{\theta_t}$ exactly solves (\ref{FPE}). In this section, we provide such examples. 

Let us consider the Fokker--Planck equations with quadratic potentials whose initial conditions are Gaussian:
\begin{equation}
V(x)=\frac{1}{2}(x-\mu)^{\textrm{T}}\Sigma^{-1}(x-\mu) \quad\mathrm{and}\quad\rho_0\sim\mathcal{N}(\mu_0,\Sigma_0).\label{conditions_example}
\end{equation}
Here $\mathcal{N}(\mu,\Sigma)$ denotes Gaussian distribution with mean $\mu$ and covariance $\Sigma$. We consider parameter space $\Theta=(\Gamma,b)\subset\mathbb{R}^{m}$ ($m=\frac{1}{2}d(d+1)+d$), where $\Gamma$ is a $d\times d$ symmetric positive definite matrix and $b\in\mathbb{R}^d$. We define the parametric map as $T_\theta(x)=\Gamma x+b$, and choose the reference measure $p=\mathcal{N}(0,I)$.
\begin{lemma}\label{coro_b}
Let $\mathcal{H}$ be the relative entropy defined in (\ref{relative entropy}) and $H$ defined in (\ref{relative entropy parameter}). For $\theta\in\Theta$, if the vector function $\nabla\left(\frac{\delta\mathcal{H}}{\delta\rho}\right)\circ T_\theta$ can be written as the linear combination of $\{ \frac{\partial T_\theta}{\partial \theta_1},...,\frac{\partial T_\theta}{\partial \theta_{m}}\}$, i.e. there exists $\zeta\in\mathbb{R}^m$, such that $\nabla\left(\frac{\delta\mathcal{H}}{\delta\rho}\right)\circ T_\theta(x)=\partial_\theta T_\theta(x)\zeta$. Then:\\ 
1) $\zeta=G(\theta)^{-1}\nabla_\theta H(\theta)$, which is the Wasserstein gradient of $H$ at $\theta$.\\
2) $\mathcal{P}_{\Theta}$ as $\mathrm{grad}_W\mathcal{H}(\rho_\theta)|_{\mathcal{P}_{\Theta}}$, then 
$ \mathrm{grad}_W\mathcal{H}(\rho_\theta)|_{\mathcal{P}_{\Theta}} = \mathrm{grad}_W \mathcal{H}(\rho_\theta)$, where $\mathrm{grad}_W\mathcal{H}(\rho_\theta)|_{\mathcal{P}_{\Theta}}$ is the gradient of $\mathcal{H}$ on the submanifold  $\mathcal{P}_{\Theta}$.
\end{lemma}
\begin{proof}
Suppose that $\zeta\in\mathbb{R}^m$ satisfies $\nabla\left(\frac{\delta\mathcal{H}}{\delta\rho}\right)\circ T_\theta(x)=\partial_\theta T_\theta(x)\zeta$, then we have
\begin{equation*}
  \int |\partial_\theta T_\theta(x)\zeta-\nabla(\frac{\delta \mathcal{H}}{\delta\rho})\circ T_\theta(x)|^2~d p(x) = 0.
\end{equation*}
By definition of $\boldsymbol{\Psi}$ in Theorem \ref{thm_about_computing_metric_of_Theta}, one can verify
\begin{equation*}
    -\nabla\cdot\left(\rho_{\theta}\left((\nabla\boldsymbol{\Psi})^{\textrm{T}}\zeta-\nabla\left( \frac{\delta\mathcal{H}}{\delta\rho}\right)\right)\right) = -\nabla\cdot\left(\rho_\theta \left(\partial_\theta T_\theta\circ T^{-1}_\theta \zeta-\nabla\left(\frac{\delta\mathcal{H}}{\delta\rho}\right)\right)\right)
\end{equation*}
Now we apply 
\eqref{lemma:local err analys} of Lemma \ref{lemma:local err analys} to obtain:
\begin{equation*}
  \int |(\nabla\boldsymbol{\Psi}(T_\theta(x)))^{\textrm{T}}\zeta-\nabla\left( \frac{\delta\mathcal{H}}{\delta\rho}\right)\circ T_\theta(x)|^2~dp(x)\leq 0.
\end{equation*}
This implies,
\begin{equation*}
  \inf_\eta \int |(\nabla\boldsymbol{\Psi}(T_\theta(x)))^{\textrm{T}}\eta-\nabla\left( \frac{\delta\mathcal{H}}{\delta\rho}\right)\circ T_\theta(x)|^2~dp(x) = \int |(\nabla\boldsymbol{\Psi}(T_\theta(x)))^{\textrm{T}}\zeta-\nabla\left( \frac{\delta\mathcal{H}}{\delta\rho}\right)\circ T_\theta(x)|^2~dp(x) = 0.
\end{equation*}
By Theorem \ref{lemma_submfld_grad_err}, we get $\zeta = G(\theta)^{-1}\nabla_\theta H(\theta)$ and $\|(T_{\theta\sharp})_*\zeta-\textrm{grad}_W\mathcal{H}(\rho_\theta)\|_{g^W(\rho_\theta)}=0$. The latter leads to $(T_{\theta\sharp})_*\zeta=\textrm{grad}_W\mathcal{H}(\rho_\theta)$. According to Theorem \ref{theorem_submfld}, $(T_{\theta \sharp})_*\zeta = \textrm{grad}_W\mathcal{H}(\rho_\theta)\vert_{\mathcal{P}_\Theta}$. As a result, we have $\textrm{grad}_W\mathcal{H}(\rho_\theta)\vert_{\mathcal{P}_\Theta}=\textrm{grad}_W\mathcal{H}(\rho_\theta)$.
\end{proof}

Back to our example with quadratic potential \eqref{conditions_example} and $T_\theta(x)=\Gamma x+b$, we can compute
\begin{equation}
\rho_\theta(x)={T_\theta}_{\sharp}p(x)=\frac{f(T_\theta^{-1}(x))}{|\det(\Gamma)|}=\frac{f(\Gamma^{-1}(x-b))}{|\det(\Gamma)|},~ f(x)=\frac{\exp(-\frac{1}{2}|x|^2)}{(2\pi)^{\frac{d}{2}}}.\nonumber
\end{equation}
Then we have,
\begin{equation*}
\nabla\left(\frac{\delta\mathcal{H}(\rho_\theta)}{\delta\rho}\right)\circ T_\theta(x)=\nabla (V+{D}\log\rho_\theta)\circ T_\theta(x) = \Sigma^{-1}(\Gamma x+b-\mu)-{D}\Gamma^{-T}x,
\end{equation*}
which is affine with respect to $x$.

Notice that $\partial_{\Gamma_{ij}}T_\theta(x)=(\dots, 0, \dots,\underset{i-\mathrm{th}}{x_j},\dots, 0, \dots)^{\textrm{T}}$ and $\partial_{b_i}T_\theta=(\dots, 0, \dots,\underset{i-\mathrm{th}}{1},\dots, 0, \dots)^{\textrm{T}}$, we can verify that $\zeta=(\Sigma^{-1}\Gamma-{D}\Gamma^{-T},\Sigma^{-1}(b-\mu))$ solves $\nabla\left(\frac{\delta\mathcal{H}(\rho_\theta)}{\delta\rho}\right)\circ T_\theta(x)=\partial_\theta T_\theta(x)\zeta$. By 1) of Lemma \ref{coro_b}, $\zeta = G(\theta)^{-1}\nabla_\theta H(\theta)$. Thus ODE (\ref{wass_grad_flow_on_para_spc}) for our example is:
\begin{align}
\dot{\Gamma}&=-\Sigma^{-1}\Gamma+{D}\Gamma^{-T}\quad \Gamma_0=\sqrt{\Sigma_0},\label{wass_grad_flow_1_example}\\
\dot{b}&=\Sigma^{-1}(\mu-b)\quad b_0=\mu_0.\label{wass_grad_flow_2_example}
\end{align}
By 2) of Lemma \ref{coro_b}, we know $\mathrm{grad_W\mathcal{H}(\rho_\theta)|_{\mathcal{P}_{\Theta}}=\mathrm{grad}_W\mathcal{H}(\rho_\theta)}$ for all $\theta\in\Theta$, which indicates that there is no error between our parametric Fokker--Planck and the original equations.

Following the equations \eqref{wass_grad_flow_1_example} and \eqref{wass_grad_flow_2_example}, we have the following corollary,
\begin{corollary}
The solution of the Fokker--Planck equation (\ref{FPE}) with condition(\ref{conditions_example}) is a Gaussian distribution for all $t>0$.
\end{corollary}
\begin{proof}
If we denote $\{\Gamma_t,b_t\}$ as the solutions to (\ref{wass_grad_flow_1_example}),(\ref{wass_grad_flow_2_example}), set $\theta_t=(\Gamma_t,b_t)$, then $\rho_t={T_{\theta_t}}_{\sharp}p$ solves the Fokker Planck Equation (\ref{FPE}) with conditions (\ref{conditions_example}). Since the pushforward of Gaussian distribution $p$ by an affine transform $T_\theta$ is still a Gaussian, we conclude that for any $t>0$, the solution  $\rho_t={T_{\theta_t}}_{\sharp}p$ is always Gaussian distribution.
\end{proof}
\begin{remark}
This is already a well known property for Ornstein–Uhlenbeck process \cite{doob1942brownian}. We provide an alternative proof using our framework.
\end{remark}

\section{Numerical methods}\label{section4}
In this section, we introduce our sampling efficient numerical method to compute the proposed parametric Fokker--Planck equations.

Before we start, we want to mention that as stated in \cite{liliuzhouzha}, when dimension $d=1$, $G(\theta)$ has explicit solution. Thus the push-forward approximation of 1D Fokker--Planck equation can be directly computed by solving the ODE system \eqref{wass_grad_flow_on_para_spc} with numerical methods, such as forward-Euler scheme. In this section, our focus is on numerical methods for \eqref{wass_grad_flow_on_para_spc} with dimension $d\geq 2$. It turns out to be very challenging to compute \eqref{wass_grad_flow_on_para_spc} by the forward-Euler scheme directly. There are two reasons. One is that there is no known explicit formula for $G(\theta)$, and direct computation based on \eqref{Metric tensor D dimension} can be expensive because it requires to solve multiple differential equations. The other is incurred by the high dimensionality, which is the main goal of this paper. To overcome the challenge of dimensionality, we choose to use deep neural networks to construct our $T(\theta)$. However, directly evaluating $G(\theta)^{-1} \nabla_{\theta} H(\theta)$ is difficult, alternative strategies must be sought.

There are a few papers investigating numerical methods for gradient flows on Riemannian manifolds, such as Fisher natural gradient \cite{martens2015optimizing} and Wasserstein gradient \cite{carrillo2019primal}. The well known JKO scheme \cite{jordan1998variational} calculates the time discrete approximation of
the Wasserstein gradient flow using an optimization formulation, 
\begin{equation}
 \partial_t\rho_t = -\textrm{grad}_W\mathcal{F}(\rho_t), \quad  \quad \rho_{k+1} = \underset{\rho\in\mathcal{P} }{\textrm{argmin} } \left\{\frac{W_2^2(\rho, \rho_{k})}{2h} + \mathcal{F}(\rho) \right\}, \label{JKO scheme}
\end{equation}
where $h$ is the time step size, $\mathcal{F}$ could be a suitable functional defined on $\mathcal{P}$. Along the line of JKO scheme, there are further developments in machine learning recently \cite{li2019affine}.

In our approach, we design schemes that computes the exact Wasserstein gradient flow directly with provable accuracy guarantee. Our algorithms are completely sample based so that they can be run efficiently under deep learning framework, and can scale up to high dimensional cases.

\subsection{Normalizing Flow as push forward maps}\label{4.1 }
We choose $T_\theta$ as the so-called normalizing flow \cite{rezende2015variational}. Here is a brief sketch of its structure:
$T_\theta$ is written as the composition of $K$ invertible nonlinear transforms:
\begin{equation*}
  T_\theta = f_K\circ f_{K-1}\circ...\circ f_2\circ f_1,
\end{equation*}
where each $f_k$ ($1\leq k\leq K$) takes the form
\begin{equation*}
f_k (x) = x + \sigma( w_k^{\textrm{T}} x +b_k)u_k.
\end{equation*}
Here $w_k,u_k\in \mathbb{R}^d$, $b_k\in\mathbb{R}$, and $\sigma$ is a nonlinear function, which can be chosen as $\textrm{tanh}$ for example. In \cite{rezende2015variational}, it has been shown that $f_k$ is invertible iff $w_k^{\textrm{T}}u_k\geq -1$.
Figure \ref{figure_pushforward gaussian} shows several snapshots of how a normalizing flow $T_\theta$ with length equal to 10 pushes forward standard Gaussian distribution to a target distribution.
\begin{figure}[!htb]
\minipage{0.09\textwidth}
  \includegraphics[width=\linewidth]{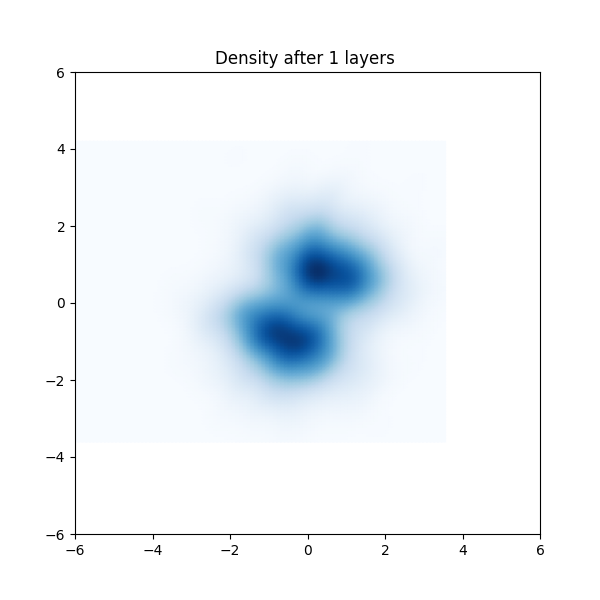}
\endminipage\hfill
\minipage{0.09\textwidth}
  \includegraphics[width=\linewidth]{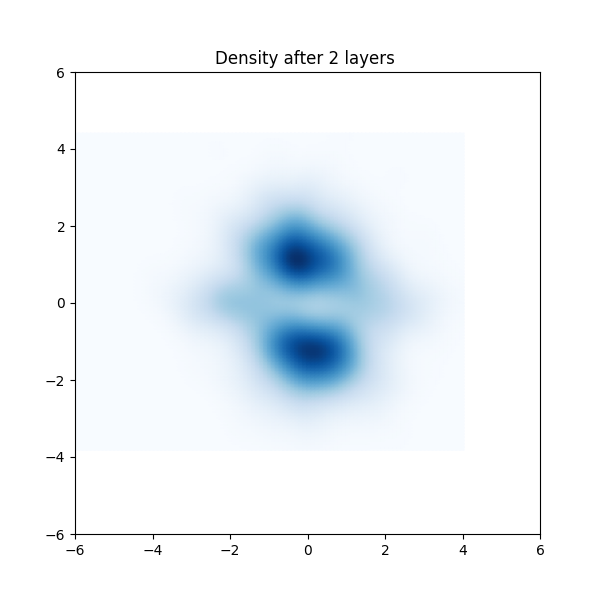}
\endminipage\hfill
\minipage{0.09\textwidth}
  \includegraphics[width=\linewidth]{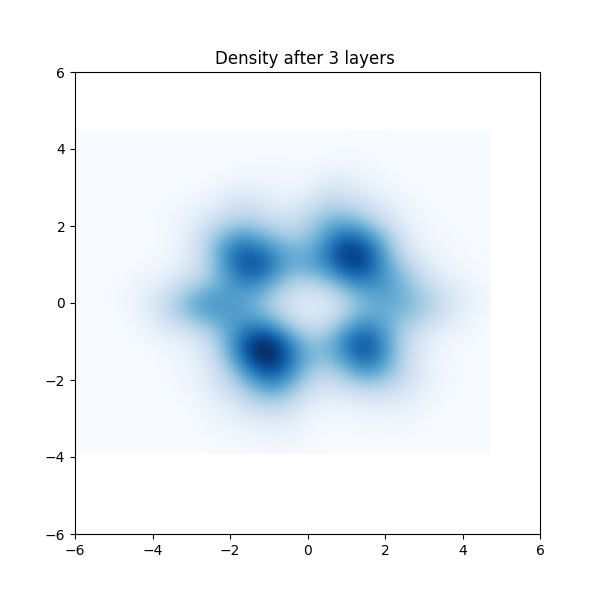}
\endminipage\hfill
\minipage{0.09\textwidth}
  \includegraphics[width=\linewidth]{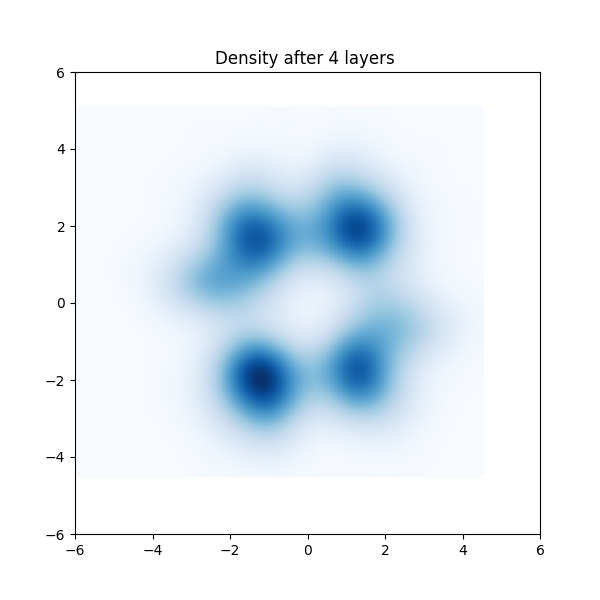}
\endminipage\hfill
\minipage{0.09\textwidth}
  \includegraphics[width=\linewidth]{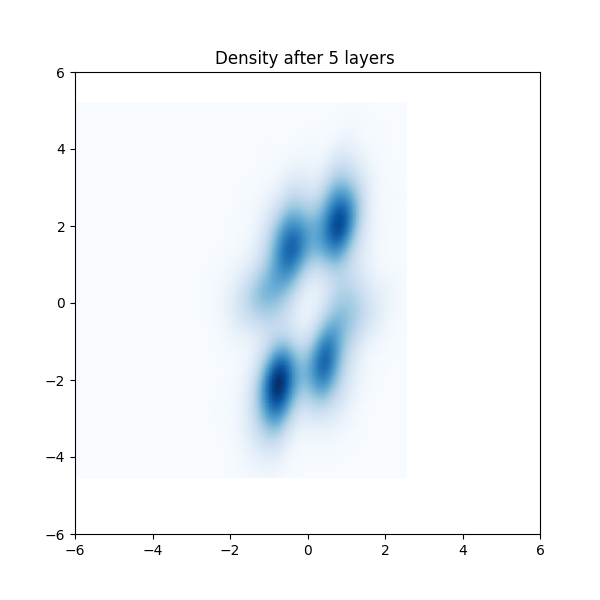}
\endminipage\hfill
\minipage{0.09\textwidth}
  \includegraphics[width=\linewidth]{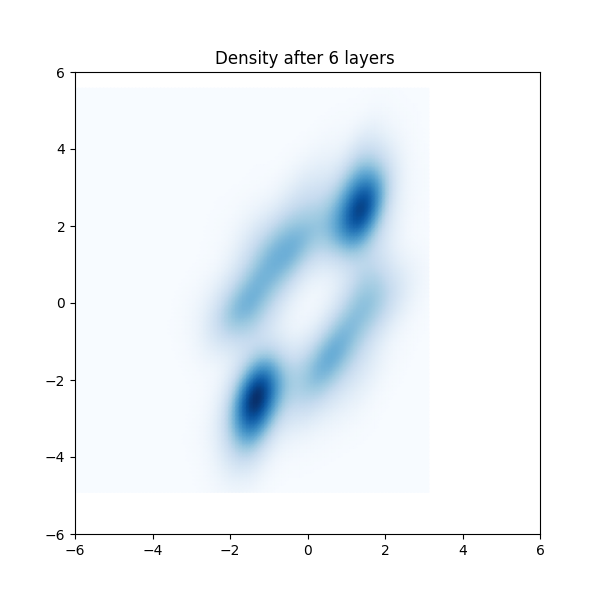}
\endminipage\hfill
\minipage{0.09\textwidth}
  \includegraphics[width=\linewidth]{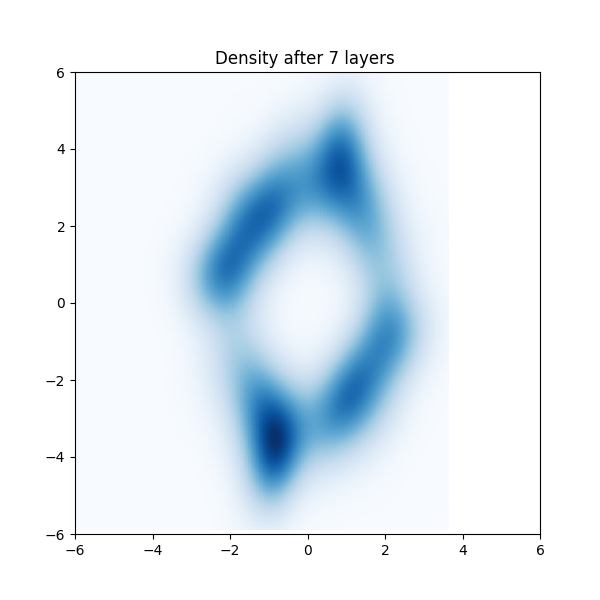}
\endminipage\hfill
\minipage{0.09\textwidth}
  \includegraphics[width=\linewidth]{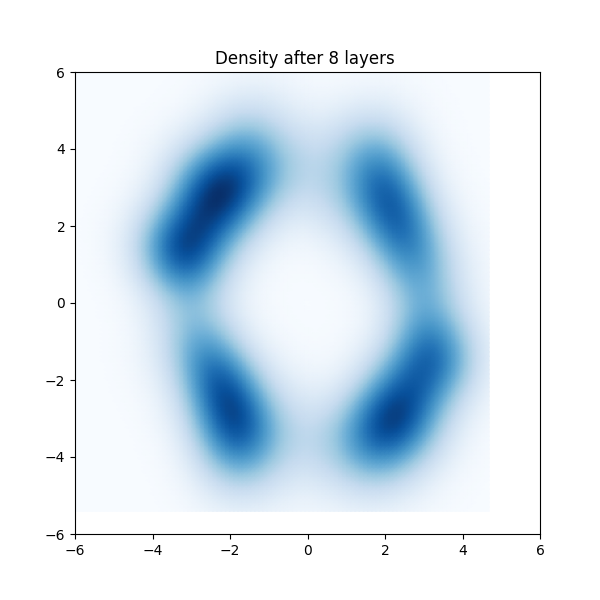}
\endminipage\hfill
\minipage{0.09\textwidth}
  \includegraphics[width=\linewidth]{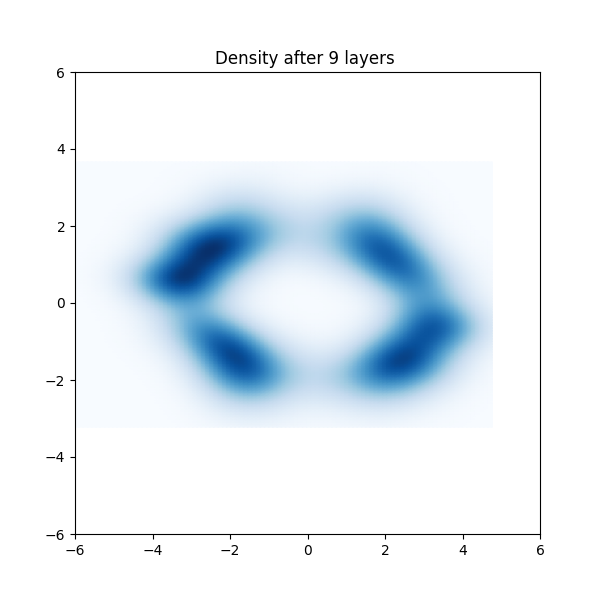}
\endminipage\hfill
\minipage{0.09\textwidth}
  \includegraphics[width=\linewidth]{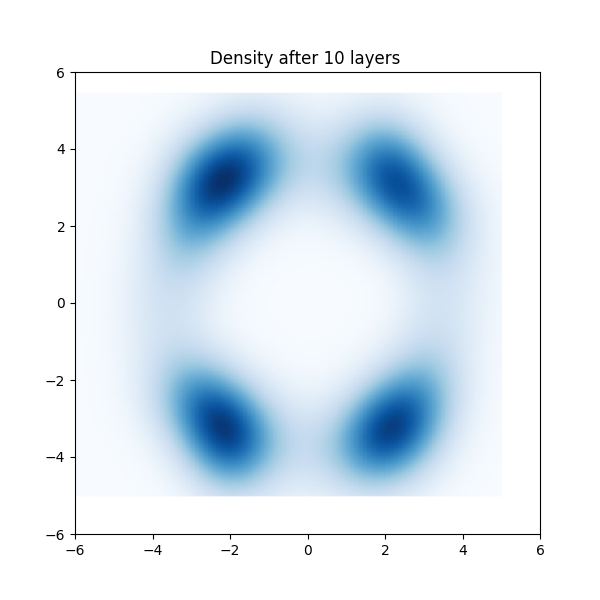}
\endminipage\hfill

\minipage{0.09\textwidth}
  \includegraphics[width=\linewidth]{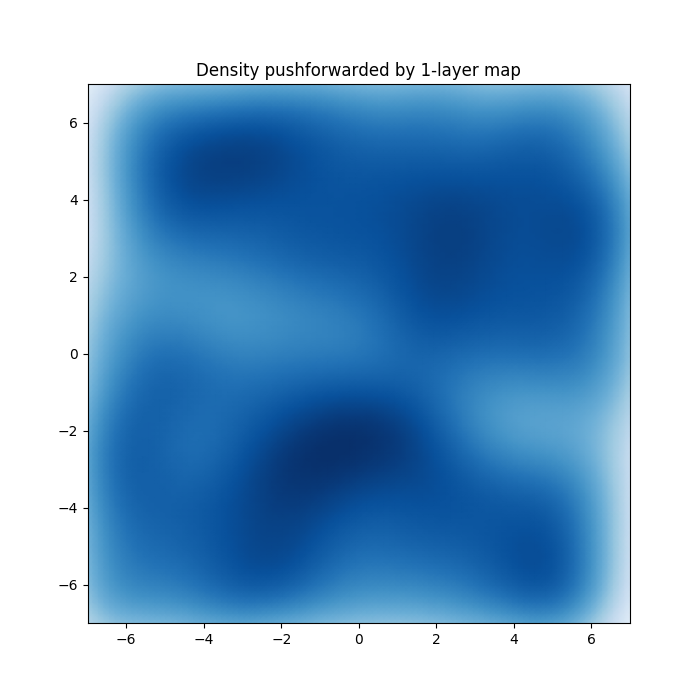}
\endminipage\hfill
\minipage{0.09\textwidth}
  \includegraphics[width=\linewidth]{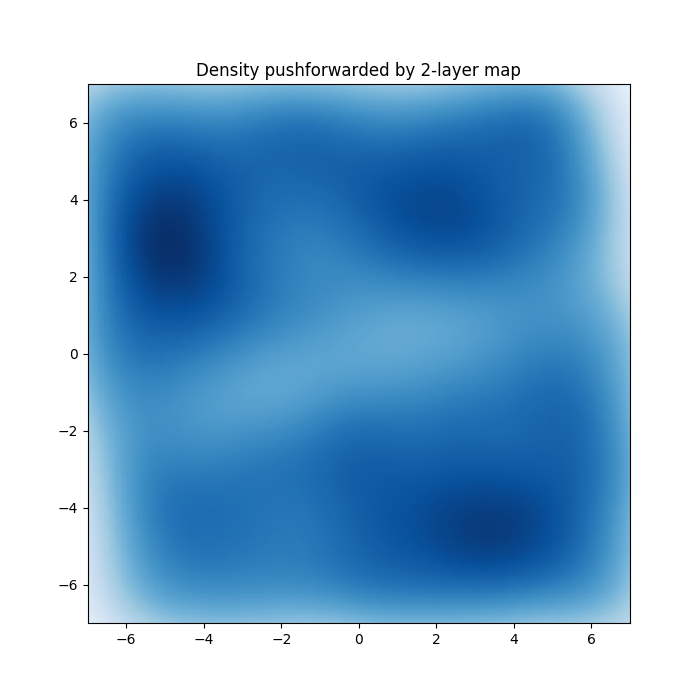}
\endminipage\hfill
\minipage{0.09\textwidth}
  \includegraphics[width=\linewidth]{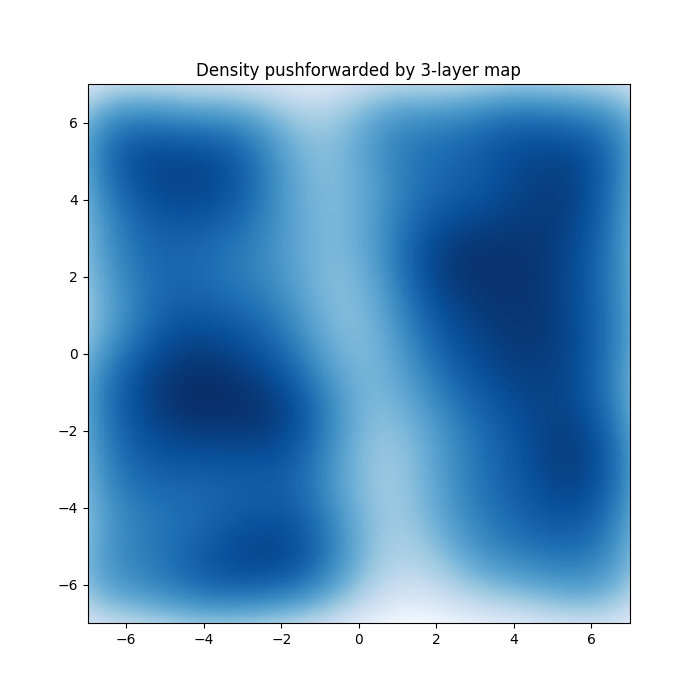}
\endminipage\hfill
\minipage{0.09\textwidth}
  \includegraphics[width=\linewidth]{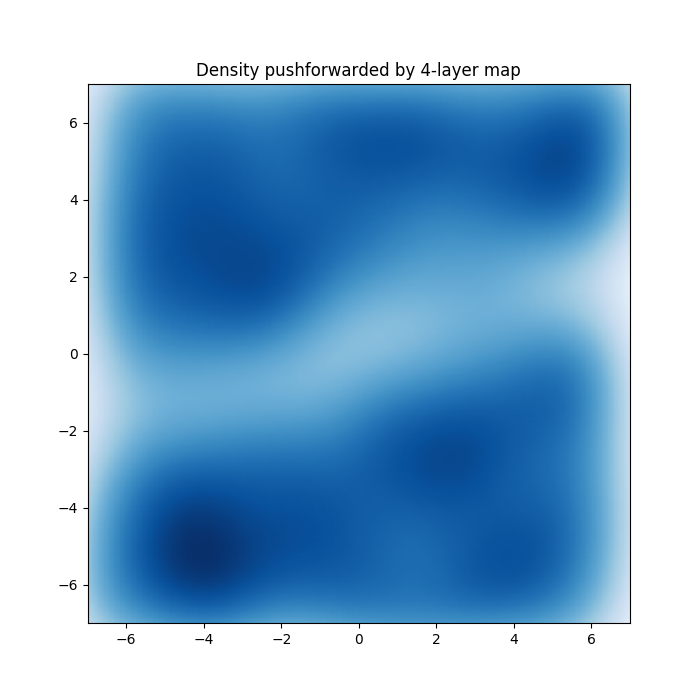}
\endminipage\hfill
\minipage{0.09\textwidth}
  \includegraphics[width=\linewidth]{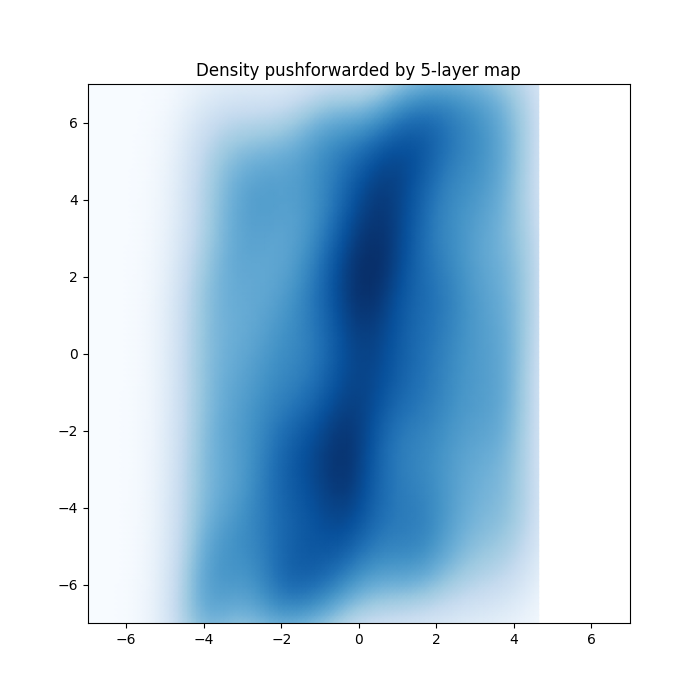}
\endminipage\hfill
\minipage{0.09\textwidth}
  \includegraphics[width=\linewidth]{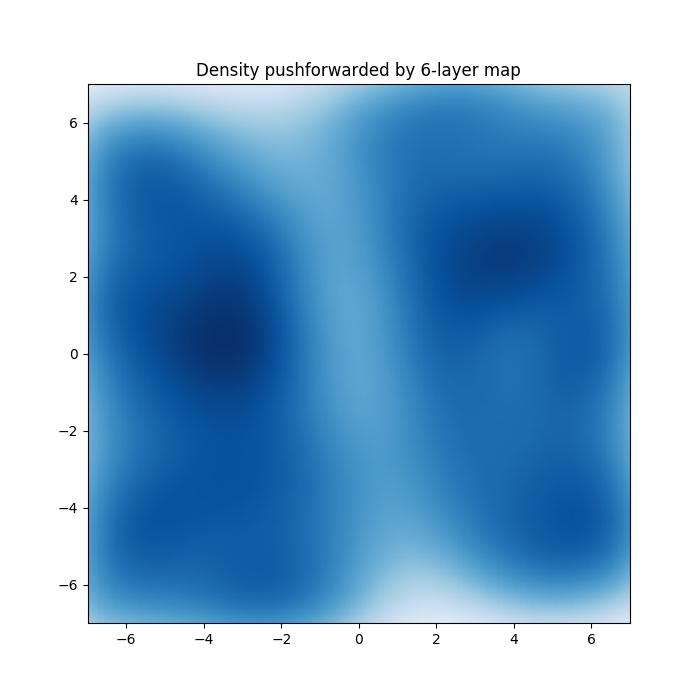}
\endminipage\hfill
\minipage{0.09\textwidth}
  \includegraphics[width=\linewidth]{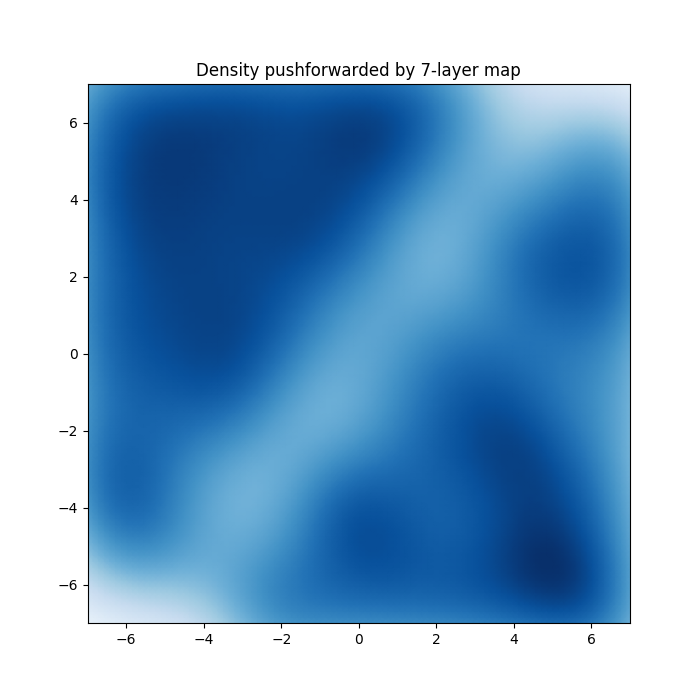}
\endminipage\hfill
\minipage{0.09\textwidth}
  \includegraphics[width=\linewidth]{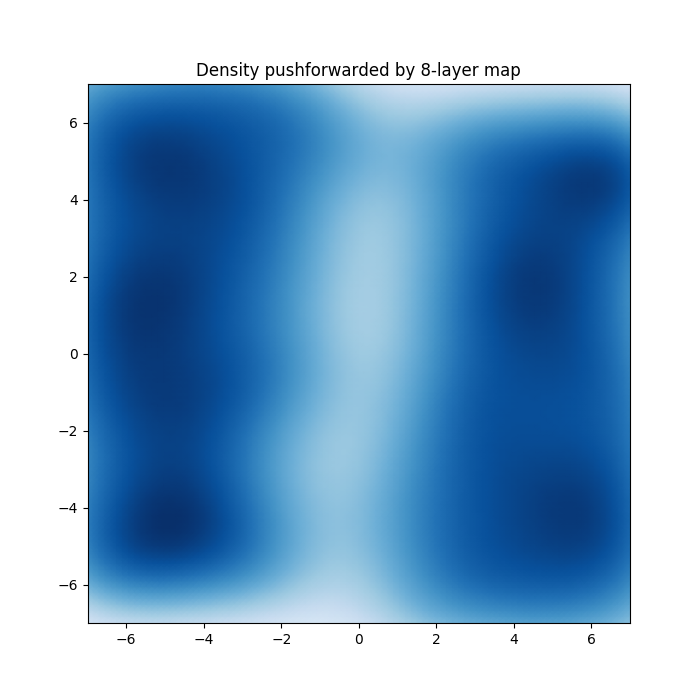}
\endminipage\hfill
\minipage{0.09\textwidth}
  \includegraphics[width=\linewidth]{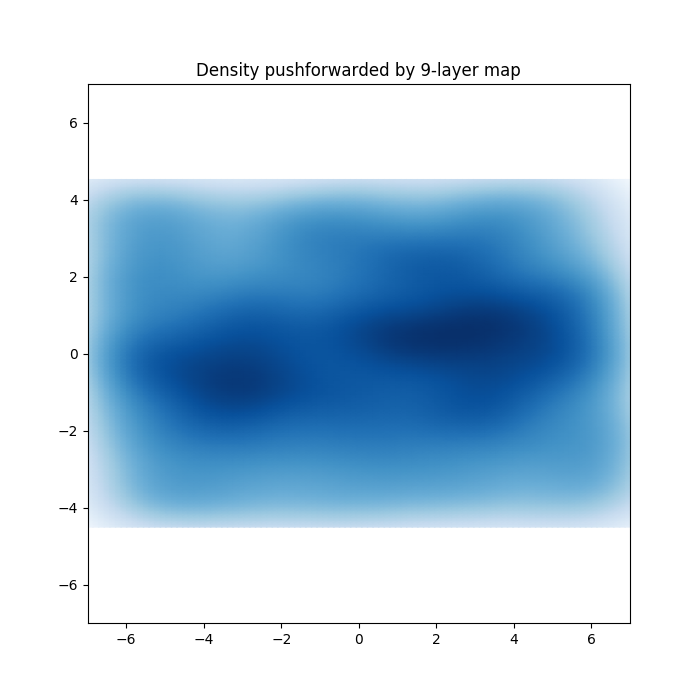}
\endminipage\hfill
\minipage{0.09\textwidth}
  \includegraphics[width=\linewidth]{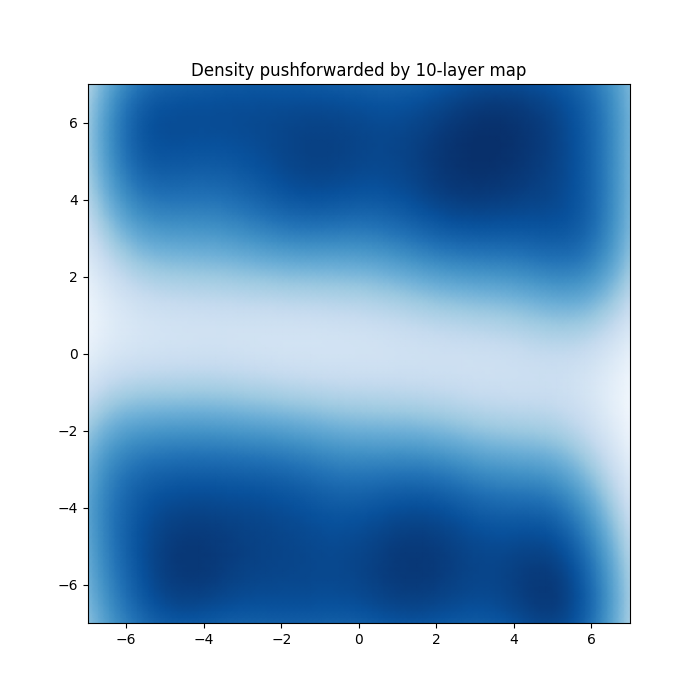}
\endminipage\hfill
\caption{Top row from left to right are the probability densities of distributions $f_{1 \sharp }p, (f_2\circ f_{1})_{\sharp}p,...,(f_{10}\circ f_9\circ ...\circ f_{1})_{\sharp } p$. The last image displays our target distribution. Bottom row displays the push-forward effect of each single-layer transformation $f_k$ ($1\leq k\leq 10$). }
\label{figure_pushforward gaussian}
\end{figure}

In a normalizing flow, the parameters are: $\theta=(w_1,u_1,b_1,...,w_K,u_K,b_K)$. The determinant of the Jacobi matrix of $T_\theta$, an important quantity for our schemes, can be explicitly computed by
\begin{equation*}
    \textbf{det}\left(\frac{\partial T_\theta(x)}{\partial x}\right) = \prod_{k=1}^K(1+\sigma'(w_k^{\textrm{T}} x_k + b_k)w_k^{\textrm{T}} u_k),
\end{equation*}
where $x_k = f_k\circ f_{k-1}\circ...\circ f_1(x)$.
Using the structure of normalizing flow, the logarithm of the density $\rho_\theta = { T_\theta }_{\sharp}p$ can be written as
\begin{equation}
  \log\rho_\theta(x) = \log p\circ T_\theta^{-1}(x) - \sum_{k=1}^K \log(1+\sigma'(w_k^{\textrm{T}}\tilde{x}_k)w_k^{\textrm{T}}u_k), \quad 
  \tilde{x}_k= f_k\circ...\circ f_1(T_\theta^{-1}(x)) = f_{k+1}^{-1}\circ...\circ f_K^{-1}(x). \label{explicit_logrho}
\end{equation}
Then we can explicitly write the relative entropy functional $H(\theta)$ defined in \eqref{relative entropy parameter} as,
\begin{equation}
  H(\theta) = \mathbb{E}_{\mathbf{X}\sim p} [ V(T_\theta(\mathbf{X}))+\mathcal{L}_\theta(\mathbf{X}) ], \label{NF relative entropy}
\end{equation}
where $\mathcal{L}_\theta $ is defined by,
\begin{equation*}
\mathcal{L}_\theta(\cdot) = \log p(\cdot) - \sum_{k=1}^K \log(1+\sigma'(w_k^{\textrm{T}}F_k(\cdot))w_k^{\textrm{T}}u_k) \quad F_k(\cdot) = f_k\circ f_{k-1}\circ...\circ f_1(\cdot). 
\end{equation*}
Once $H(\theta)$ is computed explicitly, so does the gradient $\nabla_\theta H(\theta)$.

In summary, we choose the normalizing flow because it has sufficient expression power to approximate complicated distributions on $\mathbb{R}^d$ \cite{rezende2015variational}, and the relative entropy $H(\theta)$ has a very concise form \eqref{NF relative entropy}, and its gradient can be conveniently computed.
\begin{remark}
We want to emphasize here that the normalizing flow is not the only choice for $T_\theta$. One may choose other network structures as long as they have sufficient approximation power and can compute the gradient of relative entropy efficiently.
\end{remark}

\subsection{Numerical scheme}\label{4.2 }
For the convenience of our presentation, at the beginning of this section, we first introduce the following definition. 
\begin{definition}[Orthogonal projection onto space of gradient fields]\label{Hodge Decomp}
  Consider vector field $\vec{v}\in L^2(\mathbb{R}^d;\mathbb{R}^d,\rho)$. Define $\textrm{Proj}_\rho[\vec{v}]=\nabla\psi$ as the $L^2(\rho)$-orthogonal projection of $\vec{v}$ onto the subspace of gradient fields. Where $\psi$ solves: 
  \begin{equation}
    \underset{\psi}{\min}  \left\{\int |\vec{v}(x)-\nabla\psi(x)|^2\rho(x)~dx \right\}.  \label{least square hodge decomp}
  \end{equation}
  Or equivalently $\psi$ solves $-\nabla\cdot(\rho(x)\nabla\psi(x)) = -\nabla\cdot(\rho(x)\vec{v}(x))$.
\end{definition}
\subsubsection{Proposed Double-Minimization Scheme}\label{section double min scheme}
Our numerical scheme is inspired by the following semi-implicit scheme of \eqref{wass_grad_flow_on_para_spc},
\begin{equation*}
  \frac{\theta_{k+1}-\theta_k}{h}=-G^{-1}(\theta_{k})\nabla_\theta H(\theta_{k+1}).
\end{equation*}
Equivalently, we can write it as a proximal algorithm,
\begin{equation}
 \theta_{k+1}=\underset{\theta}{\textrm{argmin}} \left\{\frac{1}{2} \langle\theta-\theta_k, G(\theta_k)(\theta-\theta_k)\rangle + h H(\theta) \right\}.\label{JKO_para_FPE}
\end{equation}
Recall $\boldsymbol{\Psi}$ as defined in Theorem \ref{thm_about_computing_metric_of_Theta}, if we denote $\psi = \boldsymbol{\Psi}^{\textrm{T}}(\theta-\theta_k)$, we have $\langle (\theta-\theta_k), G(\theta)(\theta-\theta_k)\rangle = \int|\nabla\psi|^2\rho_{\theta_k}~dx$ with $\psi$ solves the equation
\begin{equation}
  -\nabla\cdot(\rho_{\theta_k}\nabla\psi(x)) = -\nabla \cdot (\rho_{\theta_k}\partial_\theta T_{\theta_k}(T_{\theta_k}^{-1}(x))(\theta-\theta_k)). \label{psi elliptic pde}
\end{equation}
By definition \ref{Hodge Decomp}, $\nabla\psi$ is the orthogonal projection of vector field $\partial_\theta T_{\theta_k}(T_{\theta_k}^{-1}(\cdot))(\theta-\theta_k)$. Equivalently, $\psi$ can also be obtained by solving the least square problem \eqref{least square hodge decomp}.

Based on the observation that $\nabla\psi$ is obtained via orthogonal projection after replacing $\partial_\theta T_{\theta_k}(\theta-\theta_k)$ by finite difference $T_\theta-T_{\theta_k}$, we end up with the following double-minimization scheme for solving \eqref{JKO_para_FPE}
\begin{align}
\label{selected double min scheme}
\begin{split}
  & \min_\theta \left\{\left(\int \left(2 ~\nabla\phi(x)\cdot((T_\theta - T_{\theta_k})\circ T_{\theta_k}^{-1}(x))-|\nabla\phi(x)|^2 \right)\rho_{\theta_k}(x)~dx\right) + 2hH(\theta)\right\} \\ 
  & \textrm{with} ~ \phi ~\textrm{solves:}~~\min_\phi\left\{\int |\nabla\phi(x)-((T_\theta-T_{\theta_k})\circ T_{\theta_k}^{-1}(x))|^2\rho_{\theta_k}(x)~dx\right\}.
\end{split}
\end{align}

Scheme \eqref{selected double min scheme} has an equivalent saddle point optimization formulation
\begin{equation}
  \min_\theta \max_\phi \left\{\left( \int (2\nabla\phi(x)\cdot((T_\theta-T_{\theta_k})\circ T_{\theta_k}^{-1}(x)) -  |\nabla\phi(x)|^2)\rho_{\theta_k}(x) ~ dx \right) + 2hH(\theta)\right\}, \label{JKO_A}
\end{equation}
which can be directly derived from \eqref{JKO_para_FPE} via adjoint method. Their equivalence is explained in the next remark.
\begin{remark}
Here we briefly demonstrate the equivalence among the three schemes \eqref{JKO_para_FPE}, \eqref{selected double min scheme} and \eqref{JKO_A}.
Our target function $\frac{1}{2}\langle\theta-\theta_k, G(\theta_k)(\theta-\theta_k)\rangle + hH(\theta)$ can be formulated as
\begin{equation*}
  \int \frac{1}{2}|\nabla\psi(x)|^2\rho_{\theta_k}(x)~dx + hH(\theta) \quad \textrm{with the constraint:}~~ \psi~\textrm{solves}~ \eqref{psi elliptic pde}.
\end{equation*}
By introducing the dual variable $\phi$, and applying the adjoint method, we obtain
\begin{align}
   & \frac{1}{2}\langle\theta-\theta_k, G(\theta_k)(\theta-\theta_k)\rangle + hH(\theta) \nonumber\\
 = & \max_{\phi}\min_{\psi} \left\{ \int \frac{1}{2}|\nabla\psi(x)|^2\rho_{\theta_k}dx + hH(\theta) + \int \phi(x)(\nabla\cdot(\rho_{\theta_k}\nabla\psi(x)) - \nabla \cdot (\rho_{\theta_k}\partial_\theta T_{\theta_k}    (T_{\theta_k}^{-1}(x))(\theta-\theta_k)))~dx \right\} \nonumber\\
 = & \max_\phi \min_\psi \left\{ \int \left( \frac{1}{2}|\nabla\psi(x)|^2 - \nabla\phi(x)\cdot\nabla\psi(x) + \nabla\phi(x)\cdot \partial_\theta T_{\theta_k}(T_{\theta_k}^{-1}(x))(\theta-\theta_k)) \right)\rho_{\theta_k}(x)~dx + hH(\theta) \right\} \nonumber \\
 = & \max_{\phi} \left\{\int \left(-\frac{1}{2}|\nabla\phi(x)|^2 + \nabla\phi(x)\cdot\partial_\theta T_{\theta_k}(T_{\theta_k}^{-1}(x))(\theta-\theta_k)\right)\rho_{\theta_k}(x)~dx + hH(\theta)\right\}  \label{adjoint loss}
\end{align}
In implementation, we substitute $\partial_\theta T_{\theta_k}(\theta-\theta_k)$ by $T_{\theta}-T_{\theta_k}$ since the latter is tractable in computation. As a consequence, by substituting \eqref{adjoint loss} into \eqref{JKO_para_FPE} we obtain (by multiplying the entire function by $2$) the saddle scheme \eqref{JKO_A}. To verify the equivalence between \eqref{JKO_A} and \eqref{selected double min scheme}, we check the identity
\begin{align*}
  & \int (2\nabla\phi(x)\cdot((T_\theta-T_{\theta_k})\circ T_{\theta_k}^{-1}(x))-|\nabla\phi(x)|^2)\rho_{\theta_k}(x)~dx \\
  = &-\int |\nabla\phi(x)-(T_\theta-T_{\theta_k})\circ T_{\theta_k}^{-1}(x)|^2\rho_{\theta_k}(x)~dx + \underbrace{\int |(T_\theta-T_{\theta_k})\circ T_{\theta_k}^{-1}(x)|^2\rho_{\theta_k}(x)~dx}_{\textrm{Constant w.r.t.}~\phi}
\end{align*}
Thus the $\phi$-minimization process of \eqref{selected double min scheme} is equivalent to the $\phi$-maximization process of \eqref{JKO_A}. This leads to the equivalence between \eqref{selected double min scheme} and \eqref{JKO_A}. 
\end{remark}

\begin{remark}\label{max Epsilon(phi) as approx of wasserstein distance}
Our proposed schemes \eqref{selected double min scheme}, \eqref{JKO_A} can be viewed as an approximation to the JKO scheme \eqref{JKO scheme} with $\mathcal{F}$ being the relative entropy $H(\theta)$.
To see this, we denote 
\begin{equation*}
\mathcal{E}(\phi) = \int ( 2\nabla\phi(x)\cdot((T_\theta-T_{\theta_k})\circ T_{\theta_k}^{-1}(x))-|\nabla\phi(x)|^2)\rho_{\theta_k}(x)~dx,
\end{equation*}
and set $\hat{\psi}=\underset{\phi}{\textrm{argmax}}~\mathcal{E}(\phi)$. We let $\vec{v}_h(x) = \frac{1}{h}(T_\theta\circ T_{\theta_k}^{-1}(x)-x)$. Under mild conditions, one can show
\begin{equation}
  W_2^2(\rho_{\theta},\rho_{\theta_k}) = W_2^2((\textrm{Id}+h\vec{v}_h)_{\sharp} \rho_{\theta_k}, \rho_{\theta_k} ) = \int |\nabla\hat{\psi}|^2\rho_{\theta_k}~dx + o(h^2) = \max_\phi\mathcal{E}(\phi)+o(h^2).  
\end{equation}
By replacing $W_2^2(\rho_\theta,\rho_{\theta_k})$ in \eqref{JKO scheme} by its approximation $\max_\phi\mathcal{E}(\phi)$, we obtain scheme \eqref{selected double min scheme}, \eqref{JKO_A}. 
\end{remark}

Although \eqref{selected double min scheme} and \eqref{JKO_A} are mathematically equivalent, we use them for different purposes. The saddle scheme \eqref{JKO_A} is our main tool to investigate the theoretical properties of our proposed method in Section \ref{error of scheme}, because it better reflects the nature of our approximation method. In our  implementation, as discussed in Section \ref{section implementation }, we prefer the double minimization scheme \eqref{selected double min scheme}. Our experience indicates that \eqref{selected double min scheme} makes our code run more efficiently and behaves more stably than \eqref{JKO_A}.


\subsubsection{Local error of the proposed scheme   } \label{error of scheme}
We now analyze the local error of scheme \eqref{JKO_A} as well as \eqref{selected double min scheme} compared with the semi-implicit scheme \eqref{JKO_para_FPE}. Let us denote $\max_{\phi}\mathcal{E}(\phi)$ as $\widehat{W}_2^2(\theta,\theta_k)$ (Here $\widehat{W}_2$ is treated as an approximation of $L^2$-Wasserstein distance (remark \ref{max Epsilon(phi) as approx of wasserstein distance})). It is straightforward to verify $\widehat{W}_2(\theta,\theta')\geq 0$ and $\widehat{W}_2(\theta,\theta)=0$.
\noindent
Consider the following assumption,
\begin{equation}
  \widehat{W}_2^2(\theta,\theta') \geq l(|\theta-\theta'|) \quad \quad \textrm{for any} \quad \theta,\theta'\in\Theta.  \label{assumption on T_theta positive def}
\end{equation}
Here $l:\mathbb{R}_{\geq 0}\rightarrow \mathbb{R}_{\geq 0}$ satisfies $l(0)=0$. $l(r)$ is continuous, strictly increasing when $r\leq r_0$ for a positive $r_0$ and is bounded below by $\lambda_0>0$ when $r>r_0$. Notice that this assumption generally guarantees positive definiteness of $\widehat{W}_2$. Clearly, \eqref{assumption on T_theta positive def} only depends on the structure of $T_\theta$, and we expect that \eqref{assumption on T_theta positive def} holds for the neural networks used as pushforward maps, including the ones we used in this paper.

\begin{theorem}\label{thm:loc_err_scheme_JKO_A}
Suppose assumption \eqref{assumption on T_theta positive def} holds true for the class of pushforward maps $\{T_\theta\}$. Then the local error of scheme \eqref{JKO_A} is of order $h^2$, i.e., assume that $\theta_{k+1}$ is the optimal solution to \eqref{JKO_A}, then
\begin{equation}
  |\theta_{k+1}-\theta_k+hG(\theta_k)^{-1}\nabla_\theta H(\theta_{k+1})|\sim O(h^2).  \label{second order loc err}
\end{equation}
or equivalently: 
$\limsup_{h\rightarrow 0^+} \frac{|\theta_{k+1}-\theta_k+hG(\theta_k)^{-1}\nabla_\theta H(\theta_{k+1})|}{h^2}<+\infty$.
\end{theorem}

Before proving Theorem \ref{thm:loc_err_scheme_JKO_A}, we introduce a few additional notations. We define $\epsilon$ ball in parameter space as $B_\epsilon(\theta_k)=\{\theta~|~|\theta-\theta_k|\leq\epsilon\}$, let $T^{(i)}_\theta$ be the $i$th component ($1\leq i\leq d$) of map $T_\theta$. For fixed $\theta_k$ and $\epsilon>0$ small enough, we assume the following two quantities are finite
\begin{equation}
  L(\theta_k,\epsilon) = \sum_{i=1}^d ~\mathbb{E}_{x\sim p} \sup_{\theta\in B_\epsilon(\theta_k)} \left\{|\partial_\theta T_\theta^{(i)}(x)|^2 \right\}, \quad H(\theta_k,\epsilon) = \sum_{i=1}^d ~\mathbb{E}_{x\sim p} \sup_{\theta\in B_\epsilon(\theta_k)} \left\{\|\partial^2_{\theta\theta} T_\theta^{(i)}(x)\|_2^2 \right\}. \label{notation L,H}
\end{equation}
To prove Theorem \ref{thm:loc_err_scheme_JKO_A}, we need the following three lemmas:
\begin{restatable}{lemma}{lemmaBasicF}
\label{lemma_basics}
    Suppose we fix $\theta_0\in \Theta$, for arbitrary $\theta\in\Theta$ and $\nabla\phi\in L^2(\mathbb{R}^d;\mathbb{R}^d,\rho_{\theta_0})$ we consider
    \begin{equation}
     F(\theta, \nabla\phi~|~\theta_0) = \left( \int (2\nabla\phi(x)\cdot(T_\theta-T_{\theta_0})\circ T_{\theta_0}^{-1}(x) - |\nabla\phi(x)|^2) ~\rho_{\theta_0}(x) ~dx \right) + 2h  H(\theta).
     \label{F_theta_phi}
    \end{equation}
    Then $F(\theta,\nabla\phi~|~\theta_0)<\infty$, furthermore,  $F(\cdot, \nabla\phi~|~\theta_0)\in C^1(\Theta)$. We can compute
    \begin{equation}
      \partial_\theta F(\theta,\nabla\phi~|~\theta_0) =2\left(\int ~ \partial_\theta T_\theta(T_{\theta_0}^{-1}(x))^{\textrm{T}}~\nabla\phi(x)~\rho_{\theta_0}(x)~dx + h~\nabla_\theta H(\theta)\right).  \label{compute_partial_theta F} 
    \end{equation}
\end{restatable}
\begin{restatable}{lemma}{lemmaEnvelope}
\label{thm:Danskin}
Suppose we fix $\theta_0\in \Theta$ and define $J(\theta) = \underset{\nabla\phi\in L^2(\mathbb{R}^d;\mathbb{R}^d,\rho_{\theta_0})}{\sup} F(\theta,\nabla\phi~|~\theta_0) $. Then $J$ is differentiable. If we denote $\hat{\psi}_\theta = \underset{\phi}{\textrm{argmax}} \left\{ F(\theta,\nabla\phi~|~\theta_0) \right\}$, then
\begin{equation*}
  \nabla_\theta J(\theta) = \partial_\theta F(\theta,\nabla\hat{\psi}_\theta~|~\theta_0) = 2\left(\int ~ \partial_\theta T_\theta(T_{\theta_0}^{-1}(x))^{\textrm{T}}~\nabla\hat{\psi}_\theta(x)~\rho_{\theta_0}(x)~dx + h~\nabla_\theta H(\theta)\right).
\end{equation*}
\end{restatable}
\noindent
This lemma is an anology of the envelope theorem \cite{afriat1971theory} under our problem setting.

\begin{restatable}{lemma}{lemmaB}
\label{prior est |theta_k+1-theta_k|}
Under assumption\eqref{assumption on T_theta positive def}, the optimal solution of \eqref{JKO_A} $\theta_{k+1}$ satisfies, 
\begin{equation*}
  |\theta_{k+1}-\theta_k|\sim o(1)\quad \textrm{i.e.}~ ~ \lim_{h\rightarrow 0^+} |\theta_{k+1}-\theta_k|=0.
\end{equation*}
\end{restatable}
\noindent
This lemma provides {\it a prior} estimation of $|\theta_{k+1}-\theta_k|$.

We prove Lemma \ref{lemma_basics}, \ref{thm:Danskin} and 
\ref{prior est |theta_k+1-theta_k|} in Appendix \ref{pf section 4}.

\begin{proof}[ Proof of Theorem \ref{thm:loc_err_scheme_JKO_A}] 
Let us consider $F(\theta,\nabla\phi~|~\theta_k)$,
we denote $\nabla\hat{\psi}_\theta = \underset{\nabla\phi\in L^2(\mathbb{R}^d;\mathbb{R}^d,\rho_{\theta_k})}{\textrm{argmax}} \left\{ F(\theta,\nabla\phi~|~\theta_k) \right\}$. Then we can set
\begin{equation*}
\nabla\hat{\psi}_\theta = \textrm{Proj}_{\rho_{\theta_k}} [(T_\theta-T_{\theta_k})\circ T_{\theta_k}^{-1} ], \quad \textrm{and}\quad J(\theta) = \sup_{\nabla\phi\in L^2(\mathbb{R}^d;\mathbb{R}^d,\rho_{\theta_k})}F(\theta,\nabla\phi~|~ \theta_k)
\end{equation*}
Apply Lemma \ref{thm:Danskin}, we obtain:
\begin{equation*}
  \nabla_\theta J (\theta) = 2\left(\int ~ \partial_\theta T_\theta(T_{\theta_k}^{-1}(x))^{\textrm{T}}~\nabla\hat{\psi}_\theta(x)~\rho_{\theta_k}(x)~dx + h~\nabla_\theta H(\theta)\right).
\end{equation*}
Due to the differentiability of $J(\theta)$, at the optimizer $\theta_{k+1}$, the gradient must vanish, i.e.
\begin{equation}
 \left( \int \partial_\theta T_{\theta_{k+1}}(T_{\theta_k}^{-1}(x))^{\textrm{T}} ~\nabla\hat{\psi}_{\theta_{k+1}}(x)~\rho_{\theta_k}(x)~dx \right) + h\nabla_\theta H(\theta_{k+1}) = 0. \label{eq by Danskin}
\end{equation}
We use Taylor expansion at $\theta_{k+1}$ to get $T_{\theta_{k+1}}-T_{\theta_k}=\partial_\theta T_{\theta_k} (\theta_{k+1}-\theta_k) + R(\theta_{k+1},\theta_k)$, in which $R(\theta,\theta')(\cdot)\in L^2(\mathbb{R}^d;\mathbb{R}^m,\rho_{\theta_k})$, the $i$th entry of  $R(\theta,\theta')$ is $R_i(\theta, \theta' )(x) = \frac{1}{2}(\theta-\theta')^{\textrm{T}} \partial^2_{\theta\theta} T_{\tilde{\theta}_i(x)}^{(i)}(x)(\theta-\theta'),\; 1\leq i\leq m,$ where each $\tilde{\theta}_i(x)=\lambda_i(x)\theta+(1-\lambda_i(x))\theta'$ for some $\lambda_i(x) \in [0,1] $.
Then we can write: 
\begin{equation}
 \nabla\hat{\psi}_{\theta_{k+1}} = \textrm{Proj}_{\rho_{\theta_k}}[(T_{\theta_{k+1}}-T_{\theta_k})\circ T^{-1}_{\theta_k}] = \textrm{Proj}_{\rho_{\theta_k}}[\partial_\theta T_{\theta_k}\circ T_{\theta_k}^{-1} (\theta_{k+1}-\theta_k)]+\textrm{Proj}_{\rho_{\theta_k}}[R(\theta_{k+1},\theta_k)\circ T_{\theta_k}^{-1}]. \label{hodge proj of taylor expansion}
\end{equation}
On the other hand, 
\begin{equation}
  \partial_\theta T_{\theta_{k+1}} = \partial_{\theta} T_{\theta_k} +  r(\theta_{k+1},\theta_k). \label{mean value}
\end{equation}
Here $r(\theta,\theta')\in L^2(\mathbb{R}^d;\mathbb{R}^{d\times m},\rho_{\theta_k}),$ the $(i,j)$ entry of $r(\theta,\theta')(x)$ is $(\theta_{k+1}-\theta_k)^{\textrm{T}}\partial_\theta(\partial_{\theta_j} T^{(i)}_{\tilde{\theta}_{ij}(x)}(x)),\;  1\leq i\leq d, ~1\leq j\leq m,$
where each $\tilde{\theta}_{ij}(x) = \mu_{ij}(x)\theta_{k+1} + (1-\mu_{ij}(x))\theta_k $, for some $\mu_{ij}(x)\in(0,1)$.
Applying \eqref{mean value}, \eqref{hodge proj of taylor expansion} to \eqref{eq by Danskin}, we obtain
\begin{align}
 &\int \partial_\theta T_{\theta_k}(T_{\theta_k}^{-1}(x))^{\textrm{T}} \textrm{Proj}_{\rho_{\theta_k}}[\partial_\theta T_{\theta_k}\circ T_{\theta_k}^{-1}(x)(\theta_{k+1}-\theta_k)]~\rho_{\theta_k}(x)~dx \nonumber \\
 &+\int \partial_\theta T_{\theta_k}(T_{\theta_k}^{-1}(x))^{\textrm{T}} \textrm{Proj}_{\rho_{\theta_k}}[R(\theta_{k+1},\theta_k)\circ T_{\theta_k}^{-1}](x)~\rho_{\theta_k}(x)~dx \nonumber \\
 & + \int r(\theta_{k+1},\theta_k)(T_{\theta_k}^{-1}(x))^{\textrm{T}}\textrm{Proj}_{\rho_{\theta_k}}[(T_{\theta_{k+1}}-T_{\theta_k})\circ T^{-1}_{\theta_k}](x)~\rho_{\theta_k}(x)~dx \quad = -h\nabla_\theta H(\theta_{k+1}). \label{eq_for_local_err_anlys}
\end{align}
Recall definition of $\boldsymbol{\Psi}$ in Theorem \ref{thm_about_computing_metric_of_Theta}, use \eqref{loc err analys lemma 1} in lemma \ref{lemma:local err analys}, we know that the first term on the left hand side of \eqref{eq_for_local_err_anlys} equals 
\begin{equation*}
 \int \nabla\boldsymbol{\Psi}(x)\nabla\boldsymbol{\Psi}(x)^{\textrm{T}}(\theta_{k+1}-\theta_k)~\rho_{\theta_k}(x)~dx = G(\theta_k)(\theta_{k+1}-\theta_k).
\end{equation*}
By applying Cauchy–Schwarz inequality and \eqref{loc err analys lemma 2} in lemma \ref{lemma:local err analys}, we bound the $i$th entry of the second term in \eqref{eq_for_local_err_anlys} by:
\begin{align*}
 &\left( \int |\partial_{\theta}T_{\theta_k}^{(i)}(x)|^2~dp(x) \cdot \int\sum_{i=1}^d |(\theta_{k+1}-\theta_k)\partial^2_{\theta\theta} T_{\tilde{\theta}_i(x)}^{(i)}(x)(\theta_{k+1}-\theta_k)|^2~dp(x) \right)^{\frac{1}{2}}\\
& \leq \left( \mathbb{E}_{p}|\partial_{\theta}T_{\theta_k}^{(i)}(x)|^2  \cdot \mathbb{E}_p \left[\sum_{i=1}^d\|\partial^2_{\theta\theta}T_{\tilde{\theta}_i(x)}^{(i)}(x)\|_2\right] \right)^{\frac{1}{2}} |\theta_{k+1}-\theta_k|^2 \overset{\textrm{denote as}}{=} A^{(i)}|\theta_{k+1}-\theta_k|^2.
\end{align*}
To bound the third term in \eqref{eq_for_local_err_anlys}, we consider $T_{\theta_{k+1}}(x)-T_{\theta_k}(x)$, the $i$th entry can be written as
\begin{equation*}
T^{(i)}_{\theta_{k+1}}(x)-T_{\theta_k}^{(i)}(x) = \partial_\theta T_{\bar{\theta}_i(x)}(x)(\theta_{k+1}-\theta_k),
\end{equation*}
here $\bar{\theta}_i(x) = \zeta_i(x)\theta_{k+1}+(1-\zeta_i(x))\theta_k$ for some $\zeta_i(x)\in (0,1)$.
The $i$th entry of the third term of \eqref{eq_for_local_err_anlys} can be bounded by:
\begin{align*}
 & \left(\int\sum_{i=1}^d |(\theta_{k+1}-\theta_k)^{\textrm{T}}\partial_{\theta\theta}T_{\tilde{\theta}_{ij}(x)}^{(i)}(x)|^2~dp(x) \cdot \int |T^{(i)}_{\theta_{k+1}}(x)-T_{\theta_k}^{(i)}(x)|^2~dp(x) \right)^{\frac{1}{2}} \\
 & \leq \left( \mathbb{E}_p\left[\sum_{i=1}^d\|\partial^2_{\theta\theta}T_{\tilde{\theta}_{ij}(x)}(x)\|_2^2\right] \cdot \mathbb{E}_p |\partial_{\theta} T^{(i)}_{\bar{\theta}_i(x)}(x)|^2 \right)^{\frac{1}{2}} |\theta_{k+1}-\theta_k|^2 \overset{\textrm{denote as}}{ = } B^{(i)}|\theta_{k+1}-\theta_k|^2.
\end{align*}
We denote $A\in\mathbb{R}^m$ with entries $A^{(i)}$, $1\leq i\leq m$ and similarly $B\in\mathbb{R}^m$ with entries $B^{(i)}$, $1\leq i\leq m$. \eqref{eq_for_local_err_anlys}  leads to the following inequality,
\begin{equation*}
  |\theta_{k+1}-\theta_k+hG(\theta_k)^{-1}\nabla_\theta H(\theta_{k+1})|\leq \|G(\theta_k)^{-1}\|_2(|A|+|B|)~|\theta_{k+1}-\theta_k|^2.
\end{equation*}
As we have shown in Lemma \ref{prior est |theta_k+1-theta_k|} that $|\theta_{k+1}-\theta_k|\sim o(1)$ for any $\epsilon>0$ when step size $h$ is small enough, we always have $\theta_{k+1}\in B_\epsilon(\theta_k)$. Recall the notations in \eqref{notation L,H}, we have $|A|,|B|\leq \sqrt{L(\theta_k,\epsilon)H(\theta_k,\epsilon)}$. Thus we have
\begin{equation*}
  |\theta_{k+1}-\theta_k+hG(\theta_k)^{-1}\nabla_\theta H(\theta_{k+1})|\leq 2\sqrt{L(\theta_k,\epsilon)H(\theta_k,\epsilon)}\|G(\theta_k)^{-1}\|_2 |\theta_{k+1}-\theta_k|^2. 
\end{equation*}
Denote $\theta_{k+1}-\theta_k=\eta$, $G (\theta_k)^{-1}\nabla_\theta H(\theta_{k+1}) = \xi$ and $C=2\sqrt{L(\theta_k,\epsilon)H(\theta_k,\epsilon)}\|G(\theta_k)^{-1}\|_2$, the previous inequality is
\begin{equation}
  |\eta~ - ~ h~ \xi|\leq C|\eta|^2. \label{short form eq err analys}
\end{equation}
Since $|\eta- h \xi|\geq |\eta|-h|\xi|$, we have 
\begin{equation}
  C |\eta|^2\geq |\eta|-h|\xi|.  \label{short form eq err analysis 2}
\end{equation}
Solving \eqref{short form eq err analysis 2} gives 
\begin{equation*}
|\eta|\leq \frac{2|\xi|h}{1+\sqrt{1-4C|\xi|h}} \quad \textrm{or} \quad |\eta|>\frac{1+\sqrt{1-4Ch|\xi|}}{2C}.
\end{equation*}
The second inequality leads to $|\theta_{k+1}-\theta_k|>\frac{1}{2C}$ for any $h>0$, which avoids $|\theta_{k+1}-\theta_k|\sim o(1)$. Thus, when $h$ is sufficiently small, we have 
\begin{equation}
|\eta|\leq \frac{2|\xi|h}{1+\sqrt{1-4C|\xi|h}}.\label{local err analys up bound for eta}
\end{equation}
Combining \eqref{local err analys up bound for eta} and \eqref{short form eq err analys}, we have:
\begin{equation}
  |\theta_{k+1}-\theta_k+hG(\theta_k)^{-1}\nabla_\theta H(\theta_{k+1})|\leq \frac{4~C~|\xi|^2}{(1+\sqrt{1-4C|\xi|h})^2}~h^2 \leq 4 C|\xi|^2h^2.   \label{loc err}
\end{equation}
This proves the result.
\end{proof}

\begin{remark}
One may be aware of the relation between the positive definite condition \eqref{assumption on T_theta positive def} and the positive definiteness of the metric tensor $G(\theta_k)$. A positive definite $G(\theta)$ guarantees the inequality $\widehat{W}_2^2(\theta,\theta')\geq C|\theta-\theta'|^2$ for $\theta'\in B_{r_0}(\theta)$ ($r_0$ depends on $\theta$ is small enough). However, we are not able to bound $\widehat{W}_2^2(\theta,\theta')$ from below when $|\theta-\theta'|>r_0$. On the other hand, \eqref{assumption on T_theta positive def} is a locally weaker condition than positive definiteness of $G(\theta)$.  
\end{remark}

\subsubsection{Implementation}\label{section implementation }
As mentioned in \ref{section double min scheme}, we prefer double-minimization scheme \eqref{selected double min scheme} than saddle scheme \eqref{JKO_A}. We will thus implement scheme \eqref{selected double min scheme}. Let us denote
\begin{align}
 & J(\theta) = \left(\int \left(2 ~\nabla\hat{\psi}(T_{\theta_k}(x))\cdot((T_\theta(x) - T_{\theta_k}(x)))-|\nabla\hat{\psi}(T_{\theta_k}(x))|^2 \right)dp(x)\right) + 2hH(\theta)  \label{J(theta)}\\
 & \textrm{with}~~\hat{\psi}=\underset{\phi}{\textrm{argmin}}\left\{\int |\nabla\phi(T_{\theta_k})-(T_{\theta}(x)-T_{\theta_k}(x))|^2dp(x) \right\}  \label{Proj selected }
\end{align}
We then solve ODE \eqref{wass_grad_flow_on_para_spc} at $t_k$ by solving
\begin{equation}
    \theta_{k+1} = \underset{\theta}{\textrm{argmin}}~
      J(\theta),  \label{new formulation of JKO_A}
\end{equation}
Here we provide some detailed discussion on our implementation.
\begin{itemize}
    \item In our numerical computation, we approximate $\phi$ by $\psi_\nu:M\rightarrow\mathbb{R}$, which is a ReLU neural network \cite{glorot2011deep}. Here $\nu$ denotes the parameter vector of the network $\psi_\nu$.
     We know that in this case, $\psi_\nu$ is a piece-wise affine function and its gradient $\nabla\psi_\nu(\cdot)$ forms a piece-wise constant vector field. 
    
    \item The entire procedure of solving \eqref{new formulation of JKO_A} can be formulated as nested loops:
    \begin{itemize}
        \item (inner loop) Every inner loop aims at solving \eqref{Proj selected } on ReLU functions $\psi_\nu$, i.e. solving:
        \begin{equation}
            \min_\nu \left\{\mathbb{E}_{\boldsymbol{X}\sim p } |\nabla\psi_\nu(T_{\theta_k}(\boldsymbol{X}))-(T_\theta(\boldsymbol{X})-T_{\theta_k}(\boldsymbol{X}))|^2 \right\}.  \label{inner optimization problem}
        \end{equation}
        One can use Stochastic Gradient Descent (SGD) methods like RMSProp \cite{ruder2016overview} or Adam \cite{kingma2014adam} with learning rate $\alpha_{\textrm{in}}$ to deal with this inner loop optimization. In our implementation, we will stop after $M_{\textrm{in}}$ iterations. Let us denote the optimal $\nu$ in each inner loop as $\hat{\nu}$;
        \item (outer loop) We apply similar SGD method to $J(\theta)$: using Lemma \ref{thm:Danskin}, we are able to compute $\nabla_\theta J(\theta)$ as:
        \begin{equation*}
            \nabla_\theta J(\theta) = \partial_\theta \left(  \left( \int 2\nabla\hat{\psi}(x)\cdot(T_{\theta}\circ T_{\theta_k}^{-1}(x))\rho_{\theta_k}(x)~dx \right) + 2hH(\theta)\right).
        \end{equation*}
        If we treat optimal $\hat{\psi}$ as $\psi_{\hat{\nu}}$, what we need to do in each outer loop is to consider:
        \begin{equation}
           \tilde{J}(\theta) = \mathbb{E}_{\boldsymbol{X}\sim p}~ 2[\nabla\psi_{\hat{\nu}}(T_{ \theta_k }(\boldsymbol{X}))\cdot T_\theta(\boldsymbol{X})] + 2 h [V(T_{\theta}(\boldsymbol{X}))+\mathcal{L}_\theta(\boldsymbol{X})]   \label{outer optimization problem}
        \end{equation} 
        and update $\theta$ for one step by our chosen SGD method with learning rate $\alpha_{\textrm{out}}$ applied to optimize $\tilde{J}(\theta)$. In our actual computation, we will stop the outer loop after $M_{\textrm{out}}$ iterations.
   \end{itemize}
   
   \item We now present the entire algorithm for computing \eqref{wass_grad_flow_on_para_spc} based on the scheme \eqref{selected double min scheme} in Algorithm \ref{semi-backward_2}. This algorithm contains the following parameters: $T,N;M_{\textrm{out}}, K_{\textrm{out}},\alpha_{\textrm{out}}; M_{\textrm{in}},K_{\textrm{in}} ,\alpha_{\textrm{in}}.$ Recall we set reference distribution $p$ as standard Gaussian on $M=\mathbb{R}^d$.
   \begin{algorithm}
   \caption{Computing \eqref{wass_grad_flow_on_para_spc} by scheme \eqref{JKO_A} on the time interval $[0,T]$ }\label{semi-backward_2}
   \begin{algorithmic}[1]
   \STATE Initialize $ \theta $
   \FOR{$i=1,...,N$}
   \STATE Save current parameter value to $\theta_0$: $\theta_0=\theta$
   \FOR{$j=1,...M_{\textrm{out}}$}
   \FOR{$p=1,...,M_{\textrm{in}}$}
   \STATE Sample $\{\mathbf{X}_1,...,\mathbf{X}_{K_{\textrm{in}}}\}$ from $p$
   \STATE Apply one SGD (Adam) step with learning rate $\alpha_{\textrm{in}}$ to loss function of variable $\lambda$.
   \begin{equation*}
      \frac{1}{K_{\textrm{in}}}\left(\sum_{k=1}^{K_{\textrm{in}}} |\nabla\psi_\nu(T_{\theta_0}(\mathbf{X}_k))-(T_\theta(\mathbf{X}_k)-T_{\theta_0}(\mathbf{Y}_k))|^2 \right)    
   \end{equation*}
   \ENDFOR
   \STATE Sample $\{\mathbf{X}_1,...,\mathbf{X}_{K_{\textrm{out}}} \}$ from $p$ 
   \STATE Apply one SGD (Adam) step with learning rate $\alpha_{\textrm{out}}$ to loss function of variable $\theta$.
   \begin{equation*}
      \frac{1}{K_{\textrm{out}}}\left(\sum_{k=1}^{K_{\textrm{out}}} 2[\nabla\psi_\nu(T_{\theta_0}(\boldsymbol{X}_k))\cdot T_\theta(\boldsymbol{X}_k)] + 2 h [V(T_{\theta}(\boldsymbol{X}_k))+\mathcal{L}_\theta(\boldsymbol{X}_k)]\right)
   \end{equation*}
   \ENDFOR
   \STATE Set $\theta_i = \theta$
   \ENDFOR
   \STATE The sequence of probability densities $\{{T_{\theta_0}}_{\sharp}p,{T_{\theta_1}}_{\sharp}p,...,{T_{\theta_N}}_{\sharp}p\}$ will be the numerical solution of $\{\rho_{t_0},\rho_{t_1},...,\rho_{t_N}\}$, where $t_i=i\frac{T}{N}$ ($i=0,1,...,N-1,N$). Here $\rho_t$ solves the original Fokker--Planck equation \eqref{FPE}.
   \end{algorithmic}
   \end{algorithm}
\end{itemize}

\begin{remark}[Rescaling]
In our implementation, $T_\theta(\boldsymbol{X})-T_{\theta_k}(\boldsymbol{X})$ is usually of order $O(\alpha_{\textrm{out}})$, which is a small quantity. We can rescale it so that each inner loop can be solved in a more stable way with larger stepsize (learning rate). That is to say, we choose some small $\epsilon  \sim O(\alpha_{\textrm{out}})$ and consider
\begin{equation}
  \min_\theta \max_\phi ~\left\{ \underbrace{\left(\int (2\nabla\phi(x) \cdot \left(\frac{1}{\epsilon}(T_\theta-T_{\theta_k})\circ T_{\theta_k}^{-1}(x) \right) - |\nabla\phi(x)|^2)\rho_{\theta_k}(x)~dx\right)}_{\mathcal{E}_\epsilon(\phi)} + \frac{2h}{\epsilon^2}H(\theta)\right\}. \label{epsilon-adj scheme}
\end{equation}
We can also check
\begin{equation*}
     \textrm{argmax} ~\mathcal{E}_\epsilon (\phi) =   \textrm{Proj}_{\rho_{\theta_k}}[\frac{1}{\epsilon}(T_\theta-T_{\theta_k})\circ T_{\theta_k}^{-1}]
      = \frac{1}{\epsilon}\textrm{Proj}_{\rho_{\theta_k}}[(T_\theta-T_{\theta_k})\circ T_{\theta_k}^{-1}] = \frac{1}{\epsilon}\textrm{argmax}~\mathcal{E}(\phi).
\end{equation*}
Using this, we are able to verify $\max_\phi ~ \mathcal{E}_\epsilon(\phi) = \frac{1}{\epsilon^2}\max_\phi~\mathcal{E}(\phi)$. Thus the optimal solution of \eqref{epsilon-adj scheme} is
\begin{equation*}
  \textrm{argmin}_{\theta} \left\{\frac{1}{\epsilon^2}\max_\phi~\mathcal{E}(\phi) + \frac{2h}{\epsilon^2}H(\theta)\right\} = \textrm{argmin}_\theta \left\{ \max_\phi~\mathcal{E}(\phi) + 2hH(\theta) \right\}
\end{equation*}
This shows that the equivalence between the modified scheme \eqref{epsilon-adj scheme} and the original scheme \eqref{JKO_A}.

In our actual implementation, we still prefer double-minimization scheme. We solve 
\begin{equation}
    \min_\nu \left\{\mathbb{E}_{\boldsymbol{X}\sim p } \left|\nabla\psi_\nu(T_{\theta_k}(\boldsymbol{X}))-\left(\frac{T_\theta(\boldsymbol{X})-T_{\theta_k}(\boldsymbol{X})}{\epsilon}\right)\right|^2 \right\},  \label{modified inner optimization problem}
\end{equation}
instead of \eqref{inner optimization problem} in each inner loop and set: 
\begin{equation}
           \tilde{J}(\theta) = \mathbb{E}_{\boldsymbol{X}\sim p}~ 2[\nabla\psi_{\hat{\nu}}(T_{ \theta_k }(\boldsymbol{X}))\cdot T_\theta(\boldsymbol{X})] + \frac{2h}{\epsilon} [V(T_{\theta}(\boldsymbol{X}))+\mathcal{L}_\theta(\boldsymbol{X})] \label{modified outer optimization problem}
\end{equation}
in each outer loop. In actual experiments, we set $\epsilon=\alpha_{\textrm{out}}$.
\end{remark}

\begin{remark}[Sufficiently large sample size]\label{rmk:mention large sample}
It is worth mentioning that the sample size $K_{\textrm{in}},K_{\textrm{out}}$ in each SGD step (especially $K_{\textrm{in}}$) should be chosen reasonably large so that the inner optimization problem can be solved with enough accuracy. 
In our practice, we usually choose $K_{\textrm{in}}=K_{\textrm{out}}=\max\{1000,300d\}$. Here $d$ is the dimension of sample space. This is very different from the small batch technique applied to training neural network in deep learning  \cite{masters2018revisiting}.
\end{remark}

\begin{remark}[Using fixed samples]
  Our numerical experiments indicate that the same samples can be used for both the inner and outer iterations, which may reduce the computational cost of our original algorithm. 
\end{remark}

\section{Asymptotic properties and error estimations}\label{Analysis}
In this section, we establish numerical analysis for the parametric Fokker--Planck equation \eqref{wass_grad_flow_on_para_spc}.

\subsection{An important quantity}\label{5.1 }

Before our analysis, we introduce an important quantity that plays an essential role in our numerical analysis. Let us recall the optimal value of the least square problem \eqref{important_lemma} in Theorem \ref{lemma_submfld_grad_err} of section \ref{section 3.2}, or equivalently \eqref{important_lemma_equiv} of section \ref{section 3.2}, \eqref{expectation of l2 discrepancy particle} of section \ref{particle level explanation}. If we denote the upper bound of all possible values to be $\delta_0$, i.e. 
\begin{equation}
  \delta_0 = \sup_{\theta\in\Theta} ~\underset{\xi\in \mathbb{R}^m}{\min} \left\{\int \left|\sum_{k=1}^M\xi_k\nabla\psi_k(x)-\nabla\left( V(x)+{D}\log\rho_{\theta}(x) \right)\right|^2\rho_\theta(x)dx \right\},
  \label{def_delta}
\end{equation} 
where $\psi_k$ are solutions to \eqref{Hodge Dcom} in Theorem \ref{thm_about_computing_metric_of_Theta}, then this quantity provides crucial error bound between our parametric equation and original equation in the forthcoming analysis. Ideally, we hope $\delta_0$ to be sufficiently small. And this can be guaranteed if the neural network we select has universal approximation power. $\delta_0$ can be bounded by another constant with more approachable form 
\begin{equation}
  \hat{\delta}_0 = \sup_{\theta\in\Theta} ~\underset{\xi\in \mathbb{R}^m}{\min} \left\{\int \left|\sum_{k=1}^M\xi_k\frac{\partial T_\theta(x)}{\partial\theta_k}-\nabla\left( V(x)+{D}\log\rho_{\theta}(x) \right)\right|^2\rho_\theta(x)dx \right\}. \label{def_hat_delta}
\end{equation}
By \eqref{loc err analys lemma 2} of Lemma \ref{lemma:local err analys}, one can verify $\delta_0\leq \hat{\delta}_0$. From \eqref{def_hat_delta}, we observe that $\hat{\delta}_0$ is determined by the optimal linear combination of $\{\frac{\partial T_\theta}{\partial \theta_k}\}_{k=1}^M$ to approximate the vector field $\nabla( V + D\log\rho_\theta )$. 
One may understand this approximation from three different aspects. 
\begin{itemize}
  \item If $T_\theta$ is chosen as a linear combination of basis functions, i.e. $T_\theta(x)=\sum_{k=1}^M\theta_k \Phi_k(x)$, we can give an explicit estimate on $\hat{\delta}_0$. For example, if $\Phi_k(x)$ is picked as the Fourier basis and $\nabla(V+D\log\rho_\theta)\in H^s$ ($s>1$), the classical spectral method theory can be applied to obtain an estimate $\hat{\delta}_0 = O(M^{-s})$ \cite{patera1984spectral, szabo1991finite}. If Radial Basis Function is selected, an related approximation bounded can be obtained too \cite{cabrera2013vector}. 
  
  \item Having a small value for $\hat{\delta}_0$ as well as $\delta_0$ is equivalent to find a suitable $T_\theta$ such that a specific vector field $\nabla(V+D\log\rho_\theta)$ can be accurately approximated in our estimate. In other words, when neural networks are used for $T_\theta$, one needs to pick a neural network structure such that it can approximate $\nabla(V+D\log\rho_\theta)$ well. This seems to be an easier question than the task for the so-called universal approximation theory for neural networks, which requires $T_\theta$ to approximate an arbitrary function in a space.
  
  \item In our implementation, we use Normalizing Flows, a special type of deep neural networks. Our numerical examples seem to show promising performance. 
  In the existing literature, although there are several references providing the universal approximation power of neural networks \cite{yarotsky2017error, daubechies2019nonlinear}, the results are mainly focused on general ReLU networks and on the approximation power of function value, which is different from our case. To the best of our knowledge, there is no existing study discussing explicit bounds for vector field approximation by deep neural networks. We believe that the question of how $\delta_0$ or $\hat{\delta}_0$ explicitly depends on the structure of $T_\theta$ is a fundamental research problem that deserves careful investigations.
  \end{itemize}

It is also worth mentioning that $\delta_0$ is used for {\it a priori} estimate in this section, because we don't know the exact trajectory of $\{\theta_t\}$ when solving ODE \eqref{wass_grad_flow_on_para_spc}, and we take supremum over $\Theta$ to obtain $\delta_0$. Once solved for $\{\theta_t\}$, denote $\mathcal{C}$ as the set covering its trajectory, i.e. 
\begin{equation}
  \mathcal{C}=\{\theta~|~\exists ~ t\geq 0,~\textrm{s.t.} ~\theta=\theta_t\} \label{set of traj}
\end{equation}
We define another quantity $\delta_1$:
\begin{equation}
  \delta_1 = \sup_{\theta\in\mathcal{C}} ~\underset{\xi\in \mathcal{T}_{\theta}\Theta}{\min} \left\{\int |\nabla\boldsymbol{\Psi}( T_{\theta}(x))^{\textrm{T}}\xi-\nabla\left( V+{D}\log\rho_{\theta} \right)\circ T_{\theta}(x)|^2 ~ dp(x) \right\}. \label{def_delta1}
\end{equation}
Clearly, we have $\delta_1 \le \delta_0$.
We can obtain corresponding 
{\it posterior} estimates for the asymptotic convergence and error analysis by replacing $\delta_0$ with $\delta_1$.

\subsection{Asymptotic Convergence Analysis}\label{5.2 }
In this section, we consider the solution $\{\theta_t\}_{t\geq 0}$ of our parametric Fokker--Planck equation \eqref{wass_grad_flow_on_para_spc}. We define:
\begin{equation*}
  \mathcal{V} = \left\{V~\Bigg|
  \substack{\text{$V\in\mathcal{C}^2(\mathbb{R}^d)$, $V$ can be decomposed as: $V=U+\phi$, with $U,\phi\in\mathcal{C}^2(\mathbb{R}^d)$;} \\ \text{$\nabla^2 U\succeq K I 
  $ with $K>0$ and $\phi\in L^\infty(\mathbb{R}^d)$}}\right\}
\end{equation*}
As we know, for the Fokker--Planck equation \eqref{FPE}, when the potential $V\in\mathcal{V}$, $\{\rho_t\}$ will converge to the Gibbs distribution $\rho_* = \frac{1}{Z_{D}}e^{-V(x)/{D}}$ as $t\rightarrow \infty$ under the measure of KL divergence \cite{holley1987logarithmic}. For \eqref{wass_grad_flow_on_para_spc}, we wish to study its asymptotic convergence property. We come up with the following result:
\begin{theorem}[{\it a priori} estimation on asymptotic convergence]\label{thm:convergenceParaFPE}
Consider the Fokker--Planck equation \eqref{FPE} with the potential $V\in\mathcal{V}$.
Suppose $\{\theta_t\}$ solves the parametric Fokker--Planck equation \eqref{wass_grad_flow_on_para_spc}, denote $\delta_0$ as in \eqref{def_delta}. Let $\rho_*(x)=\frac{1}{Z_{D}}e^{-V(x)/{D}}$ be the Gibbs distribution of original equation \eqref{FPE}.  Then we have the inequality:
 \begin{equation}
 \mathcal{D}_{\textrm{KL}}(\rho_{\theta_t}\|\rho_*)\leq \frac{\delta_0}{\tilde{\lambda}_{D} {D}^2}(1-e^{-{D}\tilde{\lambda}_{D} t}) +  \mathcal{D}_{\textrm{KL}}(\rho_{\theta_0}\|\rho_*) e^{-{D}\tilde{\lambda}_{D} t}.  \label{ineq_convergence_paraFPE}
\end{equation}
Here $\tilde{\lambda}_{D}>0$ is the constant associated to the Logarithmic Sobolev inequality discussed in Lemma \ref{lemma:HS perturbation} with potential function $\frac{1}{{D}} V$.
\end{theorem}

To prove Theorem \ref{thm:convergenceParaFPE}, we need the following two lemmas:
\begin{lemma}\label{lemma:HS perturbation}[Holley-Stroock Perturbation]\label{lemma:h-s-perturbation}
Suppose the potential $V\in\mathcal{V}$ is decomposed as $V=U+\phi$ where $\nabla^2 U\succeq K I $ and $\phi\in L^\infty$. Let $\tilde{\lambda}=K e^{-\textrm{osc}(\phi)}$, where $\textrm{osc}(\phi)=\sup\phi-\inf\phi$. Then the following Logarithmic Sobolev inequality holds for any probability density $\rho$:
\begin{equation}
  \mathcal{D}_{\textrm{KL}}(\rho\|\rho_*)\leq \frac{1}{\tilde{\lambda}}\mathcal{I}(\rho|\rho_*).  \label{log-Sobolev}
\end{equation}
\end{lemma}
Here $\rho_*=\frac{1}{Z}e^{-V}$ and $\mathcal{I}(\rho|\rho_*)$ is the Fisher information functional defined as:
\begin{equation*}
  \mathcal{I}(\rho|\rho_*) = \int\Bigl|\nabla\log\left(\frac{\rho(x)}{\rho_*(x)}\right)\Bigr|^2\rho(x) ~ dx .
\end{equation*}
Lemma \ref{lemma:HS perturbation} is first proved in \cite{holley1987logarithmic}.

\begin{lemma}\label{lemma:convergenceParaFPE}
 For any $\theta\in \Theta$, we have:
\begin{equation}
 {D}^2 ~\mathcal{I}(\rho_\theta|\rho_*)\leq\delta_0+\nabla_\theta H(\theta)\cdot G(\theta)^{-1}\nabla_\theta H(\theta), \label{inq_lemma_convergenceParaFPE}
\end{equation}
where $\delta_0$ is defined in \eqref{def_delta}.
\end{lemma}

\begin{proof}[Proof of Lemma \ref{lemma:convergenceParaFPE}]
Let us denote  $\xi = G(\theta)^{-1}\nabla_\theta H(\theta)$ for convenience.
Suppose $\{\theta_t\}$ solves \eqref{wass_grad_flow_on_para_spc} with $\theta_0=\theta$. By Theorem \ref{theorem_submfld}, $\frac{d}{dt}\rho_{\theta_t}\Bigr|_{t=0} = -{(T_{\theta\sharp})}_*\xi $ is orthogonal projection of $-\textrm{grad}_W\mathcal{H}(\rho_\theta)$ onto $\mathcal{T}_{\rho_\theta}\mathcal{P}$ with respect to metric $g^W$. Thus the orthogonal relation gives:
\begin{align}
  g^W(-\textrm{grad}_W\mathcal{H}(\rho_\theta),-\textrm{grad}_W\mathcal{H}(\rho_\theta))= &~ g^W(\textrm{grad}_W\mathcal{H}(\rho_\theta)-{(T_{\theta\sharp})}_*\xi,\textrm{grad}_W\mathcal{H}(\rho_\theta)-{(T_{\theta\sharp})}_*\xi)\nonumber\\
  & + g^W({(T_{\theta\sharp})}_*\xi,{(T_{\theta\sharp})}_*\xi).\label{Pythagorean}
\end{align}
One can verify that the left hand side of \eqref{Pythagorean} is:
\begin{equation}
  g^W(-\textrm{grad}_W\mathcal{H}(\rho_\theta),-\textrm{grad}_W\mathcal{H}(\rho_\theta))=\int|\nabla(V(x)+{D}\log\rho_\theta(x))|^2\rho(x)~dx={D}^2~\mathcal{I}(\rho_\theta|\rho_*).\label{lemma_convergence_1}
\end{equation}
Recall the equivalence between \eqref{important_lemma} and \eqref{important_lemma_equiv} and the definition of $\delta_0$ in \eqref{def_delta}, we know that the first term on the right hand side of \eqref{Pythagorean} has an upper bound
\begin{equation}
  g^W(\textrm{grad}_W\mathcal{H}(\rho_\theta)-{(T_{\theta \sharp})}_*\xi,\textrm{grad}_W\mathcal{H}(\rho_\theta)-{(T_{\theta \sharp})}_*\xi)\leq  \delta_0.   \label{lemma_convergence_2}
\end{equation}
The second term on the right hand side of \eqref{Pythagorean} is:
\begin{align}
  g^W({(T_{\theta\sharp})}_*\xi,{(T_{\theta \sharp})}_*\xi)=(T_{\theta \sharp})^*g^W(\xi,\xi) &= G(\theta)(G(\theta)^{-1}\nabla_\theta H(\theta),~ G(\theta)^{-1}\nabla_\theta H(\theta))\nonumber\\
  &= \nabla_\theta H(\theta)\cdot G(\theta)^{-1}\nabla_\theta H(\theta) \label{lemma_convergence_3}
\end{align}
Combining \eqref{Pythagorean}, \eqref{lemma_convergence_1},\eqref{lemma_convergence_2} and \eqref{lemma_convergence_3} yields to \eqref{inq_lemma_convergenceParaFPE}.
\end{proof}

\begin{proof}[Proof of Theorem \ref{thm:convergenceParaFPE}]
Let us recall the relationship between KL divergence and relative entropy,
\begin{equation*}
  \mathcal{D}_{\textrm{KL}}(\rho\|\rho_*) = \frac{1}{{D}}\mathcal{H}(\rho)+\log(Z_{D}).
\end{equation*}
Actually, we can treat $\mathcal{D}_{\textrm{KL}}(\rho_\theta\|\rho_*)$ as a Lyapunov function for our ODE \eqref{wass_grad_flow_on_para_spc}, because by taking time derivative of $\mathcal{D}_{\textrm{KL}}(\rho_{\theta_t}\|\rho_*)$, we obtain
\begin{equation*}
  \frac{d}{dt}\mathcal{D}_{\textrm{KL}}(\rho_{\theta_t}\|\rho_*) = \frac{1}{{D}}\frac{d}{dt}\mathcal{H}(\rho_{\theta_t})=\frac{1}{{D}}\dot\theta_t\cdot\nabla H(\theta_t) = -\frac{1}{{D}} \nabla H(\theta_t)\cdot G^{-1}(\theta_t)\nabla H(\theta_t).
\end{equation*}
Using the inequality in Lemma \ref{lemma:convergenceParaFPE}, we are able to show:
\begin{equation*}
  \frac{d}{dt}\mathcal{D}_{\textrm{KL}}(\rho_{\theta_t}\|\rho_*)\leq \frac{\delta_0}{{D}} - {D}~ \mathcal{I}(\rho_{\theta_t}|\rho_*).
\end{equation*}
By Lemma \ref{lemma:h-s-perturbation}, we have:
\begin{equation*}
 \frac{d}{dt}\mathcal{D}_{\textrm{KL}}(\rho_{\theta_t}\|\rho_*)\leq \frac{\delta_0}{{D}}-{D}~\tilde{\lambda}_{D} ~\mathcal{D}_{\textrm{KL}}(\rho_{\theta_t}\|\rho_*).
\end{equation*}
Therefore we obtain, by Grownwall's inequality, the following estimate, 
\begin{equation*}
 \mathcal{D}_{\textrm{KL}}(\rho_{\theta_t}\|\rho_*)\leq \frac{\delta_0}{\tilde{\lambda}_{D} {D}^2}(1-e^{-{D}\tilde{\lambda}_{D} t}) +  \mathcal{D}_{\textrm{KL}}(\rho_{\theta_0}\|\rho_*) e^{-{D}\tilde{\lambda}_{D}  t}.
\end{equation*}

\end{proof}

\begin{remark}
Following the previous proof, we can show a similar convergence estimation for the solution $\{\rho_t\}_{t\geq 0}$ of \eqref{FPE}. Such result was first discovered in \cite{bakry1985diffusions}.

\begin{equation}
  \mathcal{D}_{\textrm{KL}}(\rho_t \|\rho_*) \leq \mathcal{D}_{\textrm{KL}}(\rho_0\|\rho_*)~e^{-{D}\tilde{\lambda}_{D} t} \quad \forall ~ t>0. \label{ineq_convergence_FPE}
\end{equation}
\end{remark}
A nominal modification of our proof for Theorem \ref{thm:convergenceParaFPE} leads to a {\it posterior} version of our asymptotic convergence analysis, which is stated in the following theorem. 
\begin{theorem}[{\it Posterior} estimation on asymptotic convergence]\label{thm:convergenceParaFPE posterior}
 \begin{equation*}
 \mathcal{D}_{\textrm{KL}}(\rho_{\theta_t}\|\rho_*)\leq \frac{\delta_1}{\tilde{\lambda}_{D} {D}^2}(1-e^{-{D}\tilde{\lambda}_{D} t}) +  \mathcal{D}_{\textrm{KL}}(\rho_{\theta_0}\|\rho_*) e^{-{D}\tilde{\lambda}_{D} t},
\end{equation*}
where $\delta_1$ is defined in \eqref{def_delta1}.
\end{theorem}

\subsection{Wasserstein error estimations}\label{5.3 }
In this subsection, we establish our error bounds for both continuous and discrete version of the parametric Fokker--Planck equation \eqref{wass_grad_flow_on_para_spc} as approximations to the original equation \eqref{FPE}.
\subsubsection{ Wasserstein error for the parametric Fokker--Planck equation}
The following theorem provides an upper bound between the solutions of \eqref{FPE} and \eqref{wass_grad_flow_on_para_spc}.

\begin{theorem}\label{thm:error_analysis_continuous_para_FPE}
Assume that $\{\theta_t\}_{t\geq 0}$ solves \eqref{wass_grad_flow_on_para_spc} and $\{\rho_t\}_{t\geq 0}$ solves \eqref{FPE}. If the Hessian of the potential function $V$ in \eqref{FPE} is bounded below by a constant $\lambda$, i.e. $\nabla^2 V\succeq \lambda ~I$. the $2$-Wasserstein difference between $\rho_t$ and $\rho_{\theta_t}$ can be bounded as
\begin{equation}
  W_2(\rho_{\theta_t}, \rho_t) \leq \Omega_\lambda(t) = \begin{cases}
    \frac{\sqrt{\delta_0}}{\lambda}(1-e^{-\lambda t})+e^{-\lambda t}W_2(\rho_{\theta_0}, \rho_0), \quad \textrm{if}~ \lambda \neq 0, \\
    \sqrt{\delta_0} t + W_2(\rho_{\theta_0}, \rho_0), \quad \textrm{if}~ \lambda = 0.
  \end{cases} \label{original_error_estimate}
\end{equation}
\end{theorem}

\noindent
To prove this inequality, we need the following lemmas.
\begin{restatable}[Constant speed of geodesic]{lemma}{lemmaC}
\label{constant speed of geodesic on P(M)}
The geodesic connecting $\rho_0,\rho_1\in\mathcal{P}(M)$ is described by,
\begin{equation}
  \begin{cases}
  \frac{\partial\rho_t}{\partial t}+\nabla\cdot(\rho_t\nabla\psi_t)=0 \\
  \frac{\partial\psi_t}{\partial t}+\frac{1}{2}|\nabla\psi_t|^2=0
  \end{cases} \quad \rho_t|_{t=0}=\rho_0,~\rho_t|_{t=1}=\rho_1.   \label{geodesic eq}
\end{equation}
Using the notation $ \dot \rho_t=\partial_t\rho_t=-\nabla\cdot(\rho_t\nabla\psi_t)\in \mathcal{T}_{\rho_t}\mathcal{P}(M)$, $g^W( \dot \rho_t,  \dot \rho_t)$ is constant for $0\leq t\leq 1$ and $ g^W( \dot \rho_t, \dot \rho_t) = W_2^2(\rho_0,\rho_1) $ for $0\leq t\leq 1$.
\end{restatable}

\begin{restatable}[Displacement convexity of relative entropy]{lemma}{lemmaD}
\label{Displacement_convex_entropy}
Suppose $\{\rho_t\}$ solves \eqref{geodesic eq}, the relative entropy $\mathcal{H}$ in \eqref{relative entropy} has potential $V$ satisfying $\nabla^2 V\succeq \lambda I$, then we have $\frac{d}{dt}g^W(\textrm{grad}_W\mathcal{H}(\rho_t),\dot\rho_t)\geq \lambda W_2^2(\rho_0,\rho_1)$. Or equivalently,  $\frac{d^2}{dt^2}\mathcal{H}(\rho_t)\geq \lambda W_2^2(\rho_0,\rho_1)$.
\end{restatable}
Lemma \ref{constant speed of geodesic on P(M)} originates from section 7.2 of \cite{ambrosio2008gradient}. A generalization of it has been proved in Lemma 5 of \cite{lott2008some}. A more general version on the displacement convexity related to Lemma \ref{Displacement_convex_entropy} has been discussed in chapter 16 and 17 of \cite{villani2008optimal}. To be self-contained, we provide direct proofs to both Lemma \ref{constant speed of geodesic on P(M)} and \ref{Displacement_convex_entropy} in Appendix \ref{pf section5}.

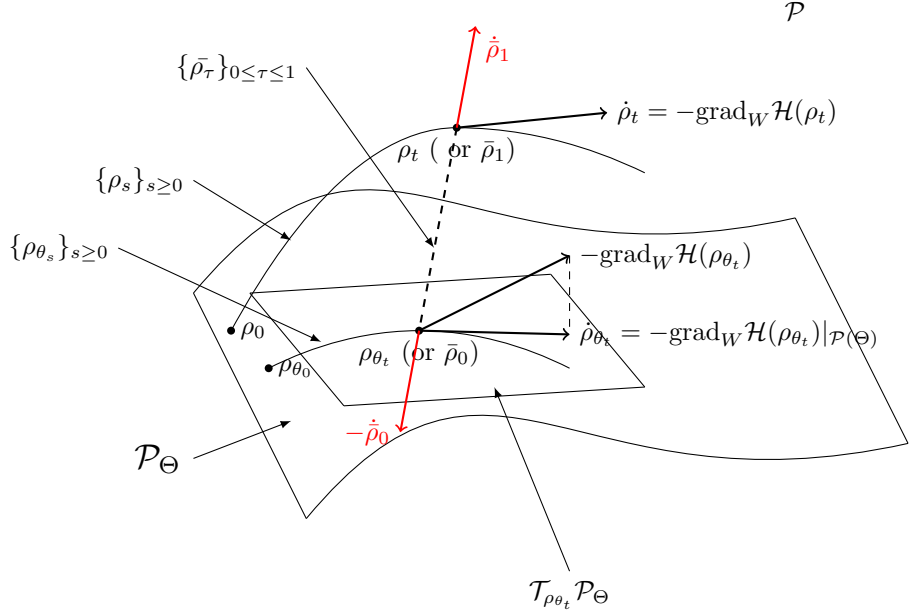
\begin{figure}
\begin{tikzpicture}

\draw (7,6)node[below]{$\mathcal{P}$};

\draw (0.5,-1) .. controls (3,2) and (4,-1) .. (8.5,0);
\draw (-1,2) .. controls (1.5,5) and (2.5,2) .. (7,3);
\draw (0.5,-1) -- (-1,2);
\draw (7,3) -- (8.5,0);

\draw (0,1) parabola bend (2,1.5) (4,1);
\draw (-0.5, 1.5) parabola bend (2.5,4.2) (5,3.6);

\draw[thick,dashed] (2.5,4.2) -- (2,1.5);

\fill (-0.5,1.5) circle[radius = 1.5pt] node[anchor = west] {$\rho_0$};
\fill (2.5,4.2) circle[radius = 1.5pt] node[anchor = north] {$\rho_t$ ( or $\bar{\rho}_1$)};
\fill (0,1) circle[radius = 1.5pt] node[anchor = west] {$\rho_{\theta_0}$};
\fill (2,1.5) circle[radius = 1.5pt] node[anchor = north] {$\rho_{\theta_t}$ (or $\bar{\rho}_0$)};

\draw (1,0.5) -- (-0.25,2);
\draw (5,0.75) -- (3.75,2.25);
\draw (1,0.5) -- (5,0.75);
\draw (-0.25,2) -- (3.75,2.25);

\draw[thick, ->] (2,1.5) -- (4,2.5) node[anchor=west] {$-\textrm{grad}_W\mathcal{H}(\rho_{\theta_t})$};
\draw[thick, ->] (2.5,4.2) -- (4.5,4.4) node[anchor=west] {$\dot\rho_t = -\textrm{grad}_W\mathcal{H}(\rho_{t})$};

\draw[thick,->] (2,1.5) -- (4,1.45) node[anchor=west] {$\dot\rho_{\theta_t} = -\textrm{grad}_W\mathcal{H}(\rho_{\theta_t})|_{\mathcal{P}(\Theta)}$};

\draw[dashed] (4,2.5) -- (4,1.45);

\draw[red, thick,->] (2.5,4.2) -- (2.75,5.55) node[anchor= north west ] {$\dot{\bar{\rho}}_1$};
\draw[red, thick,->] (2,1.5) -- (1.75,0.15) node[anchor=east] {$-\dot{\bar{\rho}}_0$}; 

\draw[-latex](-1,-0.2)node[left,scale=1.3]{$\mathcal{P}_\Theta$} to (0.3,0.3);

\draw[-latex](-1,3.5)node[left]{$\{\rho_s\}_{s\geq 0}$} to (0.3,2.7);

\draw[-latex](-2,2.6)node[left]{$\{\rho_{\theta_s}\}_{s\geq 0}$} to (0.7,1.35);

\draw[-latex](0.5,5)node[left]{$\{\bar{\rho_\tau}\}_{0\leq\tau\leq 1}$} to (2.2,2.6);

\draw[-latex](4,-1.7)node[below]{$\mathcal{T}_{\rho_{\theta_t}}\mathcal{P}_{\Theta} $ } to (3,0.75);

\end{tikzpicture}
\caption{An illustrative diagram for the proof of Theorem \ref{thm:error_analysis_continuous_para_FPE}}
\label{fig:theorem_in_5.3.1}
\end{figure}

\begin{proof}[Proof of Theorem \ref{thm:error_analysis_continuous_para_FPE}]
Figure \ref{fig:theorem_in_5.3.1} provides a sketch of our proof: For a given time $t$, the geodesic $\{\bar{\rho}_\tau\}_{0\leq \tau\leq 1}$ on Wasserstein manifold $\mathcal{P}(M)$ that connects $\rho_{\theta_t}$ and $\rho_t$ satisfies the geodesic equations \eqref{geodesic eq}. If differentiating $W_2^2(\rho_{\theta_t},\rho_t)$ with respect to time $t$ according to Theorem 23.9 of \cite{villani2008optimal}, we are able to deduce
\begin{equation}
\frac{d}{dt} W_2^2(\rho_{\theta_t}, \rho_t) = 2g^W(\dot\rho_{\theta_t}, -\dot{\bar{\rho}}_0)+2g^W(\dot\rho_t, \dot{\bar{\rho}}_1), \label{derivative_of_Wasserstein_distance}
\end{equation}
in which $ \dot{\bar{\rho}}_0 = \partial_\tau\bar{\rho}_\tau|_{\tau=0} = -\nabla\cdot(\bar{\rho}_0\nabla\psi_0)$, $\dot{\bar{\rho}}_1 = \partial_\tau\bar{\rho}_\tau|_{\tau=1} = -\nabla\cdot(\bar{\rho}_1\nabla\psi_1)$. Notice that
\begin{equation*}
 \dot\rho_{\theta_t} = (T_{\theta  \sharp})_*\dot\theta_t \quad \dot\rho_t  = -\textrm{grad}_W\mathcal{H}(\rho_t) = \nabla\cdot(\rho_t\nabla(V+{D}\log\rho_t)).
\end{equation*}
Using the definition \eqref{def_metric} of Wasserstein metric, we can compute (recall that $\rho_{\theta_t} = \bar{\rho}_0$, $\rho_t = \bar{\rho}_1$):
\begin{equation*}
  g^W(\dot\rho_{\theta_t}, \dot{\bar{\rho}}_0) = \int \nabla(V+{D}\log\bar{\rho}_0)\cdot\psi_0~\bar{\rho}_0~dx \quad g^W(\dot\rho_{t}, \dot{\bar{\rho}}_1 )  = \int \nabla(V+{D}\log\bar{\rho}_1)\cdot\psi_1~\bar{\rho}_1~dx.
\end{equation*}
Now we can write \eqref{derivative_of_Wasserstein_distance} as,
\begin{align}
 \frac{1}{2}\frac{d}{dt} W_2^2(\rho_{\theta_t}, \rho_t)=&  g^W((T_{\theta_t \sharp})_*\dot\theta_t+\textrm{grad}_W\mathcal{H}(\rho_{\theta_t}),-\dot{\bar{\rho}}_0) +g^W(-\textrm{grad}_W\mathcal{H}(\rho_{\theta_t}), -\dot{\bar{\rho}}_0)+g^W(-\textrm{grad}_W\mathcal{H}(\rho_t),\dot{\bar{\rho}}_1)  \nonumber \\
\overset{\textrm{set:~} \xi=-\dot\theta_t}{=} & g^W(\textrm{grad}_W\mathcal{H}(\rho_{\theta_t})-(T_{\theta_t \sharp})_*\xi, ~ - \dot{\bar{\rho}}_0) - (g^W(\textrm{grad}_W\mathcal{H}(\bar{\rho}_1), ~ \dot{\bar{\rho}}_1) - g^W(\textrm{grad}_W\mathcal{H}(\bar{\rho}_0), ~ \dot{\bar{\rho}}_0) ). \label{derivative_of_Wasserstein_dist_2}
\end{align}
For the first term in \eqref{derivative_of_Wasserstein_dist_2}, we use Cauchy–Schwarz inequality, \eqref{def_delta}, and Lemma \ref{constant speed of geodesic on P(M)}, which implies  $g(\dot{\bar{\rho}}_0,\dot{\bar{\rho}}_0)=W_2^2(\rho_{\theta_t}, \rho_t)$, to obtain 
\begin{align}
 g^W(\textrm{grad}_W\mathcal{H}(\rho_{\theta_t})-(T_{\theta_t\sharp})_*\xi, -\dot{\bar{\rho}}_0) \leq & \sqrt{g^W(\textrm{grad}_W\mathcal{H}(\rho_{\theta_t})-(T_{\theta_t \sharp})_*\xi, ~ \textrm{grad}_W\mathcal{H}(\rho_{\theta_t})-(T_{\theta_t \sharp})_*\xi)}\sqrt{g^W(\dot{\bar{\rho}}_0, \dot{\bar{\rho}}_0)}\nonumber\\
 \leq & \sqrt{\delta_0}W(\rho_{\theta_t},\rho_t).\label{err_analyz_ineq1}
\end{align}
For the second term in \eqref{derivative_of_Wasserstein_dist_2} , we write it as:
\begin{equation}
  g^W(\textrm{grad}_W\mathcal{H}(\bar{\rho}_1),\dot{\bar{\rho}}_1) - g^W(\textrm{grad}_W\mathcal{H}(\bar{\rho}_0), \dot{\bar{\rho}}_0) = \int_{0}^{1} \frac{d}{d\tau} g^W(\textrm{grad}_W\mathcal{H}(\bar{\rho}_\tau), \dot{\bar{\rho}}_\tau)~d\tau.   \label{need to be simplified using Bochner}
\end{equation}
By Lemma \ref{Displacement_convex_entropy}, we have:
\begin{equation}
  g^W(\textrm{grad}_W\mathcal{H}(\bar{\rho}_1),\dot{\bar{\rho}}_1) - g^W(\textrm{grad}_W\mathcal{H}(\bar{\rho}_0), \dot{\bar{\rho}}_0)
  \geq  \lambda ~ W_2^2(\rho_{\theta_t},\rho_t). \label{err_analyz_ineq2}
\end{equation}
Combining inequalities \eqref{err_analyz_ineq1}, \eqref{err_analyz_ineq2} and \eqref{derivative_of_Wasserstein_dist_2}, we get
\begin{equation*}
  \frac{1}{2}\frac{d}{dt} W_2^2(\rho_{\theta_t}, \rho_t) \leq - \lambda W_2^2(\rho_{\theta_t}, \rho_t) + \sqrt{\delta_0}~W_2(\rho_{\theta_t}, \rho_t).
\end{equation*}
This is:
\begin{equation*}
\frac{d}{dt} W_2(\rho_{\theta_t},\rho_t)\leq  -\lambda W_2(\rho_{\theta_t},\rho_t) + \sqrt{\delta_0}.   
\end{equation*}
When $\lambda\neq 0$, the Grownwall's inequality gives 
\begin{equation*}
 W_2(\rho_{\theta_t},\rho_t)\leq \frac{\sqrt{\delta_0}}{\lambda}(1-e^{-\lambda t}) + e^{-\lambda t}W_2(\rho_{\theta_0}, \rho_0).
\end{equation*}
When $\lambda = 0$, the inequality is $\frac{d}{dt} W_2(\rho_{\theta_t},\rho_t)\leq \sqrt{\delta_0}$, direct integration yields
\begin{equation*}
  W_2(\rho_{\theta_t},\rho_t)\leq \sqrt{\delta_0} t + W_2(\rho_{\theta_0}, \rho_0)~.
\end{equation*}

\end{proof}

When the potential $V$ is strictly convex, i.e. $\lambda>0$. \eqref{original_error_estimate} in Theorem \ref{thm:error_analysis_continuous_para_FPE} provides a nice estimation of the error term $W_2(\rho_{\theta_t}, \rho_t)$ at any time $t$ that is always upper bounded by $\max\{\frac{\sqrt{\delta_0}}{\lambda}, W_2(\rho_{\theta_0}, \rho_0)\}$. 

In case that the potential $V$ is not strictly convex, i.e. $\lambda$ could be $0$ or negative, the right hand side in \eqref{original_error_estimate} may increase to infinity when time $t\rightarrow\infty$. However, \eqref{ineq_convergence_paraFPE} and \eqref{ineq_convergence_FPE} reveals that both $\rho_{\theta_t}$ and $\rho_t$ stay in a small neighbourhood of the Gibbs $\rho_*$ when $t$ is large. When taking this into account, we are able to show that the error term $W_2(\rho_{\theta_t}, \rho_t)$ doesn't get arbitrarily large. In the following theorem, we provide a uniform bound for the error depending on $t$. 
\begin{theorem}\label{thm:improved_upper_bound_of_error_estimate}
Suppose $\{\rho_t\}_{t\geq 0}$ solves \eqref{FPE} and $\{\rho_{\theta_t}\}_{t\geq 0}$ solves \eqref{wass_grad_flow_on_para_spc}, the Hessian of the potential $V\in\mathcal{V}$ is bounded from below by $\lambda$, i.e. $\nabla^2V\succeq \lambda I$, then
\begin{equation}
W_2(\rho_{\theta_t}, \rho_t) \leq \min\left\{\Omega_\lambda(t),~ \sqrt{\frac{2\delta_0}{\tilde{\lambda}_{D}^2{D}^2} } + \left(\sqrt{  \left| 2K_1-\frac{2\delta_0}{\tilde{\lambda}^2_{D}{D}^2} \right|  }+\sqrt{\frac{2 K_2}{\tilde{\lambda}_{D}}}\right)e^{-\frac{\tilde{\lambda}_{D}}{2}{D} t} \right\},  \label{better_error_estimate}
\end{equation}
where function $\Omega_\lambda(t)$ is defined in \eqref{original_error_estimate},
 $E_0=W_2(\rho_{\theta_0},\rho_0)$, $K_1=\mathcal{D}_{\textrm{KL}}(\rho_{\theta_0}\|\rho_*)$, and $K_2=\mathcal{D}_{\textrm{KL}}(\rho_0\|\rho_*)$.
\end{theorem}
\begin{lemma}[Talagrand inequality \cite{otto2000generalization,villani2008optimal}]\label{lemma:Talagrand}
If the Gibbs distribution $\rho_*$ satisfies the Logarithmic Sobolev inequality \eqref{log-Sobolev} with constant $\tilde{\lambda}>0$, $\rho_*$ also satisfies the Talagrand inequality: 
\begin{equation}
 \sqrt{2 \frac{\mathcal{D}_{\textrm{KL}}(\rho\|\rho_*)}{\tilde{\lambda}}}\geq W_2(\rho,\rho_*). \quad \textrm{for any}~ \rho\in\mathcal{P}.\label{Talagrand}
\end{equation}
\end{lemma}
\begin{proof}[Proof of Theorem \ref{thm:improved_upper_bound_of_error_estimate}]
The first term is already provided in Theorem \ref{thm:error_analysis_continuous_para_FPE}, the second term is just a quick result of Theorem \ref{thm:convergenceParaFPE} and Talagrand inequality: for $t$ fixed, \eqref{ineq_convergence_paraFPE} together with Talagrand inequality \eqref{Talagrand} gives:
\begin{align*}
  W_2(\rho_{\theta_t}, \rho_*)\leq \sqrt{2\frac{\mathcal{D}_{\textrm{KL}}(\rho_{\theta_t} \|\rho_*)}{\tilde{\lambda}_{D}}}&\leq \sqrt{\frac{2\delta_0}{\tilde{\lambda}_{D}^2{D}^2}(1-e^{-\tilde{\lambda}_{D}{D} t})+ 2K_1 e^{-\tilde{\lambda}_{D}{D} t} }\\
  & \leq \sqrt{\frac{2\delta_0}{\tilde{\lambda}_{D}^2{D}^2} } + \sqrt{\left|2K_1-\frac{2\delta_0}{\tilde{\lambda}^2_{D}{D}^2} \right| } e^{-\frac{\tilde{\lambda}_{D}}{2}{D} t}.
\end{align*}
Similarly, \eqref{ineq_convergence_FPE} and \eqref{Talagrand} gives
\begin{equation*}
  W_2(\rho_t, \rho_*) \leq \sqrt{2\frac{\mathcal{D}_{\textrm{KL}}(\rho_t\|\rho_*)}{\tilde{\lambda}_{D}}}\leq\sqrt{\frac{2 K_2}{\tilde{\lambda}_{D}}} e^{-\frac{\tilde{\lambda}_{D}}{2}{D} t}.
\end{equation*}
Applying triangle inequality of Wasserstein distance $W_2(\rho_{\theta_t}, \rho_t)\leq W_2(\rho_{\theta_t},\rho_*)+W_2(\rho_t, \rho_*)$, we get \eqref{better_error_estimate}.
\end{proof}
Based on Theorem \ref{thm:improved_upper_bound_of_error_estimate}, we can obtain a uniform {\it a priori} error estimate. 
\begin{theorem}[Main Theorem on {\it a priori} error analysis of the parametric Fokker--Planck equation]\label{thm:uniform_error_bound} 
Assume $E_0=W_2(\rho_{\theta_0},\rho_0)$ and $\delta_0$ defined in \eqref{def_delta} are sufficiently small in the sense that 
\begin{equation}
  E_0<A\sqrt{\delta_0}+B,~ \sqrt{\delta_0}+E_0\leq Be^{-\mu_{D} (A+1)}.  \label{condition for uniform bds}
\end{equation}
Then the approximation error $W_2(\rho_{\theta_t}, \rho_t)$ at any time $t>0$ can be uniformly bounded by $E_0$ and $\delta_0$.
\begin{itemize}
 \item When $\lambda>0$, $W_2(\rho_{\theta_t},\rho_t)\leq \max\{\sqrt{\delta_0}/\lambda, E_0\}\sim O(\sqrt{\delta_0}+E_0)$, 
 \item When $\lambda = 0$, $W_2(\rho_{\theta_t}, \rho_t) \leq \frac{\sqrt{\delta_0}}{\mu_{D} }\log\frac{B}{\sqrt{\delta_0}+ E_0}+E_0\sim O(\sqrt{\delta_0}\log \frac{1}{\sqrt{\delta_0}+E_0}+E_0)$,
 \item When $\lambda<0$,  $W_2(\rho_{\theta_t},\rho_t) \leq  A\sqrt{\delta_0} + B^{\frac{|\lambda|}{|\lambda|+\mu_{D}}}~\left( E_0+\sqrt{\delta_0}/|\lambda| \right)^{\frac{\mu_{D}}{|\lambda|+\mu_{D}}}\sim O((E_0 +\sqrt{\delta_0})^{\frac{\tilde{\lambda}_{D}{D}}{2|\lambda|+\tilde{\lambda}_{D}{D}}})$.
\end{itemize}
Here $A,B,\mu_{D}$ are $O(1)$ constants depending on $V,{D},\rho_0,\theta_0$. Their values are given in \eqref{some notations}.
\end{theorem}
\begin{proof}[ Proof of Theorem \ref{thm:uniform_error_bound} ]
When $\lambda > 0$, by \eqref{better_error_estimate}, we have $E(t)\leq \frac{\sqrt{\delta_0}}{\lambda} + \left(E_0-\frac{\sqrt{\delta_0}}{\lambda}\right)~e^{-\lambda t}$, the right hand side can be bounded by $\max\{E_0, \frac{\sqrt{\delta_0}}{\lambda}\}$.\\

When $\lambda<0$, we denote the right hand side of \eqref{better_error_estimate} as
\begin{equation}
  E(t) = \min\left\{-\frac{1}{|\lambda|}\sqrt{\delta_0} + \left( E_0+\frac{\sqrt{\delta_0}}{|\lambda|} \right)~ e^{|\lambda| t} ,A\sqrt{\delta_0} + B e^{-\mu_{D} t}  \right\},\label{shorthand_upper_bound}
\end{equation}
where 
\begin{equation}
 \quad A = \frac{\sqrt{2}}{\tilde{\lambda}_{D}{D}}, \quad B = \sqrt{ \left|2K_1-\frac{2\delta_0}{\tilde{\lambda}^2_{D}{D}^2} \right| }+\sqrt{\frac{2 K_2}{\tilde{\lambda}_{D}}},\quad \text{and} \quad \mu_{D} = \frac{\tilde{\lambda}_{D} {D} }{2}   \label{some notations}
\end{equation}
are all positive numbers.
The first term in \eqref{shorthand_upper_bound} is increasing as a function of time $t$ while the second term is decreasing, combining $E_0<A\sqrt{\delta_0}+B$, we know $t_0 = \textrm{argmax}_{t\geq 0 } E(t)$ is unique and satisfies
\begin{equation}
  -\frac{1}{|\lambda|}\sqrt{\delta_0} + \left( E_0+\frac{\sqrt{\delta_0}}{|\lambda|} \right) ~ e^{|\lambda| t_0} = A \sqrt{\delta_0 }  + B e^{-\mu_{D} t_0}, \label{uniform_up_bdd_derivation_a}
\end{equation}
as indicated in Figure \ref{fig:main_theorem}.

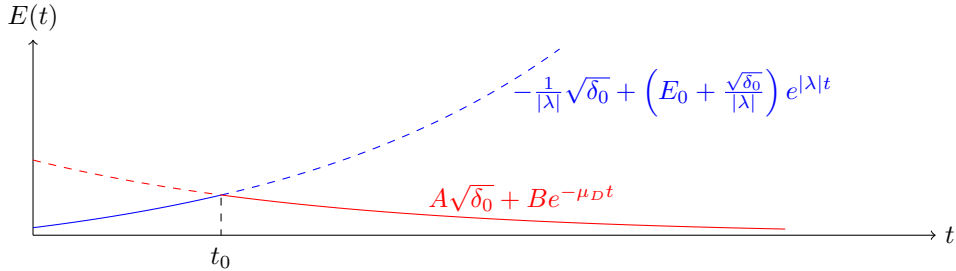
\begin{figure}[ht]
\centering
\begin{tikzpicture}

\draw[->] (0,0) -- (12,0) node[right] {$t$};
\draw[->] (0,0) -- (0,2.6) node[above] {$E(t)$};

\draw[domain=0:2.5,smooth,variable=\x,blue] plot ({\x},{0.5*exp(0.25*\x)-0.4});
\draw[domain=2.5:7,smooth,variable=\x,blue,dashed] plot ({\x},{0.5*exp(0.25*\x)-0.4});

\draw[domain=0:2.5,smooth,variable=\x,red,dashed]  plot ({\x},{exp(-0.25*\x)});
\draw[domain=2.5:10,smooth,variable=\x,red]  plot ({\x},{exp(-0.25*\x)});

\node[text=blue] at (8.5,1.9){$-\frac{1}{|\lambda|}\sqrt{\delta_0} + \left( E_0+\frac{\sqrt{\delta_0}}{|\lambda|} \right) e^{|\lambda|t}$};
\node[text=red] at (6.5,0.5) {$A\sqrt{\delta_0}+Be^{-\mu_{D} t}$};

\draw[dashed] (2.5, 0.5) -- (2.5,0);
\node at (2.5,-0.3) {$t_0$};

\end{tikzpicture}
\caption{An illustrative diagram for the proof of Theorem \ref{thm:uniform_error_bound} }
\label{fig:main_theorem}
\end{figure}

Since  $A>0$, \eqref{uniform_up_bdd_derivation_a} leads to $\left( E_0+\frac{\sqrt{\delta_0}}{|\lambda|} \right) e^{|\lambda| t_0} > B e^{-\mu_{D} t_0}$, thus
\begin{equation}
  t_0 > \frac{\log B - \log \left( E_0+\frac{\sqrt{\delta_0}}{|\lambda|} \right)}{|\lambda|+\mu_{D}}.  \label{uniform_up_bdd_derivation_b}
\end{equation}
Using \eqref{uniform_up_bdd_derivation_b}, we show
\begin{equation}
  \max_{t\geq 0} E(t) = E(t_0) = A\sqrt{\delta_0} + B ~e^{-\mu_{D} t_0} < A\sqrt{\delta_0} + B^{\frac{|\lambda|}{|\lambda|+\mu_{D}}}~\left( E_0+\frac{\sqrt{\delta_0}}{|\lambda|} \right)^{\frac{\mu_{D}}{|\lambda|+\mu_{D}}}. \label{uniform_up_bdd}
\end{equation}
As a result, $W_2(\rho_{\theta_t},\rho_t)$ can be uniformly bounded by the right hand side of \eqref{uniform_up_bdd}. Since $A,B$ are $O(1)$ coefficients, this uniform bound is dominated by the term $O(\left( E_0+\frac{\sqrt{\delta_0}}{|\lambda|} \right)^{\frac{\mu_{D}}{|\lambda|+\mu_{D}}}) = O((E_0 +\sqrt{\delta_0})^{\frac{\tilde{\lambda}_{D}{D}}{2|\lambda|+\tilde{\lambda}_{D}{D}}})$.

At last, when $\lambda=0$, by \eqref{better_error_estimate}
\begin{equation*}
  E(t) = \min\left\{\sqrt{\delta_0}t+E_0 ,A\sqrt{\delta_0} + B e^{-\mu_{D} t}  \right\},  
\end{equation*}
Let us denote $f(t)=A\sqrt{\delta_0}+Be^{-\mu_{D} t}-\sqrt{\delta_0}t-E_0$. Similar to the analysis for the case $\lambda<0$, we denote $t_0 = \textrm{argmax}_{t\geq 0}E(t)$, then $t_0$ is unique and solves $f(t_0)=0$. Since $f(t)$ is decreasing with $f(A+1)>0$, $t_0>A+1$. Then we have
\begin{equation*}
  \max_{t\geq 0} E(t) =E(t_0)= A\sqrt{\delta_0}+Be^{-\mu_{D} t_0} = \sqrt{\delta_0}t_0+E_0>\sqrt{\delta_0}(A+1)+E_0
\end{equation*}
This leads to $Be^{-\mu_{D} t_0}>\sqrt{\delta_0}+E_0$, i.e. $t_0<\frac{1}{\mu_{D}}\log\frac{B}{\sqrt{\delta_0}+E_0}$. Thus we have
\begin{equation*}
\max_{t\geq 0}E(t)=E(t_0)=\sqrt{\delta_0}t_0+E_0<\frac{\sqrt{\delta_0}}{\mu_{D}}\log\frac{B}{\sqrt{\delta_0}+E_0}+E_0.
\end{equation*}
Therefore $W_2(\rho_{\theta_t},\rho_t)$ can be uniformly bounded by the term $\frac{\sqrt{\delta_0}}{\mu_{D}}\log\frac{B}{\sqrt{\delta_0}+E_0}+E_0\sim O(\sqrt{\delta_0}\log\frac{1}{\sqrt{\delta_0}+E_0}+E_0)$.

\end{proof}

\begin{remark} In the case that $V\in\mathcal{V}$ is not convex, we can decompose $V$ by $V=U+\phi$ with $\nabla^2U\succeq K I$ ($K>0$) and $\nabla^2\phi\succeq K_\phi I$. We can still assume $\nabla^2 V\succeq\lambda I$, but $\lambda$ may be negative. One can verify that $ K_\phi<0$ and $|K_\phi|-K\geq|\lambda|$. On the other hand, one can compute $\tilde{\lambda}_{D}=\frac{K}{{D}}e^{-\frac{\textrm{osc}(\phi)}{{D}}}$. Combining them together, we provide a lower bound for $\alpha$:
\begin{equation*}
  \alpha \geq \gamma({D},U,\phi)=\frac{1}{1+2\left(\frac{|K_\phi|}{K}-1\right)e^{\frac{\textrm{osc}(\phi)}{{D}}}}
\end{equation*}
One can verify that increasing the diffusion coefficient ${D}$ or convexity $K$, or decreasing the oscillation $\textrm{osc}(\phi)$ and convexity $K_\phi$ can improve the lower bound $\gamma({D}, U, \phi)$ for the order $\alpha$.
\end{remark}

In a similar way, we can establish the corresponding {\it posterior} error estimate for $W_2(\rho_{\theta_t},\rho_t)$:
\begin{theorem}[{\it Posterior} error analysis of the parametric Fokker--Planck equation]\label{thm:uniform_error_boundposterior} 
Suppose $E_0=W_2(\rho_{\theta_0},\rho_0)$ and $\delta_1$defined in \eqref{def_delta1} satisfy the condition \eqref{condition for uniform bds} with $\delta_0$ replaced by $\delta_1$.  Then 
\begin{enumerate}
 \item When $\lambda\geq 0$, $W_2(\rho_{\theta_t},\rho_t)$ can be uniformly bounded by $O(E_0+\sqrt{\delta_1})$;
 \item When $\lambda = 0$, $W_2(\rho_{\theta_t},\rho_t)$ can be  uniformly bounded by $O(\sqrt{\delta_1}\log\frac{1}{\sqrt{\delta_1}+E_0}+E_0)$;
 \item When $\lambda<0$, $W_2(\rho_{\theta_t},\rho_t)$ can be uniformly bounded by $O((E_0 +\sqrt{\delta_1})^{\frac{\tilde{\lambda}_{D}{D}}{2|\lambda|+\tilde{\lambda}_{D}{D}}})$.
\end{enumerate}
\end{theorem}

\subsubsection{  Wasserstein error for the time discrete schemes}\label{subsection on error analysis for forward Euler scheme}
To solve \eqref{wass_grad_flow_on_para_spc} numerically, we need time discrete schemes, such as the one proposed in \eqref{JKO_A}. In this subsection, we present the error estimate in Wasserstein distance for our scheme. We begin our analysis by focusing on the forward Euler scheme, meaning that we apply forward Euler scheme to solve \eqref{wass_grad_flow_on_para_spc} and compute $\theta_k$ at each time step. We denote $\rho_{\theta_k}={T_{\theta_k}}_{\sharp}p$. We  estimate the $W_2$-error between  $\rho_{\theta_k}$ and the real solution $\rho_{t_k}$. Then we analyze the $W_2$ distance between the solutions obtained by
forward Euler scheme and our scheme \eqref{JKO_A} respectively, which in turn give us the $W_2$ error estimate for our scheme.

\begin{theorem}[{\it a priori} error analysis of forward Euler scheme]\label{Theorem_err_analyz}
Let $\theta_k$ ($k=0,1, \dots, N$) be the solution of forward Euler scheme applied to \eqref{wass_grad_flow_on_para_spc} at time $t_k=kh$ on $[0,T]$ with time step size $h=\frac{T}{N}$, $\rho_{\theta_k} = {T_{\theta_k}}_{\sharp}p$, and $\{\rho_t\}_{t\geq 0}$ solves the Fokker--Planck Equation \eqref{FPE} exactly. 
Assume that the Hessian of the potential function $V\in\mathcal{C}^2(\mathbb{R}^d)$ can be bounded from above and below, i.e. $\lambda I \preceq \nabla^2 V \preceq\Lambda I$. Then
\begin{equation}
W_2 (\rho_{\theta_k}, \rho_{t_k})\leq (\sqrt{\delta_0}  h + C h^2)\frac{1-e^{-\lambda t_k}}{1-e^{-\lambda h }}+e^{-\lambda t_k}W_2(\rho_{\theta_0},\rho_0)  \quad \textrm{for any} \quad t_k=kh, \quad 0\leq k\leq N, \label{main_result}
\end{equation}
where $C$ is a constant whose direct formula is provided in \eqref{definition of C_N}.  
\end{theorem}

In order to estimate $W_2(\rho_{\theta_k},\rho_{t_k})$, we use the triangle inequality of $W_2$ distance \cite{villani2008optimal} to separate it into three parts:
\begin{equation}
W_2(\rho_{\theta_k}, \rho_{t_k})\leq W_2(\rho_{\theta_k},\tilde{\rho}^{\star}_{t_k})+ W_2(\rho^{\star}_{t_k}, \tilde{\rho}_{t_k})+
 W_2(\tilde{\rho}_{t_k}, \rho_{t_k}).\label{tri_inequality}
\end{equation}
Here $\{\tilde{\rho}_t\}_{t_{k-1}\leq t\leq t_k}$ satisfies:
\begin{equation}
    \frac{\partial\tilde{\rho}_t}{\partial t}=\nabla\cdot(\tilde{\rho}_t\nabla V)+{D} \Delta \tilde{\rho}_t ~ , \quad \tilde{\rho}_{t_{k-1}}=\rho_{\theta_{k-1}}, \label{aux_a}
\end{equation}
and $\{\rho^{\star}_t\}_{  t\geq t_{k-1}}$ satisfies:
\begin{equation}
    \frac{\partial \rho^{\star}_t}{\partial t}=\nabla\cdot(\rho^{\star}_t\nabla (V+{D}\log \rho_{\theta_{k-1}})) ~ , \quad \rho^{\star}_{t_{k-1}}=\rho_{\theta_{k-1}}.  \label{aux_b}
\end{equation}
Figure  \ref{help to prove discrete version} shows the relations of different items used in our proof. We present three lemmas that estimate three terms in \eqref{tri_inequality} respectively. 
\begin{figure}
{\par\centering
\begin{tikzpicture}
\draw (0.5,6) node{$\mathcal{P}(\mathbb{R}^d)$};

\fill (0,0) circle[radius=1.5pt];
\fill (0,-0.3) circle[radius=1.5pt];
\fill(5/2,25/16) circle[radius=1.5pt];
\fill(4,2.6) circle[radius=1.5pt];
\fill(3.6,2.8) circle[radius=1.5pt];
\fill(3,1-0.3) circle[radius=1.5pt];
\fill(3.9,1521/400) circle[radius=1.5pt];
\fill(3.6,2.8) circle[radius=1.5pt];
\fill (5,25/9-0.3) circle[radius=1.5pt];

\draw (0,0) parabola (5,25/4);
\draw (0,-0.3) parabola (6,4-0.3);
\draw (5/2,25/16) parabola (4,2.6);
\draw (5/2,25/16) parabola (3.6,2.8);
\draw[dotted] (5/2,25/16) -- (3,1-0.3);
\draw[dotted] (3.9,1521/400) -- (5,25/9-0.3);
\draw[dotted]  (3.9,1521/400) -- (3.6,2.8);
\draw[dotted]  (3.6,2.8) -- (4,2.6);
\draw[dotted]  (4,2.6) -- (5,25/9-0.3);
\draw (-0.5,-0.3) node{$\rho_0$};
\draw (-0.5,0) node{${\rho}_{\theta_0}$};
\draw (3.6,1521/400) node{$\rho_{\theta_k}$};
\draw (5,22/9-0.4) node{$\rho_{t_k}$};
\draw (2,25/16-0.3) node{$\rho_{\theta_{k-1}}$};
\draw (3,0) node{$\rho_{t_{k-1}}$};
\draw (3.4,2.8) node{$\rho^{\star}_{t_k}$};
\draw (4.1,2.2) node{$\tilde{\rho}_{t_k}$};
\end{tikzpicture}
\par}
\caption{Trajectory of $\{{\rho}_{\theta_k}\}_{k=0,...,N}$ is our numerical solution; trajectory of $\{\rho_t\}_{t\geq 0 }$ is the real solution of the Fokker--Planck Equation; $\{\tilde{\rho}_t\}_{t\geq t_{k-1}}$ solves (\ref{aux_a});  $\{\rho^{\star}_t\}_{t\geq t_{k-1}}$ solves (\ref{aux_b}).}
\label{help to prove discrete version}
\end{figure}
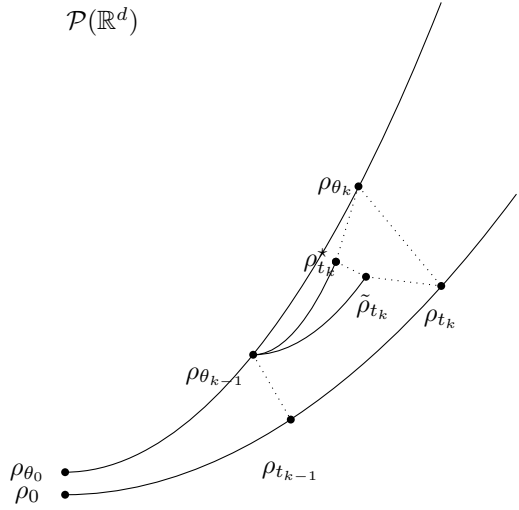

\begin{lemma}\label{Lemma_1}
 $W_2(\rho_{\theta_k},  \rho^{\star}_{t_k} )$ in (\ref{tri_inequality}) can be upper bounded by $\sqrt{\delta_0} h + O(h^2)$. 
\end{lemma}
An explicit formula for the coefficient of $h^2$ is included in the following proof.
\begin{proof}

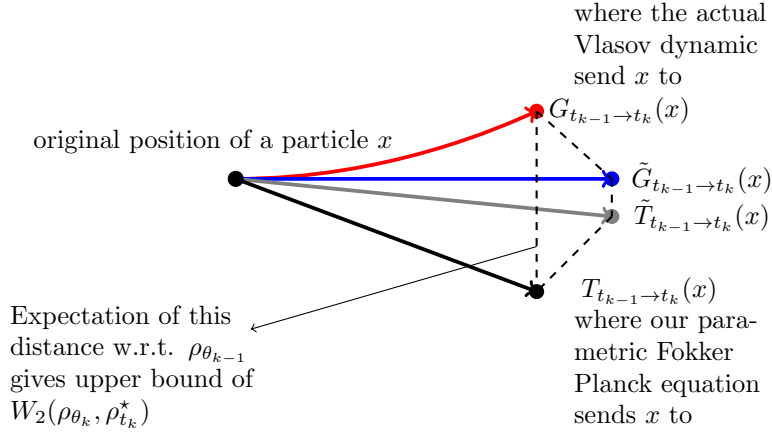
\begin{figure}
{\par\centering
\begin{tikzpicture}
\draw[->,red,line width=0.5mm] (0, 0) parabola (4,0.9);
\draw[->,blue,line width=0.5mm] (0,0) -- (5,0);
\draw[->,gray,line width=0.5mm] (0,0) -- (5,-0.5);
\draw[->,black,line width=0.5mm] (0,0) -- (4,-1.5);
\draw (0,0) circle (0.1cm);
\fill[black] (0,0) circle (0.1cm);
\fill[red] (4,0.9) circle (0.1cm);
\fill[blue] (5,0) circle (0.1cm);
\fill[gray] (5,-0.5) circle (0.1cm);
\fill[black] (4,-1.5) circle (0.1cm);
\draw (-0.3,0.5) node{original position of a particle $x$};

\draw (5.5,-1.5) node{$T_{t_{k-1}\rightarrow t_k}(x)$ };
\node[text width=3cm] at (6,-2.5){where our parametric Fokker Planck equation sends $x$ to};

\draw (6.2,-0.5) node{$\tilde{T}_{t_{k-1}\rightarrow t_k}(x)$};

\draw (6.2,0) node{$\tilde{G}_{t_{k-1}\rightarrow t_k}(x)$};

\draw (5.1,0.9) node{$G_{t_{k-1}\rightarrow t_k}(x)$ };
\node[text width=3cm] at (6,1.8){where the actual Vlasov dynamic send $x$ to};

\draw[thick, dashed] (4,0.9) -- (5,0);
\draw[thick, dashed] (5,0) -- (5,-0.5);
\draw[thick, dashed] (5,-0.5) -- (4,-1.5);
\draw[thick, dashed] (4,0.9) -- (4,-1.5);

\draw[->] (4,-0.9) -- (0.2,-2);
\node[text width=4cm] at (-1,-2.5){Expectation of this distance w.r.t. $\rho_{\theta_{k-1}}$ gives upper bound of $W_2(\rho_{\theta_{k}}, \rho^{\star}_{t_k})$};

\end{tikzpicture}
\par}
\caption{ Illustration of proof strategy for Lemma \ref{Lemma_1}} \label{lemma-proof}
\end{figure}

We establish the desired estimation by introducing several different pushforward maps as shown in Figure \ref{lemma-proof} and then applying triangle inequality.
\begin{enumerate}[label=(\alph*)]
    \item We know $\rho_{\theta_{k-1}}={T_{\theta_{k-1}}}_{\sharp}p$ and $\rho_{\theta_k}={T_{\theta_k}}_{\sharp}p$, let us denote $T_{t_{k-1}\rightarrow t_k}=T_{\theta_k}\circ T^{-1}_{\theta_{k-1}}$. Then ${\rho}_{\theta_k}={T_{t_{k-1}\rightarrow t_k}}_{\sharp}\rho_{\theta_{k-1}}$.
    
    \item Let $\xi_{k-1}=\dot\theta_{k-1}=-G(\theta_{k-1})^{-1}\nabla_\theta H(\theta_{k-1})$ and by convention, we denote $\boldsymbol{\Psi}$ as solution of \eqref{Hodge Dcom}. We consider the map $\tilde{T}_{t_{k-1}\rightarrow t_k}(\cdot)=\textrm{Id}+h\nabla\boldsymbol{\Psi}(\cdot)^{\textrm{T}}\xi_{k-1}$.
    
    \item We denote $\zeta_{\theta}(\cdot)=V(\cdot)+{D}\log \rho_{\theta}(\cdot)$. The particle version (recall \eqref{Vlasov-type SDE }) of (\ref{aux_b}) is:
    \begin{equation}
      \dot{z}_t=-\nabla \zeta_{\theta_{k-1}}(z_t) \quad 0\leq t\leq h \quad \textrm{with initial condition} ~ z_0=x\sim\rho_{\theta_{k-1}}.   \label{ODE_origin}
    \end{equation}
    we denote the solution map of \eqref{ODE_origin} by $G_{t_{k-1}\rightarrow t_k}(x)=z_{t_k}$. Then $\rho^{\star}_{t_k} = {G_{t_{k-1}\rightarrow t_k}}_{\sharp}\rho_{\theta_{k-1}}$.
    
    \item The map $G_{t_{k-1}\rightarrow t_{k}}$ is obtained by solving an ODE, in order to compare the difference with $T_{t_{k-1}\rightarrow t_k}$, we consider the ODE with fixed initial vector field: 
    \begin{equation}
       \dot{\tilde{z}}_t=-\nabla\zeta_{\theta_{k-1}}(x) \quad 0 \leq t\leq h \quad  \tilde{z}_0=x\sim\rho_{\theta_{k-1}}.\label{ODE_approximate}
    \end{equation}
    This ODE will induce the solution map $\tilde{G}_{t_{k-1}\rightarrow t_k}(\cdot)=\textrm{Id}-h\nabla\zeta_{\theta_{k-1}}(\cdot)$.
\end{enumerate}
With the maps defined in (a),(b),(c),(d), and using the triangle inequality of $W_2$ distance, we have,
\begin{align*}
  W_2( \rho_{\theta_k},\tilde{\rho}^{\star}_{t_k})&=W_2( T_{t_{k-1}\rightarrow t_k \sharp} \rho_{\theta_{k-1}}, G_{t_{k-1}\rightarrow t_k \sharp} \rho_{\theta_{k-1}} )\\
  &\leq \underbrace{W_2(T_{t_{k-1}\rightarrow t_k\sharp}\rho_{\theta_{k-1}},\tilde{T}_{t_{k-1}\rightarrow t_k\sharp} {\rho}_{\theta_{k-1}})}_\text{(A)} +\underbrace{W_2( \tilde{T}_{t_{k-1}\rightarrow t_k \sharp} {\rho}_{\theta_{k-1}},  \tilde{G}_{t_{k-1}\rightarrow t_k \sharp } {\rho}_{\theta_{k-1}})}_\text{(B)}\\
  &~ + \underbrace{W_2(\tilde{G}_{t_{k-1}\rightarrow t_k\sharp}\rho_{\theta_{k-1}}, G_{t_{k-1}\rightarrow t_k \sharp} \rho_{\theta_{k-1}})}_\text{(C)}.
\end{align*}
In the rest of the proof, We give upper bounds for distances (A),(B) and (C) respectively.
\begin{itemize}
\item[(A)] Let us define $\xi(\theta) = -G(\theta)^{-1}\nabla H(\theta)$. Now we set $\theta(\tau)=\theta_{k-1}+\frac{\tau}{h}(\theta_k-\theta_{k-1})=\theta_{k-1}+\tau\xi(\theta_{k-1})$. For any $x$, consider $x_\tau = T_{\theta(\tau)}(T_{\theta_{k-1}}^{-1}(x))$ with $0\leq \tau\leq h$, then $\{x_\tau\}_{0\leq\tau\leq 
h}$ satisfies 
\begin{equation}
\dot x_\tau = \partial_\theta T_{\theta(\tau)}(T_{\theta(\tau)}^{-1}(x_\tau))\xi(\theta_{k-1})\quad  0 \leq \tau\leq h. \label{A 1}
\end{equation}
If $x_0\sim \rho_{\theta_{k-1}}$ in \eqref{A 1}, it is clear that $x_h\sim {T_{t_{k-1}\rightarrow t_k}}_{\sharp}{\rho}_{\theta_{k-1}}$. Furthermore, we denote the distribution of $x_\tau$ as $\rho_\tau$ and $\{\psi_\tau\}$ satisfying
\begin{equation}
  -\nabla\cdot(\rho_\tau(x)\partial_\theta T_{\theta(\tau)}(T_{\theta(\tau)}^{-1}(x))\xi_{k-1}) = -\nabla\cdot(\rho_\tau(x)\nabla\psi_\tau(x)) \quad 0\leq \tau\leq h. \label{A 2}
\end{equation}
If we consider
\begin{equation*}
  \dot y_\tau = \nabla\psi_\tau(y_\tau) \quad 0\leq \tau\leq h \quad \textrm{with} ~ y_0\sim{\rho}_{\theta_{k-1}},
\end{equation*}
and denote $\varrho_\tau$ as the distribution of $y_\tau$, by continuity equation and \eqref{A 2}, we know $\rho_\tau=\varrho_\tau$ for $0\leq \tau\leq h$, thus $y_h\sim {T_{t_{k-1}\rightarrow t_k}}_{\sharp}\rho_{\theta_{k-1}}$. On the other hand, when $\tau=0$, \eqref{A 2} shows  $\nabla\psi_0(x)=\nabla\boldsymbol{\Psi}(x)^{\textrm{T}}\xi_{k-1}$. Combining them together, we bound term (A) as,
\begin{align*}
  & W_2^2(T_{t_{k-1}\rightarrow t_k \sharp} \rho_{\theta_{k-1}},{\tilde{T}}_{t_{k-1}\rightarrow t_k \sharp} \rho_{\theta_{k-1}})  \\
  \leq & \mathbb{E}_{y_0\sim \rho_{\theta_{k-1}}} |y_h-(y_0+h\nabla\psi_0(y_0))|^2  = \mathbb{E}_{y_0\sim \rho_{\theta_{k-1}}}\Bigl|\int_0^h ( \nabla\psi_\tau(y_\tau)-\nabla\psi_0(y_0) ) ~d\tau\Bigr|^2 \\
  = & \mathbb{E}_{y_0} \left| \int_0^h \int_0^\tau \frac{d}{ds}(\nabla\psi_s(y_s))~ds~d\tau  \right|^2 = \mathbb{E}_{y_0} \left|\int_0^h\int_s^h \frac{d}{ds}(\nabla\psi_s(y_s))~d\tau~ds\right|^2\\
  = & \mathbb{E}_{y_0}\left|\int_0^h (h-s)\frac{d}{ds}(\nabla\psi_s(y_s))~ds \right|^2\leq \mathbb{E}_{y_0} \int_0^h(h-s)^2~ds\int_0^h\left|\frac{d}{ds}(\nabla\psi_s(y_s))\right|^2~ds\\
  =& \frac{h^3}{3} \int_0^h \mathbb{E}_{y_0}\left|\frac{d}{ds}(\nabla\psi_s(y_s))\right|^2~ds =\frac{h^4}{3}\left(\frac{1}{h}\int_0^h \mathbb{E}_{y_s}\left|\frac{\partial \nabla\psi_s(y_s)}{\partial t}+\nabla^2\psi_s(y_s)\nabla\psi_s(y_s)\right|^2~ds\right)
\end{align*}
Notice that $y_s$ follows the distribution $\rho_s = (T_{\theta_{k-1}+s\xi(\theta_{k-1})}\circ T_{\theta_{k-1}}^{-1})_{\sharp} \rho_{\theta_{k-1}} = T_{\theta_{k-1}+s\xi(\theta_{k-1}) \sharp}p$.

If we define
\begin{align}
& \mathfrak{M}(\theta, s) = \int \left|\frac{\partial}{\partial t} \nabla\psi_s(T_{\theta(s)}(z))+\nabla^2\psi_s(T_{\theta(s)}(z))\nabla\psi_s(T_{\theta(s)}(z))\right|^2p(z)~dz \quad \textrm{with}~ \theta(s) = \theta+s\xi(\theta), \label{def_M_discrete_scheme} \\
& \textrm{and}~  \psi_s ~~\textrm{solving} ~~ -\nabla\cdot(\rho_s\nabla\psi_s) = -\nabla\cdot(\rho_s~\partial_\theta T_{\theta(s)}\circ T_{\theta(s)}^{-1}~\xi(\theta)),~ \textrm{here} ~\rho_s = T_{\theta+s\xi(\theta)\sharp}p.\nonumber
\end{align}
we are able to derive
\begin{equation}
  W_2^2(T_{t_{k-1}\rightarrow t_k \sharp} \rho_{\theta_{k-1}},{\tilde{T}}_{t_{k-1}\rightarrow t_k \sharp} \rho_{\theta_{k-1}}) \leq \frac{1}{3} \underset{0\leq s\leq h}{\textrm{sup}}\mathfrak{M}(\theta_{k-1}, s) h^4 .
\end{equation}

\item[(B)] We have
\begin{align*}
  W^2_2&( \tilde{T}_{t_{k-1}\rightarrow t_k\sharp} \rho_{\theta_{k-1}}, {\tilde{G}}_{t_{k-1}\rightarrow t_k\sharp} {\rho}_{\theta_{k-1}}) \leq \int |{\tilde{T}}_{t_{k-1}\rightarrow t_k}(x)-\tilde{G}_{t_{k-1}\rightarrow t_k}(x)|^2 \rho_{\theta_{k-1}}(x)~dx\\
  &=h^2\left( \int |\nabla\boldsymbol{\Psi}(x)^{\textrm{T}}\xi(\theta_{k-1})-(-\nabla\zeta_{\theta_{k-1}}(x))|^2 \rho_{\theta_{k-1}}(x)~dx \right)\\
  & =  h^2\left( \int|\nabla\boldsymbol{\Psi}(T_{\theta_{k-1}}(x))^{\textrm{T}}\xi(\theta_{k-1})-(-\nabla(V+{D}\log\rho_{\theta_{k-1}})\circ T_{\theta_{k-1}}(x))|^2~dp(x)   \right) \leq \delta_0 ~h^2.
\end{align*}
The last inequality is due to Theorem \ref{lemma_submfld_grad_err} and definition \eqref{def_delta}.\\

\item[(C)] Recall that $\{z_t\}$ and $\{\tilde{z}_t\}$ solve \eqref{ODE_origin} and \eqref{ODE_approximate} with initial condition $z_0=\tilde{z}_0=x$ respectively, similar to the analysis in (A), we can estimate term (C) as
\begin{align*}
  &W_2^2( \tilde{G}_{t_{k-1}\rightarrow t_k\sharp} \rho_{\theta_{k-1}} , G_{t_{k-1}\rightarrow t_k \sharp} \rho_{\theta_{k-1}}) \\
  \leq & {\mathbb{E}}_{x\sim\rho_{\theta_{k-1}}} |z_h-\tilde{z}_h|^2 = {\mathbb{E}}_{x\sim\rho_{\theta_{k-1}}} \left|\int_0^h \nabla\zeta_{k-1}(x)-\nabla\zeta_{k-1}(z_\tau)~d\tau\right|^2\\
  =&\mathbb{E}_x\left|\int_0^h\int_0^\tau \frac{d}{ds}\nabla\zeta_{\theta_{k-1}}(z_s)~ds~d\tau \right|^2 = \mathbb{E}_{x}\left|\int_0^h(h-s)\frac{d}{ds}\nabla\zeta_{\theta_{k-1}}(z_s)~ds \right|^2\\
  \leq & \mathbb{E}_x \frac{h^3}{3}\int_0^h \left|\frac{d}{ds}\nabla\zeta_{\theta_{k-1}}(z_s)\right|^2 ds = \frac{h^4}{3}\left(\frac{1}{h}\int_0^h \mathbb{E}_{z_s}\left|\nabla^2\zeta_{\theta_{k-1}}(z_s)\zeta_{\theta_{k-1}}(z_s)\right|^2~ds\right)
\end{align*}
We define
\begin{align*}
  & \mathfrak{N}(\theta, s)=\mathbb{E}_{z_s}\left|\nabla^2\zeta_\theta(z_s)\zeta_\theta(z_s)\right|^2, \quad \textrm{with} ~~~ \zeta_\theta(\cdot) = V(\cdot) + D\log\rho_\theta(\cdot),  \label{def_N_discrete_scheme}  \\
  & \dot{z}_t = -\nabla\zeta_\theta(z_t),~~z_0\sim \rho_\theta. \nonumber
\end{align*}
Similar to (A), we have:
\begin{equation*}
  W_2^2( \tilde{G}_{t_{k-1}\rightarrow t_k\sharp} {\rho}_{\theta_{k-1}} , G_{t_{k-1}\rightarrow t_k \sharp} {\rho}_{\theta_{k-1}}) \leq \frac{1}{3}\sup_{0\leq s\leq h}\mathfrak{N}(\theta_{k-1},h)h^4
\end{equation*}

\end{itemize}

Combining the estimates for terms (A),(B) and (C), and defining
\begin{equation}
  M(\theta, h) = \underset{0\leq s\leq h}{\textrm{sup}}\mathfrak{M}(\theta_{k-1}, s), \quad  N(\theta, h) = \underset{0\leq s\leq h}{\textrm{sup}}\mathfrak{N}(\theta_{k-1}, s),  \label{def M  N}
\end{equation}
we obtain
\begin{equation*}
 W_2( \rho_{\theta_k},\tilde{\rho}^{\star}_{t_k})\leq \sqrt{\delta_0} h + \frac{M(\theta_{k-1},h)+N(\theta_{k-1}, h)}{\sqrt{3}}~h^2.
\end{equation*}

\end{proof}

\begin{lemma}\label{Lemma_2}
The second term in (\ref{tri_inequality}) can be upper bounded by $O(h^2)$.
\end{lemma}
\begin{proof}
Recall that $\tilde{\rho}_t$ is defined by \eqref{aux_a} and $\rho^*_t$ is defined by \eqref{aux_b}. We can rewrite \eqref{aux_b} as:
\begin{equation*}
  \frac{\partial \rho^{\star}_t}{\partial t}=\nabla\cdot(\rho^{\star}_t(\nabla V+{D}\nabla\log\rho_{\theta_{k-1}}-D\nabla\log \rho^{\star}_t)) + {D}\Delta  \rho^{\star}_t \quad t_{k-1}\leq t\leq t_k
\end{equation*}
We consider the following Stochastic Differential Equations (SDEs) sharing the same trajectory of Brownian motion $\{\boldsymbol{B}_\tau\}_{0\leq\tau\leq h}$ and initial condition:
\begin{align}
  & dx_\tau = -\nabla V(x_\tau)d\tau + \sqrt{2{D}}~d\boldsymbol{B}_\tau \label{SDE_1}\\
  & dx^\star_\tau = -\nabla V(x^\star_\tau) d\tau + ({D}\nabla\log\rho^\star_{t_{k-1}+\tau}(x^\star_\tau) - {D} \nabla\log\rho_{\theta_{k-1}}(x^\star_\tau))d\tau + \sqrt{2{D}}~d\boldsymbol{B}_\tau \label{SDE_2} \\
  & \textrm{with initial condition:}~ x_0 = x^\star_0\sim\rho_{\theta_{k-1}}\quad\textrm{and}~ 0\leq \tau \leq h. \nonumber
\end{align}
Subtracting \eqref{SDE_1} from \eqref{SDE_2}, we get:
\begin{equation*}
  x^\star_\tau - x_\tau = \int_0^\tau \nabla V(x_s)-\nabla V(x^\star_s) + \vec{r}(x^\star_s,s) ~ds,
\end{equation*}
in which we denote $\vec{r}(x,\tau) = {D} \nabla\log\rho^\star_{t_{k-1}+\tau}(x)-{D}\nabla\log\rho_{\theta_{k-1}}(x)$ for convenience.
Hence, 
\begin{align*}
  \mathbb{E}|x^\star_\tau-x_\tau|^2 & = \mathbb{E} \left| \int_0^\tau \nabla V(x_s)-\nabla V(x^\star_s) + \vec{r}(x^\star_s,s) ~ds \right|^2 \\
  & \leq 2~\mathbb{E}\left|\int_0^\tau \nabla V(x_s)-\nabla V(x^\star_s)~ds\right|^2 + 2 ~ \mathbb{E}\left| \int_0^\tau \vec{r}(x^\star_s,s) ~ds \right|^2 \\
   & \leq 2~\mathbb{E}\left[ \tau~\int_0^\tau |\nabla V(x_s)-\nabla V(x^\star_s)|^2~ds \right] + 2~\mathbb{E}\left[\tau~\int_0^\tau|\vec{r}(x^\star_s,s)|^2~ds\right] \\
   &=2\tau\left(\int_0^\tau \mathbb{E}|\nabla V(x_s)-\nabla V(x^\star_s)|^2 + \mathbb{E}|\vec{r}(x^\star_s,s)|^2~ds \right)
\end{align*}
Since Hessian of $V$ is bounded from above by $\Lambda$, $|\nabla V(x)-\nabla V(y)|\leq \Lambda|x-y|$ for any $x,y\in\mathbb{R}^d$, we have the inequality:
\begin{equation}
  \mathbb{E}|x^\star_\tau-x_\tau|^2 \leq 2\tau\Lambda^2 ~\int_0^\tau \mathbb{E}|x^\star_s-x_s|^2 ~ds + 2\tau~\int_0^\tau\mathbb{E}|\vec{r}(x^\star_s,s)|^2~ds  \label{discrete lemma 2 Gronwall}
\end{equation}
If we define $U_\tau=\int_0^\tau \mathbb{E}|x^\star_s-x_s|^2 ~ds$ and $R_\tau = \int_0^\tau\mathbb{E}|\vec{r}(x^\star_s,s)|^2~ds$, \eqref{discrete lemma 2 Gronwall} becomes:
\begin{equation*}
  U'_\tau \leq 2\Lambda^2\tau U_\tau + 2\tau R_\tau 
\end{equation*}
By integrating this inequality, we have $U_\tau \leq \int_0^\tau 2e^{\Lambda(\tau^2-s^2)}sR_s~ds$ and $U'_\tau\leq 4\Lambda^2\tau\int_0^\tau e^{\Lambda(\tau^2-s^2)}sR_sds+2\tau R_\tau$. Therefore
\begin{equation*}
  W_2(\rho^\star_{t_k},\tilde{\rho}_{t_k}) \leq \sqrt{\mathbb{E}|x^\star_h-x_h|^2} = \sqrt{U'_h} \leq \sqrt{4\Lambda^2h\int_0^he^{\Lambda(h^2-s^2)}sR_s~ds + 2hR_h} \label{discrete lemma 2 Gronwall upper bdd 1}
\end{equation*}
Since $R_\tau$ is increasing with respect to  $\tau$, we are able to estimate
\begin{equation}
  W_2(\rho_{t_k}^{\star}, \tilde{\rho}_{t_k}) \leq \sqrt{ 4 \Lambda^2 h^2 \int_0^h e^{\Lambda (h^2-s^2)}sds + 2h} ~ \sqrt{R_h} = \sqrt{2\Lambda(e^{\Lambda h^2} - 1)h + 2h} \sqrt{R_h}.  \label{intermediate est on w2-distance}
\end{equation}
Next we estimate $R_h$. Recall $\rho^*_{t_{k-1}}=\rho_{\theta_{k-1}}$ as in \eqref{aux_b}, we have
\begin{align*}
  R_h = & \int_0^h \mathbb{E}_{x^\star_s} |D\log\rho^\star_{t_{k-1}+s}(x^\star_s)-D\log\rho^\star_{t_{k-1}}(x^\star_s)|^2~ds = D^2 \int_0^h \mathbb{E}_{x^\star_s} \left|\int_0^s \frac{\partial}{\partial t}\nabla \log \rho^\star_{t_{k-1}+t}(x^\star_s)~dt\right|^2~ds  \\
  \leq & D^2\int_0^h \mathbb{E}_{x^\star_s}\left[s\int_0^s \left|\frac{\partial}{\partial t}\nabla\log\rho^\star_{t_{k-1}+t}(x^\star_s)\right|^2~dt\right]~ds = D^2\int_0^h\int_0^s  s\int \left|\frac{\partial}{\partial t}\nabla\log\rho^\star_{t_{k-1}+t}\right|^2\rho^\star_{t_{k-1}+s}~dx~dt ~ds.
\end{align*}
By \eqref{aux_b}, one can further compute $\frac{\partial}{\partial t}\log\rho^\star_{t_{k-1}+t} = -\nabla\log\rho^\star_{t_{k-1}+t}\cdot\nabla\zeta_{\theta_{k-1}} - \Delta\zeta_{\theta_{k-1}}$.
Let us define
\begin{align*}
  \mathfrak{L}(\theta, t, s) & = \int |\nabla(\nabla\log\rho_t\cdot\nabla\zeta_\theta + \Delta\zeta_\theta)|^2\rho_s~dx \quad \textrm{with} ~~ \zeta_\theta = V+D\log\rho_\theta \\
  & \textrm{and}~~ \frac{\partial \rho_s}{\partial s}+\nabla\cdot(\rho_s\nabla\zeta_\theta) = 0 \quad \rho_0 = \rho_\theta
\end{align*}
Then we have the estimation
\begin{align*}
  R_h \leq D^2\int_0^h\int_0^s~s\cdot\left(\sup_{0\leq t\leq s\leq h}\mathfrak{L}(\theta_{k-1}, t, s)\right)~dt~ds = \frac{D^2}{3}\sup_{0\leq t\leq s\leq h}\mathfrak{L}(\theta_{k-1}, t, s)~h^3.
\end{align*}
Let us also define
\begin{equation}
  L(\theta, h) = \left(\sup_{0\leq t\leq s\leq h} \mathfrak{L}(\theta, t, s)\right)^{\frac{1}{2}}  \label{def L }
\end{equation}
Thus \eqref{intermediate est on w2-distance} becomes $W_2(\rho^\star_{t_k}, \tilde{\rho}_{t_k})\leq \sqrt{\frac{2D^2}{3}(\Lambda(e^{\Lambda h^2} - 1) + 2) } L(\theta_{k-1}, h) ~ h^2$. When the stepsize $h$ is small enough, we have $e^{\Lambda h^2}<2$. Let us denote $K(D,\Lambda) = \sqrt{\frac{2D^2}{3}(\Lambda + 2)}$. Thus we have  $W_2(\rho^\star_{t_k}, \tilde{\rho}_{t_k})\leq K(D,\Lambda)L(\theta_{k-1}, h)~h^2$.
\end{proof}
\begin{remark}
Analyzing the discrepancy of stochastic particles under different movements  provides a natural upper bound for $W_2$ distance. Both Lemma \ref{Lemma_1} and Lemma \ref{Lemma_2} are derived by making use of the particle version of their corresponding density evolution. Such proving strategy was motivated from section \ref{particle level explanation}.
\end{remark}

\begin{lemma}\label{Lemma_3}
The third term in (\ref{tri_inequality}) satisfies $W_2(\rho_{t_k},\tilde{\rho}_{t_k})\leq e^{-\lambda h} W_2(\rho_{t_{k-1}}, \rho_{\theta_{k-1}})$. Here we recall that $\lambda$ satisfies $\nabla^2 V\succeq  \lambda I$. 
\end{lemma}
\noindent
This lemma is a direct corollary of the following theorem:
\begin{theorem}
Suppose the potential $V\in C^2(\mathbb{R}^d)$ satisfying $\nabla^2 V \succeq \lambda I$ for a finite real number $\lambda$, i.e. the matrix $\nabla^2 V(x)-\lambda I$ is semi-positive definite for any $x\in\mathbb{R}^d$. Given $\rho_1,\rho_2\in\mathcal{P}$, and denote $\rho^{(1)}_t$ and $\rho^{(2)}_t$ the solutions of the Fokker--Planck equation with different initial distributions $\rho_1$ and $\rho_2$ respectively, i.e.
\begin{align*}
   \frac{\partial \rho^{(1)}_t}{\partial t}=\nabla\cdot(\rho^{(1)}_t\nabla V)+{D}\Delta\rho^{(1)}_t ~~~\rho^{(1)}_0=\rho_1,\\
   \frac{\partial \rho^{(2)}_t}{\partial t}=\nabla\cdot(\rho^{(2)}_t\nabla V)+{D}\Delta\rho^{(2)}_t ~~~\rho^{(2)}_0=\rho_2.
\end{align*}
Then
\begin{equation}
   W_2(\rho^{(1)}_t ,\rho^{(2)}_t )\leq e^{-\lambda t}W_2(\rho_1,\rho_2)\label{Lemma3_general_result}
\end{equation}
\end{theorem}
\noindent
This is a known stability result on Wasserstein gradient flows. One can find its proof in \cite{ambrosio2008gradient} or \cite{villani2008optimal}. With the results in Lemmas \ref{Lemma_1},\ref{Lemma_2},\ref{Lemma_3}, we are ready to prove Theorem \ref{Theorem_err_analyz}. 
\begin{proof}
(Proof of Theorem \ref{Theorem_err_analyz}) For convenience, we write 
\begin{equation*}
  \textrm{Err}_k=W_2(\rho_{\theta_k} , \rho_{t_k})  \quad k=0,1,...,N.
\end{equation*}
Combining Lemma \ref{Lemma_1}, Lemma \ref{Lemma_2} and Lemma \ref{Lemma_3}, the triangle inequality \eqref{tri_inequality} becomes 
\begin{equation*}
  \textrm{Err}_{k}\leq \sqrt{\delta_0}~h + \left( \frac{1}{\sqrt{3}}M(\theta_{k-1},h)+\frac{1}{\sqrt{3}}N(\theta_{k-1},h)+ K(D, \Lambda) L(\theta_{k-1}, h) \right)h^2 + e^{-\lambda h} ~\textrm{Err}_{k-1}.
\end{equation*}
Let us denote the constant $C$ depending on initial parameter $\theta_0$, time stepsize $h$ and time steps $N$:
\begin{equation}
  C(\theta_0, h, N) = \underset{0\leq k\leq N-1}{\textrm{max}}\left\{ \frac{1}{\sqrt{3}}M(\theta_{k-1},h)+\frac{1}{\sqrt{3}}N(\theta_{k-1},h)+ K(D, \Lambda) L(\theta_{k-1}, h) \right\}.   \label{definition of C_N}
\end{equation} 

In the following discussion, we will denote $C=C(\theta_0, h, N)$ for simplicity. By \eqref{definition of C_N}, We have:
\begin{equation}
  \textrm{Err}_{k}\leq\sqrt{\delta_0} h + C h^2 + e^{-\lambda h} \textrm{Err}_{k-1}   \label{discrete theorem Gronwall}
\end{equation}
Multiplying $e^{\lambda kh}$ to both sides of \eqref{discrete theorem Gronwall}, we get:
\begin{equation}
    e^{\lambda kh} \textrm{Err}_k \leq (\sqrt{\delta_0}~h + Ch^2)e^{\lambda kh} + e^{\lambda (k-1)h} \textrm{Err}_{k-1}. \label{recurrence}
\end{equation}
For any $n$, $1 \leq n \leq N$, summing \eqref{recurrence} from $1$ to $n$, we reach 
\begin{equation*}
    e^{\lambda nh} \textrm{Err}_n \leq (\sqrt{\delta_0} h + C h^2 )\left(\sum_{k=1}^n e^{\lambda kh}\right)+\textrm{Err}_0 = (\sqrt{\delta_0} h + C h^2 )\frac{e^{\lambda(n+1)h}-e^{\lambda h}}{e^{\lambda h}-1} + \textrm{Err}_0.
\end{equation*}
Recall that $t_n = nh$ for $1\leq n\leq N$, it leads to:
\begin{equation*}
  \textrm{Err}_n \leq (\sqrt{\delta_0} h + C h^2)\frac{1-e^{-\lambda t_n}}{1-e^{-\lambda h}}+e^{-\lambda t_n} \textrm{Err}_0 \quad  n=1,...,N.
\end{equation*}
\end{proof}

Theorem \ref{Theorem_err_analyz} indicates that the error $W_2(\rho_{\theta_k},\rho_{t_k})$ is upper bounded by $O(\sqrt{\delta_0})+O(C h)+O(W_2(\rho_{\theta_0},\rho_0))$. Here $O(\sqrt{\delta_0})$ is the essential error term that originates from the approximation mechanism of our parametric Fokker--Planck equation. The $O(C h)$ error term is induced by the finite difference scheme. And the $O(W_2(\rho_{\theta_0},\rho_0))$ term is the initial error.

It is worth mentioning that the error bound for forward Euler scheme in \eqref{main_result} matches the error bound for the continuous scheme \eqref{original_error_estimate} as we reduce the effects introduced by finite difference. To be more precise, under the assumption $\lim_{h\rightarrow 0} C(\theta_0, h, N) h = 0$, we have:
\begin{align*}
  &\lim_{h\rightarrow 0} (\sqrt{\delta_0}h+Ch^2)\frac{1-e^{-\lambda t}}{1-e^{-\lambda h}}+e^{-\lambda t}W_2(\rho_{\theta_0}, \rho_0)\\ =&\lim_{h\rightarrow 0}(\sqrt{\delta_0}+Ch)(1-e^{-\lambda t})\frac{h}{1-e^{-\lambda h}} + e^{-\lambda t}W_2(\rho_{\theta_0},\rho_0)=\frac{\sqrt{\delta_0}}{\lambda}(1-e^{-\lambda t}) + e^{-\lambda t}W_2(\rho_{\theta_0},\rho_0)
\end{align*}
this indicates that error bounds \eqref{main_result} and \eqref{original_error_estimate} are compatible as  $h\rightarrow 0$.

\begin{remark}[$O(h)$ error order]
  Under further assumptions that $\Theta=\mathbb{R}^m$, $T_{\theta}(x)\in C^3(\Theta\times\mathbb{R}^d)$ and 
  \begin{equation}
    \lim_{\theta\rightarrow \infty}H(\theta)=+\infty  \label{H(theta) infinite}
  \end{equation}
  we can show the finite difference error term $O(Ch)$ is of order $O(h)$. In fact, the solution obtained from forward Euler scheme is always restricted in a fixed bounded region of $\Theta$. To be more precise, suppose the initial value is $\theta_0$, we consider $\Theta_0 = \{\theta|H(\theta)\leq H(\theta_0)\}$. By \eqref{H(theta) infinite}, one can verify $\Theta_0$ is bounded and closed set and thus is compact. We set $l = \max_{\theta\in \Theta_0} |G(\theta)^{-1}\nabla_\theta H(\theta)|$. Then we consider a slightly larger set $\Theta_0^l = \{\theta~|~\textrm{there exists}~\theta'\in\Theta_0, ~\textrm{s.t.}~|\theta-\theta'|\leq l\}$. Notice that $\Theta_0^l$ is also bounded. We define
  \begin{equation*}
     \sigma^G_{\textrm{min}} = \min_{\theta\in\Theta_0^l}  \sigma_{\textrm{min}}(G(\theta)) \quad \sigma^H_{\textrm{max}} = \max_{\theta\in\Theta_0^l} \sigma_{\textrm{max}}(\nabla_{\theta\theta}^2 H(\theta)).
  \end{equation*}
  Here $\sigma_{\textrm{max}}(A),\sigma_{\textrm{min}}(A)$ denotes the maximum and the minimum singular values of matrix $A$.
  We can show that for any time step size $h<\min\{\frac{2\sigma_{\textrm{min}}^G}{{\sigma}_{\textrm{max}}^H}, 1 \}$, the numerical solution $\{\theta_k\}_{k=1}^N$ obtained by applying forward-Euler scheme to \eqref{wass_grad_flow_on_para_spc} is included in $\Theta_0$. To prove this, we first show $\theta_1\in\Theta_0$, we consider
  \begin{align*}
     H(\theta_1)=H(\theta_0-hG(\theta_0)^{-1}\nabla_\theta H(\theta_0)) 
     = & H(\theta_0)-h\xi^{\textrm{T}}G(\theta_0)\xi + \frac{h^2}{2} \xi^{\textrm{T}} \nabla^2_{\theta\theta}H(\tilde{\theta})\xi  \\
     \leq &  H(\theta_0) - h\sigma_{\textrm{min}}^G|\xi|^2 + \frac{h^2}{2}\sigma_{\textrm{max}}^H |\xi|^2 \leq H(\theta_0)
  \end{align*}
  Here we denote $\xi = G(\theta_0)^{-1}\nabla_\theta H(\theta_0)$. The second equality is due to $T_\theta(x)\in C^3(\Theta\times\mathbb{R}^d)$ and thus $H(\cdot)\in C^2(\Theta)$. We notice that $\tilde{\theta}=\theta_0+\tau(hG(\theta_0)^{-1}\nabla_\theta H(\theta_0))$ with $0\leq \tau\leq 1$ and thus $\tilde{\theta}\in\Theta_0^l$. Since $H(\theta_1)\leq H(\theta_0)$, we know $\theta_1\in \Theta_0$. Applying a similar argument with $\theta_0$ being replaced by $\theta_1$, we can further prove $\theta_2\in\Theta_0$. By induction, we can prove $\{\theta_k\}_{k=1}^N\subset\Theta_0$. Since $\mathfrak{M}(\theta, s),\mathfrak{N}(\theta, s),\mathfrak{L}(\theta, s)$ depend continuously on $\theta, s$, their supreme values on compact set $\Theta_0\times [0, 1]$ must be finite so we know $C(\theta_0, h, N)$ in \eqref{definition of C_N} is upper bounded by a constant independent of $h$ as well as $N$ (recall $N=\frac{T}{h}$). Thus the error term $O(Ch)$ is of $O(h)$ order.
\end{remark}

Similar to the discussion in previous sections, we can naturally extend Theorem \ref{Theorem_err_analyz} to a {\it posterior} estimate.
\begin{theorem}[{\it posterior} error analysis of forward Euler scheme] \label{Theorem_err_analyz post} 
\begin{equation*}
W_2 (\rho_{\theta_k}, \rho_{t_k})\leq (\sqrt{\delta_1}  h + C h^2)\frac{1-e^{-\lambda t_k}}{1-e^{-\lambda h }}+e^{-\lambda t_k}W_2(\rho_{\theta_0},\rho_0)  \quad \textrm{for any} \quad t_k=kh, \quad 0\leq k\leq N.  
\end{equation*}
The explicit definition of the constant $C $ is in \eqref{definition of C_N}.  
\end{theorem}

Up to this point, we mainly analyze the error term for the forward Euler scheme. In our numerical implementation, we adopt the scheme \eqref{JKO_A}, which turns out to be a semi-implicit scheme with $O(h^2)$ local error. In the following discussion, we compare the difference between the numerical solutions of our semi-implicit scheme and forward Euler scheme.

Recall that the parametric Fokker--Planck equation \eqref{wass_grad_flow_on_para_spc} is an ODE: $\dot\theta = - G(\theta)^{-1}\nabla_\theta H(\theta)$. We consider two numerical schemes:
 \begin{align}
    \theta_{n+1} =& \theta_{n}-hG(\theta_n)^{-1}\nabla_\theta H(\theta_n) \quad\theta_0=\theta,~n=1,2,...,N\quad \textrm{forward Euler scheme},\label{explicit_scheme}\\
    \hat{\theta}_{n+1} = & \hat{\theta}_{n}-hG(\hat{\theta}_{n})^{-1}\nabla_\theta H(\hat{\theta}_{n+1}) \quad\hat{\theta}_0=\theta,~n=1,2,...,N\quad \textrm{semi-implicit Euler scheme}.  \label{semi-implicit_scheme}
 \end{align}
 We denote $F(\theta')=G(\theta')^{-1}\nabla_\theta F(\theta'')$, and set:
 \begin{align*}
 L_1 =& \underset{1\leq n\leq N}{\max}\left\{\|F(\theta_n)-F(\hat{\theta}_n)\|/\|\theta_n-\hat{\theta}_n\|\right\}, \quad L_2=\underset{1\leq k\leq N}{\max}\{\|\nabla_\theta H(\hat{\theta}_n)-\nabla_\theta H(\hat{\theta}_{n+1})\|/\|\hat{\theta}_n-\hat{\theta}_{n+1}\|\}, \\
 M_1=&\max_{1\leq n\leq N}\{\|G(\hat{\theta}_n)^{-1}\|\},\quad M_2=\max_{1\leq n\leq N}\{\|\nabla_\theta H(\hat{\theta}_n)\|\},
 \end{align*}
 where $\|\cdot \|$ is a vector norm (or its corresponding matrix norm). 
 \begin{theorem} [Relation between forward Euler and proposed semi-implicit schemes]\label{rel forward Euler and semi-implicit}
The numerical solutions $\theta_n$ and $\hat{\theta}_n$ of the forward Euler and semi-implicit schemes with time stepsize $h$ and $Nh=T$ satisfy
 \begin{equation*}
  \|\theta_n-\hat{\theta}_n\|\leq ((1 + L_1 h )^n-1)\frac{M_1^2M_2L_2}{L_1}h \quad n=1,2,...,N
 \end{equation*}
  \end{theorem}

This result implies that 
 $\|\theta_n-\hat{\theta}_n\|$ can be upper bounded by $(e^{L_1 T }-1)\frac{M_1^2M_2L_2}{L_1}h$.
 When assuming the upper bounds $L_1,L_2,M_1,M_2\sim O(1)$ as $h\rightarrow 0$ (or equivalently $N \rightarrow \infty$), the differences between our proposed semi-implicit scheme and forward Euler scheme can be bounded by $O(h)$. As a consequence, we are able to establish $O(h)$ error bound for our proposed scheme \eqref{JKO_A}.
 
 \begin{proof}[Proof of Theorem \ref{rel forward Euler and semi-implicit}]
 If we subtract \eqref{semi-implicit_scheme} from \eqref{explicit_scheme},
 \begin{equation*}
   (\theta_{n+1}-\hat{\theta}_{n+1}) = (\theta_n-\hat{\theta}_n)-h(G(\theta_n)^{-1}\nabla_\theta H(\theta_n) - G(\hat{\theta}_{n})^{-1}\nabla_\theta H(\hat{\theta}_{n+1}))
 \end{equation*}
 and denote $e_n=\theta_n-\hat{\theta}_n$ and $F(\theta) = G(\theta_n)^{-1}\nabla_\theta H(\theta)$, we may rewrite this equation as
 \begin{equation*}
   e_{n+1} = e_n-h(F(\theta_n)-F(\hat{\theta}_n)+G(\hat{\theta}_n)^{-1}(\nabla_\theta H(\hat{\theta}_n)-\nabla_\theta H(\hat{\theta}_{ n+1 }))).
 \end{equation*}
 Recall the definitions of $L_1,L_2,M_1$, we have
 \begin{equation*}
   \|e_{n+1}\|\leq\|e_n\|+hL_1\|e_n\|+hM_1L_2\|\hat{\theta}_{n+1}-\hat{\theta}_n\|.
 \end{equation*}
 By the semi-simplicit scheme, we have
 \begin{equation*}
    \hat{\theta}_{n+1}-\hat{\theta}_n = -hG(\hat{\theta}_n)^{-1} \nabla_\theta H(\hat{\theta}_{n+1})
 \end{equation*}
 Thus $|\hat{\theta}_{n+1}-\hat{\theta}_n\|\leq hM_1M_2$. This gives us a recurrent inequality,
 \begin{equation*}
    \|e_{n+1}\|\leq\|e_n\|+hL_1\|e_n\| + M_1^2M_2L_2 h^2,
 \end{equation*}
which implies
 \begin{equation*}
    \left(\|e_{n+1}\| +\frac{M_1^2M_2L_2}{L_1}h \right)\leq (1+hL_1)\left(\|e_n\| +\frac{M_1^2M_2L_2}{L_1}h\right)\quad n=0,1,...,N-1.
 \end{equation*}
This leads to:
 \begin{equation*}
   \|e_n\|\leq ((1+hL_1)^n-1) \frac{M_1^2M_2L_2}{L_1}h.
 \end{equation*}
 When we solve the ODE on $[0,T]$ with $h=T/N$, we have $(1+hL_1)^n\leq(1+hL_1)^N=\left(1+\frac{L_1 T }{N}\right)^{N}\leq e^{L_1T}$. This means all terms $\{\|e_n\|\}_{1\leq n\leq N}$ can be upper bounded by $(e^{L_1T}-1)\frac{M_1^2M_2L_2}{L_1}h$.
 \end{proof}

\noindent
\begin{remark}\label{remark on relu approx}
In order to make our argument clear and concise, we omitted the errors introduced by the approximation of ReLU function $\psi_\nu$. Careful analysis on how well $\nabla\psi_\nu$ can approximate a general gradient field is among our future research directions.
\end{remark}

\begin{remark}\label{remark on SGD}
 The convergence property of the Stochastic Gradient Descent method (mainly Adam method) used in our Algorithm \ref{semi-backward_2} is not discussed in details. One can check its convergence analysis in the paper \cite{kingma2014adam}. Based on our experiences, for most of the smooth potential functions $V \in \mathcal{V}$ with diffusion coefficient ${D}$ not too small (i.e. ${D} > 0.1$), our algorithm shows convergent behavior and produces accurate result when checking against the true solution if it is possible.  
\end{remark}

\section{Numerical examples}\label{section 6}
In this section, we consider solving the Fokker--Planck equation (\ref{FPE}) on $\mathbb{R}^d$ with initial condition $\rho_0(x)=\mathcal{N}(0,I_d)$\footnote{We can set initial value $\theta_0$ so that $T_{\theta_0}=Id$ and thus $\rho_0={T_{\theta_0}}_{\sharp}p$ is standard Gaussian distribution.} by using Algorithm \ref{semi-backward_2}. We demonstrate several numerical examples with different potential functions $V$. 
In the following experiments, unless specifically stated, we choose the length of normalizing flow $T_\theta$ as $60$. We set $\psi_\nu:\mathbb{R}^d\rightarrow \mathbb{R}$ as ReLU network with number of layers equals $6$ and hidden dimension equals $20$. We use Adam (Adaptive Moment Estimation) Stochastic Gradient Descent method \cite{kingma2014adam} with default parameters ${D}_1=0.9,{D}_2 = 0.999; \epsilon=10^{-8}$. For the parameters of Algorithm \ref{semi-backward_2}, we choose $\alpha_{\textrm{out}}=0.005$, $\alpha_{\textrm{in}}=0.0005$. We follow Remark \ref{rmk:mention large sample} to choose $K_{\textrm{in}},K_{\textrm{out}} = \max\{1000,300d\}$. Based on our experience, we set $M_{\textrm{out}}=O(\frac{h}{\alpha_{\textrm{out}}})$. The suitable value of $M_{\textrm{in}}$ can be chosen after several quick tests to make sure that every inner optimization problem \eqref{inner optimization problem} can be solved.

Our Python code is uploaded to Github, which can be  downloaded from website \url{https://github.com/LSLSliushu/Parametric-Fokker--Planck-Equation}.

\subsection{Quadratic Potential}
Our first set of examples uses quadratic potential $V$. In this case, we can compute the explicit solution of \eqref{FPE}. These examples are used for the verification purpose, because we can check the results with exact solutions.
\subsubsection{2D cases}
We take $d=2$, and set $V(x)=\frac{1}{2}(x-\mu)^{\textrm{T}}\Sigma^{-1}(x-\mu)$, with $\mu = [3, 3]^{\textrm{T}}$ and $\Sigma= \textrm{diag}([0.25, 0.25])$. The solution of \eqref{FPE} is:
\begin{equation*}
   \rho_t = \mathcal{N}(\mu(t), \Sigma(t)) \quad \mu(t) = (1-e^{-4t})\mu , ~ \Sigma(t) =\left[\begin{array}{cc}
       \frac{1}{4}+\frac{3}{4}e^{-8t} &  \\
        & \frac{1}{4}+\frac{3}{4}e^{-8t}
   \end{array}\right]~~ t\geq 0.
\end{equation*}
We solve the equation in time interval $[0, 0.7]$ with time stepsize $0.01$. We set $M_{\textrm{out}}=20$ and $M_{\textrm{in}}=100$.

To compare against the exact solution, we set $M=6000$ and sample $\{\boldsymbol{X}_1,...,\boldsymbol{X}_M\}\sim {T_{\theta_k}}_{\sharp}p$ at time $t_k$ and use:
\begin{equation*}
\hat{\mu}^k=\frac{1}{M}\sum_{j=1}^M\boldsymbol{X}_j, \quad  \hat{\Sigma}^k = \frac{1}{M-1}\sum_{j=1}^M (\boldsymbol{X}_j-\hat{\mu}_k)(\boldsymbol{X}_j-\hat{\mu}_k)^{\textrm{T}}
\end{equation*}
to compute for its empirical mean and covariance of $\hat{\rho}_k$. We plot the blue curves $\{\hat{\mu}^{(k)}\}$, $\{\hat{\mu}^{(k)}_2\}$, $\{(\hat{\Sigma}^{(k)}_{11}, \hat{\Sigma}^{(k)}_{22})\}$, $\{(\hat{\mu}^{(k)}_1,\hat{\Sigma}^{(k)}_{11})\}$ in Figure \ref{plot_statistics_isotropic_2D}, these plots properly captures the exponential convergence exhibited by the explicit solution (red curves) $\{\mu(t)\}$, $\{\mu_2 (t)\}$, $\{(\Sigma_{11}(t), \Sigma_{22}(t))\}$, $\{(\mu_1(t) , \Sigma_{11}(t))\}$.
\begin{figure}[!htb]
\minipage{0.24\textwidth}
  \includegraphics[width=\linewidth]{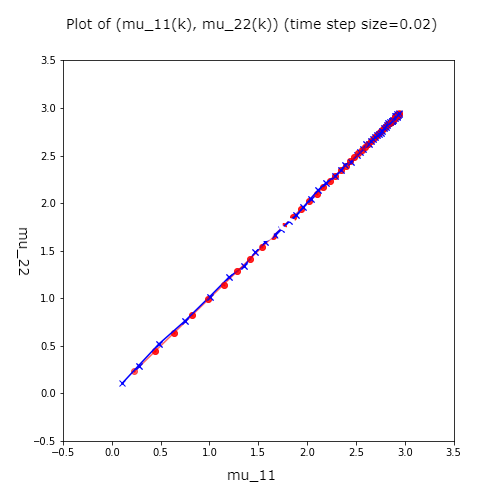}
  \caption{$\{\hat{\mu}^{(k)}\}$}
\endminipage\hfill
\minipage{0.24\textwidth}
  \includegraphics[width=\linewidth]{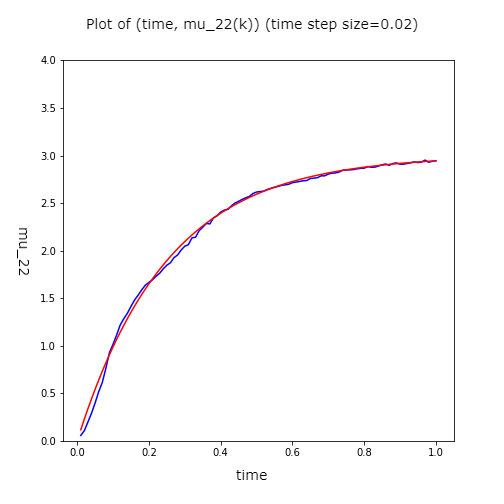}
  \caption{$\{\hat{\mu}_1^{(k)}\}$}
\endminipage\hfill
\minipage{0.24\textwidth}
  \includegraphics[width=\linewidth]{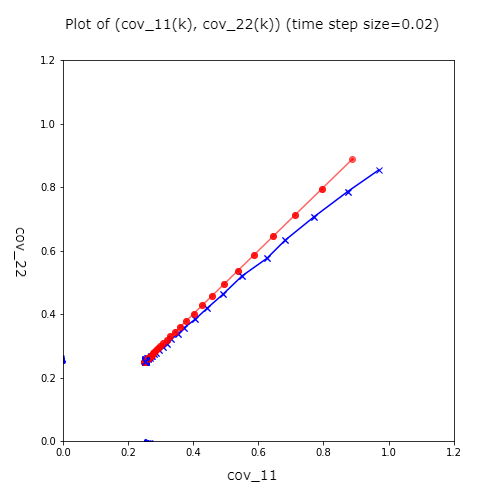}
  \caption{$\{(\hat{\Sigma}^{(k)}_{11},\hat{\Sigma}^{(k)}_{22})\}$}
\endminipage\hfill
\minipage{0.24\textwidth}%
  \includegraphics[width=\linewidth]{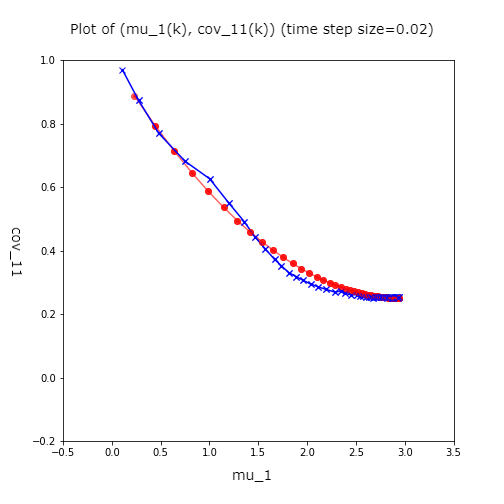}
  \caption{$\{(\hat{\mu}_{(k)},\hat{\Sigma}_{11}^{(k)})\}$}
\endminipage
\caption{Plot of empirical statistics (numerical solution: blue; real solution: red) }\label{plot_statistics_isotropic_2D}
\end{figure}

We also exam the network $\psi_{\hat{\nu}}$ trained at the end of each outer iteration. Generally speaking, the gradient field $\nabla\psi_{\hat{\nu}}$ reflects the movements of the particles under the Vlasov-typed dynamic \eqref{Vlasov-type SDE } at every time step. Here are the graph of $\psi_{\hat{\nu}}$ at $k=10,k=140$ (Figure \ref{Graph of psi iso Gaussian 1}, Figure \ref{Graph of psi iso Gaussian 2}). As we can see from these graphs, the gradient field is in the same direction, but judging from the variation of two $\psi_{\hat{\nu}}$s, when $k=10$, $|\nabla\psi_{\hat{\nu}}|$ is much greater than its value at $k=140$. This is because when $t=140$, the distribution is already close to the Gibbs distribution, the particles no longer need to move for a long distance to reach their final destination.
\begin{figure}[!htb]
\minipage{0.46\textwidth}
\centering
  \includegraphics[width=0.8\linewidth]{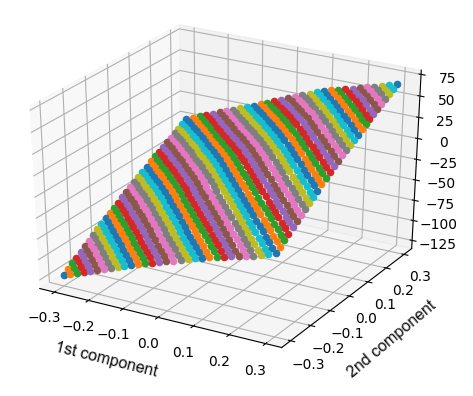}
  \caption{Graph of $\psi_{\hat{\nu}}$ after $M_{\textrm{out}}=20$ outer iterations at $k=10$th time step}\label{Graph of psi iso Gaussian 1}
\endminipage\hfill
\minipage{0.46\textwidth}
\centering
  \includegraphics[width=0.8\linewidth]{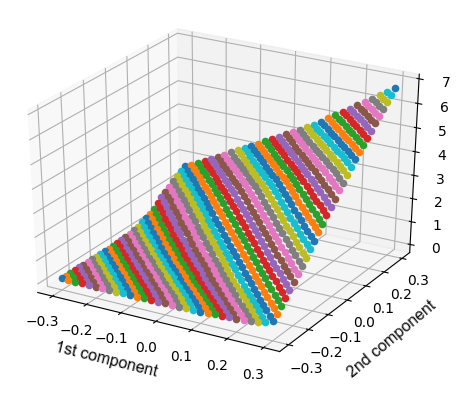}
  \caption{Graph of $\psi_{\hat{\nu}}$ after $M_{\textrm{out}}=20$ outer iterations at $k=140$th time step}\label{Graph of psi iso Gaussian 2}
\endminipage
\end{figure}

In the next example, we apply our algorithm to the Fokker--Planck equation with non-isotropic potential 
\begin{equation*}
   V(x)=\frac{1}{2}(x-\mu)^{\textrm{T}}\Sigma^{-1}(x-\mu) \quad \mu=\left[\begin{array}{c}
        3 \\
        3 
   \end{array}\right] ~\textrm{and}~ \Sigma = \left[\begin{array}{cc}
        1 &  \\
        & \frac{1}{4}
   \end{array}\right].
\end{equation*}
One can verify that the solution to \eqref{FPE} is
\begin{equation*}
  \rho_t = \mathcal{N}(\mu_t, \Sigma_t) \quad \mu_t = \left[\begin{array}{c}
       3(1-e^{-t}) \\
       3(1-e^{-4t}) 
  \end{array}\right], ~ \Sigma_t=\left[\begin{array}{cc}
       1 &  \\
       & \frac{1}{4}(1+3e^{-8t})
  \end{array}\right].
\end{equation*}
We use the same parameters as before. We solve \eqref{FPE} on time interval $[0, 1.4]$ with time step size $0.005$. 

Similarly, we also plot the empirical mean trajectory, one can compare it with the true solution $\mu(t)=(3(1-e^{-t}), 3(1-e^{-4t}))$. Both the curvature and the exponential convergence to $\mu$ are captured by our numerical result. To demonstrate the effectiveness of our formulation, we also compare the mean trajectory obtained by our result
(Figure \ref{mean_traj of wass grad flow}) with the mean trajectory obtained by computing the flat gradient flow
$\dot\theta = -\nabla_\theta H(\theta )$ (Figure \ref{mean_traj of flat grad flow}). It reveals very different behavior of the Wasserstein gradient ($G(\theta)^{-1}\nabla_\theta $) flow and the flat gradient ($\nabla_\theta$) flow. Clearly, our approximation based on Wasserstein gradient flow captures the exact mean function much more accurately.
\begin{figure}[!htb]
\minipage{0.40\textwidth}
  \includegraphics[width=0.8\linewidth]{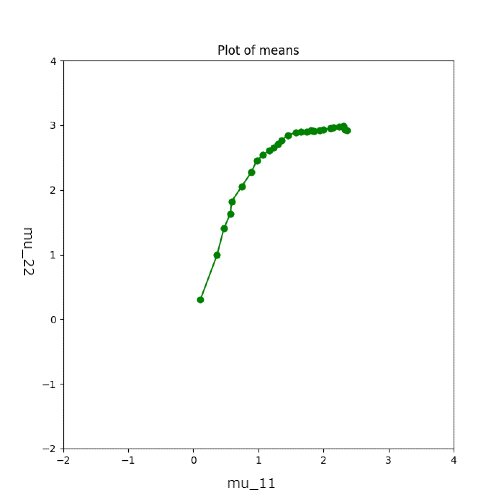}
  \caption{mean trajectory of $\{\rho_{\theta_t}\}$ w.r.t. $\dot\theta = -G(\theta)^{-1}\nabla_\theta H(\theta)$}\label{mean_traj of wass grad flow}
\endminipage\hfill
\minipage{0.40\textwidth}
  \includegraphics[width=0.8\linewidth]{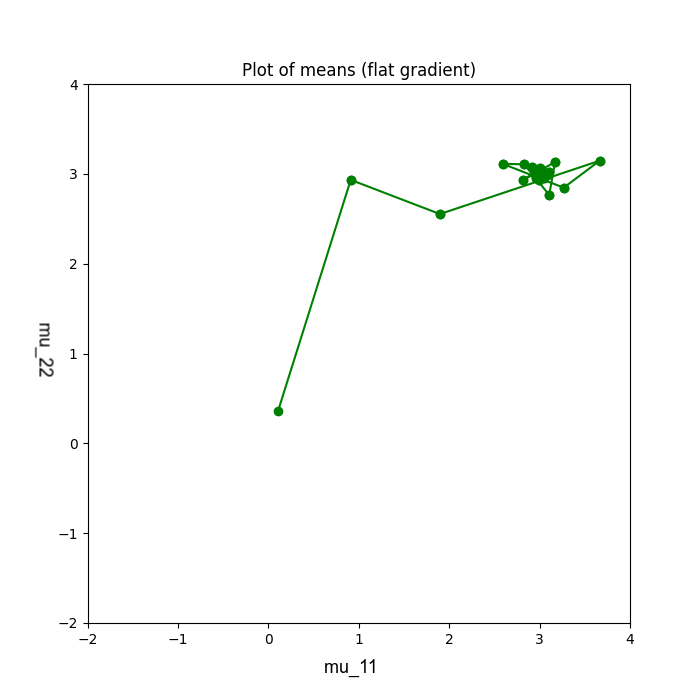}
  \caption{mean trajectory of $\{\rho_{\theta_t}\}$ w.r.t. $\dot\theta = -\nabla_\theta H(\theta)$}\label{mean_traj of flat grad flow}
\endminipage
\end{figure}
We compare the graph of trained $\psi_{\hat{\nu}}$ at different time steps $k=10,140$ (Figure \ref{nonisotropic gauss psi graph k=10}, \ref{nonisotropic gauss psi graph k=140}). The directions of $\nabla\psi_{\hat{\nu}}$ at $k=10$ and $k=140$ is different from the previous example. This is caused by the non-isotropic quadratic (Gaussian) potential $V$ used in this example. 
\begin{figure}[!htb]
\minipage{0.46\textwidth}
\centering
  \includegraphics[width=0.8\linewidth]{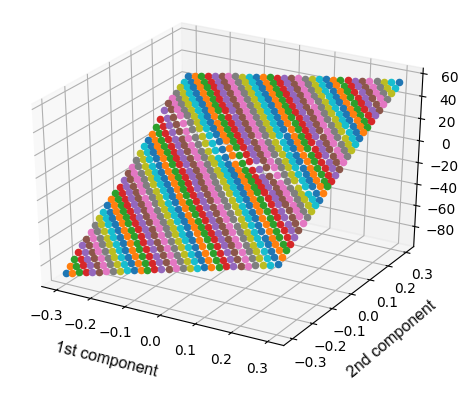}
  \caption{Graph of $\psi_{\hat{\nu}}$ after $M_{\textrm{out}}=20$ outer iterations at $k=10$th time step}\label{nonisotropic gauss psi graph k=10}
\endminipage\hfill
\minipage{0.46\textwidth}
\centering
  \includegraphics[width=0.8\linewidth]{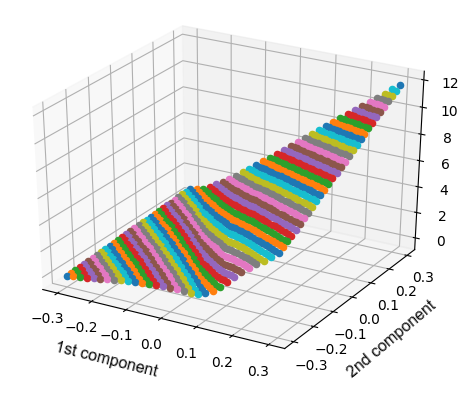}
  \caption{Graph of $\psi_{\hat{\nu}}$ after $M_{\textrm{out}}=20$ outer iterations at $k=140$th time step}\label{nonisotropic gauss psi graph k=140}
\endminipage\hfill
\end{figure}

\subsubsection{Verification of the error estimate}
We verify the $O(h)$ error estimation discussed in \ref{subsection on error analysis for forward Euler scheme} based on numerical experiments with quadratic potentials. We consider $V(x) = |x-\mu|^2$ defined on $\mathbb{R}^2$ with $\mu = (12.0, 12.0)$ and $\rho_0$ as standard Gaussian on time interval $[0, 1]$. We run our algorithm with several different time step size $h=0.01, 0.05, 0.08, 0.1, 0.2, 0.3$ and record their corresponding mean trajectory $\{\hat{\mu}^{(k)}\}$ as defined in Section 6.1.1. During this process, we need to adjust our hyperparameters $\alpha_{\textrm{in}}, \alpha_{\textrm{out}}, M_{\textrm{in}}, M_{\textrm{out}}$ correspondingly in order to guarantee the convergence of Adam method. Denote $\{\mu(t_k)\}$ as the real solution. We compute the average $l^2$ error of mean values as $\textrm{AveErr}(h) = \frac{1}{N}\sum_{k}|\hat{\mu}^{(k)}-\mu(t_k)|$. 
We pick $h$ in a range larger than $0.01$ because when $h$ is smaller, the influence from  the approximation error $\delta_0$ of normalizing flow $T_\theta$ as well as initial error $W_2(\rho_0, \rho_{\theta_0})$
start to dominate the overall error.  Figure \ref{Fig_linear_rel} exhibits the linear relationship between our numerical error $\textrm{AveErr}(h)$ and time step size $h$, which confirms our theoretical estimates.
\begin{figure}[ht]
\centering
  \includegraphics[scale=0.6]{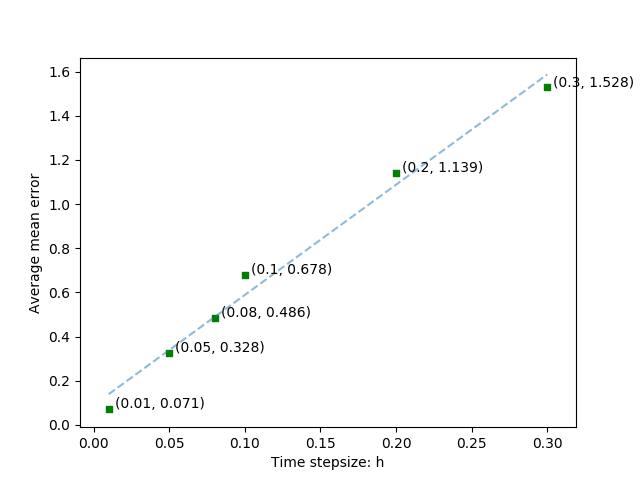}
  \caption{  Numerical errors versus time stepsize $h$.  }\label{Fig_linear_rel}
\end{figure}
\begin{remark}
 The reason of choosing quadratic potential is because its corresponding Fokker--Planck equation has an explicit solution. The reason that we focus on the average error of mean vectors is mainly due to computational accuracy and convenience: one can approximate the error of the mean vector of a distribution by computing the arithmetic average of samples, which is faster and more accurate than computing for the $L^2$-Wasserstein error among two distributions.
\end{remark}

\subsubsection{Higher dimension}
We implement our algorithm in higher dimensional space. In the next example, we take $d=10$, and consider the quadratic potential
\begin{equation*}
    V(x) = \frac{1}{2}(x-\mu)^{\textrm{T}}\Sigma^{-1}(x-\mu) \quad \Sigma=\textrm{diag}(\Sigma_A,I_2,\Sigma_B,I_2,\Sigma_C) \quad \mu = (1,1,0,0,1,2,0,0,2,3)^{\textrm{T}}.
\end{equation*}
Here we set the diagonal blocks as:
\begin{equation*}
  \Sigma_A = \left[\begin{array}{cc}
       \frac{5}{8} & -\frac{3}{8} \\
       -\frac{3}{8}   & \frac{5}{8}
  \end{array}\right] ~~ \Sigma_B = \left[\begin{array}{cc}
      1 &  \\
       & \frac{1}{4}
  \end{array}\right] ~~ \Sigma_C = \left[\begin{array}{cc}
       \frac{1}{4} &  \\
       &  \frac{1}{4}
  \end{array}\right].
\end{equation*}

We solve the equation in time interval $[0, 0.7]$ with time stepsize $0.005$. We set $M_{\textrm{out}}=20$ and $M_{\textrm{in}}=100$. To demonstrate the results, $6000$ samples from the reference distribution $p$ are drawn and pushforwarded by using our computed map $T_{\theta_k}$. We plot a few snapshots of the pushforwarded points (from $t=0.05$ to $t=0.70$) in Figure \ref{sample plots 10d gauss}. One can check that the distribution of our numerical computed samples gradually converges to the Gibbs distribution $\mathcal{N}(\mu,\Sigma)$.

We solve \eqref{FPE} on time interval $[0,2]$ with time step size $h=0.005$. We set $K_{\textrm{in}}=K_{\textrm{out}}=3000$ and choose $M_{\textrm{out}}=30$, $M_{\textrm{in}}=100$. To demonstrate the results, $6000$ samples from the reference distribution $p$ are drawn and pushforwarded by using our computed map $T_{\theta_k}$. We exhibit the projection of the samples on $0-1$, $4-5$ and $8-9$ plane in Figure \ref{sample plots 10d gauss} at time $t=2.0$. One can verify that the distribution of our numerical computed samples converges to the Gibbs distribution $\mathcal{N}(\mu,\Sigma)$.
\begin{figure}[!htb]
\minipage{0.3\textwidth}  \includegraphics[width=\linewidth]{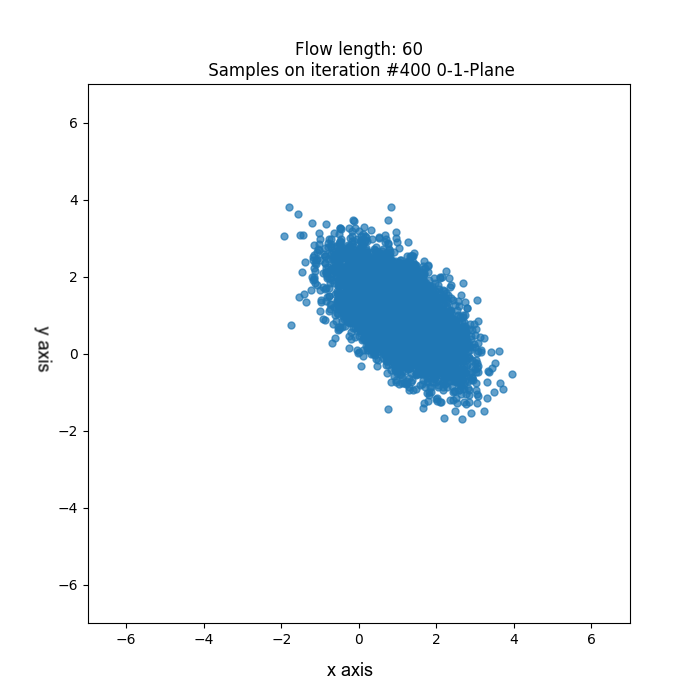}
  \caption{projection of samples on 0-1 plane}
\endminipage\hfill
\minipage{0.3\textwidth}
  \includegraphics[width=\linewidth]{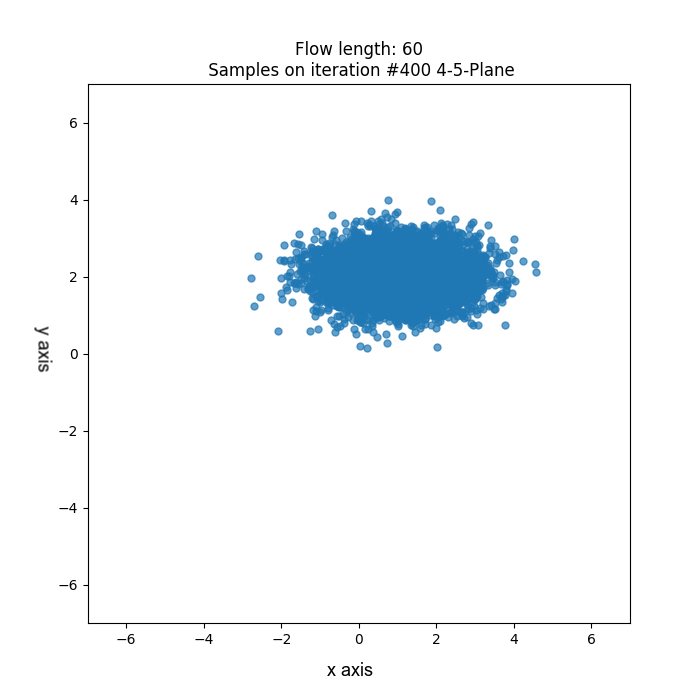}
  \caption{projection of samples on 4-5 plane}
\endminipage\hfill
\minipage{0.3\textwidth}%
  \includegraphics[width=\linewidth]{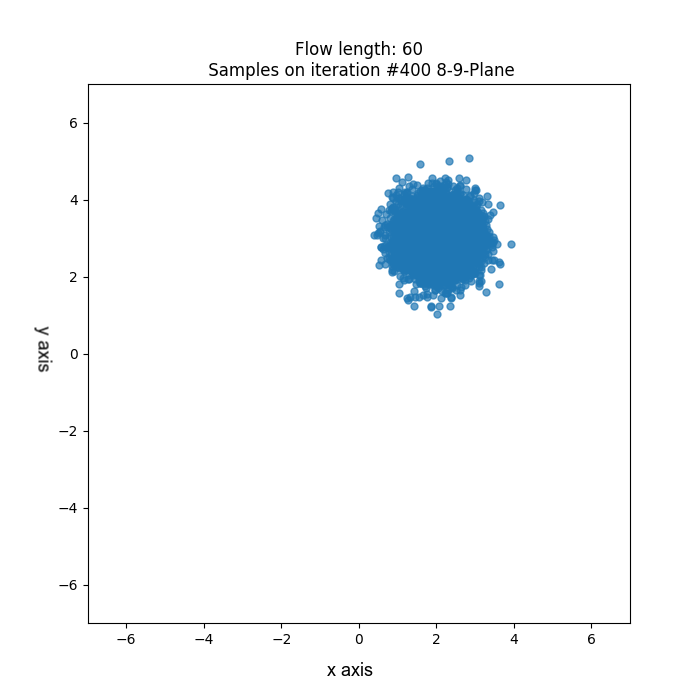}
  \caption{projection of samples on 8-9 plane}
\endminipage
\caption{Sample points of computed $\rho_{\theta_t}$ projected on different planes at $t=2.0$}\label{sample plots 10d gauss}
\end{figure}
The explicit solution to the Fokker--Planck equation is always Gaussian distribution $\mathcal{N}(\mu(t),\Sigma(t))$ with mean $\mu(t)$ and covariance matrix $\Sigma(t)$:
\begin{align*}
\mu(t) = &(1-e^{-t}, 1-e^{-t}, 0, 0, 1-e^{-t}, 2(1-e^{-4t}), 0, 0, 2(1-e^{-4t}), 3(1-e^{-4t}))^{\textrm{T}}\\
\Sigma(t) = & \textrm{diag}(\Sigma_A(t), I, \Sigma_B(t), I, \Sigma_C(t))\\
\textrm{with}~ & ~ \Sigma_A(t) = \left[\begin{array}{cc}
    \frac{5}{8} + f(t) & -\frac{3}{8} + f(t) \\
    -\frac{3}{8} +f(t) & \frac{5}{8} + f(t)
\end{array}\right],~ \Sigma_B(t) = \left[\begin{array}{cc}
  1   &  \\
     & \frac{1+3e^{-8t}}{4}
\end{array}\right] , ~ \Sigma_C(t) = \left[\begin{array}{cc}
   \frac{1+3e^{-8t}}{4}  &  \\
     & \frac{1+3e^{-8t}}{4}
\end{array}\right]\\
\textrm{here}~& ~f(t) = -\frac{2}{7} e^{-t} +\frac{1}{3} e^{-2t} + \frac{55}{168}e^{-8t}
\end{align*}
To compare against the exact solution, we set sample size $M=6000$ and compute the empirical mean $\hat{\mu}^k$ and covariance $\hat{\Sigma}^k$ of our numerical solution $\hat{\rho}_k$ at time $t_k$. We evaluate the error between $\hat{\mu}^{(k)}$ and $\mu(t_k)$; $\hat{\Sigma}^{(k)}$ and $\Sigma(t_k)$. We plot the error curves of $\|\hat{\mu}^{(k)}-\mu(t_k)\|_2$ (Figure \ref{10 dim gauss mean error}) and $\|\hat{\Sigma}^{(k)} - \Sigma(t_k)\|_F$ (Figure \ref{10 dim gauss covariance error}). Here $\|\cdot\|_F$ is the matrix Frobenius norm.
\begin{figure}[!htb]
\minipage{0.33\textwidth}
  \includegraphics[width=\linewidth]{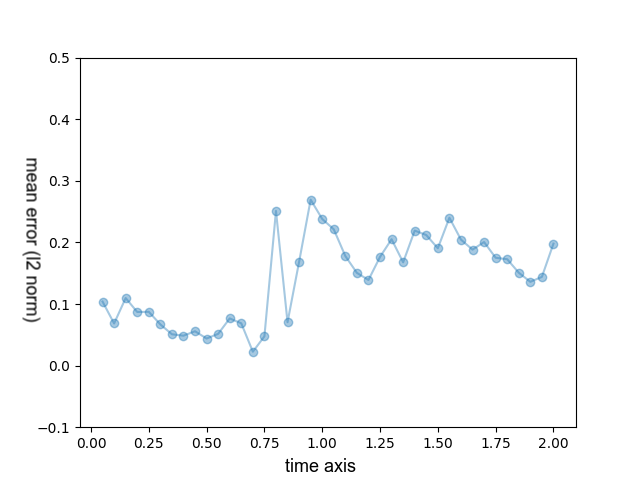}
  \caption{mean error ($l_2$)}\label{10 dim gauss mean error}
\endminipage\hfill
\minipage{0.33\textwidth}
  \includegraphics[width=\linewidth]{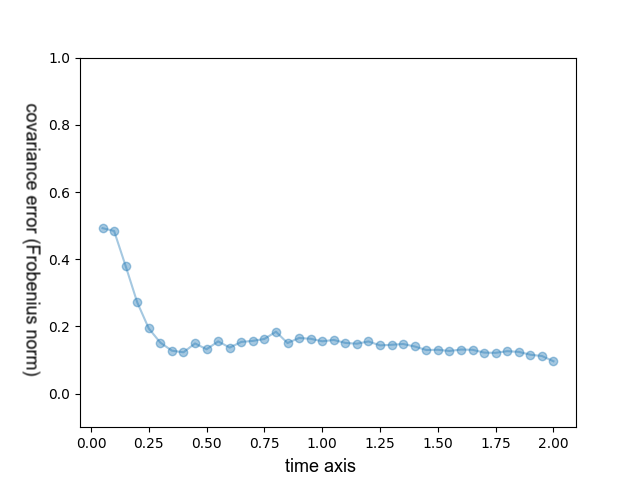}
  \caption{covariance error ($\|\cdot\|_F$)}\label{10 dim gauss covariance error}
\endminipage\hfill
\minipage{0.33\textwidth}
   \includegraphics[width=\textwidth]{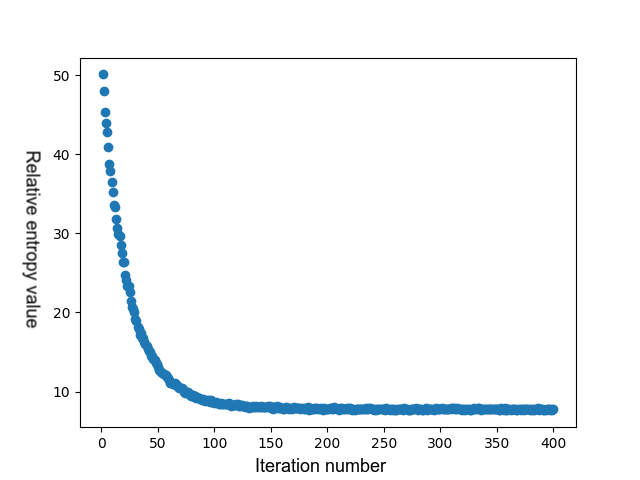}
    \caption{Plot of $\{H(\theta)\}$}\label{10 dim gauss dissipative verification}
\endminipage\hfill
\end{figure}
Figure \ref{10 dim gauss dissipative verification} captures the exponential decay of  $H$ along its Wasserstein gradient flow, this verifies the entropy dissipation property of the Fokker--Planck equation with convex potential function $V$.

In this case, we take a closer look at the loss in the inner loops. Figure \ref{plot inner loop loss 10 d gauss} shows the first $10$ (out of $20$) loss plots when applying SGD method to solve \eqref{modified inner optimization problem} with $k=200$ ($t=200\cdot h=1.0$). The remaining loss plots from the 11th outer iteration to 20th iteration are similar to the plots in the second row. The situations are similar for other time step $k$. We believe that $M_{\textrm{in}}=100$ works well in this problem, the SGD method we used can thoroughly solve the variational problem \eqref{modified inner optimization problem} for each outer loop.

\begin{figure}[tbhp]
\centering
\subfloat[1st iteration]{\includegraphics[width=0.2\linewidth]{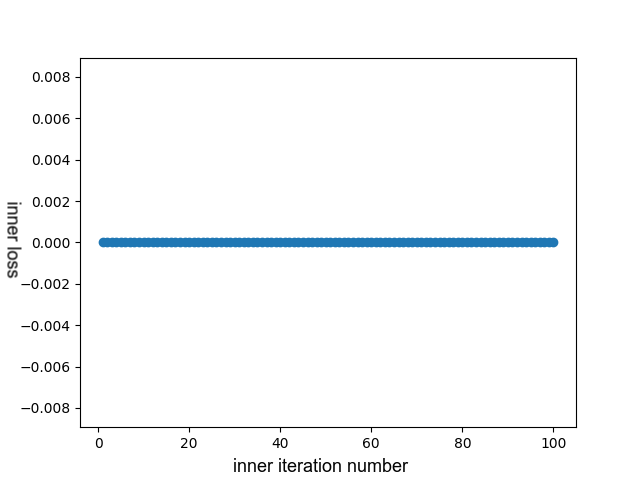}}
\subfloat[2nd iteration]{\includegraphics[width=0.2\linewidth]{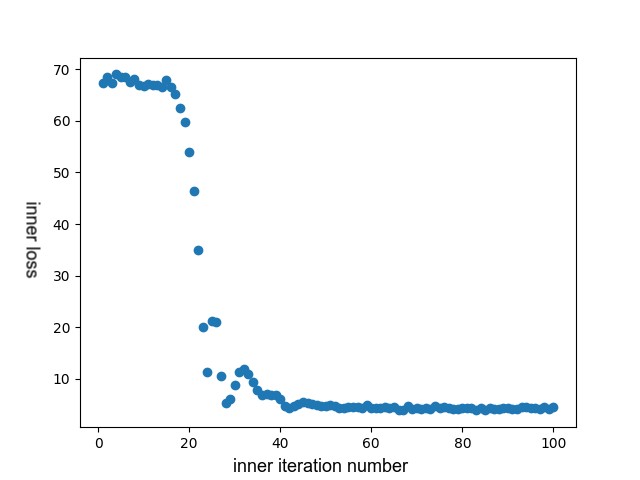}}
\subfloat[3th iteration]{\includegraphics[width=0.2\linewidth]{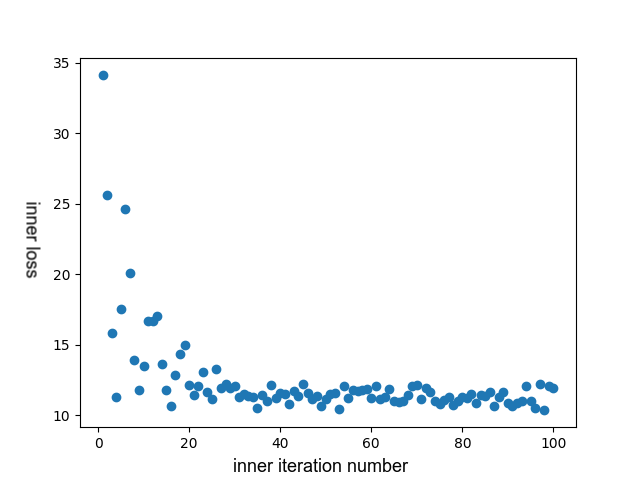}}
\subfloat[4th iteration]{\includegraphics[width=0.2\linewidth]{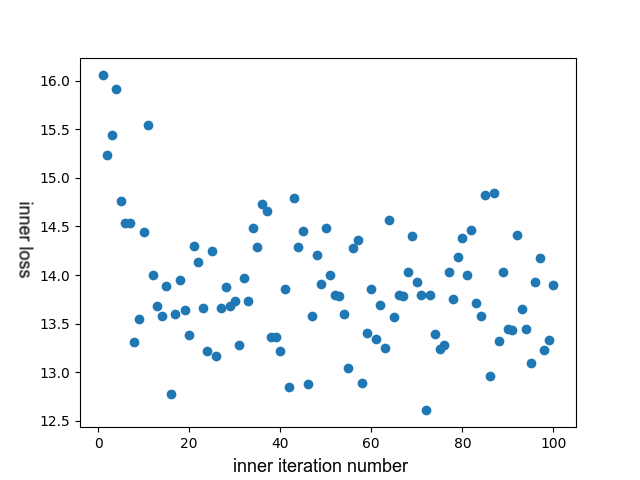}}
\subfloat[5th iteration]{\includegraphics[width=0.2\linewidth]{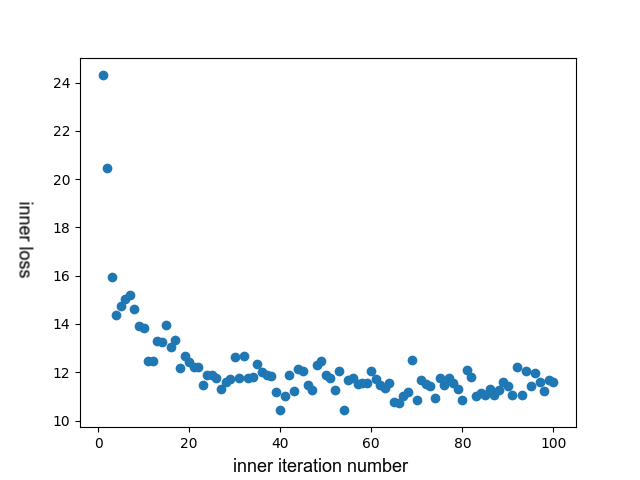}}
\newline
\subfloat[6th iteration]{\includegraphics[width=0.2\linewidth]{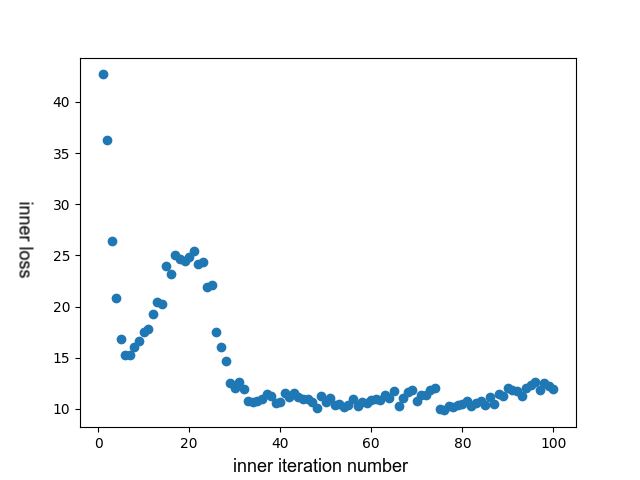}}
\subfloat[7th iteration]{\includegraphics[width=0.2\linewidth]{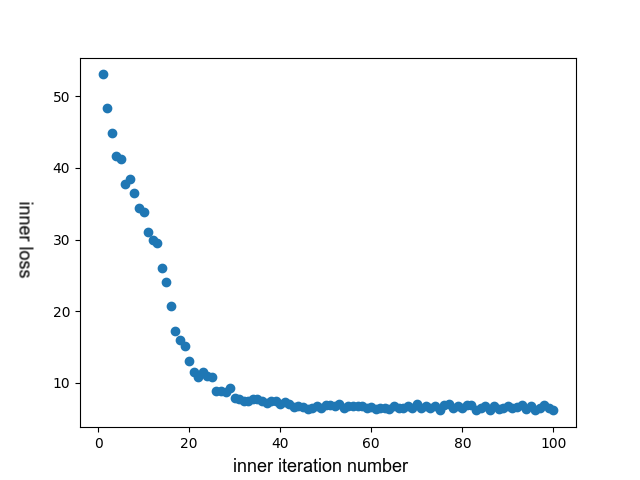}}
\subfloat[8th iteration]{\includegraphics[width=0.2\linewidth]{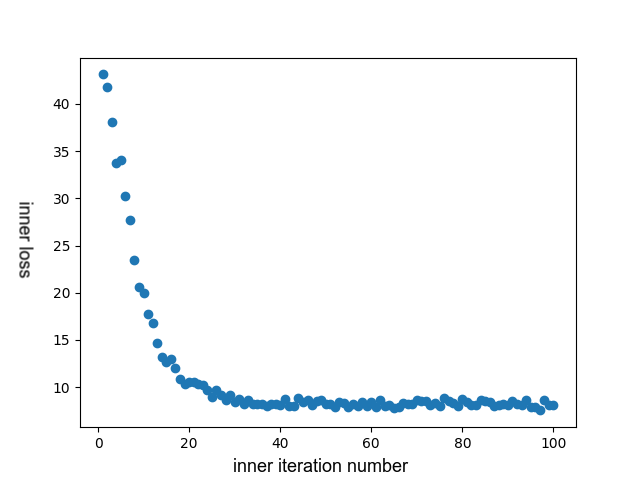}}
\subfloat[9th iteration]{\includegraphics[width=0.2\linewidth]{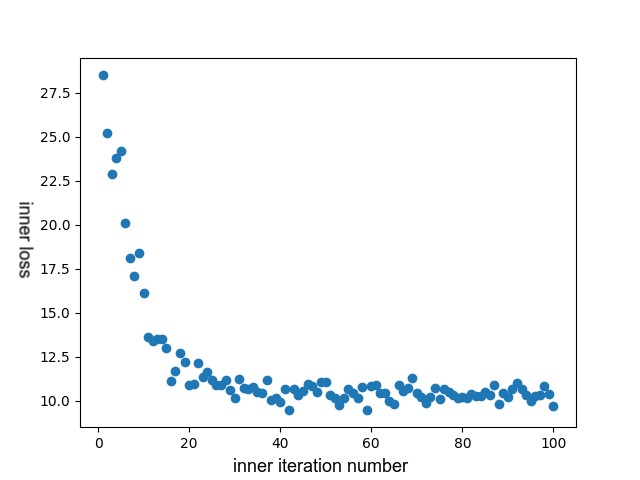}}
\subfloat[10th iteration]{\includegraphics[width=0.2\linewidth]{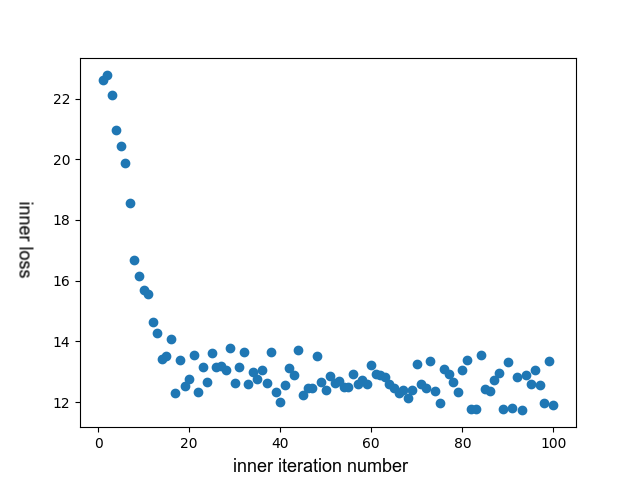}}
\caption{Plots of inner loop losses}
\label{plot inner loop loss 10 d gauss}
\end{figure}

\subsection{Experiments with more general potentials}
In this section, we exhibit two examples with more general potentials in higher dimensional space.
\subsubsection{Styblinski-Tang potential}\label{st ex}
In this example, we set dimension $d=30$, and consider the Styblinski–Tang function \cite{simulationlib}
\begin{equation*}
 V(x)=\frac{3}{50}\left(\sum_{i=1}^{d} x_i^4-16x_i^2+5x_i\right).
\end{equation*}
We solve \eqref{FPE} with potential $V$ on time interval $[0,3]$ with time step size $h=0.005$. We set $K_{\textrm{in}}=K_{\textrm{out}}=9000$ and $M_{\textrm{in}}=100$, $M_{\textrm{out}}=30$.

To exhibit sample results, due to the symmetric structure of the potential function, we  project the sample points in $\mathbb{R}^{30}$ to some random plane, such as $5-15$ plane used in this paper. The sample plots and their estimated densities are presented in Figure \ref{sample plot S-T}.

\begin{figure}[tbhp]
\centering
\subfloat[t=0.30]{\includegraphics[width=0.16\linewidth]{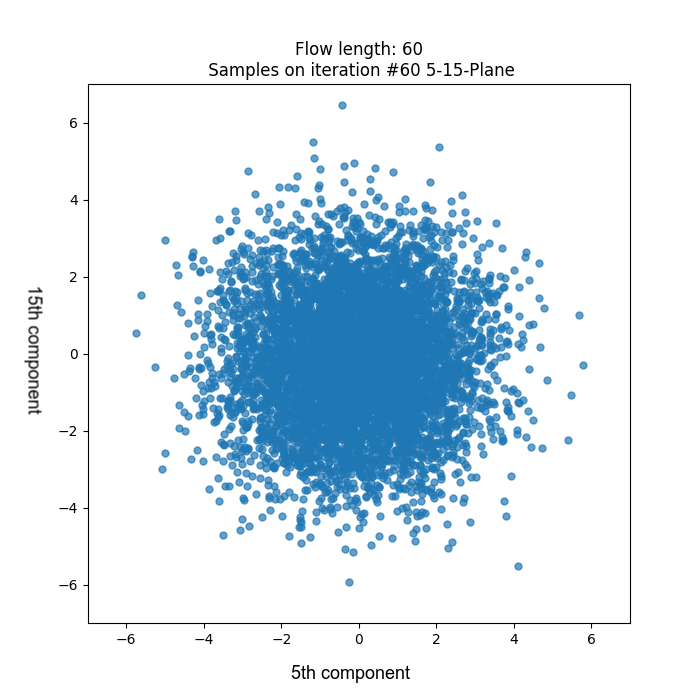}}
\subfloat[t=0.60]{\includegraphics[width=0.16\linewidth]{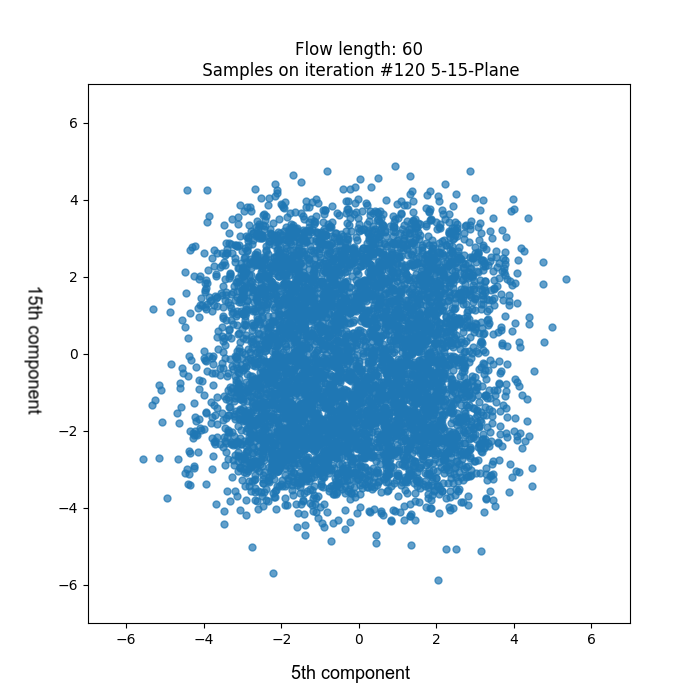}}
\subfloat[t=0.90]{\includegraphics[width=0.16\linewidth]{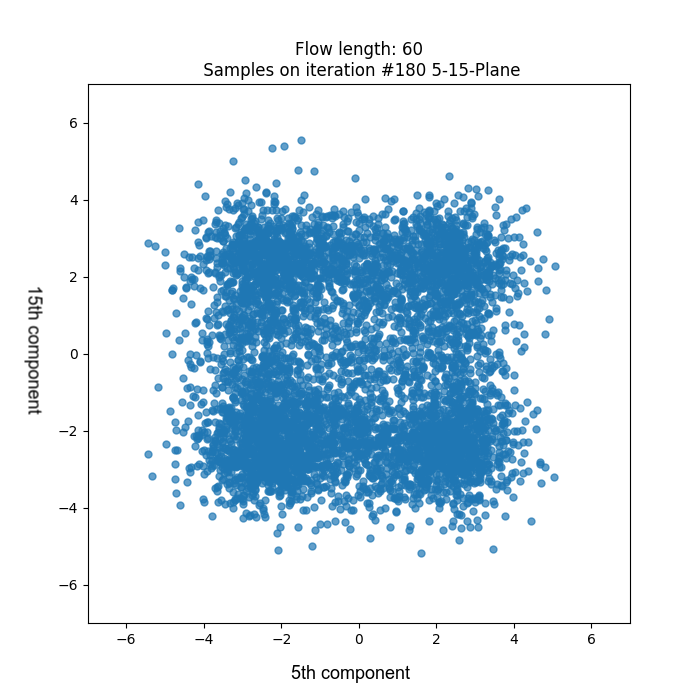}}
\subfloat[t=1.20]{\includegraphics[width=0.16\linewidth]{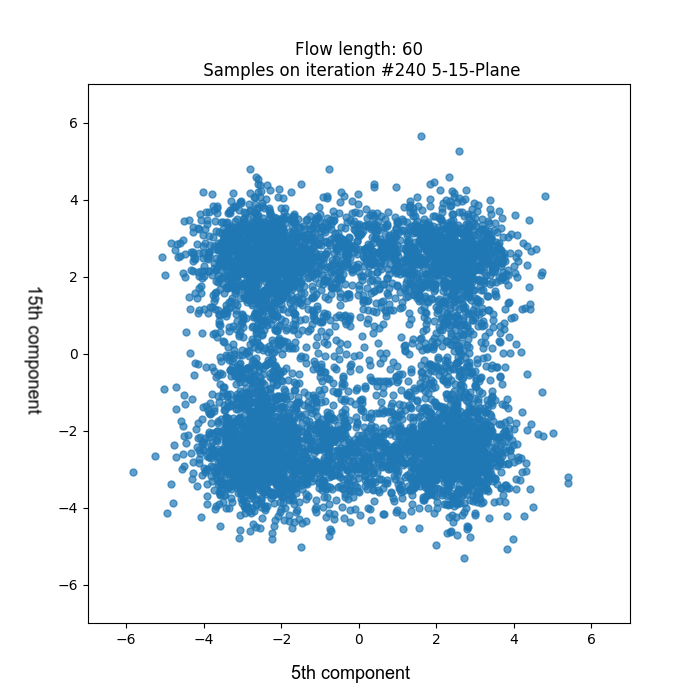}}
\subfloat[t=1.50]{\includegraphics[width=0.16\linewidth]{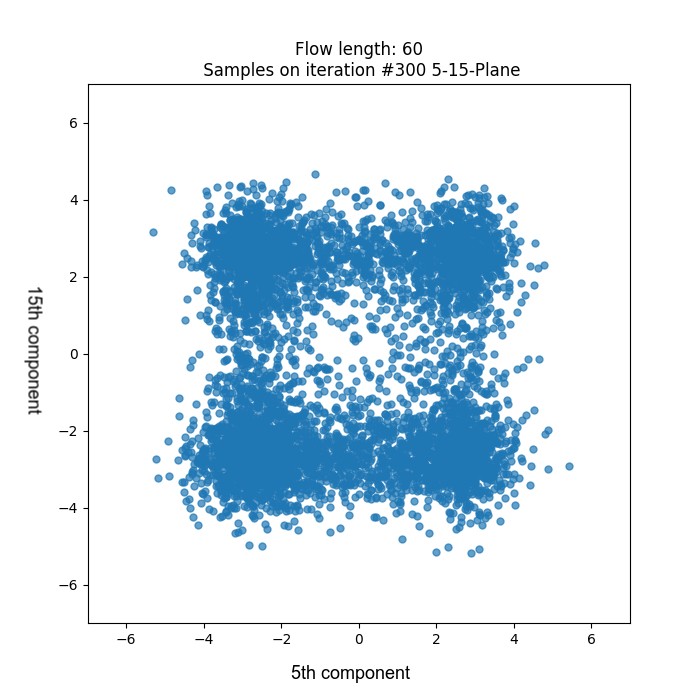}}
\subfloat[t=1.80]{\includegraphics[width=0.16\linewidth]{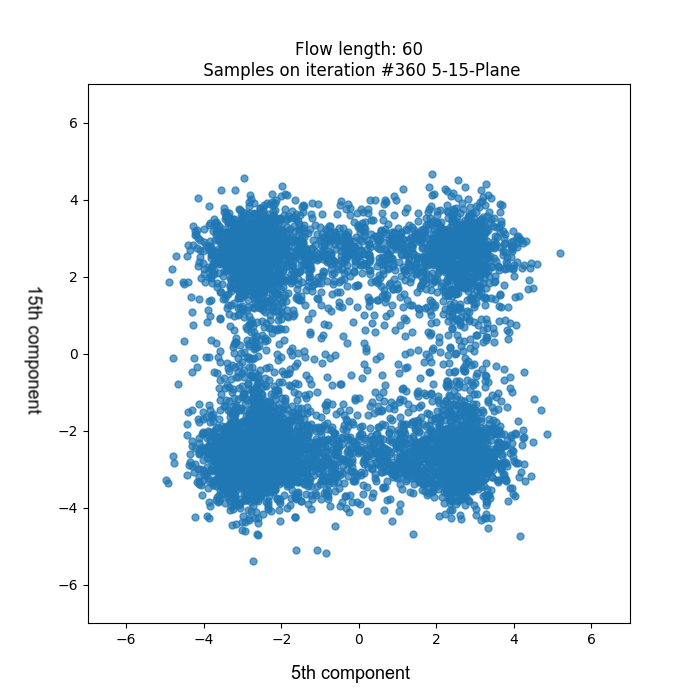}}
\newline
\subfloat[t=0.30]{\includegraphics[width=0.16\linewidth]{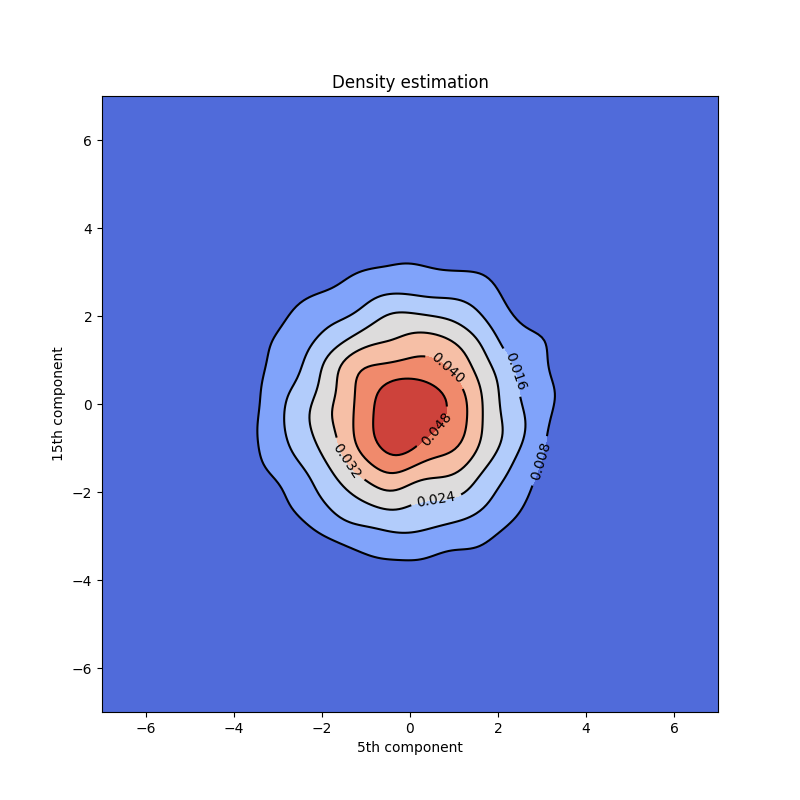}}
\subfloat[t=0.60]{\includegraphics[width=0.16\linewidth]{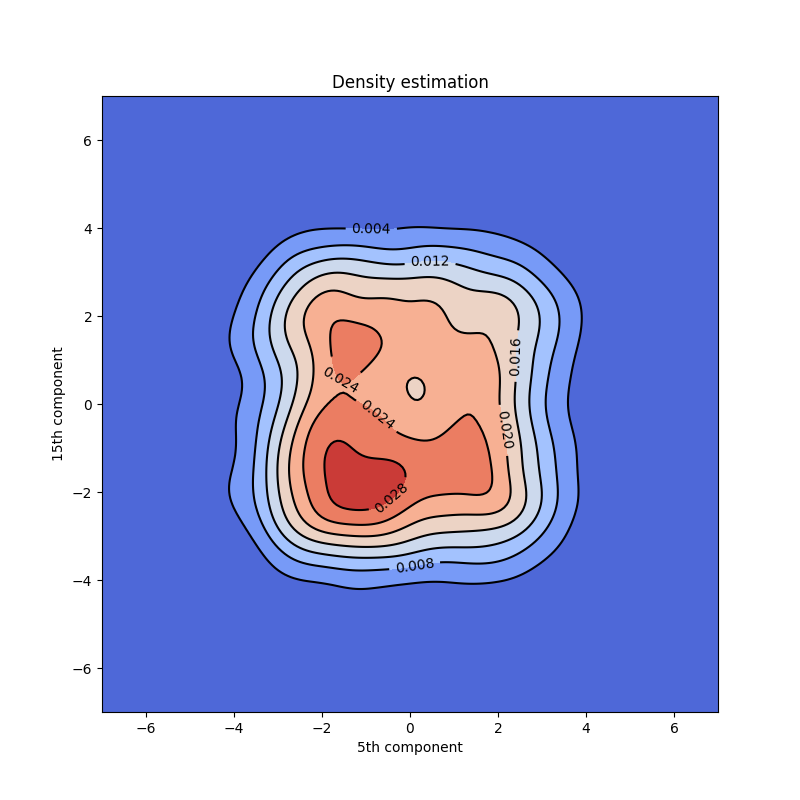}}
\subfloat[t=0.90]{\includegraphics[width=0.16\linewidth]{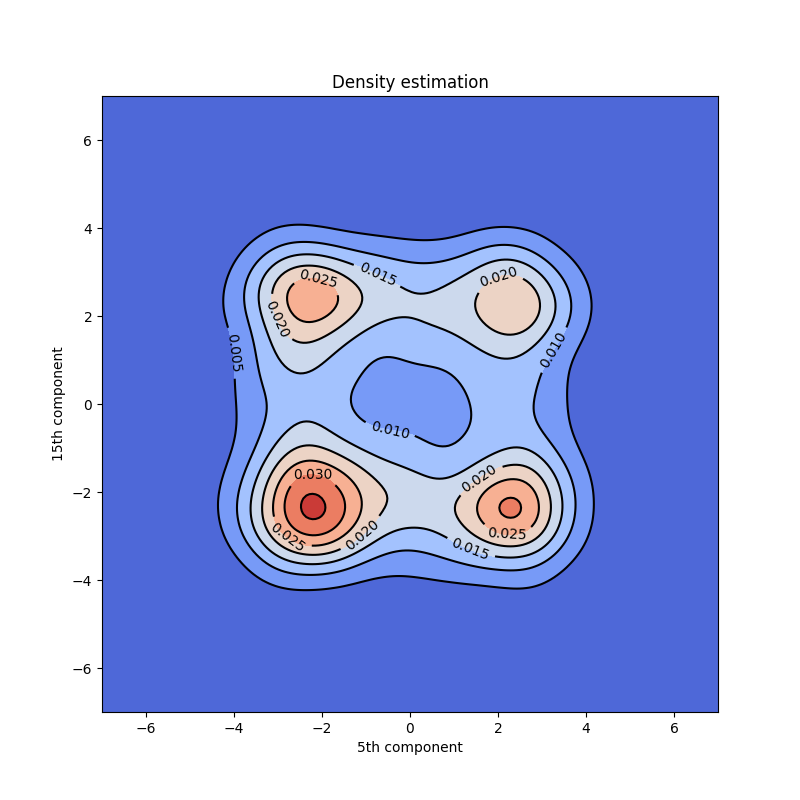}}
\subfloat[t=1.20]{\includegraphics[width=0.16\linewidth]{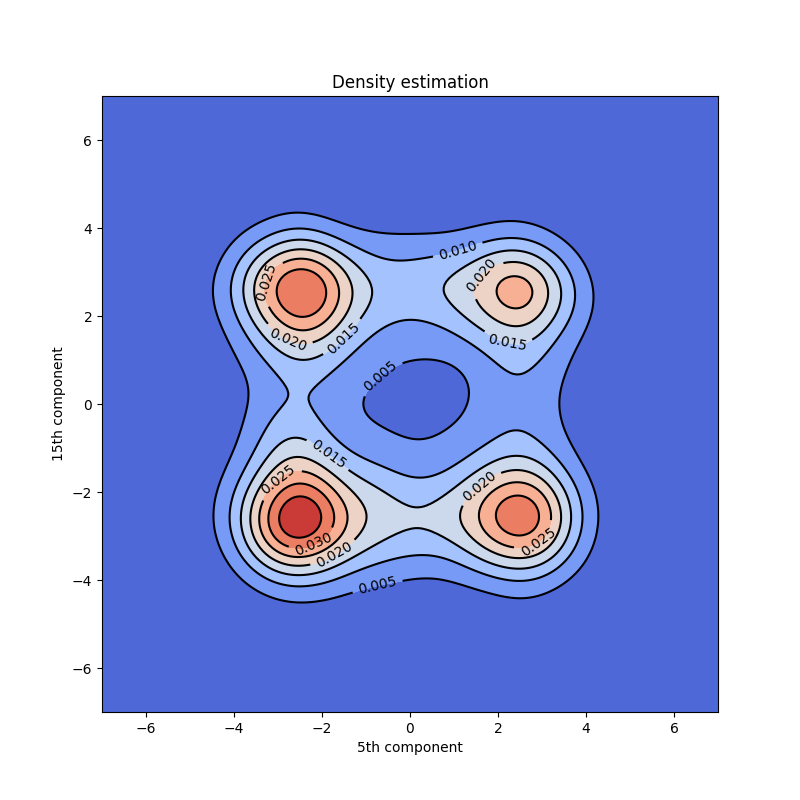}}
\subfloat[t=1.50]{\includegraphics[width=0.16\linewidth]{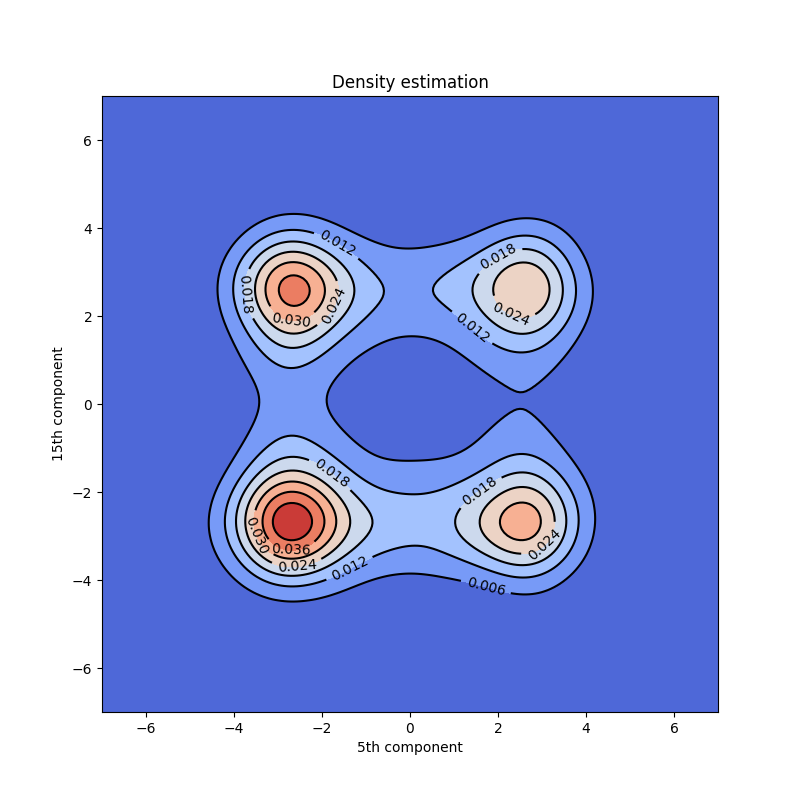}}
\subfloat[t=1.80]{\includegraphics[width=0.16\linewidth]{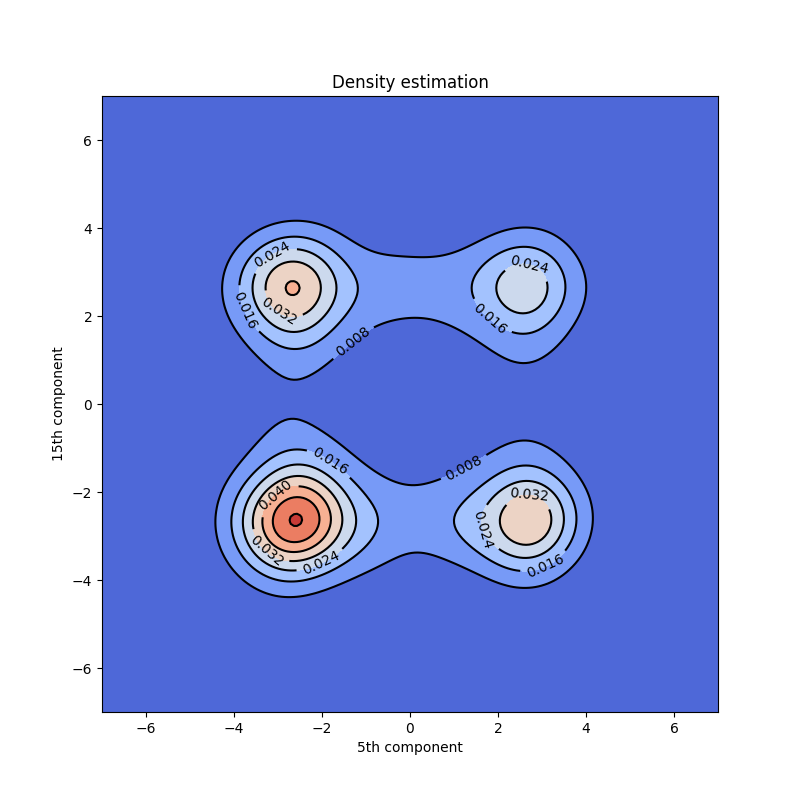}}
\caption{Sample points and estimated densities of $\rho_{\theta_t}$ on $5-15$ plane at different time nodes}
\label{sample plot S-T}
\end{figure}

In this special example, the potential function is the direct addition of same functions, we can exploit this property and show that any marginal distribution 
\begin{equation*}
  \varrho_j(x_j,t)=\int...\int \rho(x, t)~dx_1...dx_{j-1}dx_{j+1}...dx_d
\end{equation*}
of the solution $\rho_t$ solves the following the 1D Fokker--Planck equation:
\begin{equation}
  \frac{\partial\varrho(x,t)}{\partial t}=\frac{\partial}{\partial x}(\varrho(x,t) ~ V'(x))+{D}\Delta\varrho(x,t) \quad \varrho(\cdot, 0)=\mathcal{N}(0,1) \quad \textrm{with}~ V(x)=\frac{3}{50}(x^4-16x^2+5x). \label{1 d FPE}
\end{equation}
We then solve the SDE associated to \eqref{1 d FPE}:
\begin{equation}
  dX_t = - V'(X_t)~dt + \sqrt{2{D}} dB_t \quad X_0\sim\mathcal{N}(0,1). \label{Euler Maruyama 1D SDE}
\end{equation}
Since \eqref{Euler Maruyama 1D SDE} is an SDE in one dimensional space, we can solve it with high accuracy by Euler-Maruyama scheme \cite{kloeden2013numerical} and use it as a benchmark for our numerical solution. The following Figure \ref{1D compare } exhibits both the estimated densities for our numerical solutions (marginal distribution on the 15th component) and the solution of \eqref{Euler Maruyama 1D SDE} given by Euler-Maruyama scheme with step size $0.005$. The sample sizes for both solutions equal to $6000$.
\begin{figure}[!htb]
\minipage{0.165\textwidth}
  \includegraphics[width=\linewidth]{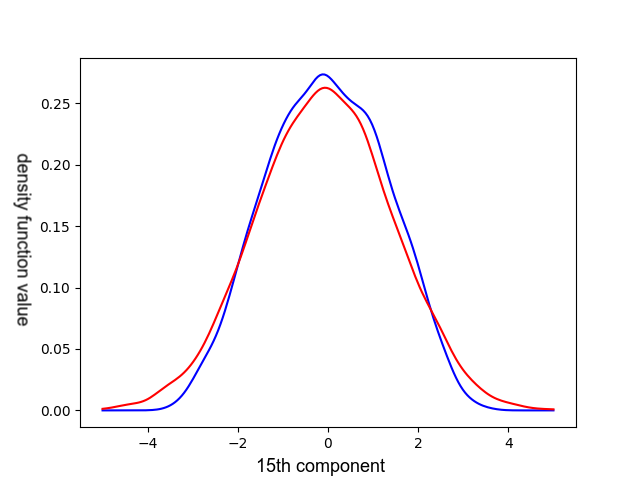}
  \caption*{t=0.30}
\endminipage\hfil
\minipage{0.165\textwidth}
  \includegraphics[width=\linewidth]{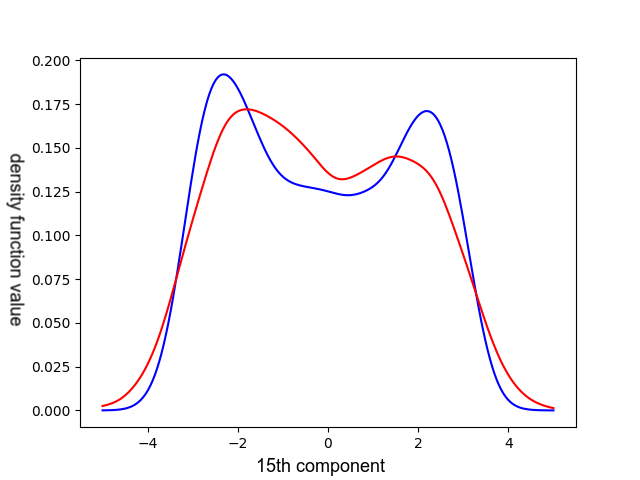}
  \caption*{t=0.60}
\endminipage\hfill
\minipage{0.165\textwidth}
  \includegraphics[width=\linewidth]{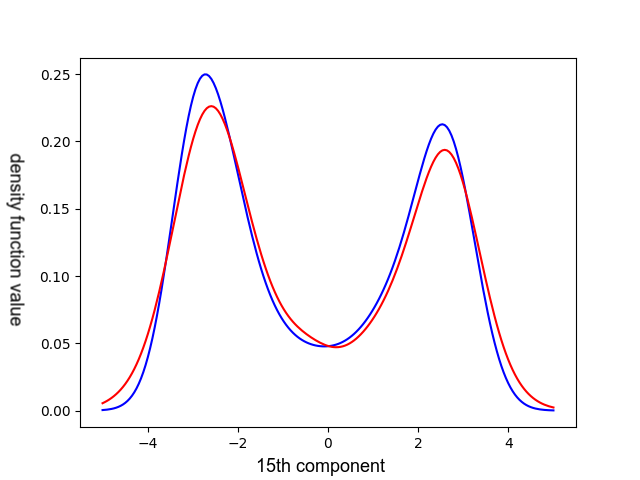}
  \caption*{t=0.90}
\endminipage\hfill
\minipage{0.165\textwidth}%
  \includegraphics[width=\linewidth]{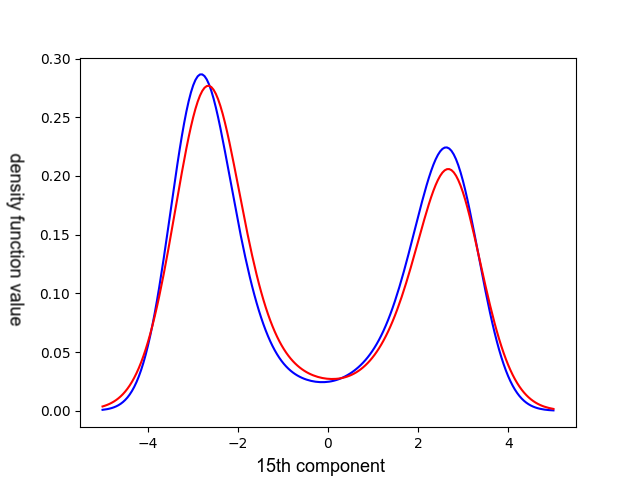}
  \caption*{t=1.20}
\endminipage
\minipage{0.165\textwidth}%
  \includegraphics[width=\linewidth]{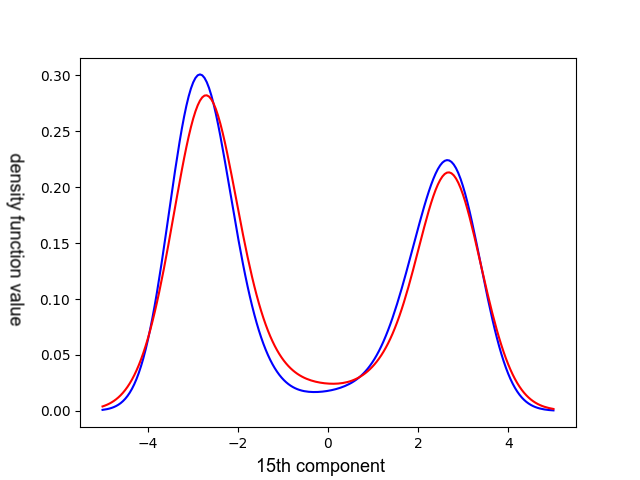}
  \caption*{t=1.50}
\endminipage
\minipage{0.165\textwidth}%
  \includegraphics[width=\linewidth]{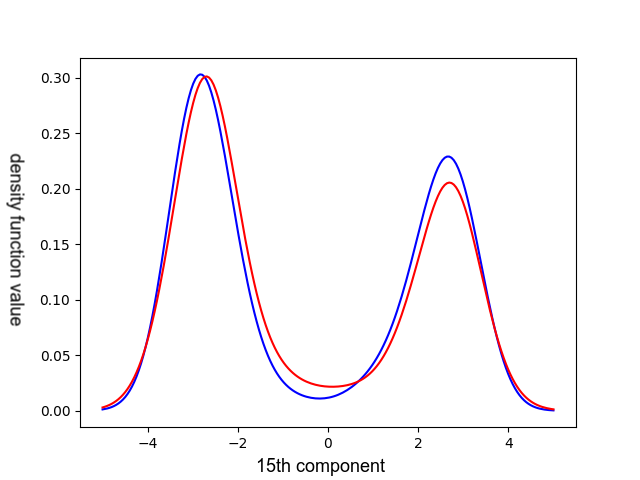}
  \caption*{t=1.80}
\endminipage
\caption{Estimated densities of our numerical solution(red) (projected onto the 15th component) and the solution given by Euler Maruyama scheme(blue) }\label{1D compare }
\end{figure}

We also illustrate the graphs of $\psi_{\hat{\nu}}$ on $5-15$ plane trained at different time steps in Figure \ref{Graph psi S-T }. 
\begin{figure}[!htb]
\minipage{0.33\textwidth}
  \includegraphics[width=0.8\linewidth]{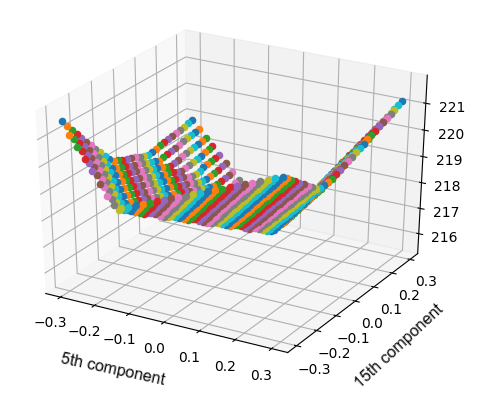}
  \caption*{$\psi_{\hat{\nu} } $ at $k=30$}
\endminipage\hfill
\minipage{0.33\textwidth}
  \includegraphics[width=0.8\linewidth]{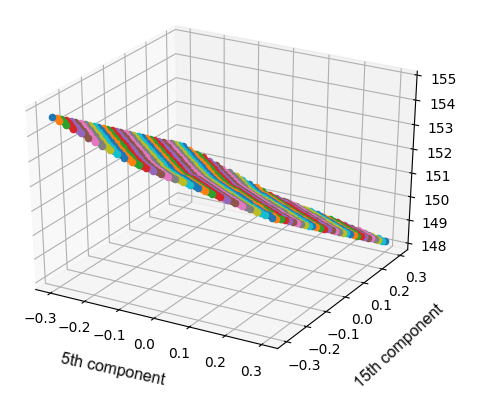}
  \caption*{$\psi_{\hat{\nu}}$ at $k=150$}
\endminipage\hfill
\minipage{0.33\textwidth}
  \includegraphics[width=0.8\linewidth]{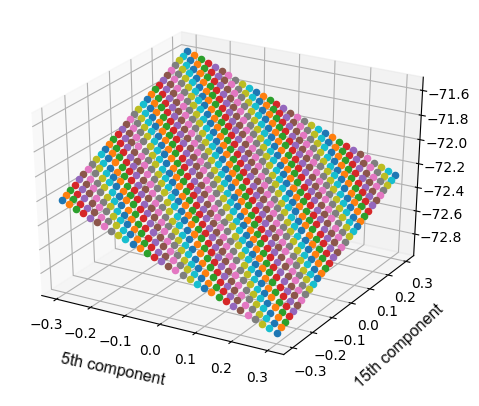}
  \caption*{$\psi_{\hat{\nu}}$ at  $k=360$}
\endminipage\hfill
\caption{Graph of $\psi_{\hat{\nu}}$ on $5-15$ plane trained at different time steps}\label{Graph psi S-T }
\end{figure}

\subsubsection{Affects of different initial distributions}
Different initial conditions $\rho_0$ affect the behavior of solutions of neural parametric Fokker--Planck equations differently, especially on the convergence speed to the Gibbs distribution. Here is an example. We consider $V$ as Styblinski-Tang potential in $\mathbb{R}^2$. We compute the solutions with three different initial distributions given as Gaussian distributions with covariances 
\begin{equation*}
\Sigma_1=\left[\begin{array}{cc} 
1 &  \\
 & 1
\end{array}\right],\Sigma_2=\left[\begin{array}{cc}
  \frac{13}{8} & \frac{5}{8} \\
 \frac{5}{8} & \frac{13}{8}
\end{array}\right],\Sigma_3=\left[\begin{array}{cc}
  \frac{13}{8} & -\frac{5}{8} \\
-\frac{5}{8} & \frac{13}{8}
\end{array}\right], \end{equation*}
respectively. Although the solutions converge to the Gibbs distribution, as expected from the theory, regardless of the initial density, their convergence speed may be different. 
Figure \ref{fig_c} shows the initial distributions and the corresponding densities (which are the estimations of the samples obtained from our algorithm) at $t=1.0$. 
As we can observe, the numerical result produced by $\rho_0=\mathcal{N}(0,\Sigma_1)$ is already close to Gibbs distribution at $t=1.0$, while numerical results associated to $\Sigma_2, \Sigma_3$ still have noticeable differences from Gibbs. They seems to be trapped in intermediate metastable statuses that are clearly influenced by the orientations in initial distributions.
\begin{figure}[!htb]
\minipage{0.16\textwidth}
  \includegraphics[width = \linewidth]{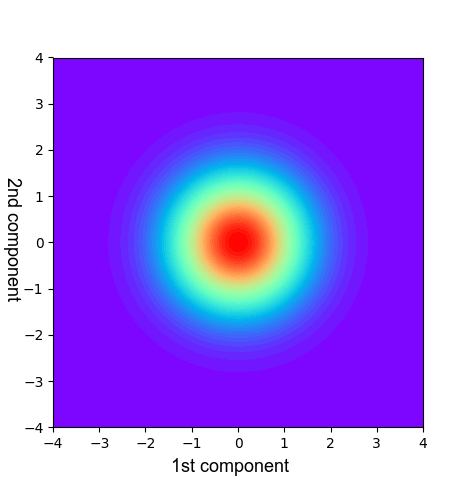}
  \caption*{$\rho_0 = \mathcal{N}(0, \Sigma_1)$}
\endminipage\hfill
\minipage{0.16\textwidth}
  \includegraphics[width = \linewidth]{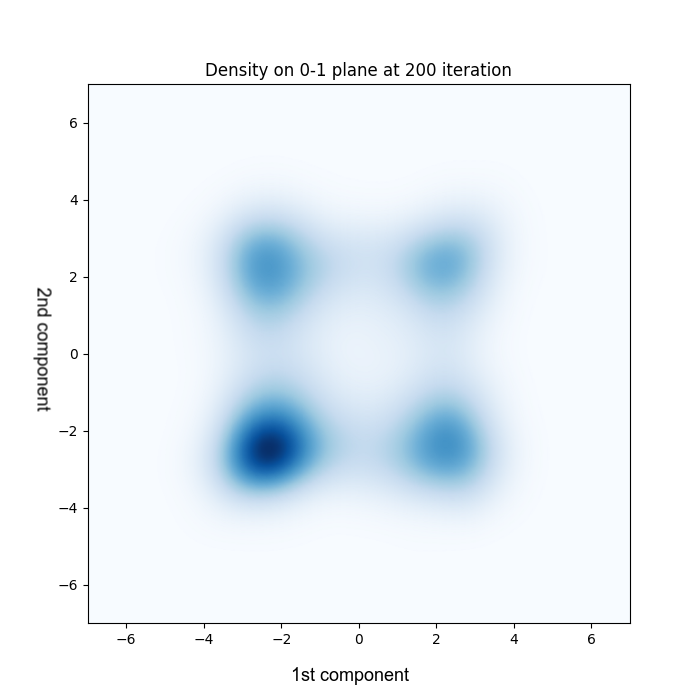}
  \caption*{$t=1.0$}
\endminipage\hfill
\minipage{0.16\textwidth}
  \includegraphics[width = \linewidth]{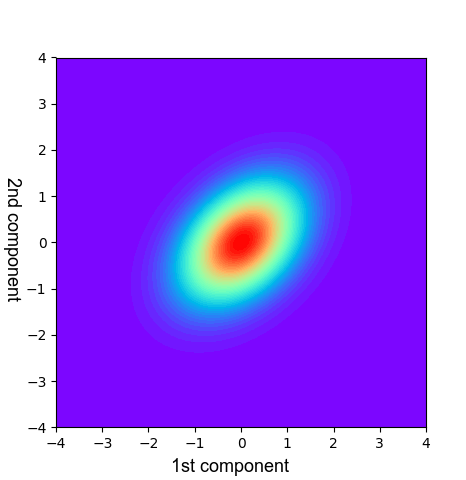}
  \caption*{$\rho_0 = \mathcal{N}(0, \Sigma_2)$}
\endminipage\hfill
\minipage{0.16\textwidth}
  \includegraphics[width = \linewidth]{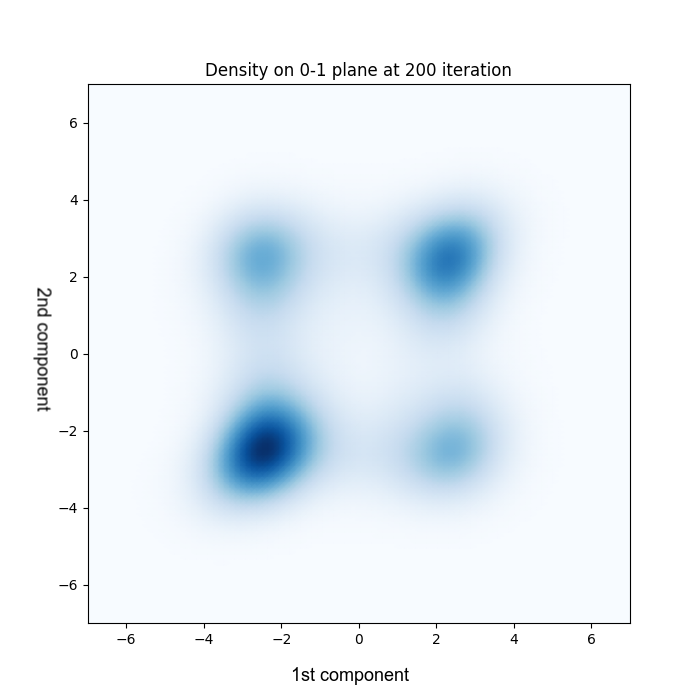}
  \caption*{$t=1.0$}
\endminipage\hfill
\minipage{0.16\textwidth}
  \includegraphics[width = \linewidth]{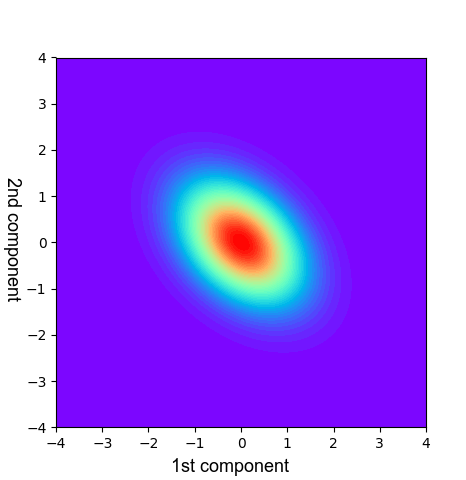}
  \caption*{$\rho_0 = \mathcal{N}(0, \Sigma_3)$}
\endminipage\hfill
\minipage{0.16\textwidth}
  \includegraphics[width = \linewidth]{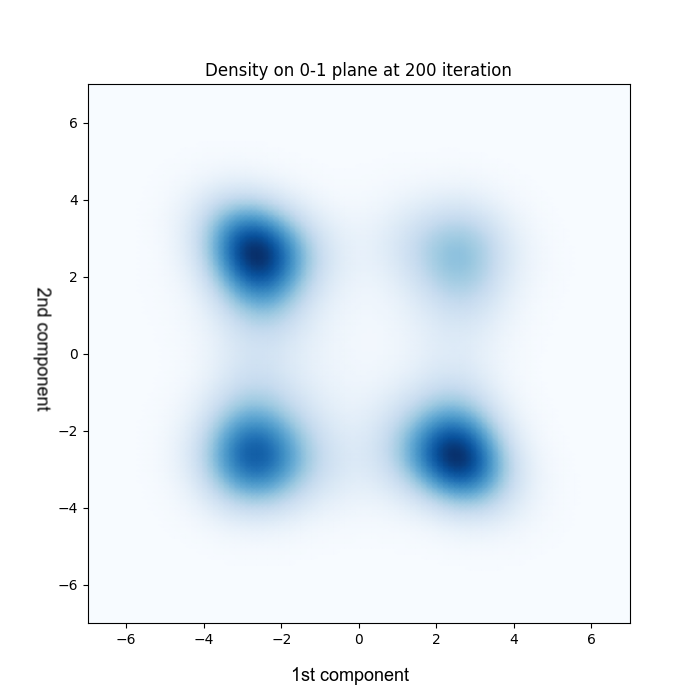}
  \caption*{$t=1.0$}
\endminipage\hfill
\caption{Different behaviors of numerical solution with different  $\rho_0$s}\label{fig_c}
\end{figure}

In general, we believe that the choice of $\rho_0$ affects the behavior of numerical solution.  Choosing a suitable $\rho_0$ may shorten the computing time in the training process.

\subsubsection{Solving the equation with different diffusion coefficients}
The different behaviors of the Fokker--Planck equation caused by different diffusion coefficients ${D}$ can be captured by our algorithm. As the following figure shows, we apply our method to solve Fokker--Planck equation with Styblinski-Tang potential function with ${D}=0.1,1.0,10.0$ and exhibit samples points and estimated density surfaces at the time $t=3.0$.
\begin{figure}[ht]
\centering
\subfloat[Samples ${D}=10$]{
\includegraphics[width=0.2\linewidth]{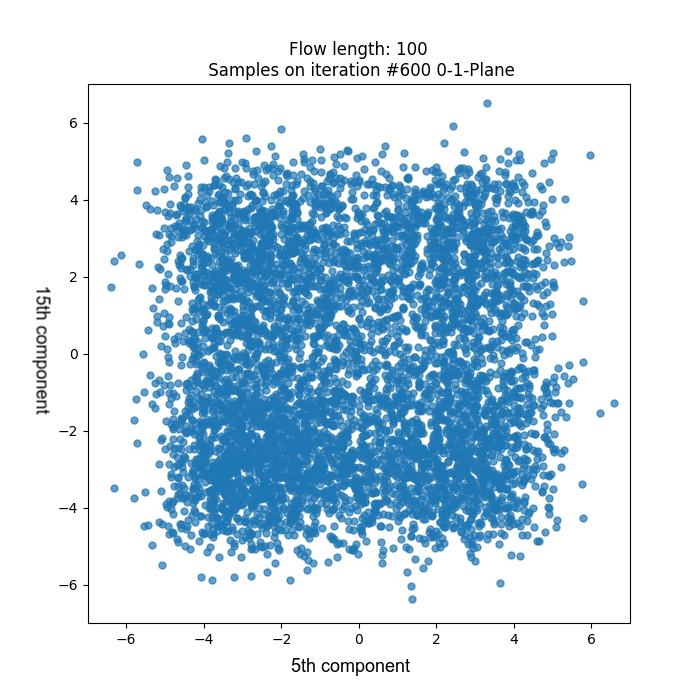}} \hspace{1cm} 
\subfloat[Samples ${D}=1$]{
\includegraphics[width=0.2\linewidth]{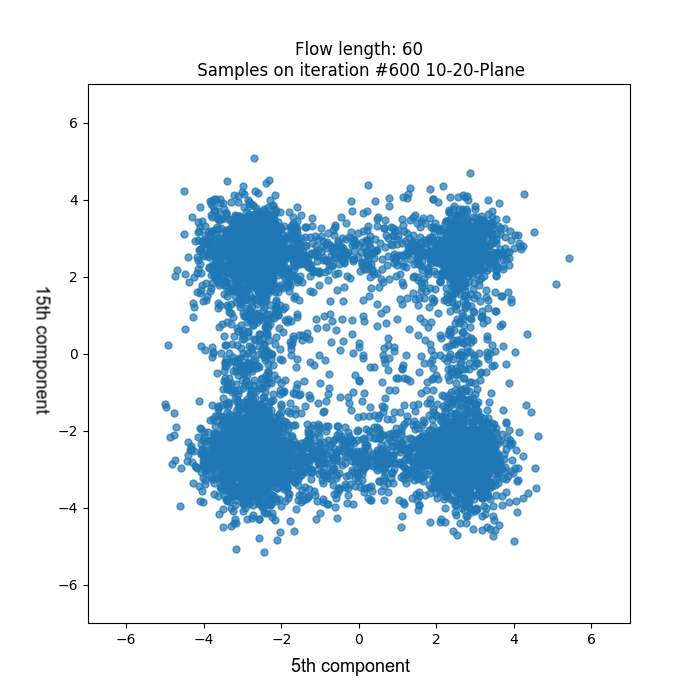}} \hspace{1cm}
\subfloat[Samples ${D} = 0.1$]{
\includegraphics[width=0.2\linewidth]{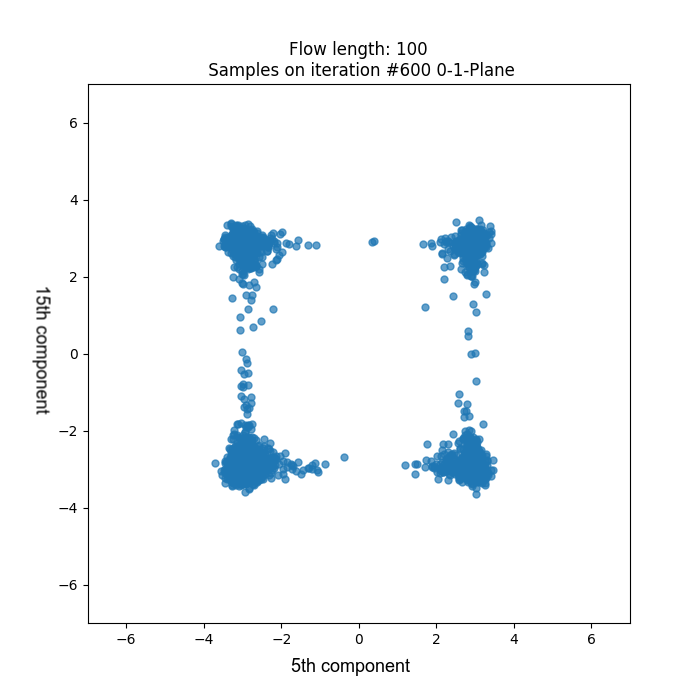}}
\qquad
\subfloat[Density ${D}=10$]{
\includegraphics[width=0.32\textwidth]{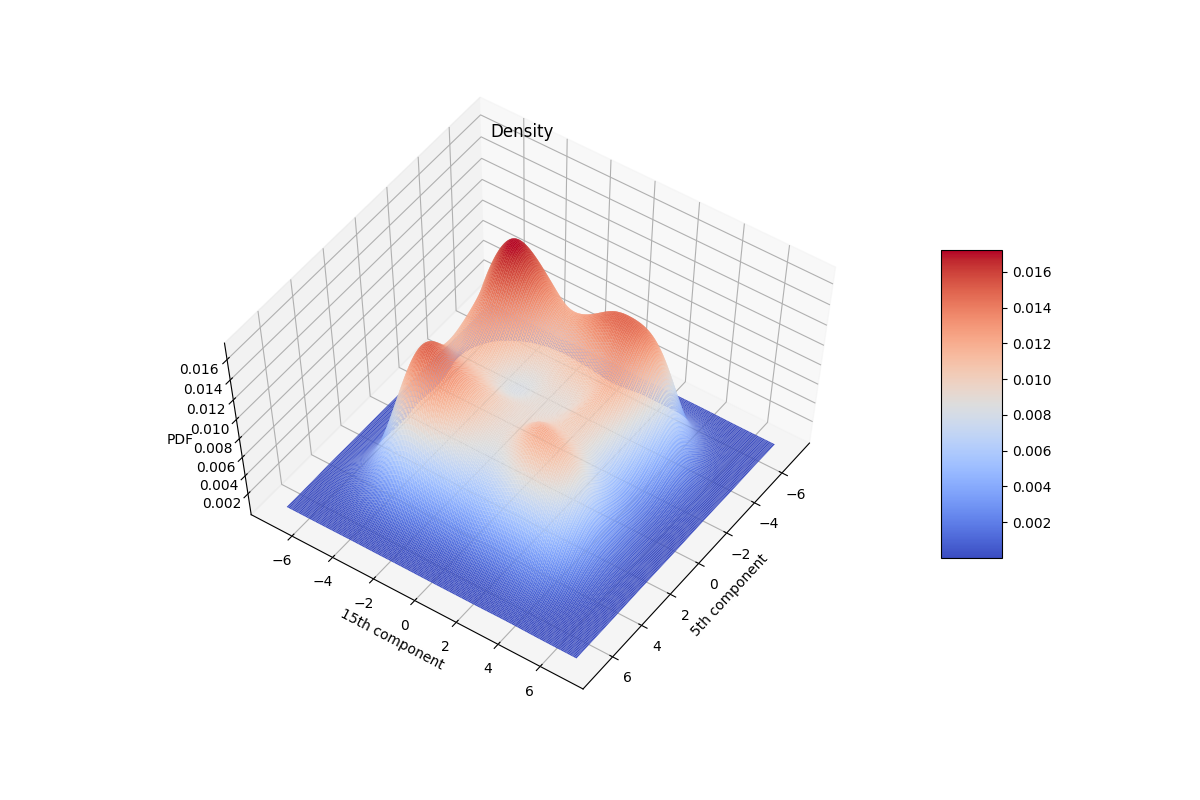}}
\subfloat[Density ${D}=1$]{
\includegraphics[width=0.32\textwidth]{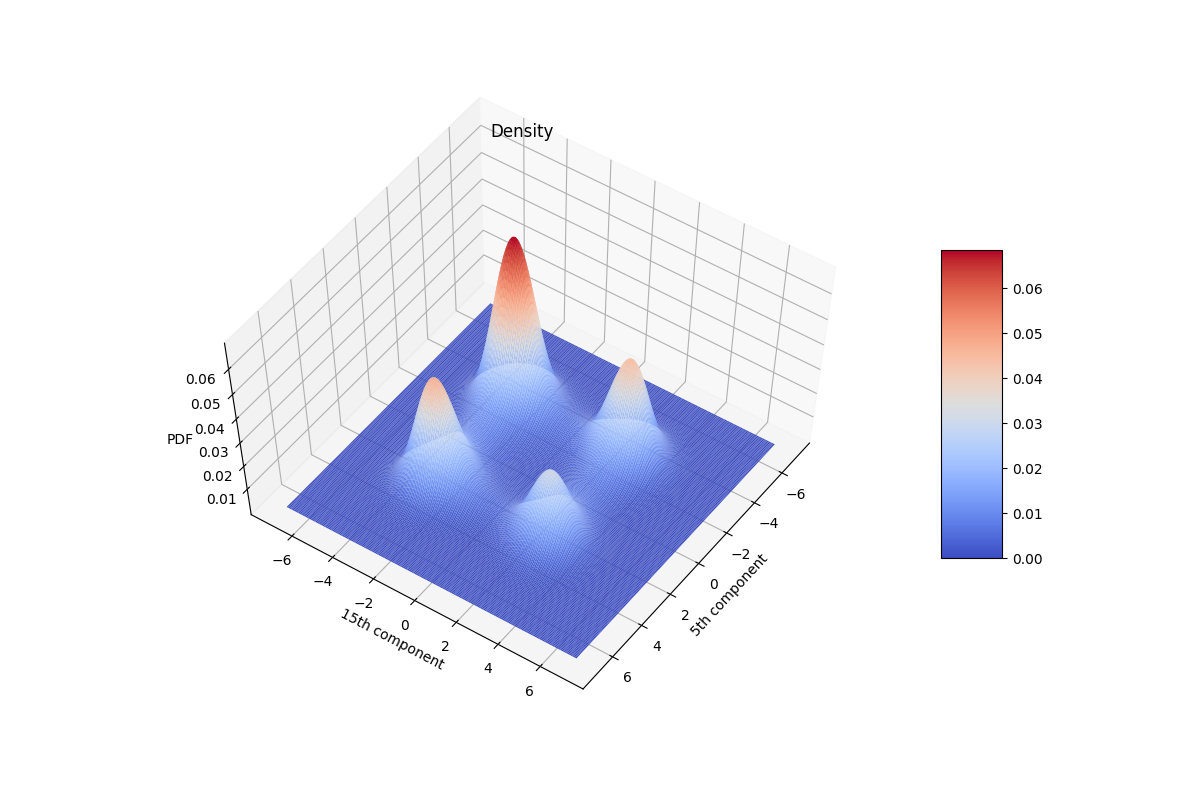}}
\subfloat[Density ${D}=0.1$]{
\includegraphics[width=0.32\textwidth]{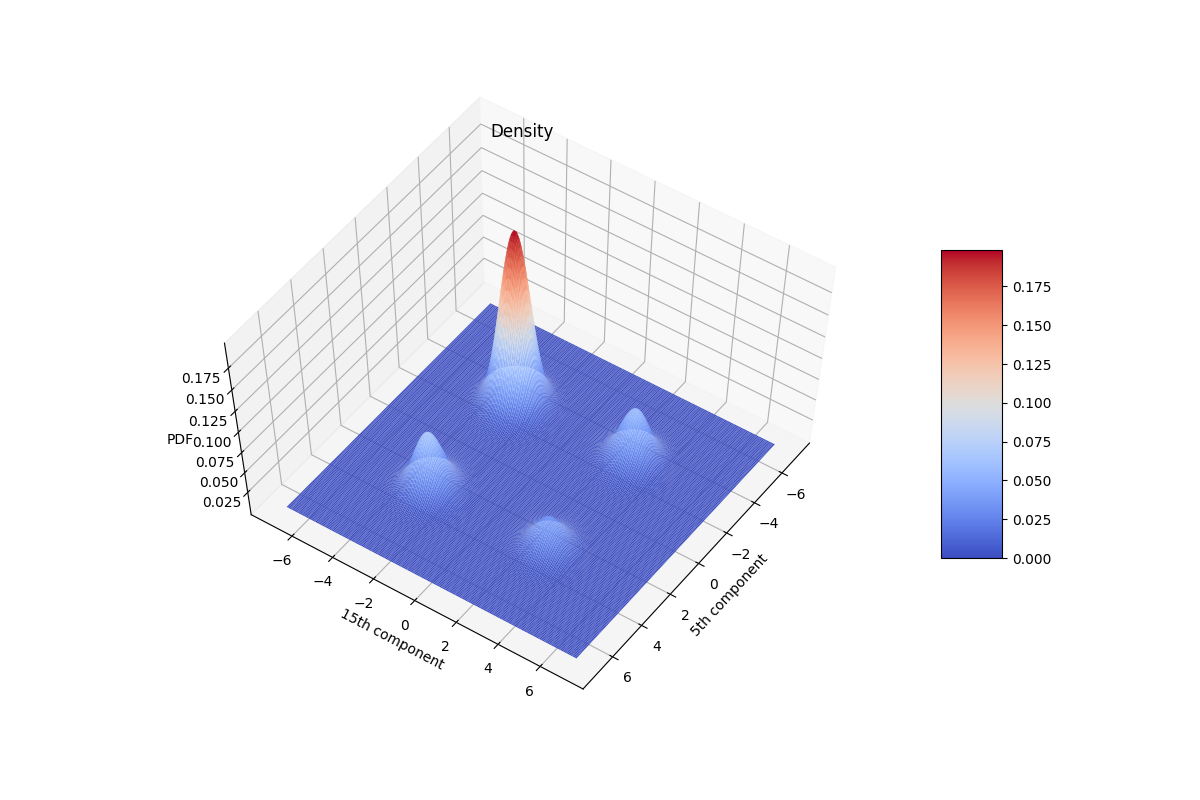}}
\caption{Samples and estimated densities at $t=3.0$, from left to right: ${D}=10$, ${D}=1.0$, ${D}=0.1$}
\label{fig:globfig}
\end{figure}

\subsubsection{Rosenbrock potential}
In this example, we set dimension $d=10$. We consider the Rosenbrock typed function \cite{rosenbrock1960automatic}:
\begin{equation*}
 V(x)=\frac{3}{50}\left(\sum_{i=1}^{d-1} 10 (x_{k+1}-x_k^2)^2 + (x_k-1)^2 \right),
\end{equation*}
which involve interactions among its coordinates.
We solve the corresponding \eqref{FPE} on time interval $[0,1]$ with step size $h=0.005$. We set the length of normalizing flow $T_\theta$ as 100. We set $K_{\textrm{in}}=K_{\textrm{out}}=3000$ and $M_{\textrm{in}}=100$, $M_{\textrm{out}}=60$.\\
Here are the sample results, we exhibit the projection of sample points on the $1-2$, $7-8$ and $9-10$ plane in Figure \ref{samples Rsbrck}. Blue samples are obtained from our numerical solution while red samples are obtained by applying Euler-Maruyama scheme with the same step size.

\begin{figure}[tbhp]
\centering
\subfloat[$t=0.05$]{\includegraphics[width=0.24\linewidth]{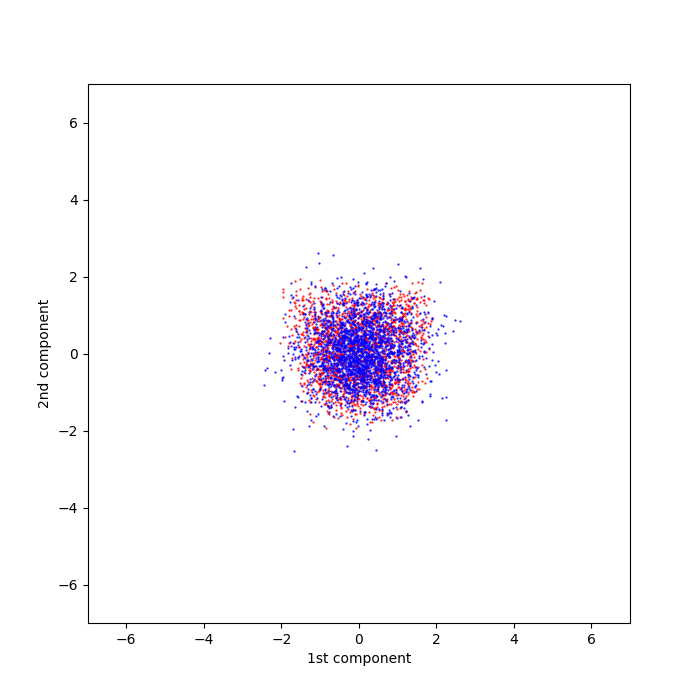}}
\subfloat[$t=0.35$]{\includegraphics[width=0.24\linewidth]{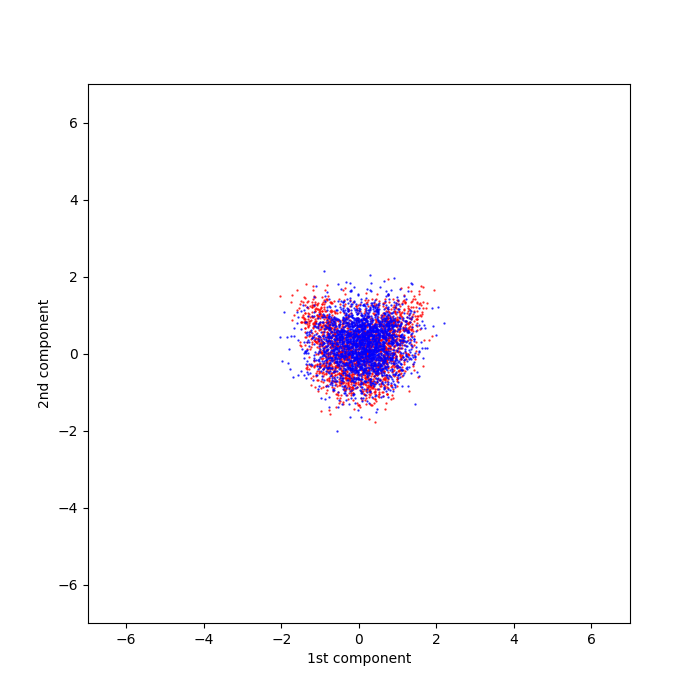}}
\subfloat[$t=0.50$]{\includegraphics[width=0.24\linewidth]{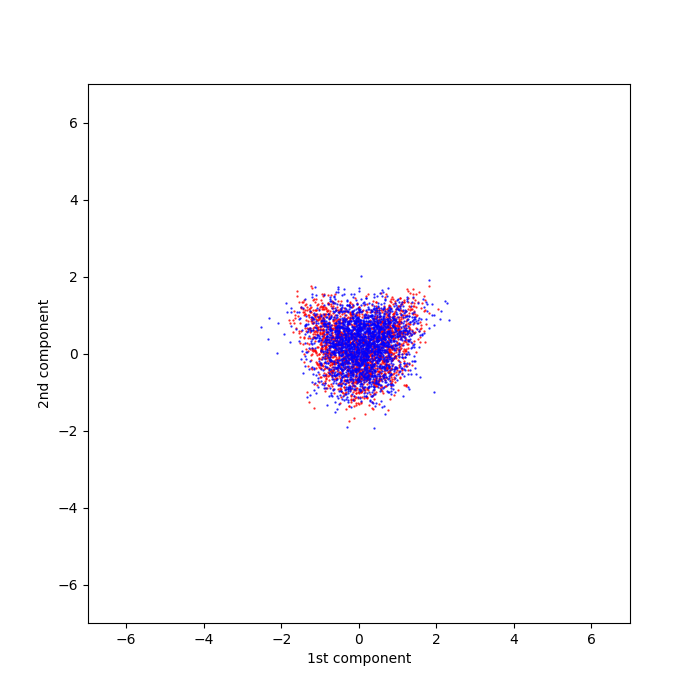}}
\subfloat[$t=1.00$]{\includegraphics[width=0.24\linewidth]{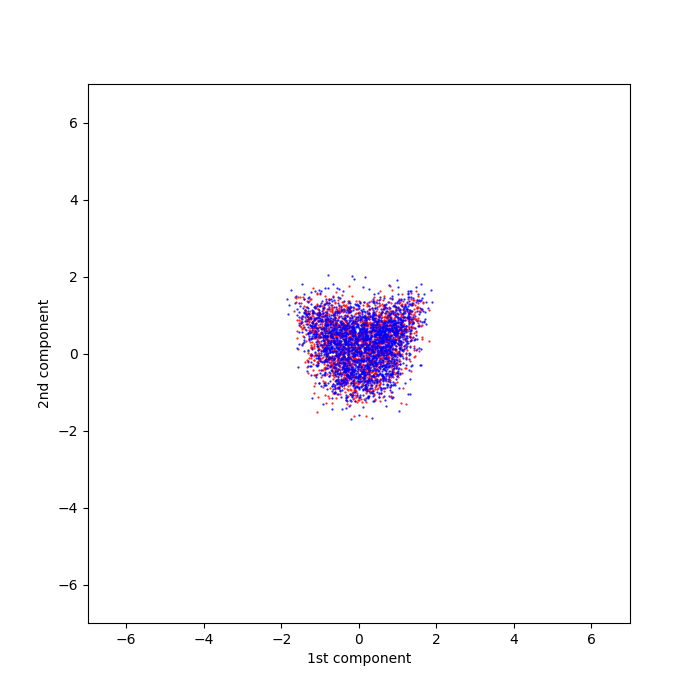}}
\newline
\subfloat[$t=0.05$]{\includegraphics[width=0.24\linewidth]{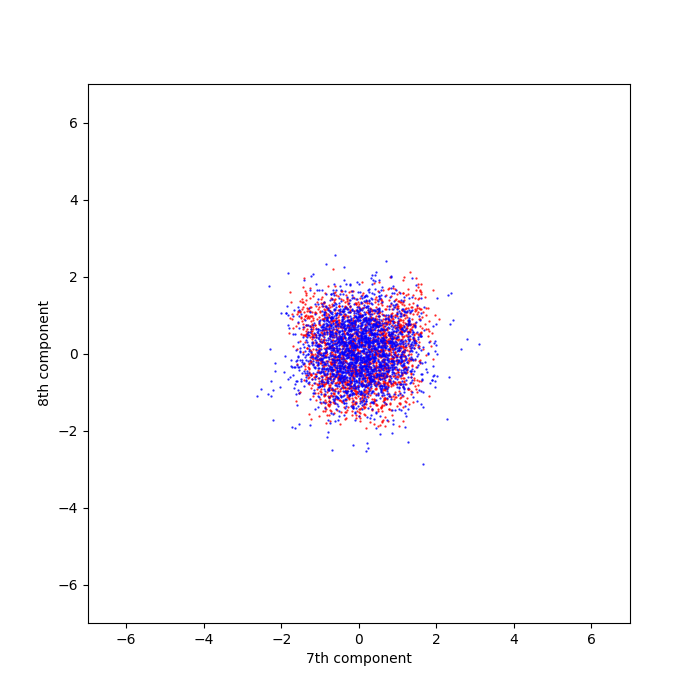}}
\subfloat[$t=0.35$]{\includegraphics[width=0.24\linewidth]{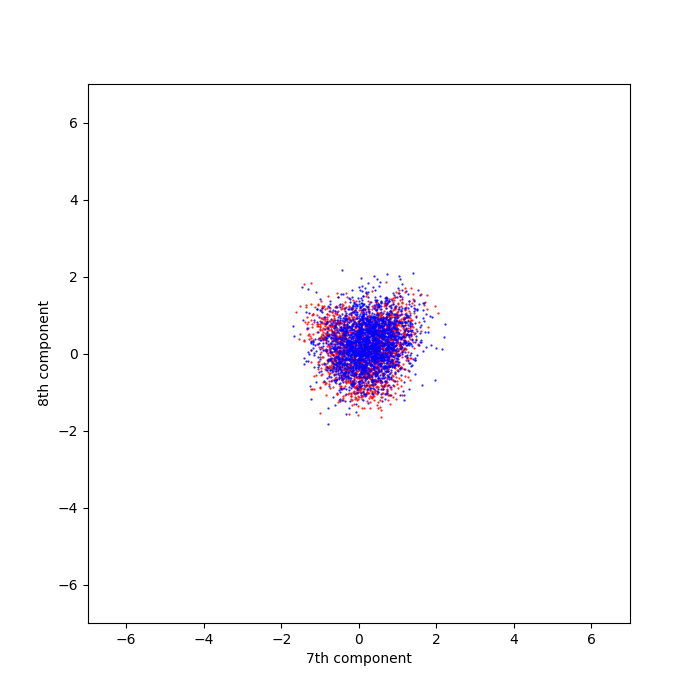}}
\subfloat[$t=0.50$]{\includegraphics[width=0.24\linewidth]{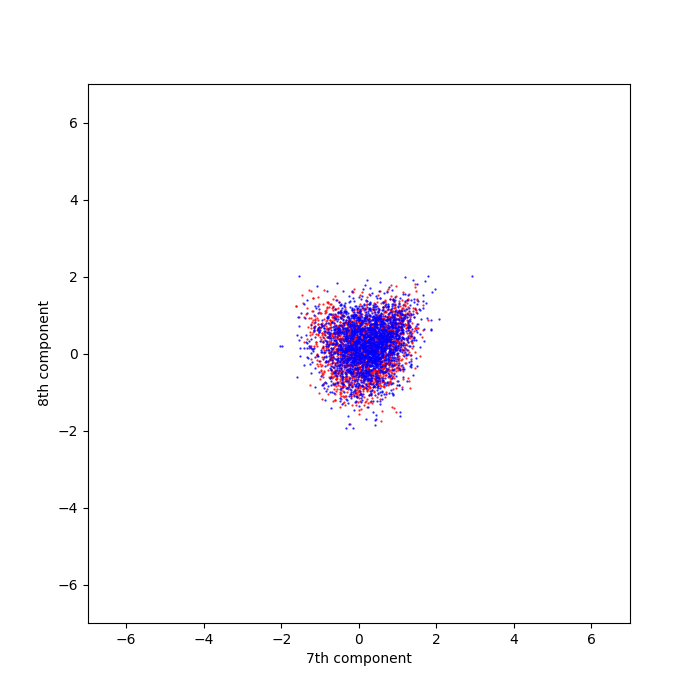}}
\subfloat[$t=1.00$]{\includegraphics[width=0.24\linewidth]{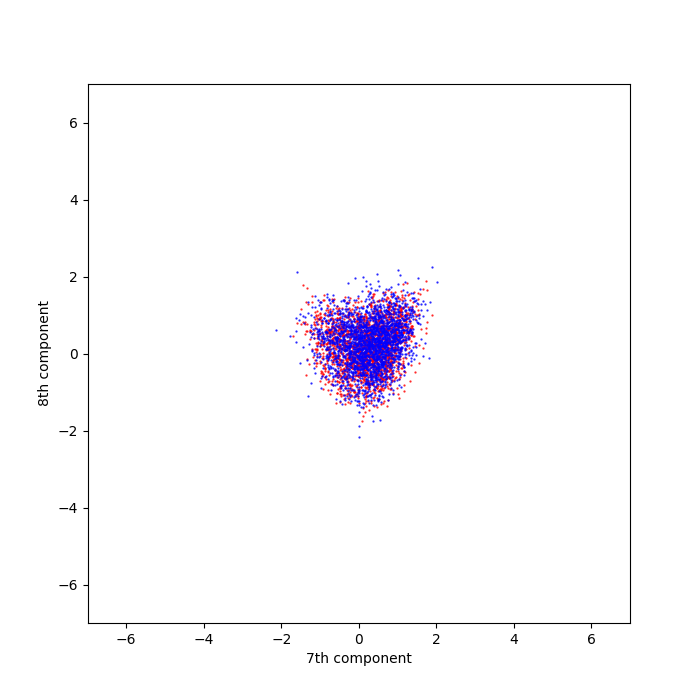}}
\newline
\subfloat[$t=0.05$]{\includegraphics[width=0.24\linewidth]{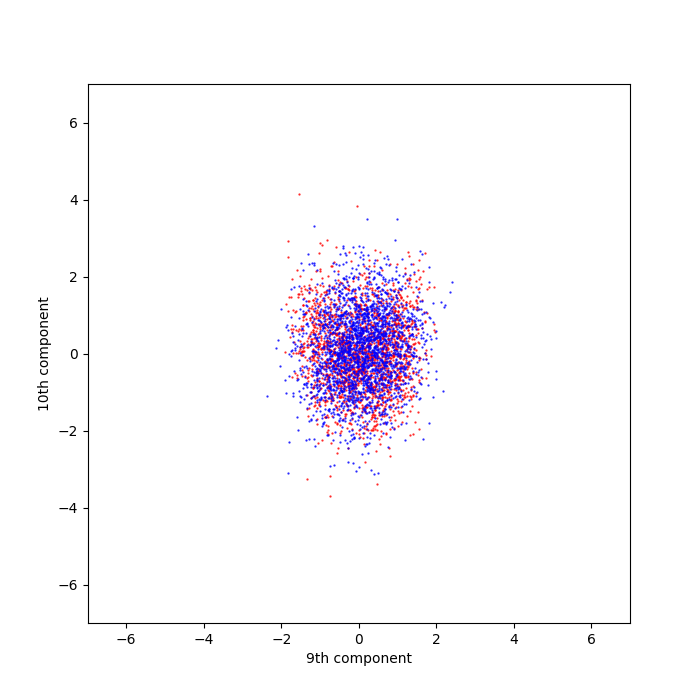}}
\subfloat[$t=0.35$]{\includegraphics[width=0.24\linewidth]{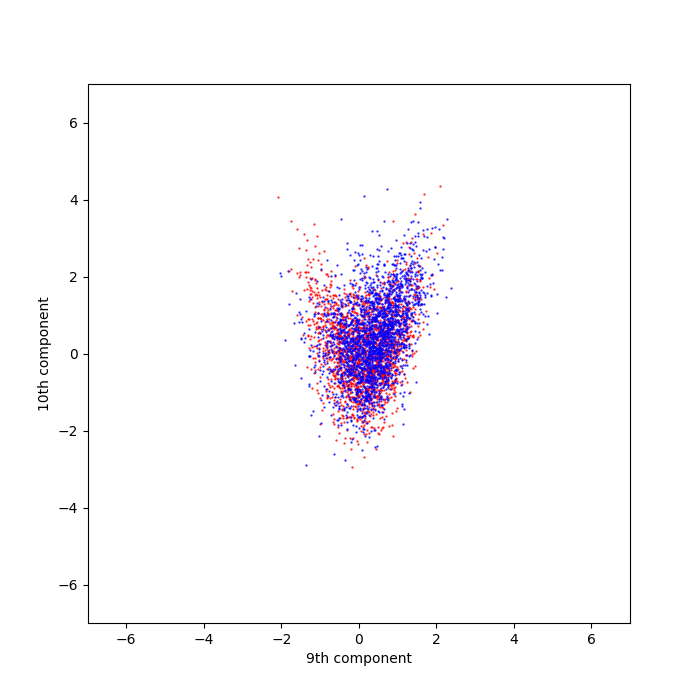}}
\subfloat[$t=0.50$]{\includegraphics[width=0.24\linewidth]{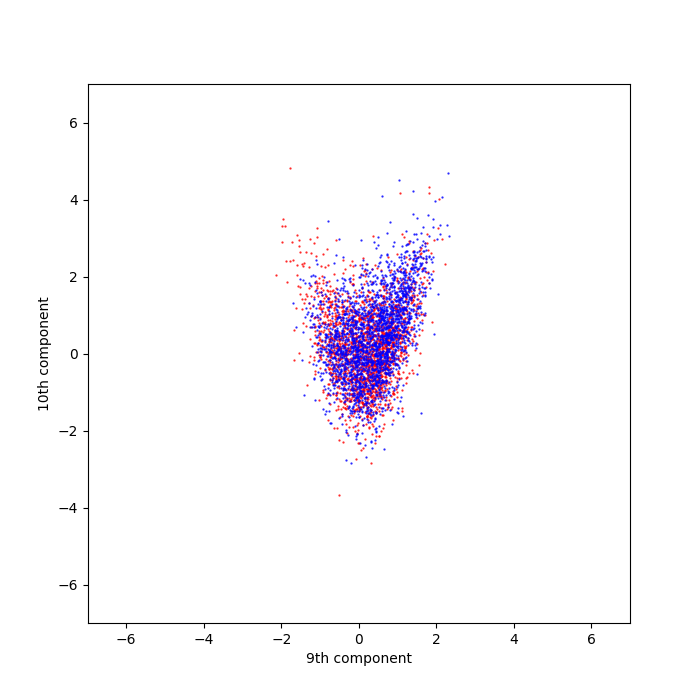}}
\subfloat[$t=1.00$]{\includegraphics[width=0.24\linewidth]{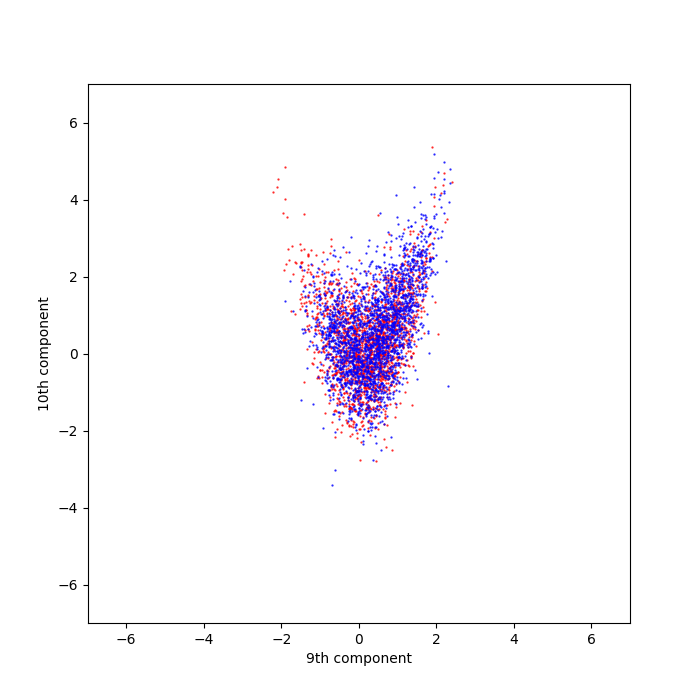}}
\caption{Samples of our numerical solution (blue) and Euler-Maruyama (red) on different planes at different time nodes}
\label{samples Rsbrck}
\end{figure}


\subsection{Discussion on time consumption}
we should point out that the running time of our algorithm depends on the following three aspects: 
\begin{enumerate}
   \item[(i)] Dimension $d$ of the problem; potential function $V$;
   \item[(ii)] The size of normalizing flow $T_\theta$ and fully connected neural network $\psi_\nu$;
   \item[(iii)] Number of time steps $N$; outer iterations $M_{\textrm{out}}$; inner iterations $M_{\textrm{in}}$; sample size $K_{\textrm{out}}$ and $K_{\textrm{in}}$.
\end{enumerate} 
Among them, the networks in (ii) are selected according to (i). The hyper-parameters $M_{\textrm{out}},M_{\textrm{out}},K_{\textrm{out}},K_{\textrm{in}}$ in (iii) are chosen based on our trial and error as well as Remark \ref{rmk:mention large sample} stated earlier in this paper. 

All numerical examples reported in this paper are computed on a Laptop with Intel Core™ i5-8250U CPU @ 1.60GHz × 8 processor. For most of the high dimensional examples ($d\geq 10$), we choose the length of $T_\theta$ between $60$ and $100$; for the ReLU network $\psi_\nu$, we set its number of layers equal to $6$ with hidden dimension $20$. We set $M_{\textrm{out}}\sim 50, M_{\textrm{in}}\sim 100$ and choose sample sizes $K_{\textrm{out}}, K_{\textrm{in}}$ according to Remark \ref{rmk:mention large sample}. The total running time is ranged in $20-40$ hours.
    
    We observe that the running time of our algorithm is dominated by the inner loop of Algorithm 4.1, i.e. the part for optimizing over $\psi_\nu$. The cost associated with this part can be estimated as  $O(N\cdot M_{\textrm{out}}\cdot M_{\textrm{in}}\cdot(K_{\textrm{in}}  t_a + t_b) )$, where $t_a$ denotes the time cost of using backpropagation to evaluate the gradient w.r.t. $\nu$ of each $|\nabla\psi_\nu(T_{\theta_0}(\mathbf{X}_k))-(T_\theta(\mathbf{X}_k)-T_{\theta_0}(\mathbf{Y}_k))|^2$ in every inner loop of Algorithm 4.1, and $t_b$ denotes the time for updating $\nu$ by Adam method. Here $t_a, t_b$ both depend on $d,V$ and the sizes of networks $T_\theta,\psi_\nu$. According to our experiences, for most of the cases, $t_a$ is of the order of magnitude around $10^{-5}s$ and $t_b$ is around $10^{-2}s$. 
    
    Although the cost for our current implementation of the train process is still high, we want to remind that there is a distinct advantage in the sampling application, namely that the network training just needs to be done once. The trained network can be reused to generate samples, regardless the sample size, from distribution $\rho_t$ by pushing forward samples from the reference distribution $p$ with negligible additional cost. This is in sharp contrast to the classical MCMC sampling techniques, which requires to solve the SDE associated with Fokker--Planck equation by numerical methods, such as Euler-Maruyama scheme, for every sample.

\section{Discussion}
In this paper, we design and analyze an algorithm for computing the high dimensional Fokker--Planck equations. Our approach is based on transport information geometry with probability formulations arisen in deep learning generative models. We first introduce the parametric Fokker--Planck equations, a set of ODE, to approximate the original Fokker--Planck equation. The ODE can be viewed as the ``spatial'' discretization of the PDE using neural networks. We propose a variational version of the semi-implicit Euler scheme and design a discrete time updating algorithm to compute the solution of the parametric Fokker--Planck equations. Our method is a sampling based approach that is capable to handle high dimensional cases. It can also be viewed as an alternative of the JKO scheme used in conjunction with neural networks. More importantly, we prove the asymptotic convergence and error estimates, both under the Wasserstein metric, for our proposed scheme. 

We hope that our study may shed light on  principally designing deep neural networks 
and other machine learning approaches to 
compute solutions of high dimensional PDEs, and systematically analyzing their error bounds for understandable and trustworthy computations.
Our parametric Fokker--Planck equations are derived by approximating the density function in free energy using neural networks, and then following the rules in calculus of variation to get its Euler-Langrange equation. The energy law and principles in variational framework build a solid foundation for our ``spatial'' discretization that is able to inherit many desirable physical properties shared by the PDEs, such as relative entropy dissipation in a neural network setting. Our numerical scheme provides a systemic mechanism to design sampling efficient algorithms, which are critical for high dimensional problems. One distinction of our method is that, contrary to the data dependent machine learning studies in the literature, our approach does not require any knowledge of the ''data'' from the PDEs. In fact, we generate the ``data" to compute the numerical solutions, just like the traditional numerical schemes do for PDEs. More importantly, we carried out the numerical analysis, using tools such as KL divergence and Wasserstein metric from the transport information geometry, to study the the asymptotic convergence and error estimates in probability space. We emphasize that the Wasserstein metric provides a suitable geometric structure to analyze the convergence behavior in generative models, which are widely used in machine learning field.
For this reason, we believe that our investigations can be adopted to understand many machine learning algorithms, and to design efficient sampling strategies based on pushforward maps that can generate flows of samples in generative models.

We also believe that the approaches in algorithm design and error analysis developed in this study can be extended to other equations, such as porous media equation, Schr\"{o}dinger equation, and Schr\"{o}dinger bridge system, and many more. Those topics are worth to be further investigated in the future.

\section{Acknowledgment}
This work was partially supported by National Science Foundation grants DMS-1620345 and DMS-1830225 and by ONR grant N000141310408. The work of the second author was supported by a start-up fund from the University of South Carolina and NSF grant RTG:2038080. The work of the third author was partially supported by a grant from Shenzhen Research Institute of Big Data.

\appendix
\section{Proof of Lemma \ref{lemma:local err analys}}\label{pf hodge lemma}
\begin{customlemma}{3.3}
Suppose $\vec{u},\vec{v}$ are two vector fields defined on $\mathbb{R}^d$, suppose $\varphi,\psi$ solves $-\nabla\cdot(\rho\nabla\varphi)=-\nabla\cdot(\rho\vec{u})$ and $-\nabla\cdot(\rho\nabla\psi) = -\nabla\cdot(\rho\vec{v})$, or equivalently, $\textrm{Proj}_{\rho}[\vec{u}]=\nabla\varphi$ and $\textrm{Proj}_{\rho}[\vec{v}]=\nabla\psi$ (cf. Definition \ref{Hodge Decomp}). Then:
\begin{align}
   & \int \vec{u}(x)\cdot \nabla\psi(x)\rho(x)~dx = \int  \nabla\varphi(x)\cdot \nabla\psi(x) \rho(x)~dx; \tag{\ref{loc err analys lemma 1}}\\
  &  \int |\nabla\psi(x)|^2\rho(x)~dx\leq \int |\vec{v}(x)|^2\rho(x)~dx. \tag{\ref{loc err analys lemma 2}}
\end{align}
\end{customlemma}
\begin{proof}[Proof of Lemma \ref{lemma:local err analys}]
For \eqref{loc err analys lemma 1}:
\begin{equation*}
  \int \vec{u}(x) \cdot \nabla\psi(x)\rho(x)~dx = \int -\nabla\cdot(\rho(x)\vec{u}(x))\psi(x)~dx =  \int-\nabla\cdot(\rho(x)\nabla\varphi(x))\psi(x)~dx = \int \nabla\varphi(x)\cdot\nabla\psi(x)\rho(x)~dx.
\end{equation*}
For \eqref{loc err analys lemma 2}:
\begin{align*}
\int  |\vec{v}(x)|^2\rho(x)~dx = & \int ( |\nabla\psi(x)|^2 + 2(\vec{v}(x) -  \nabla\psi(x))\cdot\nabla\psi(x) + |\vec{v}(x)-\nabla\psi(x)|^2)\rho(x)~dx \\
= & \int (|\nabla\psi(x)|^2 + |\vec{v}(x)-\nabla\psi(x)|^2)\rho(x)~dx \geq \int |\nabla\psi(x)|^2\rho(x)~dx.
\end{align*}
The second equality is due to \eqref{loc err analys lemma 1}.
\end{proof}

\section{Proof of Theorem \ref{theorem_submfld} }\label{pf on subM}

\begin{customthm}{3.7}
 Suppose $\{\theta_t\}_{t\geq 0 }$ solves (\ref{wass_grad_flow_on_para_spc}). Then $\{\rho_{\theta_t}\}$ is the gradient flow of $\mathcal{H}$ on probability submanifold $\mathcal{P}_{\Theta}$. Furthermore, at any time $t$, $\dot\rho_{\theta_{t}}=\frac{d}{dt} \rho_{\theta_t} \in \mathcal{T}_{\rho_{\theta_{t}}}\mathcal{P}_\Theta$ is the orthogonal projection of $-\textrm{grad}_W\mathcal{H}(\rho_{\theta_t})\in \mathcal{T}_{\rho_{\theta_t}}\mathcal{P}$ onto the subspace $\mathcal{T}_{\rho_{\theta_t}}\mathcal{P}_\Theta$ with respect to the Wasserstein metric $g^W$.
\end{customthm}
Theorem \ref{theorem_submfld} easily follows from the following two general results about manifold gradient. 
\begin{theorem}\label{general_A}
Suppose $(N,g^N),(M,g^M)$ are Riemannian Manifolds. Suppose $\varphi:N\rightarrow M$ is isometric. Consider $\mathcal{F}\in \mathcal{C}^\infty(M)$, define $F=\mathcal{F}\circ\varphi\in\mathcal{C}^\infty(N)$.
Suppose $\{x_t\}_{t\geq 0}$ is the gradient flow of $F$ on $N$:
\begin{equation*}
\dot{x}=-\textrm{grad}_N F(x).
\end{equation*}
Then $\{y_t=\varphi(x_t)\}_{t\geq 0}$ is the gradient flow of $\mathcal{F}$ on $M$. That is, $\{y_t\}$ satisfies $\dot y = -\textrm{grad}_M\mathcal{F}(y)$.
\end{theorem}
\begin{proof}
Since we always have $\dot y_t = \varphi_* \dot x_t = -\varphi_*\textrm{grad}_N F(x_t)$, we only need to show that $\varphi_*\textrm{grad}_N F(x_t)=\textrm{grad}_{M}\mathcal{F}(\varphi(x_t))$. Fix the time $t$, consider any curve $\{\xi_\tau\}$ on $N$ passing through $x_t$ at $\tau = 0$,
since $\varphi$ is isometry, we have $g^N = \varphi^*g^M$, thus:
\begin{equation*}
    \frac{d}{d\tau} F(\xi_\tau)\Big\vert_{\tau = 0} = g^N(\textrm{grad}_NF(x_t),\dot\xi_0)= \varphi^*g^M(\textrm{grad}_NF(x_t),\dot\xi_0) = g^M(\varphi_*\textrm{grad}_NF(x_t),\varphi_*\dot\xi_0).
\end{equation*}
On the other hand, denote $\eta_\tau = \varphi(\xi_\tau)$, we have:
\begin{equation*}
\frac{d}{d\tau}F(\xi_\tau)\Big\vert_{\tau=0} = \frac{d}{d\tau}\mathcal{F}(\eta_\tau)\Big\vert_{\tau=0}=g^M(\textrm{grad}_M\mathcal{F}(y_t),\dot \eta_0)=g^M(\textrm{grad}_M\mathcal{F}(y_t),\varphi_* \dot \xi_0).
\end{equation*}
As a result, $g^M(\varphi_*\textrm{grad}_NF(x_t)-\textrm{grad}_M\mathcal{F}(y_t),\varphi_*\dot\xi_0)=0$ for all $\dot\xi_0\in T_{x_t}N$.
Since $\varphi_*$ is surjective, we have $\varphi_*\textrm{grad}_N F(x_t)=\textrm{grad}_{M}\mathcal{F}(\varphi(x_t))$.

\end{proof}

\begin{theorem}\label{general_B}
Suppose $(M,g^M)$ is Riemannian manifold, $M_{\textrm{sub}}\subset M$ is the submanifold of $M$. Assume $M_{\textrm{sub}}$ inherits metric $g^M$, i.e. define $\iota:M_{\textrm{sub}}\rightarrow M$ as the inclusion map, which induces a metric tensor on $M_{\textrm{sub}}$ as $g^{M_{\textrm{sub}}}=\iota^*g^M$.
\noindent
For any $\mathcal{F}\in \mathcal{C}^\infty(M)$, denote the restriction of $\mathcal{F}$ on $M_{\textrm{sub}}$ as $\mathcal{F}^{\textrm{sub}}$. Then the gradient $\textrm{grad}_{M_\textrm{sub}}\mathcal{F}^{\textrm{sub}}(x)\in T_x M_\textrm{sub}$ is the orthogonal projection of $\textrm{grad}_M\mathcal{F}(x)\in T_x M$ onto subspace $T_x M_\textrm{sub}$ with respect to the metric $g^M$ for any $x\in M_\textrm{sub}$. 
\end{theorem}

\begin{proof}
For any $x\in M_\textrm{sub}$, consider any curve $\{\gamma_\tau\}$ on $M_\textrm{ sub }$ passing through $x$ at $\tau=0$. We have
\begin{equation*}
  \frac{d}{d\tau}\mathcal{F}^{\textrm{sub}}(\gamma_\tau)\Big\vert_{\tau=0} =  g^{M_\textrm{sub}}(\textrm{grad}_{M_\textrm{sub}}\mathcal{F}^{\textrm{sub}}(x),\dot\gamma_0)=g^M(\iota_*\textrm{grad}_{M_\textrm{sub}}\mathcal{F}^{\textrm{sub}}(x),\iota_*\dot\gamma_0)=g^M(\textrm{grad}_{M_\textrm{sub}}\mathcal{F}^{\textrm{sub}}(x),\dot\gamma_0).
\end{equation*}
The last equality is because $\iota_*$ restricted on $TM_\textrm{sub}$ is identity. On the other hand, $\mathcal{F}^{\textrm{sub}}(\gamma_\tau)=\mathcal{F}(\gamma_\tau)$ for all $\tau$. We also have:
\begin{equation*}
    \frac{d}{d\tau}\mathcal{F}^{\textrm{sub}}(\gamma_\tau)\Big\vert_{\tau = 0}=g^M(\textrm{grad}_M\mathcal{F}(x) , \dot\gamma_0).
\end{equation*}
Combining them we know
\begin{equation*}
    g^M(\textrm{grad}_{M_\textrm{sub}}\mathcal{F}^{\textrm{sub}}(x)-\textrm{grad}_M\mathcal{F}(x), v) = 0 \quad \forall ~ v\in T_x M_\textrm{sub}\Rightarrow \textrm{grad}_{M_\textrm{sub}}\mathcal{F}^{\textrm{sub}}(x)-\textrm{grad}_M\mathcal{F}(x)\perp_{g^M} T_x M_\textrm{sub},
\end{equation*}
which proves this result.
\end{proof}
\begin{proof}(Theorem \ref{theorem_submfld})
To prove the first part of Theorem \ref{theorem_submfld}, we apply Theorem \ref{general_A} with $(N,g^N)=(\Theta,G)$, $M=\mathcal{P}_\Theta$ with its metric inherited from $(\mathcal{P},g^W)$ and $\varphi=T_{(\cdot)\sharp}$. To prove the second part, we apply Theorem \ref{general_B} with $(M,g^M)=(\mathcal{P},g^W)$, $M_{\textrm{sub}}=\mathcal{P}_\Theta$.
\end{proof}

\section{Proof of Lemma \ref{lemma_basics} \ref{thm:Danskin} and \ref{prior est |theta_k+1-theta_k|}}\label{pf section 4}
\begin{customlemma}{4.6}
  Suppose we fix $\theta_0\in \Theta$, for arbitrary $\theta\in\Theta$ and $\nabla\phi\in L^2(\mathbb{R}^d;\mathbb{R}^d,\rho_{\theta_0})$ we consider
    \begin{equation}
     F(\theta, \nabla\phi~|~\theta_0) = \left( \int (2\nabla\phi(x)\cdot(T_\theta-T_{\theta_0})\circ T_{\theta_0}^{-1}(x) - |\nabla\phi(x)|^2) ~\rho_{\theta_0}(x) ~dx \right) + 2h  H(\theta).\tag{\ref{F_theta_phi}}
    \end{equation}
    Then $F(\theta,\nabla\phi~|~\theta_0)<\infty$, furthermore,  $F(\cdot, \nabla\phi~|~\theta_0)\in C^1(\Theta)$. We can compute
    \begin{equation}
      \partial_\theta F(\theta,\nabla\phi~|~\theta_0) =2\left(\int ~ \partial_\theta T_\theta(T_{\theta_0}^{-1}(x))^{\textrm{T}}~\nabla\phi(x)~\rho_{\theta_0}(x)~dx + h~\nabla_\theta H(\theta)\right).  \tag{\ref{compute_partial_theta F}} 
    \end{equation}
\end{customlemma}

\begin{proof}
 To show $F(\theta,\nabla\phi~|\theta_0)<\infty$, we write
 \begin{equation*}
   F(\theta,\nabla~|\theta_0) = \underbrace{\int 2\nabla\phi\cdot T_\theta(T_{\theta_0}^{-1}(x))\rho_{\theta_0}~dx}_{\text{$A$}}-\underbrace{\int2\nabla\phi(T_{\theta_0}(x))\cdot xdp(x)}_{\text{$B$}} - \underbrace{\int|\nabla\phi(x)|^2\rho_{\theta_0}(x)~dx}_{\text{$C$}} + 2hH(\theta).
 \end{equation*}
 By Cauchy–Schwarz inequality, the first two terms can be estimated as
 \begin{equation*}
    |A-B|\leq 2\|\nabla\phi\|_{L^2(\rho_{\theta_0})}\left(\int|T_\theta(x)|^2dp(x) +  \int x^2dp(x)\right).
 \end{equation*}
 Recall \eqref{Theta_condition_A} and $p$ having finite second order moment, we know the first two terms are finite. In addition $C=\|\nabla\phi\|_{L^2(\rho_{\theta_0})}^2<\infty$. We thus have shown $F(\theta,\nabla\phi~|\theta_0)<\infty$. 
 
 To show $F(\cdot,\nabla\phi~|\theta_0)\in C^1(\Theta)$, recall $T_\theta(x)\in C^2(\Theta\times\mathbb{R}^d)$ as mentioned in \ref{3.1}. We know the relative entropy $H(\cdot)\in C^1(\Theta)$, thus we only need to prove for $\tilde{F}(\cdot,\nabla\phi~|\theta_0)=F(\cdot,\nabla\phi~|\theta_0)-2hH(\theta)$. We consider $\xi\in\mathbb{R}^m$ with $|\xi|$ small enough and $\theta+\xi\in \Theta$. Then the difference
 \begin{equation}
   \tilde{F}(\theta+\xi,\nabla\phi~|\theta_0) - \tilde{F}(\theta,\nabla\phi~|\theta_0) = \int 2\nabla\phi(x)\cdot(T_{\theta+\xi}-T_\theta)\circ T_{\theta_0}^{-1}(x) ~\rho_{\theta_0}(x) ~dx \label{lemmabasicF eq 1}
 \end{equation}
 We denote the $i$th component of $T_\theta$ as $T_\theta^{(i)}$, $1\leq i\leq d$. By Taylor expansion (w.r.t. $\theta$), we have $T^{(i)}_{\theta+\xi}(x)-T^{(i)}_\theta(x)=\partial_\theta T_\theta^{(i)}(x)^{\textrm{T}}\xi+\frac{1}{2}\xi^{\textrm{T}}\partial^2_{\theta\theta}T_{\theta+\lambda_i(x)\xi}(x)\xi$ with $\lambda_i(x)\in [0,1],$ then the right hand side of \eqref{lemmabasicF eq 1} is
 \begin{equation}
     \underbrace{\left(\int 2\partial_\theta T_\theta(T_{\theta_0}^{-1}(x))^{\textrm{T}}\nabla\phi(x)\rho_{\theta_0}~dx\right)^{\textrm{T}}\xi}_{\text{Denote as $\mathcal{J}(\theta)^{\textrm{T}} \xi  $}} + \int \left(\sum_{i=1}^d \partial_{x_i} \phi\cdot(\xi^{\textrm{T}}\partial^2_{\theta\theta}T^{(i)}_{\theta+\lambda_i(x)\xi}(T_{\theta_0}^{-1}(x)) \xi) \right)\rho_{\theta_0}~dx  \label{lemmabasicF eq 2}
 \end{equation}
 By Cauchy-Schwarz inequality, the sum in the second term of \eqref{lemmabasicF eq 2} can be estimated as
 \begin{equation*}
    \left(\sum_{i=1}^d |\partial_{x_i}\phi|^2\right)^{\frac{1}{2}}\cdot\left(\sum_{i=1}^d|\xi^{\textrm{T}}\partial^2_{\theta\theta}T_{\theta+\lambda_i(x)\xi}^{(i)}(T_{\theta_0}^{-1}(x)) \xi|^2\right)^{\frac{1}{2}}\leq |\nabla\phi|\cdot\left(\sum_{i=1}^d \|\partial^2_{\theta\theta}T_{\theta+\lambda_i(x)\xi}^{(i)}(T_{\theta_0}^{-1}(x))\|_2^2\right)^{\frac{1}{2}}|\xi|^2
 \end{equation*}
 Let us recall \eqref{notation L,H} and the absolute value of the second term in \eqref{lemmabasicF eq 2} can be upper bounded by
 \begin{equation*}
    \left(\int|\nabla\phi|^2\rho_{\theta_0}~dx\right)^{\frac{1}{2}}\cdot\left(\int \sum_{i=1}^d \|\partial^2_{\theta\theta}T_{\theta+\lambda_i(x)\xi}^{(i)}(x)\|_2^2 dp(x) \right)^{\frac{1}{2}}|\xi|^2 \leq \|\nabla\phi\|^2_{L^2(\rho_{\theta_0})}\cdot \sqrt{H(\theta_0,|\xi|)}|\xi|^2.
 \end{equation*}
 This inequality is due to \eqref{notation L,H}. As a result, we have
 \begin{equation}
  \frac{|\tilde{F}(\theta+\xi,\nabla\phi~|\theta_0) - \tilde{F}(\theta,\nabla~|\theta_0) - \mathcal{J}(\theta)^{\textrm{T}}\xi|}{|\xi|}\leq \|\nabla\phi\|^2_{L^2(\rho_{\theta_0})}\cdot \sqrt{H(\theta_0,|\xi|)}~|\xi|. \label{lemmabasicF eq  3}
 \end{equation}
 Since $H(\theta_0,\epsilon)$ is increasing w.r.t. $\epsilon$, when we send $|\xi|\rightarrow 0$, the upper bound in \eqref{lemmabasicF eq  3} approaches to $0$. This verifies the differentiability of $\tilde{F}(\cdot,\nabla\phi~|\theta_0)$.Thus $F(\cdot,\nabla\phi~|\theta_0)$ is also differentiable and $\partial_\theta F(\theta,\nabla\phi~|\theta_0)=\mathcal{J}(\theta)+2h\nabla_\theta H(\theta)$. At last, to show that $F(\cdot,\nabla\phi~|\theta_0)\in C^1(\Theta)$, we only need to prove the continuity of $\mathcal{J}(\theta)$. One only need to notice that
 \begin{equation*}
 2\partial_\theta T^{(i)}_{\theta'}(T_{\theta_0}^{-1}(x))^{\textrm{T}}\nabla\phi(x) \leq |\partial_{\theta'} T^{(i)}_\theta(T_{\theta_0}^{-1}(x)) |^2+|\nabla\phi(x)|^2\leq L_2(T^{-1}_{\theta_0}(x)|\theta) + |\nabla\phi(x)|^2 
 \quad \forall~\theta',~|\theta'-\theta|<r(\theta).
 \end{equation*}
 The last inequality is due to condition \eqref{Theta_condition_B}. Since $ L_2(T^{-1}_{\theta_0}(x)|\theta) + |\nabla\phi(x)|^2 \in L^1(\rho_{\theta_0})$, then by dominated convergence theorem, we are able to prove the continuity of $\partial_\theta F(\theta, \nabla\phi~|\theta_0)$.
\end{proof}

\begin{customlemma}{4.7}
    Suppose we fix $\theta_0\in \Theta$ and define $J(\theta) = \underset{\nabla\phi\in L^2(\mathbb{R}^d;\mathbb{R}^d,\rho_{\theta_0})}{\sup} F(\theta,\nabla\phi~|~\theta_0) $. Then $J$ is differentiable. If we denote $\hat{\psi}_\theta = \underset{\phi}{\textrm{argmax}} \left\{ F(\theta,\nabla\phi~|~\theta_0) \right\}$, then
    \begin{equation*}
      \nabla_\theta J(\theta) = \partial_\theta F(\theta,\nabla\hat{\psi}_\theta~|~\theta_0) = 2\left(\int ~ \partial_\theta T_\theta(T_{\theta_0}^{-1}(x))^{\textrm{T}}~\nabla\hat{\psi}_\theta(x)~\rho_{\theta_0}(x)~dx + h~\nabla_\theta H(\theta)\right).
    \end{equation*}
\end{customlemma}

\begin{proof}
  Let us denote $\Xi_\theta = (T_\theta-T_{\theta_0})\circ T_{\theta_0}^{-1}$. Then for any $\xi\in\mathbb{R}^m$ such that $\theta+\xi\in\Theta$, we set $\hat{\psi}_{\theta+\xi}=\underset{\phi}{\textrm{argmax}}\{F(\theta+\xi, \nabla\phi~|~\theta_0)\}$. Then according to Definition \ref{Hodge Decomp}, $\hat{\psi}_\theta, \hat{\psi}_{\theta+\xi}$ solves
  \begin{equation}
      -\nabla\cdot(\rho_{\theta_0}\nabla\hat{\psi}_\theta) =  -\nabla\cdot(\rho_{\theta_0}\Xi_{\theta}) \quad -\nabla\cdot(\rho_{\theta_0}\nabla\hat{\psi}_{\theta+\xi}) =  -\nabla\cdot(\rho_{\theta_0}\Xi_{\theta+\xi}).  \label{lemma thm danskins eq  1}
  \end{equation}
  Subtracting the two equations, then multiply $\hat{\psi}_{\theta+\xi}-\hat{\psi}_\theta$ on both sides and integrate yields
  \begin{equation*}
     \int |\nabla\hat{\psi}_{\theta+\xi}-\nabla\hat{\psi}_\theta|^2\rho_{\theta_0} ~dx = \int (\nabla\hat{\psi}_{\theta+\xi} - \nabla\hat{\psi}_\theta)\cdot(\Xi_{\theta+\xi} - \Xi_\theta)\rho_{\theta_0}~dx.
  \end{equation*}
  Then by Cauchy–Schwarz inequality, we derive
  \begin{equation*}
    \int |\nabla\hat{\psi}_{\theta+\xi}-\nabla\hat{\psi}_\theta|^2\rho_{\theta_0} ~dx \leq \int |\Xi_{\theta+\xi} - \Xi_\theta|^2\rho_{\theta_0}~dx.
  \end{equation*}
  Now since $\Xi_{\theta_\xi}(x)-\Xi_{\theta}(x) = (T_{\theta+\xi}-T_{\theta})\circ T_{\theta_0}^{-1}(x)$, by mean value theorem, the $i$th component of $\Xi_{\theta+\xi}(x)-\Xi_\theta(x)$ can be written as $\partial_\theta T^{(i)}_{\theta+\lambda_i(x)\xi}(T_{\theta_0}^{-1}(x))^{\textrm{T}}\xi$ with $\lambda_i(x)\in [0,1]$. Then recall the definition of $L(\theta,\epsilon)$ in \eqref{notation L,H}, we can verify 
  \begin{equation*}
    \int |\Xi_{\theta+\xi}-\Xi_{\theta}|^2\rho_{\theta_0}~dx = \int |T_{\theta+\xi}(x)-T_\theta(x)|dp(x) \leq L(\theta, |\xi|)|\xi|^2.
  \end{equation*}
  Thus we have the following estimation
  \begin{equation}
    \int |\nabla\hat{\psi}_{\theta+\xi}-\nabla\hat{\psi}_\theta|^2\rho_{\theta_0}~dx\leq L(\theta, |\xi|)|\xi|^2  \label{est of difference on nabla psi}
  \end{equation}
  Now let us consider $J(\theta+\xi)-J(\theta)$
  \begin{align}
     J(\theta+\xi) - J(\theta) & = F(\theta+\xi, \nabla\hat{\psi}_{\theta+\xi}~|~\theta_0) - F(\theta, \nabla\hat{\psi}_{\theta}~|~\theta_0) \nonumber \\
     & = \underbrace{F(\theta+\xi, \nabla\hat{\psi}_{\theta+\xi}~|~\theta_0) - F(\theta, \nabla\hat{\psi}_{\theta+\xi}~|~\theta_0)}_{\text{$A$}} + \underbrace{F(\theta, \nabla\hat{\psi}_{\theta+\xi}~|~\theta_0) - F(\theta, \nabla\hat{\psi}_\theta~|~\theta_0)}_{\text{$B$}}. \label{thm danskins J(theta+xi)-J(theta)} 
  \end{align}
  Now according to Lemma \ref{lemma_basics}, $F(\cdot,\nabla\phi~|~\theta_k)\in C^1(\Theta)$. By mean value theorem, term $A$ can be written as
  \begin{align}
  A = & F(\theta+\xi, \nabla\hat{\psi}_{\theta+\xi}~|~\theta_0) - F(\theta, \nabla\hat{\psi}_{\theta+\xi}~|~\theta_0) = \partial_\theta F(\theta+\tau\xi,\nabla\hat{\psi}_{\theta+\xi}~|~\theta_0)\xi \quad\quad \textrm{with}~ \tau\in[0,1]  \nonumber \\
    = & \partial_\theta F(\theta,\nabla\hat{\psi}_\theta~|~\theta_0)^{\textrm{T}}\xi+ (\underbrace{\partial_\theta F(\theta+\tau\xi,\nabla\hat{\psi}_\theta~|~\theta_0) - \partial_\theta F(\theta,\nabla\hat{\psi}_\theta~|~\theta_0)}_{\text{$r_1(\theta,\xi)$}})^{\textrm{T}}\xi \nonumber 
    \\
    & + (\underbrace{\partial_\theta F(\theta+\tau\xi,\nabla\hat{\psi}_{\theta+\xi}~|~\theta_0)-\partial_\theta F(\theta+\tau\xi,\nabla\hat{\psi}_\theta~|~\theta_0)}_{\text{$r_2(\theta,\xi)$}})^{\textrm{T}}\xi. \nonumber
  \end{align}
  
 Term $B$ can be computed as
  \begin{align}
     B & = F(\theta,\nabla \hat{\psi}_{\theta+\xi}~|~\theta_0) - F(\theta,\nabla\hat{\psi}_\theta~|~\theta_0) = \int (2(\nabla\hat{\psi}_{\theta+\xi}-\nabla\hat{\psi}_\theta)\cdot\Xi_\theta - (|\nabla\hat{\psi}_{\theta+\xi}|^2-|\nabla\hat{\psi}_\theta|^2))\rho_{\theta_0}~dx \nonumber\\
      & = 2\int (\nabla\hat{\psi}_{\theta+\xi}-\nabla\hat{\psi}_\theta)\cdot(\Xi_\theta - \nabla\hat{\psi}_\theta)\rho_{\theta_0}~dx - \int |\nabla\hat{\psi}_{\theta+\xi} - \nabla\hat{\psi}_\theta|^2\rho_{\theta_0}~dx = - \int |\nabla\hat{\psi}_{\theta+\xi} - \nabla\hat{\psi}_\theta|^2\rho_{\theta_0}~dx. \nonumber
  \end{align}
  The last equality is due to integration by parts and \eqref{lemma thm danskins eq  1}. 
  
  Now substituting $A$ and $B$ in \eqref{thm danskins J(theta+xi)-J(theta)}  yields
  \begin{equation*}
     J(\theta+\xi) - J(\theta) = \partial_\theta F(\theta,\nabla\hat{\psi}_\theta~|~\theta_0) + r_1(\theta,\xi)^{\textrm{T}}\xi + r_2(\theta,\xi)^{\textrm{T}}\xi -\|\nabla\hat{\psi}_{\theta+\xi} - \nabla\hat{\psi}_\theta\|_{L^2(\rho_{\theta_0})}^2 
  \end{equation*}
  We can estimate
  \begin{equation}
    \frac{\left|J(\theta+\xi)-J(\theta)-\partial_\theta F(\theta,\nabla\hat{\psi}_\theta~|~\theta_0)^{\textrm{T}}\xi\right|}{|\xi|} \leq |r_1(\theta,\xi)|+|r_2(\theta,\xi)| + \frac{1}{|\xi|}\|\nabla\hat{\psi}_{\theta+\xi} - \nabla\hat{\psi}_\theta\|_{L^2(\rho_{\theta_0})}^2  \label{est on pf differentiable J(theta)}
  \end{equation}
  Now we prove the right hand side of \eqref{est on pf differentiable J(theta)} approaches to $0$ as $\xi\rightarrow 0$. Since $\partial_\theta F(\cdot,\nabla\hat{\psi}_\theta~|~\theta_0)\in C^1(\Theta)$, using continuity, we know $\lim_{\xi\rightarrow 0}r_1(\theta,\xi)=0$. For $r_2(\theta,\xi)$, when $|\xi|$ is sufficiently small, we have
  \begin{align*}
     |r_2(\theta,\xi)| =& \left|\int \partial_\theta T_{\theta+\tau\xi}(T_{\theta_0}^{-1}(x))^{\textrm{T}}(\nabla\hat{\psi}_{\theta+\xi}(x)-\nabla\hat{\psi}_{\theta}(x))\rho_{\theta_0}(x)~dx\right|   \\
     \leq & \left( \int \|\partial_\theta T_{\theta+\tau\xi}(x)\|_F^2dp(x) \right)^{\frac{1}{2}}\left(\int |\nabla\hat{\psi}_{\theta+\xi}-\nabla\hat{\psi}_\theta|^2\rho_{\theta_0}~dx\right)^{\frac{1}{2}} \leq \sqrt{\|L_2(\cdot|\theta)\|_{L^1(p)}}\sqrt{L(\theta,|\xi|)}|\xi|
  \end{align*}
  The last inequality is due to \eqref{Theta_condition_B} (when $|\xi|$ is small enough so that $|\xi|<r(\theta)$) and \eqref{est of difference on nabla psi}. Using this we are able to show $\lim_{\xi\rightarrow 0}r_2(\theta,\xi)=0$. Using \eqref{est of difference on nabla psi} again, we can verify $\frac{1}{|\xi|}\|\nabla\hat{\psi}_{\theta+\xi} - \nabla\hat{\psi}_\theta\|_{L^2(\rho_{\theta_0})}^2 \leq L(\theta, |\xi|)|\xi|\rightarrow 0$ as $\xi\rightarrow 0$. Thus $J$ is differentiable at $\theta$ and we know $\nabla_\theta J(\theta) = \partial_\theta F(\theta,\nabla\hat{\psi}_\theta~|~\theta_0)$. We complete the proof by applying \eqref{compute_partial_theta F} of Lemma \eqref{lemma_basics}.

\end{proof}

\begin{customlemma}{4.8}
  Under assumption\eqref{assumption on T_theta positive def}, the optimal solution of \eqref{JKO_A} $\theta_{k+1}$ satisfies, 
  \begin{equation*}
    |\theta_{k+1}-\theta_k|\sim o(1)\quad \textrm{i.e.,}~ ~ \lim_{h\rightarrow 0^+} |\theta_{k+1}-\theta_k|=0.
  \end{equation*}
\end{customlemma}

\begin{proof}[Proof of Lemma \ref{prior est |theta_k+1-theta_k|}]
Recall the function to be minimized in \eqref{JKO_A} is $J(\theta) = \widehat{W}_2^2(\theta,\theta_k)+2hH(\theta)$. If choosing $\theta=\theta_k$ in \eqref{JKO_A}, we have $J(\theta_k)=2hH(\theta_k)$. Thus $J(\theta_{k+1})\leq J(\theta_k)=2hH(\theta_k)$. Since $H(\theta_k)\geq 0$, this leads to $\widehat{W}_2^2(\theta_{k+1},\theta_k)\leq 2hH(\theta_k)$. When $h$ is small enough, $|\theta_{k+1}-\theta_k|\leq l^{-1}(2hH(\theta_k))$, here $l^{-1}$ is the inverse function of $l$ defined on $[0,l(r_0)]$. We know $l^{-1}(0)=0$ and $l^{-1}$ is also continuous and increasing function. This leads to $\lim_{h\rightarrow 0^+}|\theta_{k+1}-\theta_k|\leq\lim_{h\rightarrow 0^+} l^{-1}(2hH(\theta_k))=0$.
\end{proof}

\section{Proofs for Lemma \ref{constant speed of geodesic on P(M)} and \ref{Displacement_convex_entropy}}\label{pf section5}

\begin{customlemma}{5.7}
The geodesic connecting $\rho_0,\rho_1\in\mathcal{P}(M)$ is described by,
\begin{equation}
  \begin{cases}
  \frac{\partial\rho_t}{\partial t}+\nabla\cdot(\rho_t\nabla\psi_t)=0 \\
  \frac{\partial\psi_t}{\partial t}+\frac{1}{2}|\nabla\psi_t|^2=0
  \end{cases} \quad \rho_t|_{t=0}=\rho_0,~\rho_t|_{t=1}=\rho_1.   \tag{\ref{geodesic eq}}
\end{equation}
Using the notation $ \dot \rho_t=\partial_t\rho_t=-\nabla\cdot(\rho_t\nabla\psi_t)\in \mathcal{T}_{\rho_t}\mathcal{P}(M)$, $g^W( \dot \rho_t,  \dot \rho_t)$ is constant for $0\leq t\leq 1$ and $ g^W( \dot \rho_t, \dot \rho_t) = W_2^2(\rho_0,\rho_1) $ for $0\leq t\leq 1$.
\end{customlemma}

\begin{proof}
Recall the definition \eqref{def_metric} of Wasserstein metric $g^W$, $g^W( \dot \rho_t, \dot \rho_t)=\int |\nabla\psi_t|^2\rho_t~dx$. Since $\{\rho_t\}$ is the geodesic on $(\mathcal{P}(M),g^W)$, the speed $g^W(\sigma_t,\sigma_t)$ remains constant. To directly verify this, we compute the time derivative:
\begin{equation*}
  \frac{d}{dt} g^W( \dot \rho_t, \dot \rho_t) = \frac{d}{dt} \left(\int |\nabla\psi_t|^2 \rho_t~dx \right) = \int \frac{\partial}{\partial t}|\nabla\psi_t|^2\rho_t~dx + \int |\nabla\psi_t|^2\partial_t\rho_t~dx.
\end{equation*}
Using the first equation in \eqref{geodesic eq}, we obtain
\begin{equation*}
  \int |\nabla\psi_t|^2\partial_t\rho_t~dx = \int |\nabla\psi_t|^2\cdot(-\nabla\cdot(\rho_t\nabla\psi_t))~dx = \int \nabla(|\nabla\psi_t|^2)\cdot\nabla\psi_t\rho_t~dx,
\end{equation*}
Taking the spatial gradient of the second equation in \eqref{geodesic eq}, we have 
\begin{equation*}
  \partial_t(\nabla\psi_t) = -\nabla(\frac{1}{2}|\nabla\psi_t|^2).
\end{equation*}
Then 
\begin{equation*}
  \int \frac{\partial}{\partial t} |\nabla\psi_t|^2\rho_t~dx = \int 2\partial_t(\nabla\psi_t)\cdot\nabla\psi_t\rho_t~dx = \int -\nabla(|\nabla\psi_t|^2)\cdot\nabla\psi_t\rho_t~dx.
\end{equation*}
Adding them together, we verify $\frac{d}{dt}g^W( \dot \rho_t, \dot \rho_t)=0$, hence $\int_0^1 g^W( \dot \rho_t, \dot \rho_t)~dt=W_2^2(\rho_0,\rho_1)$. Thus we know $g^W( \dot \rho_t, \dot \rho_t)=W_2^2(\rho_0,\rho_1)$ for any $0\leq t\leq 1$.
\end{proof}

\begin{customlemma}{5.8}
  Suppose $\{\rho_t\}$ solves \eqref{geodesic eq}, the relative entropy $\mathcal{H}$ in \eqref{relative entropy} has potential $V$ satisfying $\nabla^2 V\succeq \lambda I$, then we have $\frac{d}{dt}g^W(\textrm{grad}_W\mathcal{H}(\rho_t),\dot\rho_t)\geq \lambda W_2^2(\rho_0,\rho_1)$. Or equivalently,  $\frac{d^2}{dt^2}\mathcal{H}(\rho_t)\geq \lambda W_2^2(\rho_0,\rho_1)$.
\end{customlemma}

\begin{proof}
Let us write:
\begin{equation*}
 g^W(\textrm{grad}_W\mathcal{H}(\rho_t),\dot\rho_t) =  \int\nabla(V+{D}\log\rho_t)\cdot\nabla\psi_t~\rho_t~dx .
\end{equation*}
Then:
\begin{equation*} 
  \frac{d}{dt}g^W(\textrm{grad}_W\mathcal{H}(\rho_t),\dot\rho_t) = \frac{d}{dt}\left( \int \nabla(V+{D}\log\rho_t)\cdot\nabla\psi_t~\rho_t~dx  \right) = \int (\nabla \psi_t^{\textrm{T}}\nabla^2V \nabla\psi_t + \textrm{Tr}( \nabla^2 \psi_t \nabla^2 \psi_t )) ~ \rho_t ~dx.
\end{equation*}
The second equality can be carried out by direct calculations.
One can check \cite{villani2003topics} or \cite{villani2008optimal} for its complete derivation. 
Using $\nabla^2V\succeq\lambda I$, we get
\begin{equation*}
  \frac{d}{dt} g^W(\textrm{grad}_W\mathcal{H}(\rho_t),\dot\rho_t) \geq \int \lambda |\nabla\psi_t|^2\rho_t~dx=\lambda ~ g^W(\dot\rho_t,\dot\rho_t)=\lambda W_2^2(\rho_0,\rho_1).
\end{equation*}
The last equality is due to Lemma \ref{constant speed of geodesic on P(M)}.
By the definition of Wasserstein gradient \eqref{gradflow}, we have $\frac{d}{dt}\mathcal{H}(\rho_t)=g^W(\textrm{grad}_W\mathcal{H}(\rho_t),\dot\rho_t)$, we also proved $\frac{d^2}{dt^2}\mathcal{H}(\rho_t)\geq \lambda W_2^2(\rho_0,\rho_1)$.
\end{proof}

\bibliographystyle{siamplain}
\bibliography{references}
\end{document}